%% file: main.tex
\documentclass[10pt, oneside]{amsart} 
\usepackage[lmargin=1in,rmargin=1in,
bmargin=1in, tmargin=1in]{geometry}
\usepackage[dvipsnames]{xcolor}
\usepackage{amsmath, amssymb,tikz}
\usepackage{tikz-cd}
\usepackage{tikz}
\usepackage{todonotes}
\usepackage{mathrsfs}
\usepackage{extarrows}
\usepackage{graphicx}
\usepackage[breaklinks, pagebackref]{hyperref}
\usepackage{IEEEtrantools}
\usepackage{mathrsfs}
\usepackage[utf8]{inputenc}
\usepackage[english]{babel}
\usepackage{comment}
\usepackage{csquotes}
\usepackage[shortlabels]{enumitem}

\usepackage{etoolbox,refcount}
\usepackage{multicol}
\usepackage[shortlabels]{enumitem}
\usepackage{nicematrix}  
\usepackage{tikz-cd}
\usepackage{tikz}
\usetikzlibrary{decorations.pathreplacing,matrix,calc,positioning}
\usepackage{nicematrix}
\tikzset{ 
    table/.style={
        matrix of math nodes,
        row sep=-\pgflinewidth,
        column sep=-\pgflinewidth,
        nodes={rectangle,text width=3em,align=center},
        text depth=1.25ex,
        text height=2.5ex,
        nodes in empty cells,
        left delimiter=[,
        right delimiter={]},
        ampersand replacement=\&
    }
}
\usetikzlibrary{decorations.pathreplacing, calligraphy}

\def\mmath#1{\text{\scalebox{1.1}{$#1$}}}
\def\smath#1{\text{\scalebox{0.9}{$#1$}}}

\def\mfrac#1#2{\mmath{\frac{#1}{#2}}}
\def\sfrac#1#2{\smath{\frac{#1}{#2}}}

\usepackage[dvipsnames]{xcolor}
\hypersetup{
    colorlinks=true,
    linkcolor=blue,
    filecolor=pink, 
    citecolor=red,          
    urlcolor=gray,
    pdftitle={GSp6HNR},
    pdfpagemode=FullScreen,
    }

\usepackage{tikz}

\makeatletter
\newcommand*{\encircled}[1]{\relax\ifmmode\mathpalette\@encircled@math{#1}\else\@encircled{#1}\fi}
\newcommand*{\@encircled@math}[2]{\@encircled{$\m@th#1#2$}}
\newcommand*{\@encircled}[1]{%
  \tikz[baseline,anchor=base]{\node[draw,circle,outer sep=0pt,inner sep=.2ex] {#1};}}
\makeatother

\usepackage{pict2e,picture}

\makeatletter
\newcommand{\pnrelbar}{%
  \linethickness{\dimen2}%
  \sbox\z@{$\m@th\prec$}%
  \dimen@=1.1\ht\z@
  \begin{picture}(\dimen@,.4ex)
  \roundcap
  \put(0,.2ex){\line(1,0){\dimen@}}
  \put(\dimexpr 0.5\dimen@-.2ex\relax,0){\line(1,1){.4ex}}
  \end{picture}%
}

\newcommand{\precneq}{\mathrel{\vcenter{\hbox{\text{\prec@neq}}}}}
\newcommand{\prec@neq}{%
  \dimen2=\f@size\dimexpr.04pt\relax
  \oalign{%
    \noalign{\kern\dimexpr.2ex-.5\dimen2\relax}
    $\m@th\prec$\cr
    \noalign{\kern-.5\dimen2}
    \hidewidth\pnrelbar\hidewidth\cr
  }%
}
\makeatother

\DeclareFontEncoding{LS1}{}{}
\DeclareFontSubstitution{LS1}{stix}{m}{n}
\newcommand*\kay{%
  \text{%
  \fontencoding{LS1}%
  \fontfamily{stixscr}%
  \fontseries{\textmathversion}%
  \fontshape{n}%
  \selectfont\symbol{"6B}}}
\makeatletter
  \newcommand*\textmathversion{\csname textmv@\math@version\endcsname}
  \newcommand*\textmv@normal{m}
  \newcommand*\textmv@bold{b}
\makeatother

\DeclareFontEncoding{LS1}{}{}
\DeclareFontSubstitution{LS1}{stix}{m}{n}

\makeatletter
\usepackage{capt-of}

\DeclareFontEncoding{LS1}{}{}
\DeclareFontSubstitution{LS1}{stix}{m}{n}

\makeatletter

\DeclareFontEncoding{LS1}{}{}
\DeclareFontSubstitution{LS1}{stix}{m}{n}

\makeatletter

\DeclareFontEncoding{LS1}{}{}
\DeclareFontSubstitution{LS1}{stix}{m}{n}
\newcommand*\ess{%
  \text{%
  \fontencoding{LS1}%
  \fontfamily{stixscr}%
  \fontseries{\textmathversion}%
  \fontshape{n}%
  \selectfont\symbol{"73}}}
\makeatletter

\usepackage{setspace}\setdisplayskipstretch{}

\usepackage{dynkin-diagrams}  
\input{macros.tex}

\usepackage{setspace}
\usepackage[all]{xy}
\usepackage{textcomp}
\usepackage{amsbsy}
\usepackage{datetime}
\renewcommand{\dateseparator}{-}
\renewcommand{\today}{\the\year \dateseparator \twodigit\month
\dateseparator \twodigit\day}
\title{Horizontal norm compatibility of cohomology classes for  $\mathrm{GSp}_{6}$} 
\author{Syed Waqar Ali Shah}    
\date{}
\begin{document}
\begin{abstract} We establish  abstract  horizontal 
norm relations involving the unramified Hecke-Frobenius  polynomials  that    correspond under the Satake isomorhpism to the degree  eight  spinor $L$-factors  of  $ \mathrm{GSp}_{6} $. These relations    apply  to    classes     in   the  degree seven motivic  cohomology of the Siegel modular sixfold obtained 
via   Gysin   pushforwards of Beilinson's  Eisenstein symbol 
pulled back on one  copy in a triple   product of modular curves. The proof is based on   a novel    approach  that circumvents the failure of the so-called   multiplicity one hypothesis in our setting, which precludes the applicability  of an  existing technique.  In a sequel, we combine our result with  the     previously established vertical  norm  relations for these classes to  obtain   new    Euler systems    for  the    eight dimensional  Galois representations associated with    certain     non-endoscopic  cohomological  cuspidal automorphic representations of $ \mathrm{GSp}_{6} $.   
\end{abstract}    
\maketitle
\tableofcontents
\input{Introduction}

\input{Statement}
\input{DoubleCosets}
\input{Convolutions} 
\bibliographystyle{amsalpha}    
\bibliography{refs}{}  
\Addresses
\end{document}

%% file: macros.tex
\usepackage{amsmath,amsthm,amssymb}
\usepackage{colonequals}
\usepackage{proof-at-the-end}

\newcommand{\QQ}{\mathbb{Q}}

\newcommand{\ZZ}{\mathbb{Z}}

\newcommand{\tH}{H'}
\newcommand{\tfh}{\mathfrak{h}}
  
\newcommand{\tU}{U'}

\newcounter{dummypart}



\newcommand{\ch}{\mathrm{ch}}

\newcommand{\Ab}{\mathbb{A}}   
    
\newcommand{\pr}{\mathrm{pr}}

\DeclareMathOperator{\supp}{Supp}

\DeclareMathOperator{\Oscr}{\mathscr{O}}

\numberwithin{equation}{section}

\newtheorem{theorem}[equation]{Theorem}
\newtheorem{proposition}[equation]{Proposition}
\newtheorem{lemma}[equation]{Lemma}
\newtheorem{corollary}[equation]{Corollary}
\newtheorem*{corollary*}{Corollary}

\newtheorem{theoremx}{Theorem}

\theoremstyle{definition}

\newtheorem*{notation*}{Notation}

\theoremstyle{definition}

\newtheorem{definition}[equation]{Definition}
\theoremstyle{remark}
\newtheorem{remark}[equation]{Remark}
\newtheorem{example}{Example}[section]   

\newtheorem{notation}{Notation}[section] 
\newtheorem{note*}[equation]{Note}

\tikzcdset{scale cd/.style={every label/.append style={scale=#1},
    cells={nodes={scale=#1}}}}
\usepackage{tikz-cd}
\usepackage{mathrsfs}

\DeclareFontFamily{U}{wncy}{}
\DeclareFontShape{U}{wncy}{m}{n}{<->wncyr10}{}
\DeclareSymbolFont{mcy}{U}{wncy}{m}{n}
\DeclareMathSymbol{\sha}{\mathord}{mcy}{"58}

\makeatletter
\renewcommand*\env@matrix[1][\arraystretch]{%
  \edef\arraystretch{#1}%
  \hskip -\arraycolsep
  \let\@ifnextchar\new@ifnextchar
  \array{*\c@MaxMatrixCols c}}
\makeatother

\newcommand{\Yvar}[1]{\mathcal{Y}_{\varepsilon_{#1}}}

\newcommand{\SL}{\mathrm{SL}}

\newcommand{\Gb}{\mathbf{G}}

\newcommand{\Hb}{\mathbf{H}}  
   
\newcommand{\Ht}{H'}

\newcommand{\Addresses}{{% additional braces for segregating \footnotesize
  \bigskip
  \footnotesize \medskip

  (Shah) \textsc{Department of Mathematics, University of California, Santa Barbara, CA 93106-3080}\par\nopagebreak
  \textit{E-mail address}: \texttt{swshah@ucsb.edu}

}}

\usepackage{mathtools,bm}
\providecommand\given{}
\newcommand\givensymbol[1]{%
  \nonscript\;\delimsize#1\allowbreak\nonscript\;\mathopen{}%
}
\DeclarePairedDelimiterX\Set[1]\{\}{%
\renewcommand\given{\givensymbol{\vert}}%
  #1%
}

\usepackage{stmaryrd}

\newcommand{\GG}{\mathbb{G}}
\newcommand{\OO}{\mathcal{O}}
\newcommand{\GL}{\mathrm{GL}}

\newcommand{\RR}{\mathbb{R}}

%% file: Introduction.tex
\section{Introduction}

Ever since the pioneering work of Kolyvagin,  the machinery of Euler systems has become  a standard tool for probing the structure of Selmer groups of global Galois representations  and for    establishing  specific  instances of Bloch-Kato and Iwasawa main conjectures.  Recently, there has been an interest in constructing Euler systems  for Galois representations found in the cohomology of Siegel modular varieties. 
In  \cite{LSZ}, the authors constructed an Euler system for certain four dimensional  Galois representations found in the middle degree cohomology of the $  \mathrm{GSp}_{4} $ Siegel modular variety.      They also introduced a  new  technique of  using local zeta integrals   that  has been applied with great success in many other settings (\cite{GS}, \cite{hsu2020euler}, \cite{gu21}, \cite{Dis}).   

The natural successor  of  $ \mathrm{GSp}_{4} $ in  Euler system based investigations is  the Siegel modular variety attached to  $  \mathrm{GSp}_{6}$. This is a sixfold 
whose middle degree  cohomology realizes the composition of the spin representation with the $\mathrm{GSpin}_{7} $-valued Galois representation associated under  Langlands correspondence with certain 
cohomological  
cuspidal automorphic representations of $ \mathrm{GSp}_{6}$  \cite{Kret},   \cite{Buzz}.     
A standard paradigm for constructing  Euler systems for such geometric  Galois representations is via pushforwards of  a   special family  of  motivic cohomology classes known as  \emph{Eisenstein symbols}. 
A natural candidate  class 
in the  $ \mathrm{GSp}_{6} $  setting    is the  pushforward of the Eisenstein symbol pulled 
back  on one copy in a triple product of modular curves. Besides having the correct numerology, this particular choice of pushforward is  motivated   by a period integral of  Pollack and Shah \cite{PS},  who showed that integrating certain cusp forms of $ \mathrm{GSp}_{6} $ against an Eisenstein series on one copy in a  triple product of $\GL_{2} $  retrieves the degree eight (partial)   spinor  $L$-function  for  that     cusp  form. In \cite{gclj}, the authors  use  this   period integral to relate the regulator of our candidate class in Deligne-Beilinson cohomology to non-critical special values of the spinor $L$-function, thereby providing  
 evidence that it  sits   at the bottom of a non-trivial  
 Euler system  whose behaviour can be explicitly    tied to special  $L$-values.     

To construct an Euler system above this  class,  one needs to produce  classes going up the abelian tower over  $  \QQ $  that   satisfy among themselves  two kinds of norm 
relations. One of these is the \emph{vertical} relations that see variation along the $\ZZ_{p}$-extension and are Iwasawa theoretic in nature.  These have already been verified  in   \cite{AJ} using a general method later  axiomatized in \cite{loe}.  
 The other and typically more challenging kind is  the \emph{horizontal} relations that see variation along ray class extensions and involve local $L$-factors   of the Galois representation.  These present an even  greater  challenge in the $ \mathrm{GSp}_{6} $ case,  since one is dealing with a non-spherical pair of groups and  the  so-called multiplicity one hypothesis on a local space of linear functionals fails to hold. In particular,  the  technique of   local zeta  integrals of \cite{LSZ} and  its variants cannot be applied in this situation to establish horizontal norm compatibility. 
 
The purpose  of  this article is to establish the ideal version of this compatibility   using  a    fairly     general  method developed by us  in   a  companion article   \cite{CZE}, thereby completing the Euler system construction envisioned in \cite{AJ}.    
For convenience and to free up notations that play no role outside the proof of our norm relations, we have chosen to cast our result in the framework of abstract cohomological Mackey (CoMack) functors\footnote{the more relaxed notion of ``Mackey functor" is referred to as a ``cohomology functor" in \cite{loe}}. 
The application to $p$-adic \'{e}tale   cohomology and the actual Euler system construction  is recorded in a sequel \cite{EulerGSp6}. In future, we also expect to establish an explicit reciprocity law relating this Euler system to special values of the spinor $L$-function by means of a $p$-adic $L$-function, thereby making progress on the Bloch-Kato and Iwasawa main conjectures in this setting.      

\subsection{Main result} 
Let $ \Gb = \mathrm{GSp}_{6} $, $ \tilde{\Gb} = \Gb \times \GG_{m} $  and  $ \Hb = \GL_{2} \times_{\GG_{m}} \GL_{2}   \times_{\GG_{m}}  \GL_{2} $ where the products in $ \mathbf{H} $ are  fibered  over the   determinant map.  There is a natural embedding $ \iota : \Hb  \hookrightarrow \Gb $ and if $ \mathrm{sim} : \Gb \to \GG_{m} $ denotes the  similitude map, then post composing $ \iota $ with  $ 1_{\Gb}  \times \mathrm{sim} : \Gb \to \tilde{\Gb} $ gives us an embedding $$  \tilde{\iota} :  \Hb   \hookrightarrow   \tilde{\Gb}  $$  via  which we view $ \Hb $ as a subgroup of $ \tilde{\Gb}  $.  For $ \ell $ a rational prime,  let $ G_{\ell} $    denote the groups of $ \QQ_{\ell} $-points of $ \Gb $ and  
let  $ \mathcal{H}_{R}  $  denote the spherical Hecke algebra of $ G_{\ell}$ with coefficients in a ring $ R $.  For $ c $ an integer, let $ \mathfrak{H}_{\ell,c}(X)  \in \mathcal{H}_{\ZZ[\ell^{-1}]}[X]  $ 
denote the unique polynomial in $ X $ such that for any (irreducible)  unramified representation $ \pi_{\ell} $ of $ G_{\ell} $ and any spherical vector $ \varphi_{\ell} \in \pi_{\ell}   $, $$ \mathfrak{H}_{\ell,c}(\ell^{-s})  \cdot \varphi_{\ell}  = L(s+c, \pi_{\ell}  ,  \mathrm{Spin}) ^{-1} \cdot \varphi_{\ell}  $$  
for all $ s \in \mathbb{C} $. Here  $ L(s,  \pi_{\ell}, \mathrm{Spin}) $ denotes the spinor $ L$-factor  of $ \pi_{\ell} $ normalized as in  \cite{Asgari}. 
Fix  any  finite set $ S $ of rational  primes and let  $ G$,  $ \tilde{G} $, $ H $ denote the group of $ \ZZ_{S} \cdot \Ab_{f}^{S} $-points of $ \Gb $,  $ \tilde{\Gb} $, $ \Hb $ 
respectively. 
Fix  also   a neat compact open subgroup $ K \subset G $ such that $ K $ is unramified at primes away from $ S $.    Let $ \mathcal{N} $ denote the set of all square free  products of primes outside  $ S $  (where the empty product  means  $ 1  $) and for $ n \in  \mathcal{N} $, denote  $$ K[n]  = K \times  \prod_{\ell \nmid n} \ZZ_{\ell}^{\times}  
\prod_{\ell \mid n } ( 1+ \ell \ZZ_{\ell})      \subset  \tilde{G}  .   $$ 
Let $ \mathcal{O} $ be a characteristic zero integral domain such that $ \ell \in \mathcal{O}^{\times} $ for all $ \ell \notin S $. Denote by $ \mathcal{S} =  \mathcal{S}_{\mathcal{O}} $ the  $ \mathcal{O} $-module  of  all locally constant compactly supported  functions  $ \chi  :\mathrm{Mat}_{2 \times 1} ({\Ab_{f}}) \setminus  \left \{ 0 \right \} 
\to \mathcal{O}  $ such that $ \chi  = 
f_{S} \otimes\chi  ^{S }$ where $ f_{S } $ is a fixed  function on $  \mathrm{Mat}_{2 \times 1} (\ZZ_{S})  $ that is invariant under $  \Hb(\ZZ_{S} )   $ under the natural left action of $ H $ on such functions.   We  view the association  $ V \mapsto \mathcal{S}(V) $ that sends a compact open subgroup $ V $ of $ H $ to the $ V$-invariants of $ \mathcal{S}  $ as a CoMack functor for $ H $.   
Let $ U = H \cap K[1] $ and let 
$$ \phi = 
f_{S}  
\otimes \ch(\widehat{\ZZ}^{S}) \in  \mathcal{S}_{
}(U)  $$ where $\widehat{\ZZ}^{S} = \prod_{\ell \notin S } \ZZ_{\ell} $ denotes integral adeles away from $ S $. Finally, let  $ \mathrm{Frob}_{\ell} $ denote $ \ch(\ell \ZZ_{\ell}^{\times}) $.

\begin{theoremx}[Theorem \ref{mainteoglobal}] \label{introB}  For any $\mathcal{O}$-\text{Mod} valued cohomological Mackey functor  $ M_{\tilde{G}} $ for  $  \tilde{G} $, 
any Mackey pushforward   $ \tilde{\iota}_{*} :   \mathcal{S} \to   M_{\tilde{G}} $ and any integer $ c $,   there exists a collection of classes $ y_{n} \in M_{\tilde{G}}(K[n])   $ indexed by $ n \in \mathcal{N} $ such that $ y_{1} = \tilde { \iota } _{U, K[1],*}(\phi) $ and $$ [\mathfrak{H}_{\ell,c}(\mathrm{Frob}_{\ell})]_{*}( y_{n}  )   = \mathrm{pr}_{K[n\ell],  K [n],*}(y_{n\ell})   $$
for all $ n , \ell \in \mathcal{N} $ such that $ \ell $ is a prime and $ \ell   \nmid  n  $.   
\end{theoremx}

Here for a  locally constant compactly supported function $f : \tilde{G} \to \mathcal{O} $, $  [ f  ]_{*} $ denotes the covariant action of $ f $  and $ \pr_{*} $ denotes the trace 
map of the functor $ M_{\tilde{G}} $.   For sufficiently negative $ c $, the Hecke polynomial  $ \mathfrak{H}_{\ell, c}(X) $ has coefficients in $  \mathcal{H}_{\ZZ}  $. For such  $ c $, the condition on invertibility of primes outside $ S $  in  $ \mathcal{O} $   can be dropped.

In the intended application, the   functor $ \mathcal{S} $ over $ \QQ $ parametrizes weight-$ k $ Eisenstein classes in the first motivic cohomology of the modular curve. Its  composition with   the \'{e}tale regulator admits a $ \ZZ_{p}$-valued  version by   \cite{distpolylog},  which ensures integrality of    classes  in Galois cohomology corresponding to all choices of integral Schwartz functions.    
The set $ S $ corresponds to the set of ``bad primes" where the behaviour of   Eisenstein classes is pathological and the function $ f_{S} $ is therefore not perturbed for Euler system purposes. The  functor for $ \tilde{G} $ is the degree seven  absolute \'{e}tale   cohomology on which $ \ch(\ell \ZZ_{\ell}^{\times} ) $ acts covariantly  as  
 arithmetic Frobenius.   Moreover the 
 pushforward $ \tilde{\iota} _{*} $ is    obtained via  the Gysin  triangle in Ekedahl's ``derived" category of lisse \'{e}tale $ p $-adic sheaves along with certain branching laws of coefficient sheaves on the underlying Shimura varieties.   The abstract formalism of functors used above applies to this cohomology theory  by various   results established in  
 \cite[Appendix A]{GS}. 
 \begin{remark} The   bottom class $ y_{1} $ in our Euler system is meant to be a geometric incarnation of the Rankin-Selberg period integral of Pollack-Shah \cite{PS}\footnote{This integral is denoted by  $I(\phi, s)$ in \emph{loc.\ cit.}} and is  expected to  be related to certain special values of the degree  eight   spinor $L $-function via this period.  See also  %S%ee
 \cite[\S 5]{gclj}.    
 \end{remark}

\subsection{Our  approach}  While Theorem \ref{introB} is the key relation required for an Euler system,  its proof relies on  a  far more fundamental and purely local relation  that lies   at the heart of our approach. In a nutshell, our approach posits that if the convolutions of all ‘twisted’ restrictions to \( H_{\ell} = \Hb ( \QQ_{\ell} ) \) of the Hecke-Frobenius polynomial with the unramified Schwartz function \( \phi_{\ell} = \ch \left ( \begin{smallmatrix} \ZZ_{\ell} \\ \ZZ_{\ell} \end{smallmatrix} \right ) \) fall in the image of certain trace maps, then Theorem \ref{introB} follows. This local relation is also exactly what is needed in \cite{EulerGSp6}, as it allows us to synthesize the results of \cite{AJ} with our own.

We state this relation precisely.    In analogy with the global situation, let $ \mathcal{S}_{\ell }$ denote the set of all $ \mathcal{O}$-valued locally constant compactly supported functions on  $  \mathrm{Mat}_{2 \times 1 }  ( \QQ_{\ell}  )   $. Again, this is a smooth $H_{\ell}$-representation which we view as a  CoMack functor for $ H_{\ell} $.  Denote $ \tilde{G}_{\ell} = \Gb(\QQ_{\ell}) $ and   $ \tilde{K}_{\ell} = \tilde{\Gb} (\ZZ_{\ell}) $. For a  compactly supported function $ \tilde{\mathfrak{H}} : \tilde{G}_{\ell}  \to \OO $ and $ g \in \tilde{G}_{\ell}  
$, 
the 
\emph{$ (H_{\ell},g) $-restriction   of    $ \tilde{\mathfrak{H}} $} is the function   $$ \mathfrak{h}_{g} : H_{\ell}   \to \mathcal{O}  \quad \quad  h  \mapsto \tilde{\mathfrak{H}}(hg) . $$   If $    \tilde{\mathfrak{H} }  $ is $ \tilde{K}_{\ell}$-biinvariant,  then $ \mathfrak{h}_{g} $ is left invariant under $ U_{\ell} = H_{\ell} \cap \tilde{K}_{\ell}  $ and right invariant under $ H_{\ell, g} = H_{\ell} \cap g \tilde{K}_{\ell} 
 g^{-1} $. It therefore induces   
an $ \mathcal{O} $-linear map $ \mathfrak{h}_{g,*}  :   \mathcal{S}_{\ell}  (U_{\ell})   \to  \mathcal{S}_{\ell  } 
 (H_{\ell , g }) $. 
 Let $ V_{\ell , g} $ denote the subgroup of all elements in $ H_{\ell, g} $ whose similitude lies in $ 1 + \ell \ZZ_{\ell} $. 

\begin{theoremx}[Theorem \ref{mainzeta}]   \label{introC}  Suppose in the notation above,   $ \tilde{\mathfrak{H}} = \mathfrak{H}_{\ell,c}(\mathrm{Frob}_{\ell}) $ where $ c $ is any integer.  Then $  \mathfrak{h}_{g,*}(\phi_{\ell})  $ 
lies in the image of the trace map $ \pr_{*} :   \mathcal{S}_{\ell}  (V_{\ell,g}) \to  \mathcal{S}_{\ell} (H_{\ell, g})  $ for every $ g \in \tilde{G}_{\ell} 
  $. 
\end{theoremx}

Results analogous  to Theorem \ref{introC} were obtained in \cite{CZE}, which   strengthen the norm relations of \cite{GS} and \cite{LSZ} to their ideal (motivic)  versions.     The machinery of \cite{CZE}  takes Theorem \ref{introC} as input and gives Theorem \ref{introB} as output,  and can also easily incorporate vertical norm compatibility once a local result has been established, say,   in the style of \cite{loe}.   
Our approach 
has also been  successfully  applied  in   forthcoming works   to obtain new  Euler systems for certain exterior square motives  in the cohomology of $ \mathrm{GU}_{2,2}$ Shimura varieties  \cite{EulerGU22} and for certain rank seven motives of type $G_{2}$ \cite{EulerG2}.   All these results taken together point towards an intrinsic ``trace-imbuing" property of Hecke polynomials attached to Langlands $L  $-factors that seems to be  preserved under   twisted restrictions on suitable reductive  
 subgroups.  We 
hope to explain this property   more conceptually at a future  point.

\subsection{Outline} We prove Theorem \ref{introC} by explicitly computing the convolutions of twisted  restrictions of $\tilde{\mathfrak{H}}  = \mathfrak{H}_{\ell,c}(\mathrm{Frob}_{\ell}) $ with $ \phi_{\ell}     $. As this is rather involved,  we have divided the article into two parts, the first containing mainly statements and the second  their proofs.  Below we provide an  outline  of  the key steps.   

Note first of all that if $ \mathfrak{h}_{g,*}(\phi_{\ell} ) $ lies in the image of the trace map, so does $ \mathfrak{h}_{\eta g \gamma,*}(\phi_{\ell}  ) $ for any $ \eta \in H_{\ell}  $ and $ \gamma \in \tilde{K}_{\ell}  $. Thus it suffices to compute $ \mathfrak{h}_{g,*}(\phi_{\ell} ) $ for $ g $ running over  a choice of representatives for 
 $ H_{\ell} \backslash H_{\ell}  \cdot \supp(\mathfrak{\tilde{H}}  ) / \tilde{K}_{\ell}   $.   
Since multiplies of $ \ell - 1 $ obviously lie in the images   of trace maps that concern us,   it also suffices to compute these functions modulo $ \ell - 1 $. This allows us to completely bypass the computation of $ \mathfrak{H}_{\ell,c}(X) $ by a property of Kazhdan-Lusztig polynomials. It is also straightforward to restrict attention to  $ \mathfrak{H} : = \mathfrak{H}_{\ell,c}(1) \pmod{\ell - 1} $ by first restricting $ \tilde{\mathfrak{H}}  $ to $ G_{\ell} $. The problem is then reduced to computing $ U_{\ell} $-orbits on certain double coset spaces  $ K_{\ell} g K_{\ell}/ K_{\ell} $  where $K_{\ell} = \Gb(\ZZ_{\ell} ) $ and   $ \ch(K_{\ell} g K_{\ell})  $ is a Hecke operator  in  $ \mathfrak{H} $.    
The key technique that allows us to compute  these orbits  is a recipe of decomposing parahoric double cosets  proved in \cite[\S 5]{CZE}. It is originally due to Lansky  \cite{Lansky} in the setting of  Chevalley groups.

However even with the full force of this recipe, directly computing  the $ U_{\ell}  $-orbits on  all the relevant double coset spaces  is a  rather  formidable  task, particularly because the  pair $ (\Hb , \Gb) $ is not spherical. See also Remark \ref{formidableremark}. 
What makes  this computation much more tractable is  the introduction of an intermediate group that allows us to  compute the twisted  restrictions in two steps. In the first step, we compute the restrictions   of $ \mathfrak{H} $ with respect to the group $ H' _{\ell }   = \Hb'(\QQ_{\ell}) $ where $ \Hb' = \GL_{2} \times _{\GG_{m}} \mathrm{GSp}_{4} . $   
The pair $ (\Hb', \Gb) $ is spherical, and a relatively straightforward 
computation shows that there are three $H'_{\ell}  $-restrictions corresponding to the representative elements 
$$  
\tau_{0} =  \left(\begin{smallmatrix}  
1 & & & &    \\ 
& 1 & &   \\
& &  1  & & \\
& & & 1 \\
& & & & 1 \\ 
& &  & & & 1 
\end{smallmatrix} 
\right ), 
\quad\quad
\tau_{1} = \left(\begin{smallmatrix}  
\ell & & & &  1 \\ 
&\ell  & & 1 \\
& & \ell & & \\
& & & 1 \\
& & & & 1 \\ 
& &  & & & 1 
\end{smallmatrix} 
\right ), 
\quad \quad 
\tau_{2} = \left(\begin{smallmatrix}
\ell \, \,  & & & & \ell^{-1}   \\   
&   \ell \,  \, & & \ell^{-1}\\[0.1em] 
& &  1 \, \,  &  \\   
& & & \ell^{-1} \\
& & &  & \ell^{-1}      \\ 
& & & & & 1  
\end{smallmatrix}\right) $$
in $ G_{\ell} $. This is expected since a general ``Schr\"{o}der type" decomposition holds for the quotient    $  H'_{\ell} \backslash  G_{\ell} / K_{\ell} $ by a result of  Weissauer \cite[\S 12]{Endoscopy}. We  denote the $ (H'_{\ell}, \tau_{i}) $-restrictions of $ \mathfrak{H} $ by $ \mathfrak{h}_{i} $. This step is recorded in \S \ref{firstrestrictions} and justifications  are   provided in \S \ref{U'orbitssec}.  

The second step is to compute the $ H_{\ell} $-restrictions of $\mathfrak{h}_{i} $ for $ i = 0,1,2 $.  This essentially turns out to be a study of $ \GL_{2} \times_{\GG_{m}}  \GL_{2} $-orbits on $ \mathrm{GSp}_{4} $-double cosets. Since $ (\GL_{2} \times _{\GG_{m}} , \GL_{2} , \mathrm{GSp}_{4} )$ is also a spherical pair, this is again  straightforward for $ i = 0 $ and even for $ i =1 $ as the projection of $ H'_{\ell} \cap \tau_{1} K_{\ell} \tau_{1}^{-1} $ to  the   $ \mathrm{GSp}_{4}(\QQ_{\ell}) $-component  turns out to be a non-special maximal compact open subgroup of  $ \mathrm{GSp}_{4}(\QQ_{\ell}) $. The more challenging case of $ i = 2 $ is handled by  comparing the  double cosets with a subgroup of $ \mathrm{GSp}_{4}(\QQ_{\ell}) $ deeper than the Iwahori  subgroup that sits in the projection of the twisted intersection.  For $ \mathfrak{h}_{0} $ (resp., $ \mathfrak{h}_{1}$), there turn out to be three (resp., four) restrictions  indexed again  by certain  ``Schr\"oder type" representatives. 
For $ \mathfrak{h}_{2} $ however, there turn out to be  $ \ell + 3 $ restrictions. 
We use the symbols $ \varrho $, $ \varsigma $, $\vartheta $ for the set of distinct representatives of $H_{\ell} \backslash H_{\ell} \cdot \supp(\mathfrak{H})  /K_{\ell}  $ which correspond to the $H_{\ell}  $-restrictions of $ \mathfrak{h}_{0}, \mathfrak{h}_{1}, \mathfrak{h}_{2} $ respectively. The diagram below organizes these restrictions in a  tree.  

\begin{center}  
\begin{tikzcd}
                           &                                                  &                              & \mathfrak{H} \arrow[dd] \arrow[rrrd] \arrow[lld]             &                              &                              &                                                                          &                              &                                      \\
                           & \mathfrak{h}_{0} \arrow[rd] \arrow[d] \arrow[ld] &                              &                                                              &                              &                              & \mathfrak{h}_{2} \arrow[lld] \arrow[ld] \arrow[d] \arrow[rd] \arrow[rrd] &                              &                                      \\
\mathfrak{h}_{\varrho_{0}} & \mathfrak{h}_{\varrho_{1}}                       & \mathfrak{h}_{\varrho_{2}}   & \mathfrak{h}_{1} \arrow[lld] \arrow[ld] \arrow[d] \arrow[rd] & \mathfrak{h}_{\vartheta_{0}} & \mathfrak{h}_{\vartheta_{1}} & \mathfrak{h}_{\vartheta_{2}}                                             & \mathfrak{h}_{\vartheta_{3}} & \mathfrak{h}_{\tilde{\vartheta}_{k}} \\
                           & \mathfrak{h}_{\varsigma_{0}}                     & \mathfrak{h}_{\varsigma_{1}} & \mathfrak{h}_{\varsigma_{2}}                                 & \mathfrak{h}_{\varsigma_{3}} &                              &                                                                          &                              &                                     
\end{tikzcd}
\end{center} 
Here the branch indexed by $ \tilde{\vartheta}_{k} $ actually designates  $ \ell - 1 $ branches, one for each value of $ k \in \left \{0,1, 2,3,\ldots, \ell-2\right \} $. Thus $ H_{\ell} \backslash  H_{\ell} \cdot \supp(\mathfrak{H}) / K_{\ell} $ consists of  $ 3 + 4 + (4 + \ell -1 ) =  \ell + 10  $ elements.  The corresponding  $ \ell + 10 $ restrictions are recorded  in  \S \ref{secondrestrictions} and proofs of various claims are provided  in \S\ref{Uorbitssec}.   Once these restrictions are obtained, the final step is to compute their covariant  convolution   with  $ \phi_{\ell} $. We show in \S \ref{convolutionsection} that all resulting convolutions vanish modulo $ \ell  - 1 $ except for $ \mathfrak{h}_{\vartheta_{3},*}(\phi_{\ell}) $. 
A necessary and sufficient criteria established in \cite[\S 3.5]{CZE} allows us to easily determine that $ \mathfrak{h}_{\vartheta_{3},*}(\phi_{\ell}) $ lies  in the image of the appropriate trace map and thus deduce the truth of Theorem \ref{introC}. 

\begin{remark}   For comparison, the $ \mathrm{GSp}_{4} $ setting studied in \cite[\S 9]{CZE} involved  only $ 2 $ restrictions, which explains why the test vector of \cite[Corollary 3.10.5]{LSZ} only required two terms to produce the  $ L$-factor.  
\end{remark} 

\begin{remark}  The mysterious vanishing of all but one of the convolutions modulo $ \ell - 1 $ and 
 the simplicity of $ \mathfrak{h}_{\vartheta_{3},*}(\phi_{\ell}) $  strongly suggest that a more conceptual proof of our result is possible. 
\end{remark}

\subsection{Acknowledgements}   I would like to express my gratitude to Antonio Cauchi and Joaquín Rodrigues  Jacinto, whose work on Beilinson conjectures and vertical norm relations in the $\mathrm{GSp}_{6}$ setting served as the inspiration for this article.  I am especially indebted to Antonio Cauchi for his  careful  explanation of the unfeasibility of a related construction    and for his unwavering support 
throughout the course of this project. In addition, I  thank  Aaron Pollack, Andrew Graham,  Christophe Cornut, Barry Mazur, Daniel Disegni,  David Loeffler and  Wei Zhang  for   several    valuable conversations in relation to the broader  aspects of this work. I am also grateful to  Francesc Castella,  Naomi Sweeting and  Raúl Alonso Rodríguez for some  useful  comments and  suggestions.   At various stages, the software  MATLAB\textsuperscript{®} was used for performing  and organizing  symbolic matrix manipulations, which proved 
 invaluable in composing many of the  proofs.

%% file: Statement.tex
\part{Statements of results}
\section{General notation}    \label{generalnot}     The notations introduced here are used throughout this article except  for  \S \ref{HNRglobal}.  For aesthetic reasons, we work  with an arbitrary local field of characteristic zero, though we only need the results over $ \QQ_{\ell} $.  

Let  $F $ denote  a local field of characteristic zero, $ \Oscr_{F} $ its ring of integers, $ \varpi $ a uniformizer, $ \kay = \Oscr_{F} / \varpi \Oscr_{F} $ its residue field and $ q = | \kay | $.   For $ a \geq 0 $ an integer, we let $ [\kay_{a}] \subset \Oscr_{F} $ denote a fixed set of representatives for $ \kay_{a} = \Oscr_{F} / \varpi^{a} \Oscr_{F} $ and we omit the subscript $ a $ when $ a = 1 $. We let $ 0,1,-1 \in [\kay] $ denote the elements that represent $ 0,1,-1 \in \kay $  respectively.   For $ n $ an integer,  let $ 1_{n} $ denote the $ n \times n $ identity matrix and $ J_{2n} = \left ( \begin{smallmatrix} &  1_{n} \\ - 1_{n} \end{smallmatrix}  \right )  $ denote the standard $ 2n \times 2n $ 
symplectic matrix. We define $ \mathrm{GSp}_{2n} $ to be the group scheme over $ \ZZ $  whose $R$-points for a ring $ R $  are given by $$ \mathrm{GSp}_{2n}(R) =  \left \{ (g ,c )\in \GL_{2n}(R) \times R^{\times}  \, | \, g^{t} J_{2n}  g = c   J_{2n}  \right \} . $$
Note that $ \mathrm{GSp}_{2} $ is the general linear group $ \GL_{2} $. We let $ \mathrm{sim} :  \mathrm{GSp}_{2n} \to \GG_{m} $, $ (g,c) \mapsto c $  denote the similitude map and refer to an element $ (g, c) \in \mathrm{GSp}_{2n}(R) $ simply by $ g $.  The following group schemes will be used throughout:  

\begin{multicols}{2} 
\begin{itemize} 
\item $ \Hb   = \GL_{2} \times_{\GG_{m} }  \GL_{2}  \times_{\GG_{m}}  \GL_{2}  $,
\item $ \Hb_{1} = \GL_{2} $,  
\item $ \Hb_{2} = \GL_{2} \times_{\GG_{m}} \GL_{2} $,
\item $ \Hb'  = \GL_{2} \times_{\GG_{m}} \mathrm{GSp}_{4} $,
\item $ \Hb'_{2} =  \mathrm{GSp}_{4} $,
\item $ \mathbf{G} \, \,  =  \mathrm{GSp}_{6} $
\end{itemize} 
\end{multicols} 
\noindent  where all the products are  fibered over similitude maps. We define $ H $, $H_{1} $, $ H_{2} $,  $ H ' $,  $ H_{2}' $, $ G $  to be respectively the group of $ F $-points of the algebraic groups above and $ U $, $U_{1} $, $U_{2} $,  $ U' $,  $ U_{2} '  $, $ K $ to be the group of $ \Oscr_{F} $-points. 
We define projections   
\begin{alignat*}{9}   
\pr_{1}:\Hb 
& \longrightarrow \Hb_{1}  & \quad \quad \quad
&& \pr_{2} :  \Hb 
& \longrightarrow  \Hb_{2} & \quad\quad\quad   
&& \pr_{1}': \Hb  '  &\longrightarrow  \Hb_{1}    &\quad\quad\quad   
&& \pr_{2}': \Hb' & \longrightarrow \Hb'_{2} 
\\ 
(h_{1},h_{2},h_{3}) 
& \longmapsto h_{1} &\quad \quad  \quad    
&& (h_{1}, h_{2}, h_{3} ) & \longmapsto (h_{2},h_{3}) &\quad\quad\quad 
&& (h_{1}, h_{2} )  
& \longmapsto h_{1} & 
&&  (h_{1}, h_{2}) & \longmapsto  h_{2} 
\end{alignat*}

and  embeddings 
\begin{alignat*}{9}  \jmath _{2} :  \Hb_{2}  &   \longrightarrow    \Hb'_{2}  &  \quad \quad   &&    \jmath :  \Hb   &   \longrightarrow    \Hb'    \quad \quad    &       \iota' :  \Hb'   &   \longrightarrow    
\Gb\\
\left ( \left ( 
\begin{smallmatrix} a & b \\ c & d \end{smallmatrix} \right ) , 
\left ( 
\begin{smallmatrix} a' & b' \\ c' & d' \end{smallmatrix} \right    )   \right )  &  \longmapsto  \left (   \begin{smallmatrix} a & & b \\  & a ' & & b' \\ 
c &  & d \\ &  c' & & d '  \end{smallmatrix} \right )   &  \quad \quad  &&    (h_{1}, h_{2}, h_{3})  & \longmapsto (h_{1},  \jmath_{2}(h_{2}, h_{3})      )   & \quad \quad  \left ( \begin{psmallmatrix} a & b \\ c & d \end{psmallmatrix} ,  \begin{psmallmatrix} A &  B \\ C & D  \end{psmallmatrix}   \right )  & \longmapsto  \begin{psmallmatrix} a &  &  b  \\  & A & &  B    \\ c  & & d \\  & C & & D \end{psmallmatrix}  & 
\end{alignat*}
via which we consider $ U_{2}$, $ H_{2} $, $ U $, $ H  $, $U'$, $H' $ to be subgroups of $ U_{2} '$, $ H_{2}' $, $ U' $, $ H' $, $ K $, $ G $   respectively. We let $$ \iota : \Hb \to \Gb $$ 
denote the composition $ \iota ' 
\circ  \jmath $ via which we view $ U $, $ H $ as subgroups of $  K $, $  G $ respectively. If $ R $ is a commutative ring with identity  and $ L_{1}, L_{2} $ are compact open subgroups of $ G $, we write $ \mathcal{C}_{R}(L_{1} \backslash G / L_{2}) $ for the set of $ R $-valued    compactly  supported functions $ f : G \to R $ that are left $ L_{1} $-invariant and right $L_{2} $-invariant.   Similar notations will be used for functions on $ H $ and $ H ' $.    
\begin{definition} Given a function $ \mathfrak{F} : G \to R $ and an element $ g \in G $, we define the \emph{$ (H', g)$-restriction} of $ \mathfrak{F} $ to be the function $ \mathfrak{f}_{g} : H'  \to  R $ given by $ \mathfrak{f}_{g}(h) = \mathfrak{F}(hg) $ for all  $ h \in H ' $. We similarly define $ (H,g) $-restriction  of $ \mathfrak{F} $ and $ (H,\eta) $-restrictions  of functions on $ H'  $   and   $ \eta \in H' $.  
\end{definition}  
It is easy to see that if $ \mathfrak{F} \in \mathcal{C}_{R}(K\backslash G /K) $, then $ \mathfrak{f}_{g}  \in \mathcal{C}_{R}(U' \backslash  H' / H_{g}' ) $ where  $ H'_{g} = H '\cap g K g^{-1 }$. If $ 
 \eta \in H' $, then the $ (H,\eta)$-restriction of $ \mathfrak{f}_{g}$  coincides with  the  $ (H, \eta g) $-restriction of $ \mathfrak{F} $ and  lies in  $ \mathcal{C}_{R}(U\backslash G / H_{\eta g}) $ where $ H_{\eta g } =   H \cap \eta H_{g} ' \eta^{-1}  =  H \cap  \eta g K g^{-1} \eta^{-1}   $.

\section{Spinor  Hecke polynomial}    

\subsection{Root datum of $\Gb$} Let  $  \mathbf{A} =  \GG_{m}^{4} $  and  $ \mathrm{dis} :   \mathbf{A} \to \mathbf{G} $ to be the embedding given by $$   (u_{0}, u_{1},u_{2},  u_{3} ) \mapsto  \mathrm{diag}( u_{1} , u_{2} , u_{3}  ,  u_{0} u_{1} ^{-1}, u_{0}  u _{2}  ^{-1} ,  u_{0} u _{3} ^{-1}  ) .  $$
Then $ \mathrm{dis} $
identifies $ \mathbf{A} $ with a maximal (split) torus in $ \mathbf{G} $. We let $ A $, $ A^{\circ} = A \cap K  $ denote respectively  the group of $ F $, $ \Oscr_{F} $-points of $ \mathbf{A}  $.   Let $ e_{i} : \mathbf{A} \to \GG_{m} $ be the projection onto the $ i $-th component, $ f_{i} : \GG_{m} \to \mathbf{A} $ be the cocharacter inserting $ u $ into the $ i $-th component with $ 1 $ in the remaining components. We will let $$ \Lambda =  \ZZ f_{0} \oplus \cdots \oplus \ZZ f_{3}  $$ 
denote the cocharacter lattice.  An   element $  a_{0} f_{0}  + \ldots + a_{3} f_{3}  \in \Lambda  $ will  also  be  denoted by $ (a_{0}, \ldots, a_{3} ) $. The set $ \Phi  \subset X^{*}(\mathbf{A} ) $ of roots of $ \mathbf{G} $ are
\begin{itemize}  \setlength\itemsep{0.3em}      
\item   $  \pm ( e_{i} - e_{j}  ) $ for $ 1 \leq i < j  \leq  3 $, 
\item  $  \pm (   e_{i} + e_{j}  - e_{0} )  $ for $ 1 \leq i  < j \leq 3   $
\item    $   \pm (  2 e_{i} - e_{0}  ) $ for $ i = 1 , 2 , 3 $
\end{itemize}
which makes an irreducible root system of type $ C_{3} $. We choose  $$  \alpha_{1} = e_{1} - e_{2}, \quad \quad  \alpha_{2} = e_{2} - e_{3}, \quad \quad  \alpha_{3} = 2 e_{3} - e_{0}  $$ as our simple roots and let  $ \Delta = \left \{ \alpha_{1} , \alpha_{2} ,  \alpha_{3}  \right \}  $. This determines a subset $   \Phi ^ { + }  \subset \Phi $ of positive roots. The  
 resulting  half sum of positive roots is 
 \begin{equation}  \label{halfsum} \delta =   -3 e_{0} + 3e_{1} + 2 e_{2} +   e_{3}  \in X^{*}( \mathbf{A}) 
 \end{equation}  
and  the  highest root is $ \alpha_{0}  =  2 e_{1} - e_{0} $.  
The simple  coroots corresponding to $ \alpha_{i} $ for $ i = 0, 1,  2, 3 $ are   $$ \alpha_{0}  ^ { \vee } =  f_{1}, \quad\quad 
\alpha_{1}^{\vee} = f_{1} -  f_{2},  \quad\quad     \alpha_{2}^{\vee} = f_{2}  - f_{3},  \quad  \quad   \alpha_{3}^{\vee} = f_{3}  $$  and their $ \ZZ $ span in $ \Lambda $ is denoted by $ Q ^{\vee} $. The set $ \Delta $ determines a dominance order on $ \Lambda $. Explicitly, an  element    $  \lambda = (a_{0} , \ldots, a_{3} )  \in \Lambda $ is  dominant iff 
$$ a_{1} \geq a_{2} \geq a_{3}  \, \text{ and }  \,   2 a_{3} - a_{0} \geq 0 . $$ It is anti-dominant if all these inequalities hold in reverse. We denote the set  of dominant cocharacters by $ \Lambda^{+ } $.  
Let $ W $ denote the Weyl group of $ (\Gb, \mathbf{A}) $ and $ s_{i} $ be the reflection associated with  $ \alpha_{i} $, $ i = 0, \ldots,  3  $.  The action of $  s_{i} $ on $ \Lambda $ is given  as  follows:  
\begin{itemize}[after = \vspace{\smallskipamount+1.5pt}]  
\setlength\itemsep{0.3em}   
\item   $   s_{i} $ acts by switching $ f_{i} \leftrightarrow f_{i+1} $ for $ i = 1,2   $,
\item $  s_{3}   $ acts by sending $ f_{0} \mapsto f_{0} + f_{3} $, $ f_{3} \mapsto   - f_{3}   $, 
\item $ s_{0} = s_{1} s_{2} s_{3} s_{2} s_{1} $ acts by sending $ f_{0} \mapsto f_{0} +                f_{1} $, $ f_{1}  \mapsto  -  f_{1}  $.  
\end{itemize}
We have $ W = \langle s_{1}, s_{2}, s_{3}  \rangle  \simeq  (\ZZ /2 \ZZ )^{3} \rtimes S_{3}  $ where $ S_{3} $ denotes the group of permutations of three elements that acts on $ ( \ZZ / 2 \ZZ ) ^{3} $ in the obvious manner.    
\subsection{Iwahori Weyl group}   

Let $ I $ denote  the  Iwahori subgroup  of $ G $  corresponding to (the alcove determined by) the simple affine roots $ \Delta_{\mathrm{aff}}  =  \left \{     \alpha_{1}, \alpha_{2}, - \alpha_{0} + 1   \right \}$.    Explicitly, $ I $ is the compact open subgroup of $ K $ whose reduction modulo $ \varpi $ is the Borel subgroup of $ \mathbf{G}(\kay) $ determined by $ \Delta $.   Let $ W_{\mathrm{aff}} $ and $ W_{I} $ denote respectively the  affine Weyl and  Iwahori Weyl groups of the pair $ (\mathbf{G}, \mathbf{A}) $.   We  view   $ W_{\mathrm{aff}} $ as a subgroup of the group of affine transformations of $ \Lambda \otimes \RR $. Given $ \lambda \in \Lambda $, we let $ t(\lambda) $ denote translation by $ \lambda $ map on $ \Lambda \otimes \RR $  and  write    $  \varpi^{\lambda} $ for the element  $ \lambda(\varpi) \in A $.   Let    $v : A / A ^{\circ}  \to \Lambda $ be the inverse of the isomorphism  $ \Lambda \to A / A ^{\circ} $ given by  $ \lambda  \mapsto \varpi^{-\lambda} A^{\circ} $. Then 
\begin{itemize} [before = \vspace{\smallskipamount}, after =  \vspace{\smallskipamount}]    \setlength\itemsep{0.1em}  

\item $ W_{\mathrm{aff}} = t(Q^{\vee} ) \rtimes W $
\item $ W_{I} =   N_{G}(A) / A^{\circ} =  A/A^{\circ} \rtimes W  \stackrel{v}{\simeq}     \Lambda \rtimes W $,
\end{itemize} 
where $ N_{G}(A)$ denotes the normalizer  of $ A $ in $ G $. The set  $S_{\mathrm{aff} } = \left \{ s_{1} , s_{2} ,s_{3} ,  t ( \alpha _ {  0 }  ^ { \vee }    )    s_{0}  \right \} $  is a generating  set  for $ W_{\mathrm{aff} } $ and  the pair $ (W_{\mathrm{aff}} , S_{\mathrm{aff}}) $ forms   a Coxeter system of type $ \tilde{C}_{3}   $.  Identifying $ W_{I} $ with $ \Lambda \rtimes W $ as above,  we can consider  $ W_{\mathrm{aff}}   $    a   subgroup of  $  W_{I } $ via  $ W_{\mathrm{aff}} =  t(Q^{\vee})   \rtimes W  \hookrightarrow t(\Lambda) \rtimes W  $.     The quotient   $$\Omega : = W_{I} /  W_{\mathrm{aff}} $$ is  then  an infinite cyclic group and  we  have a canonical isomorphism $ W_{I} \cong  W_{\mathrm{aff}}  \rtimes \Omega $.   
We  let   $$ \ell :  W_{I}  \to \ZZ $$  denote the induced length function with respect $ S_{\mathrm{aff}} $. Given $ \lambda \in  \Lambda $, the minimal length of elements in $ t  ( \lambda ) W $ is  achieved by a unique element. This length is given by \begin{equation}   \label{lmin}       \ell_{\mathrm{min}} ( t(\lambda)   )    :    =      \sum_{  \alpha \in \Phi_{\lambda}}  | \langle \lambda , \alpha \rangle |    +   \sum _ { \alpha  \in  
\Phi^{\lambda }  }  (  \langle \lambda , \alpha \rangle   -  1   )   
\end{equation}    where $  \Phi_ { \lambda}    = \left \{  \alpha \in  \Phi   ^    {    + } \, | \,   \langle  \lambda ,  \alpha   \rangle  \leq 0  \right \} $ and  $ \Phi ^    {\lambda }    =  \left \{   \alpha  \in  \Phi ^ { + }  \,, |  \langle  \lambda ,  \alpha  \rangle   > 0  \right \}   $.     When $ \lambda $ is dominant, this is also the minimal length of elements in $ W t(\lambda) W $.  Consider the following elements in $ N_{G} (A) $:   

$$   w_{1} : =  \scalebox{1.1}{$      \left ( \begin{smallmatrix}
 0&  1&  &&  & \\[0.1em] 
 1&  0&  &  &  & \\ 
 &  &  1&  &  & \\
&  &  & 0 &1  & \\[0.1em] 
 &  &  & 1 & 0& \\ 
 &  &  &  &  & 1
\end{smallmatrix} \right )$}, \quad     
w_{2}   :    =  \left( \scalebox{1.1}{$ \begin{smallmatrix}
 1&  &  &&  & \\ 
 &   0 & 1 &  &  & \\[0.1em] 
 &  1 &  0 &  &  & \\ 
&  &  & 1  &   & \\ 
 &  &  &   & 0&1  \\[0.1em]  
 &  &  &  &  1  &  0 
\end{smallmatrix}$}\right) , \quad      w_{3}  : =   \scalebox{1.1}{$ \left( \begin{smallmatrix}
 1&  &  &  &  & \\ 
 &   1  &  &  &  & \\ 
 &    &   0  &  &  &1\\ 
&  &  &   -1  &  & \\ 
 &  &  &   &  -1 &  \\ 
 &  &  1  &  &  &  0 
\end{smallmatrix} \right ) $} $$ 

$$  w_  {  0   }   : =  \scalebox{1.1}{$\left (  \begin{smallmatrix}
 0&  &  &  \scalebox{1}{$\sfrac{1}{\varpi}$} &  & \\ 
 &  \,  1&  &  \\ 
 &  &  \,  1& \\ 
 \varpi&  &  & 0 \\ 
 &  &  &  &   -1  \\ 
 &  &  &  &  &  -1
\end{smallmatrix}  \right )$},  \quad     \rho   =   \scalebox{0.95} {$      \left ( \begin{smallmatrix}
 &  &  & &  &  1\\ 
 &  & &  &  1& \\ 
 &   &   & 1 &  &\\ 
&     &  \varpi  & &   & \\ 
 &  \varpi&  &   & &  \\ 
     \varpi &  & &   &    &  
\end{smallmatrix} \right ) $}   .  $$ 
The   classes of $ w_{0} , w_{1} ,w_{2}, w_{3} $ in $ W_{I} $   represent $ t(\alpha_{0} ^{\vee} ) s_{0},   s_{1} , s_{2}, s_{3} $ respectively and the  reflection $ s_{0} $   is  represented  by  $ w_{\alpha_{0}} :  =  \varpi^{f_{1}}  w_{0}     =      w_{1} w_{2} w_{3} w_{2} w_{1}  $. The class of  $ \rho $ represents $ \omega : = t(-f_{0} ) s_{3} s_{2} s_{3} s_{1} s_{2} s_{3}  $ which is a generator of $ \Omega $ and the conjugation by $ \omega $ acts by switching $ s_{0} \leftrightarrow  s_{3} $, $ s_{1}  \leftrightarrow s_{2} $. That is,   it induces    an  automorphism of the extended   Coxeter-Dynkin  diagram  
 $$    \scalebox{0.9}{ \dynkin[extended,Coxeter,
edge length=1cm,
labels={0,1,2,3}]
C{3}}$$
where the labels below the vertices correspond to $ w_{i } $. Note also that $ \rho^{2} = \varpi^{(2,1,1,1)}  \in A  $ is central.   We will   henceforth    use the   letters    $ w_{i} $, $ \rho $ to denote both the matrices and the their classes in $ W_{I} $ if no confusion can arise. When referring to action of simple reflections in $ W $ on $ \Lambda $ however, we will  stick to the letters $ s_{i} $.     
\subsection{The  Hecke polynomial} Let  $ \ZZ [ \Lambda ] $ denote the group algebra of $ \Lambda $.  For $ \lambda \in \Lambda $, we let $ e ^ { \lambda } \in \ZZ [ \Lambda ] $ denote\footnote{this is done to distinguish the addition in $ \Lambda $ from addition in the group algebra}  the element corresponding to $\lambda $ 
and $ e ^ { W \lambda } \in \ZZ [ \Lambda ] $ denote the the (formal)  sum    of elements in the orbit $ W \lambda $.
 We will denote $   y_{i} : =  e^{f_{i} }   \in     \ZZ [ \Lambda ]$ for $ i = 0 , \ldots 3 $, so that $$ \ZZ [ \Lambda ] = \ZZ [ y_{0} ^{ \pm } , \ldots, y_{3} ^ { \pm } ] . $$
Let $ \mathcal{R}  =  \mathcal{R}_{q} 
 $ denote the ring $ \ZZ [ q ^ { \pm \frac{1}{2} } ]   $. The dual group of $ \mathbf{G} $ has an $ 8 $-dimensional representation called the  \emph{spin} representation.  Its    highest (co)weight is $ f_{0} + f_{1} + f_{2} + f_{3} $ which is minuscule. Thus its (co)weights are $ \frac{1}{2} ( 2 f_{0} + f_{1} + f_{2} + f_{3} )   +  \frac{1}{2} ( \pm f_{1}  \pm  f_{2}  \pm  f_{3} )   $ and its characteristic (Satake) polynomial is
\begin{align*}  
\mathfrak{S} _{ \mathrm{spin} } ( X )  =   &  ( 1-  y_{0} X ) ( 1- y_{0} y_{1} X) ( 1-  y_{0} y_{2} X ) ( 1 - y_{0} y_{3} X) \\[1pt]
& 
( 1-  y_{0} y_{1} y_{2} X ) ( 1-  y_{0} y_{1} y_{3} X )( 1- y_{0} y_{2}  y_{3} X ) ( 1 - y_{0} y_{1} y_{2} y_{3} X ) \in  \ZZ [ \Lambda ] ^ { W } (X)    .  
\end{align*}
\noindent Let $ \mathcal{H}_{\mathcal{R}}(K \backslash G / K ) $ denote the spherical Hecke algebra with coefficients in $ \mathcal{R} $ that is defined with respect to a measure on $ G $ giving  $ K $ measure one.   Let $$  \mathscr{S} :  \mathcal{H}_{\mathcal{R}}(K \backslash G / K ) \to \mathcal{R} [  \Lambda ]  ^ {  W } $$ denote the Satake  isomorphism. If $ P = P(X) \in \mathcal{H}_{\mathcal{R}}(K\backslash G / K )[X] $ is a polynomial, then $ \mathscr{S}(P) $ means the polynomial in $ \mathcal{R}[\Lambda]^{W}[X] $ obtained by applying $ \mathscr{S} $ to the coefficients of the powers of $ X $ in $ P $.      
\begin{definition}  For $ c \in \ZZ $, we define the degree 8 \emph{spinor Hecke polynomial}  $ \mathfrak{H}_{\mathrm{spin},c}(X) \in  \mathcal{H}_{\mathcal{R}}(G) [X] $ to be unique polynomial    such that $ \mathscr{S} ( \mathfrak{H}_{\mathrm{spin},  c }    )    =  \mathfrak{S}_{\mathrm{spin}}( q^{-c} X) $.   
\end{definition}
To work with this  Hecke polynomial and to describe the decompositions of the double coset operators appearing in it later on, it would be convenient to record the following.

\begin{lemma}   \label{lminwords}     For each  $ \lambda \in  \Lambda ^ { + }  $  below,      the element  $ w =  w _{\lambda}   \in  W_{I}  $ specified is the unique element in $ W_{I} $ of minimal possible length  such that  $ K \varpi ^ { \lambda } K = K   w   K $.

\begin{itemize}    [after = \vspace{\smallskipamount-1pt}] 
\setlength\itemsep{0.4em}      
\item $ \lambda = (1,1,1,1) $, $  w = \rho $, 
\item $ \lambda = (2,2,1,1) $, $ w =  w_{0} \rho^{2} $,  
\item $ \lambda = (2,2,2,1) $, $ w = w_{0}  w_{1} w_{0} \rho^{2} $, 

\item  $ \lambda = (3,3,2,2) $, $ w =   w_{0} w_{1} w_{2} w_{3}  \rho^{3}  $, 
\item $  \lambda = ( 4,3,3,3)  $,  $  w =    w_{0} w_{1} w_{0} w_{2} w_{1} w_{0}   \rho^{4}  $,    
\item $  \lambda = ( 4,4,2,2 )    $, $ w =   w_{0} w_{1} w_{2} w_{3} w_{2} w_{1} w_{0}   \rho^{4}  $.   
\end{itemize}
\end{lemma}

\begin{remark} We point out  that the translation component   of each $ w_{\lambda} $ above (i.e., the $ \Lambda $-component in  $  W_{I} = \Lambda \rtimes W $) is $ t(- \lambda^{\mathrm{opp}})  $ 
where $ \lambda^{\mathrm{opp}} $ is the  anti-dominant element in the Weyl orbit $ W \lambda $.     The minimal  
possible length in each case is computed using (\ref{lmin}) and   that $ \ell ( w _{\lambda} ) =  \ell_{\min} ( t ( -  \lambda ^ { \mathrm{opp} }   )   )    =  \ell _ { \min }  ( t ( \lambda ) )    $.  

\end{remark}   

\begin{notation} For convenience,  we will notate     
\[ \upsilon_{0} = w_{0} , \quad \upsilon_{1} = w_{0} w_{1} w_{0} , \quad  \upsilon _{2} : =   w_{0} w_{1} w_{2} w_{3}  , \quad 
 \upsilon_{3}  : = w_{0} w_{1} w_{0} w_{2} w_{1} w_{0} , \quad  \upsilon_{4} =  w_{0} w_{1} w_{2} w_{3} w_{2} w_{1} w_{0}  \]    Given $ g \in G $, we let $ (K gK ) $ denote the characteristic function $ \ch(K g K ) : G \to \ZZ $ of the double coset $ K g K$. 
 For an even integer  $ k $,  we let $ \rho^{k} (K g K ) $ denotes the function  $ \ch(K g \rho^{k} K ) $. We will use  similar notation for sums of such functions and for functions on $ H'$ and $ H $.    
\end{notation} 
\begin{proposition}\label{Gsp6Heckepolynomial}     The coefficients of $ \mathfrak{H}_{\mathrm{spin},c}(X) $ lie in $ \mathcal{H}_{\ZZ[q^{-1}]}(K \backslash G  / K ) $ for all $ c \in \ZZ $. If we define 
{
\setlength{\abovedisplayskip}{0.7em}
\setlength{\belowdisplayskip}{0em}
\begin{align*}  
\mathfrak{H} (X)   &  =  ( K  )  -   (K \rho K ) X +   \mathfrak{A}   X^{2}     -   \mathfrak{B} X^{3}   + (   \mathfrak{C}    +    2 \rho^{2}   \mathfrak{A}  )X ^{4} 
 - \rho^{2} \mathfrak{B} X^{5}  +   \rho^{4} \mathfrak{A} X^{6}  \\[1pt]  
 & - ( K \rho ^ { 7 }  K )   X ^ { 7 }   +   ( K \rho ^ { 8 }  K )       X ^{ 8 } \, \in  \, \mathcal{H}_{\ZZ } ( K  \backslash G / K  ) [ X ]
\end{align*}}
where 
\begin{itemize}[before = \vspace{-1pt}, after = \vspace{\smallskipamount+1pt}]  
\setlength\itemsep{0.5em}   
\item $  \mathfrak{A} =  ( K   \upsilon_{1} \rho^{2}    K ) +    2 (   K \upsilon _{0} \rho^{2}  K  )   +   4  (  K \rho^{2} K    )  $,
\item $   \mathfrak{B}  = ( K \upsilon_{2}  \rho  ^{3}  K  ) +   4 (   K     \rho ^ { 3 }  K   )    $,
\item $ \mathfrak{C}  =    ( K   \upsilon_{3} \rho^{4}     K   )   +   (  K  \upsilon_{4} \rho^{4}   K  )  $,    
\end{itemize}  
then  $ \mathfrak{H}_{\mathrm{spin},c} (X) $ is congruent to  
$ \mathfrak{H} (X) $ modulo $ q - 1 $  for all $ c \in \ZZ $. 
\end{proposition} 
\begin{proof}  Since the half sum of positive roots  (\ref{halfsum})   lies in $ X^{*}(\mathbf{A})    $, the first claim is obvious from the discussion in \cite[\S 4.4]{CZE}. Solving the plethysm problem for exterior powers of the  spin  representation by combining $ i $ choices of coweights $  \big ( 1 , \mfrac{1}{2} , \mfrac{1}{2}  , \mfrac{1}{2} 
\big  )  +  \big ( 0,  \pm \mfrac{1}{2}, \pm  \mfrac{1}{2}, \pm \mfrac{1}{2} \big    ) $ for $ i = 0, \ldots, 8 $  or simply  by expanding $ \mathfrak{S}_{\mathrm{spin}}(X) $, we see that             
{\setlength{\abovedisplayskip}{1em}
\setlength{\belowdisplayskip}{1em}
\begin{align*}      \mathfrak{S} _ { \mathrm{spin} } ( X)  &   = 1 - e ^{ W(1,1,1,1)} X  \\
& +  \big ( e ^ { W ( 2,2,2,1) } + 2 e ^{W(2,2,1,1)}   +  4  e ^{ (2,1,1,1)} \big ) X^{2}   -  \big (  e^{W(3,3, 2,2)}  + 4  e ^{W(3,2,2,2)}   \big  )     X ^{3}  \\
& +   \big (    e ^ {  W ( 4,4,2,2)  }  + e ^ {   W     ( 4,3,3,3) }  + 2  e ^ { W (  4,3,3,2) }   +   4   e ^ { W ( 4,3,2,2)}  +   8  e ^ {  ( 4,2,2,2) }   \big  ) X^{4}   \\
&    - \big (  e ^ { W ( 5,4,3,3) }  +   4   e  ^{    W  (  5, 3 , 3 ,3 ) }   \big  )  X ^ { 5  }    +   \big ( e ^ { W ( 6, 4,4,3) } +   2      e ^ {W (6,4,3,3)}  + 4 e ^ {  ( 6,3,3,3) }   \big  )  X ^ { 6  }    \\   &    -   e ^ { W ( 7,4,4,4)} X^{7} +  e ^ { ( 8  ,  4 , 4,   4 ) }  X ^ { 8 } 
\end{align*}} 
The claim now follows by Lemma \ref{lminwords} and \cite[Corollary  4.9.4]{CZE}.  
\end{proof}

\begin{remark} The exact coefficients in the Hecke polynomial are polynomial expressions in $ q $ translated by (possibly negative) powers of $ q $. They  can be found explicitly using Sage by computing appropriate Kazhdan-Lusztig polynomials $ P_{\sigma,\tau}(q) $ for $ \sigma, \tau \in W_{I}    $. See \cite[Remark 10.1.3]{AES} for an example.  
\end{remark}

\section{Restriction to \texorpdfstring{$\GL_{2} \times \mathrm{GSp}_{4}$}{GL2 × GSp4}} 
\label{firstrestrictions} In what follows, we will denote
\begin{align}  \label{heckepoly} \mathfrak{H}  =  \mathfrak{H}(1) 
 &  =  ( 1 + \rho^{8} ) (K) - ( 1 + \rho^{6} )  (K\rho K) + (1  +  2 \rho^{2}  +  \rho^{4} ) \mathfrak{A} - ( 1  + \rho^{2} ) \mathfrak{B}  + \mathfrak{C}  
\end{align} 
considered as an element of $ \mathcal{\mathcal{C}}_{\ZZ} ( K \backslash G /  K ) $. Note that $ \mathfrak{H} \equiv \mathfrak{H}_{\mathrm{spin}, c}(1)  $ modulo $ q - 1 $  for all $ c \in \ZZ $ by Proposition \ref{Gsp6Heckepolynomial}. Note also that $ \rho^{k }$ for even $ k $ is an element of $ H $ (and $ H'$).    We wish to write the $ \tH $-restrictions of $ \mathfrak{H} $. To this end, let us introduce the following elements in $ G $: 
$$ 
\tau_{0} = 1_{G}, 
\quad\quad
\tau_{1} = \left(\begin{smallmatrix}  
\varpi & & & &  1 \\ 
&\varpi & & 1 \\
& & \varpi  & & \\
& & & 1 \\
& & & & 1 \\ 
& &  & & & 1 
\end{smallmatrix} 
\right ), 
\quad \quad 
\tau_{2} = \left(\begin{smallmatrix}
\varpi \, \,  & & & & \mfrac{1}{\varpi}   \\   
&     \varpi \,  \, & & \mfrac{1}{\varpi}  \\[0.1em] 
& &  1 \, \,  &  \\   
& & &   \mfrac{1}{\varpi} \\
& & & &  \mfrac{1}{\varpi}      \\ 
& & & & & 1  
\end{smallmatrix}\right). $$
For $ w \in W_{I} $, we denote  $ \mathscr{R}(w) =    U' \backslash K w K / K $. When listing elements of $ \mathscr{R}(w) $, we will only write the representative element and it will be understood that no two elements represent the same double coset. Similar convention will be used for other double coset spaces.    
\begin{proposition}    
\setlength\itemsep{0.4em} With notations and conventions as above,    
\begin{itemize}    
\item $ \mathscr{R}(\rho)= 
\left\{\varpi^{(1,1,1,1)},\, 
\tau_{1} \right \}$,     
\item $ \mathscr{R}(\upsilon_{0}\rho^{2})=
\left \{\varpi^{(2,2,1,1)},\,
\varpi^{(2,1,2,1)},\,
\varpi^{(1,1,0,0)}\tau _{1}\right \}   $,   
\item $ \mathscr{R} (\upsilon_{1}\rho^{2})=
\left\{\varpi^{(2,2,2,1)},\, \varpi^{(2,1,2,2)},\,   \varpi^{(1,1,1,0)}\tau _{1},\, \varpi^{(1,1,0,1)}\tau_{1}, \, 
\varpi^{(2,1,1,1)}\tau_{2}\right\}$,
\item $ \mathscr{R} (\upsilon_{2}\rho^{3})
=\left\{\varpi^{(3,3,2,2)},\, \varpi^{(3,2,3,2)},\, \varpi^{(2,2,1,1)}\tau_{1},\,  \varpi^{(2,1,2,1)}\tau_{1},\, 
\varpi^{(2,2,0,1)}\tau_{1},\, 
\varpi^{(2,1,1,2) }\tau_{1},\,     
\varpi^{(3,2,1,2)}\tau_{2}  
\right \},$   
\item $ \mathscr{R} (\upsilon_{3}\rho^{4})
=\left\{\varpi^{(4,3,3,3)},\, 
\varpi^{(3,2,2,2)}\tau_{1},\, 
\varpi^{(4,2,2,3)}\tau_{2} \right \},$
\item $ \mathscr{R}
 (\upsilon_{4}\rho^{4})=
\left\{\varpi^{(4,4,2,2)},\, \varpi^{(4,2,4,2)}, \, \varpi^{(3,3,1,1) }\tau_{1},\, 
\varpi ^ {(3,2,0,1)}\tau_{1},\,       \varpi^{(4,3,1,2)}\tau_{2}\right\} $. 
\end{itemize}      
\label{GSp6decompose}
Moreover,  $ H'\tau_{i} K \in H '   \backslash G / K $ are pairwise distinct for $ i = 0 , 1 , 2  $.   
\end{proposition}      
\begin{proof}  
A proof of this  is provided   in  \S \ref{U'orbitssec}.  
\end{proof} 
\begin{remark} A quick check on our lists of  representatives for each $ \mathscr{R}(w) $ above is through computing their classes in $ K \backslash G / K $. These   should  return $ \varpi^{\lambda} $ on the diagonal where   $ \lambda   $ corresponds to $ w $ in  Lemma    \ref{lminwords}. The distinctness of our representatives  is also  easily checked using a  Cartan style decomposition proved in \S \ref{Cartansec}.  What is difficult however  is establishing that these  represent   \emph{all} the orbits of $ U' $ on $ K w K / K $ and this is where bulk of the work lies.  
\end{remark} 
\begin{corollary}  
$ H ' \backslash H ' \cdot  \supp (\mathfrak{H} ) / K =   \left \{ \tau_{0}, \tau_{1}, \tau_{2} \right \} $.  In particular if $ g \in G $ is such that  $ H'gK \neq H'\tau_{i}K $ for $ i = 0,1,2 $, then $ (H', g) $-restriction of $ \mathfrak{H} $ is zero.     
\end{corollary} 
\begin{proof} 
The is clear  from the expression (\ref{heckepoly})  and Proposition  \ref{GSp6decompose}.    
\end{proof}
For $ i = 0 , 1 , 2 $, we let $  \mathfrak{a}_{i} $, $   \mathfrak{b}_{i} $, $ \mathfrak{c}_{i} $, $  \tfh_{i}  \in  \mathcal{C}_{\ZZ}(U' \backslash H' /  H_{\tau_{i}}') $ denote the $ (\tH, \tau_{i}) $-restriction of $ \mathfrak{A}$, $ \mathfrak{B} $, $ \mathfrak{C}$, $ \mathfrak{H} $ respectively.    
Here for $ g \in G $, $ \tH_{g} $ denotes   the compact open subgroup  $ H ' \cap g K g^{-1} $ of $ H ' $.   As before, we omit writing $ \ch $ for characteristic functions.  By Proposition \ref{GSp6decompose}, we have $$ (K \rho K ) = (U' \varpi^{(1,1,1,1)}K) + ( U' \tau_{1}K ) . $$  Since $ U' \varpi^{\lambda} K  \subset H' K $ for any $ \lambda \in \Lambda $ and $ U' \tau_{1} K \subset H' \tau_{1} K $,  the $ (H', \tau_{i}) $-restrictions of $ (K \rho K) $ for $ i = 0, 1, 2 $  are given by    $$  (U' \varpi^{(1,1,1,1)} U') , \quad   \quad  (U' H_{\tau_{1}}') , \quad  \quad   0   $$ respectively. Proceeding in a similar fashion, we find that  \\[-1em] 
\begin{align*} 
\mathfrak{a}_{0} & = (U' \varpi^{(2,2,2,1)} U' ) + (\tU \varpi^{(2,1,2,2)}  \tU) +  2 (\tU \varpi^{(2,2,1,1)}\tU) + 2  (\tU \varpi^{(2,1,2,1)}  \tU) + 4  (\tU \varpi^{(2,1,1,1)} \tU  ),\\[0.1em]  
\mathfrak{a}_{1} & = (\tU  \varpi^{(1,1,1,0)}   \tH_{\tau_{1}} ) +  (U ' \varpi^{(1,1,0,1)}    \tH_{\tau_{1}} )  +  2 (\tU  \varpi^{(1,1,0,0)} \tH_{\tau_{1}} ), \\[0.1em]  
\mathfrak{a}_{2} & = (\tU \varpi^{(2,1,1,1)}  \tH_{\tau_{2}})  ,   \\[0.1em]
\mathfrak{b}_{0} & =  (\tU \varpi^{(3,3,2,2)} \tU ) + (\tU \varpi^{(3,2,3,2)}  \tU) +  4  (\tU \varpi^{(3,2,2,2)} \tU ),  \\[0.1em] 
\mathfrak{b}_{1} & =  (\tU  \varpi^{(2,2,1,1)}   \tH_{\tau_{1}} ) +  (\tU  \varpi^{(2,1,2,1)}    \tH_{\tau_{1}} )  +  (\tU  \varpi^{(2,2,0,1)} \tH_{\tau_{1}} )  +   (\tU  \varpi^{(2,1,1,2)} \tH_{\tau_{1}} )  + 4   (\tU  \varpi^{(2,1,1,1)} \tH_{\tau_{1}} ) ,   \\[0.1em] 
\mathfrak{b}_{2} & =   (\tU  \varpi^{(3,2,1,2)}  H_{\tau_{2}} ' )  , \\[0.1em]
\mathfrak{c}_{0} & =   (\tU   \varpi^{(4,3,3,3)}  \tU  ) + (\tU \varpi^{(4,4,2,2)}   \tU  ) +   (\tU \varpi^{(4,2,4,2)} \tU  )   ,  \\[0.1em]
\mathfrak{c}_{1} & =    (\tU   \varpi^{(3,2,2,2)} \tH_{\tau_{1}} ) + (\tU \varpi^{(3,3,1,1)}   \tH_{\tau_{1}}) +   (\tU \varpi^{(3,2,0,1)} \tH_{\tau_{1}} )  ,  \\[0.1em]
\mathfrak{c}_{2} & =     (\tU    \varpi^{(4,2,2,3)} \tH_{\tau_{2}} ) + (\tU \varpi^{(4,3,1,2)}   \tH_{\tau_{2}}). 
\end{align*} 
Using expression (\ref{heckepoly}), we find that 
\begin{align} 
\label{tfh0}  \tfh_{0} &=  
(1 + \rho^{8}) (\tU) - (1+\rho^{6}) (\tU \varpi^{(1,1,1,1)} \tU ) +(1+ 2 \rho^{2} + \rho^{4}) \mathfrak{a}_{0} -(1+\rho^{2}) \mathfrak{b}_{0} +  \mathfrak{c}_{0} , 
\\
 \tfh _{1}   & =    - ( 1 + \rho^{6} ) (U   ' H_{\tau_{1}}') +  ( 1 +  2 \rho^{2} + \rho^{4} ) \mathfrak{a}_{1} - ( 1 + \rho^{2} ) \mathfrak{b}_{1}   +  \mathfrak{c}_{1}  ,         \label{tfh1}  \\
\label{tfh2} \tfh_{2}     &  =   ( 1 + 2 \rho^{2}  +  \rho^{4} )  \mathfrak{a}_{2}  - ( 1 + \rho^{2} ) \mathfrak{b}_{2}  + \mathfrak{c}_{2}  
\end{align}
where the the central elements $ \rho^{2k} $ distribute over Hecke operators as before. 
\begin{remark} The particular choice of $ \tau_{1}, \tau_{2} $ is motivated by the structure of the group $ H' \cap \tau_{i} K \tau_{i}^{-1} $ which  is  convenient for decomposing double cosets involving these groups (see \S  \ref{H'strucsec}). Note that $\tau_{i}$  very closely related to the ``Schr\"{o}der's representatives" for the double coset $ H'\backslash G/ K  $  given  in  \cite[Chapter 12]{Endoscopy}.  
\end{remark} 
\section{Restriction to \texorpdfstring{ $\GL_{2}\times\GL_{2}\times\GL_{2}$}{GL2 × GL2  × GL2 }}     \label{secondrestrictions}     In this section, we  record  the twisted restrictions of $ \tfh_{0} $, $ \tfh_{1} $, $ 
 \tfh_{2} $ with respect to $ H $.  For $ i = 0, 1, 2 $ and $ h \in H' $, we let   $ \mathscr{R}_{i}(h) $ denote the  double coset space $   U  \backslash U' h H_{\tau_{i}}'/ H_{\tau_{i}}' $. The convention  used  in \S \ref{firstrestrictions} for listing elements of double  coset spaces will also be be applied to $ \mathscr{R}_{i}(h) $.    

\subsection{$H$-restrictions of $\tfh_{0}$}  To write the restrictions of $ \mathfrak{h}_{0} $, we  introduce the following elements of $H '= \GL_{2}(F) \times_{F^{\times}  }  \mathrm{GSp}_{4}(F) $:
\begin{equation}    \label{varrhomatrices} \quad \quad   \varrho_{0} = 1_{H'}, \quad    \varrho_{1}  = \left (  \scalebox{0.9}{$ \left ( \begin{matrix} \varpi \\ &  1   \end{matrix} \right )$} ,  \left ( \begin{smallmatrix}   \varpi & &  & 1  \\ & \varpi &  1 \\[0.3em] &  & 1 \\ &  & &1 \end{smallmatrix} \right ) \right ) , \quad  \varrho _{2} =   
\left (  \scalebox{0.9}{$\left ( \begin{matrix} \varpi \\ &  \varpi   \end{matrix} \right )$} ,    
\left ( \begin{smallmatrix} 
\varpi^{2}  & & &   1   \\ 
& \varpi^{2}  & 1 \\[0.3em]   
& & 1 \\ & & &  1  \end{smallmatrix}  \right )  \right ) 
\end{equation} 
which we also view as elements of $ G  $ via $ \iota' $.   
\begin{proposition}   \label{R0scrhard}    With notations  and conventions  as above, we have  
\begin{itemize}   
\setlength\itemsep{0.1em}   
\item $ \mathscr{R}_{0}(\varpi^{(1,1,1,1)}) =  \left \{
\varpi^{(1,1,1,1)},\, 
\varrho_{1} \right \} $,   
\item $ \mathscr{R}_{0}(\varpi^{(2,2,2,1)})=  \left \{ \varpi^{(2,2,2,1)},\,  \varpi^{(2,2,1,2)},\,   \varpi^{(1,1,1,0)}\varrho_{1} \right \}, $ 
\item $ \mathscr{R}_{0}(\varpi^{(2,1,2,2)}) = \left \{ \varpi^{(2,1,2,2)},\,   \varpi^{(1,0,1,1)}\varrho _{1},\,   \varrho_{2} \right \}, $ 
\item  $ \mathscr{R}_{0}(\varpi^{(3,2,3,2)} ) = \left \{  \varpi^{(3,2,3,2)},\,   \varpi^{(3,2,2,3)},\,    \varpi^{(2,1,2,1)}\varrho_{1}, \,   
\varpi^{(2,1,1,2)}\varrho_{1}, \, \varpi^{(2,1,2,0)}\varrho_{1}, \,   \varpi^{(1,1,0,1)}\varrho_{2} \right \}, $  
\item  $ \mathscr{R}_{0}(\varpi^{(4,2,4,2)} ) = 
\left \{\varpi^{(4,2,4,2)},\,  \varpi^{(4,2,2,4)},\, 
\varpi^{(3,1,3,1)}\varrho_{1},\,  \varpi^{(3,1,1,3)}\varrho_{1},\, \varpi^{(2,1,2,0)}\varrho_{2}        \right \}.   $
\end{itemize}
Moreover $ H \varrho_{i} U' \in H \backslash H' / U'  $ are pairwise distinct for $ i = 0 , 1 , 2 $.  
\end{proposition} 
\begin{proof}   A  proof of this is given in \S\ref{UU'orbsubsec}. 
\end{proof}
By Lemma \ref{pVbijection}, the representatives of $ \mathscr{R}_{0}(\varpi^{\lambda}) $ depend only on those for $ U_{2} \backslash U_{2}' \varpi^{\pr_{2}(\lambda)} U_{2}'/U_{2}' $. Then one easily obtains the following from Proposition   \ref{R0scrhard}.  
\begin{corollary}  \label{R0screasy} We  have 
\begin{itemize} 
\item $ \mathscr{R}_{0}(\varpi^{(2,2,1,1)} ) = \left \{ \varpi^{(2,2,1,1)} \right \}   ,   $
\item $ \mathscr{R}_{0}(\varpi^{(2,1,2,1)}) =  \left \{ \varpi^{(2,1,2,1)},\, \varpi^{(2,1,1,2)},\,  \varpi^{(1,0,1,0)}\varrho_{1} \right \} , $
\item $ \mathscr{R}_{0}(\varpi^{(3,3,2,2)}) = \left \{  \varpi^{(3,3,2,2)}, \,  
\varpi^{(2,2,1,1)} \varrho_{1} \right \}  ,    $
\item  $ \mathscr{R}_{0}(\varpi^{(4,3,3,3)}) =   \left \{    \varpi^{(4,3,3,3)}, \,  \varpi^{(3,2,2,2)}\varrho_{1}, \,  \varpi^{(2,2,1,1)}\varrho_{2}    \right \}  , $  
\item  $ \mathscr{R}_{0}(\varpi^{(4,4,2,2)}) =   \left \{  \varpi^{(4,4,2,2)}   \right \}   . $   
\end{itemize}
\end{corollary} 
The last two results describe the the $U$-orbits of all the double coset spaces arising from  (\ref{tfh0}) up to translation by the central element $ \rho^{2} $. 
This implies the next  claim.  
\begin{corollary} $ H \backslash H \cdot \supp(\mathfrak{h}_{0}) / U'  = \left \{ \varrho _{0}, \varrho _{1}, \varrho _{2} \right \} $.    
\end{corollary}    
For $ i = 0 ,1 , 2 $, we let $ \mathfrak{a}_{\varrho_{i}} ,  \mathfrak{b}_{\varrho_{i}} , \mathfrak{c}_{\varrho_{i}} ,  \mathfrak{h}_{\varrho_{i}} \in \mathcal{C}_{\ZZ}(U \backslash H /  H_{\varrho_{i}}) $ denote the $ (H, \varrho_{i})$-restriction of $ \mathfrak{a}_{0}, \mathfrak{b}_{0}, \mathfrak{c}_{0}$, $ \tfh_{0} $ respectively where as before,  we let $ H_{\varrho_{i}} $ denote $ H \cap \varrho_{i}   K   \varrho_{i} ^{-1}   $ as before. From    Proposition \ref{R0scrhard} and Corollary \ref{R0screasy}, we find that \\[-0.8em]    
\begin{align*} 
\mathfrak{a}_{\varrho_{0}} & =
(U\varpi^{(2,2,2,1)}U)+(U\varpi^{(2,2,1,2)}U)
+(U\varpi^{(2,1,2,2)}U) +  2 (U \varpi^{(2,2,1,1)}U) + 2  (U \varpi^{(2,1,2,1)}  U ) \,  \\
&  \quad  +  2   (U \varpi^{(2,1,1,2)}   U )   +   
4  (U \varpi^{(2,1,1,1)} U ) , \\ 
\mathfrak{a}_{\varrho_{1}} & = (U   \varpi^{(1,1,1,0)} H_{\varrho_{1}}  ) +  (U  \varpi^{(1,0,1,1)}    H_{\varrho_{1}})  +  2  ( U \varpi^{(1,0,1,0)}  H_{\varrho_{1}}  ) ,  
 \\ 
\mathfrak{a}_{\varrho_{2}} & = (U H_{\varrho_{2}})  ,  \\
\mathfrak{b}_{\varrho_{0}} & =  (U  \varpi^{(3,3,2,2)} U ) +  (U \varpi^{(3,2,3,2)}   U )  +  (U \varpi^{(3,2,2,3)} U )  + 4   ( U  \varpi^{(3,2,2,2)} U )  ,     \\
\mathfrak{b}_{\varrho_{1}} & =  (U  \varpi^{(2,2,1,1)}   H _{\varrho_{1}} ) + ( U  \varpi^{(2,1,2,1)  }  H_{\varrho_{1}} )  +  ( U \varpi^{(2,1,1,2)} H_{\varrho_{1}} )   +   (U \varpi^{(2,1,2,0)} H_{\varrho_{1}} )  + 4  (U \varpi^{(2,1,1,1)} H_{\varrho_{1}})   ,    \\
\mathfrak{b}_{\varrho_{2}} & =   (U  \varpi^{(1,1,0,1)}  H_{\varrho_{2}})    ,   \\
\mathfrak{c}_{\varrho_{0}} & =   ( U  \varpi^{(4,3,3,3)}   U ) + (U  \varpi^{(4,4,2,2)}   U  ) +   (U \varpi^{(4,2,4,2)} U  )      + ( U \varpi^{(4,2,2,4)} U )    ,   \\
\mathfrak{c}_{\varrho_{1}} & =    (U \varpi^{(3,2,2,2)} H_{\varrho_{1}} ) + (U \varpi^{(3,1,3,1)}   H _{\varrho_{1}}) +   (U \varpi^{(3,1,1,3)} H _{\varrho_{1}} )   ,   \\
\mathfrak{c}_{\varrho_{2}} & =     (U  \varpi^{(2,2,1,1)} H_{\varrho_{2}} ) + (U \varpi^{(2,1,2,0)} H_{\varrho_{2}})  . 
\end{align*}
From the expression (\ref{tfh0}), we  get
\begin{align} \label{tfhvarrho0}            
\tfh_{\varrho_{0}} &=  
(1 + \rho^{8}) (U) - (1+\rho^{6}) ( U 
   \varpi^{(1,1,1,1)}  U ) +(1+ 2 \rho^{2} + \rho^{4}) \mathfrak{a}_{\varrho_{0}} -(1+\rho^{2}) \mathfrak{b}_{\varrho_{0}} +  \mathfrak{c}_{\varrho_{0}}  , 
\\ \label{tfhvarrho1}     \tfh _{\varrho_{1}}   & =    - ( 1 + \rho^{6} ) (U H_{\varrho_{1}}) +  ( 1 +  2 \rho^{2} + \rho^{4} ) \mathfrak{a}_{\varrho_{1}} - ( 1 + \rho^{2} ) \mathfrak{b}_{\varrho_{1}}    +  \mathfrak{c}_{\varrho_{1}}    ,       \\
\label{tfhvarrho2} 
 \tfh_{\varrho_{2}}    &  =   ( 1 + 2 \rho^{2}  +  \rho^{4} )  \mathfrak{a}_{\varrho_{2}}  - ( 1 + \rho^{2} ) \mathfrak{b}_{\varrho_{2}}  + \mathfrak{c}_{\varrho_{2}}   .   
\end{align}
\subsection{$H$-restrictions of $\tfh_{1}$} 
We consider  the following elements in $ H ' $: 
\begin{equation}  \label{varsigmamatrices}     \hspace{-1em} 
\sigma_{0} = 1_{H'} , \quad \quad    
\sigma_{1} =  w_{2} ,  \quad \quad  
\sigma_{2 } =  \varrho_{1} \varpi^{-(1,1,1,1)}, \quad \quad   
\sigma_{3} =  \varrho_{1} . 
\end{equation} 
where $ \varrho_{i} $ are as in (\ref{varrhomatrices}).  For $ i = 0 , 1, 2 , 3 $, let $ \varsigma_{i} \in G $ denote $ \sigma_{i} \tau_{1} $. Also let   $  \psi  =   \left (  \begin{psmallmatrix} 1 \\[0.1em]  1 & 1 \end{psmallmatrix} , 1  _ { H_{2}}   \right )   \in H  $.  
\begin{proposition}  \label{R1scrhard} With notations and conventions as above, we have  
\begin{itemize}
\item $\mathscr{R}_{1}(\varpi^{(1,1,1,0)})  =   \left \{ \varpi^{(1,1,1,0)},\, \varpi^{(1,1,0,1)}\sigma_{1},\, \varpi^{(1,1,1,0)}\sigma_{2} 
\right \}  $, 
\item $\mathscr{R}_{1}(\varpi^{(1,1,0,1)})=  \left \{
\varpi^{(1,1,0,1)},\,   \varpi^{(1,1,1,0)}\sigma_{1},\,
\varpi^{(1,0,0,0)}\sigma_{2},\,
\varpi^{(1,1,0,1)}\sigma_{2},\,
\varpi^{(1,0,1,1)}\sigma_{2},\,
\varpi^{-(0,1,0,0)} \sigma_{3} 
\right \}  $, 
\item $\mathscr{R}_{1}(\varpi^{(1,1,0,0)})= \left\{
\varpi^{(1,1,0,0)},\,  \varpi^{(1,1,0,0)}\sigma_{1},\,  \varpi^{(1,0,1,0)}\sigma_{2} \right \} $,
\item $\mathscr{R}_{1}(\varpi^{(2,2,1,1)})= \left\{
\varpi^{(2,2,1,1)},\, 
\varpi^{(2,2,1,1)}\sigma_{1},\,   \varpi^{(2,2,1,1)}\sigma_{2},\, \varpi^{(2,0,1,1)}\sigma_{2}
\right\} $,
\item $\mathscr{R}_{1}(\varpi^{(2,1,2,1)})= \left \{  \varpi^{(2,1,2,1)},\,  
\varpi^{(2,1,1,2)}\sigma_{1},\,
\varpi^{(2,1,2,1)}\sigma_{2},\,
\varpi^{(2,1,2,0)}\sigma_{2},\,
\varpi^{(2,1,1,0)}\sigma_{2},\,  \varpi^{(1,0,1,0)}\sigma_{3}  \right \}, $
\item $\mathscr{R}_{1}(\varpi^{(2,2,0,1)})=  \left \{
\varpi^{(2,2,0,1)},\,
\varpi^{(2,2,1,0)}\sigma_{1},\,
\varpi^{(2,0,2,1)}\sigma_{2},\,
\varpi^{(2,0,2,0)}\sigma_{2},\,
\varpi^{(2,0,1,0)}\sigma_{2},\, \varpi^{(1,-1,1,0)}\sigma_{3} \right \}$,
\item $\mathscr{R}_{1}(\varpi^{(2,1,1,2)})=   \left \{
\varpi^{(2,1,1,2)},\,
\varpi^{(2,1,2,1)}\sigma _{1},\,    
\varpi^{(2,1,0,1)}\sigma_{2},\,  \varpi^{(2,1,1,2)}\sigma_{2},\,  \varpi^{(1,0,0,1)}\sigma_{3}
\right \},$
\item $\mathscr{R}_{1}(\varpi^{(2,1,1,1)})=   \left \{
\varpi^{(2,1,1,1)},\,
\varpi^{(2,1,1,1)}\sigma_{1},\,   \varpi^{(2,1,1,1)}\sigma_{2}
\right \},$
\item $ \mathscr{R}_{1}(\varpi^{(3,2,2,2)}) = $ \begin{minipage}[t]{1.0\textwidth}$ 
\left \{\varpi^{(3,2,2,2)},\,  \varpi^{(3,2,2,2)}\sigma_{1},\, \varpi^{(3,2,1,1)}\sigma_{2},\, \varpi^{(3,2,2,2)}\sigma_{2},\,  \varpi^{(3,1,1,2)}\sigma_{2},\, 
\varpi^{(3,2,1,2)}\psi\sigma_{2}, \, \varpi^{(2,1,1,1)}\sigma_{3}\right \},$ 

\end{minipage} 
\item $\mathscr{R}_{1}(\varpi^{(3,3,1,1)})= \left \{  \varpi^{(3,3,1,1)},\,  \varpi^{(3,3,1,1)}\sigma_{1},\, \varpi^{(3,0,2,1)}\sigma_{2}  \right \}  ,  $
\item $\mathscr{R}_{1}(\varpi^{(3,2,0,1)}) = \left \{  \varpi^{(3,2,0,1)}, \,  \varpi^{(3,2,1,0)}\sigma_{1},\,  \varpi^{(3,1,3,1)}\sigma_{2},\,  \varpi^{(3,1,2,0)}\sigma_{2},\, \varpi^{(2,0,2,0)}\sigma_{3}\right \}   $. 
  
\end{itemize}
Moreover  $ H\sigma_{i}H_{\tau_{1}}'  \in H \backslash  H'  / H_{\tau_{i}}'  $ are pairwise distinct  for $ i = 0,1,2,3 $.     
\end{proposition}    
\begin{proof} A proof  of this  is provided in 
\S \ref{UHtau1orbitsec}. 
\end{proof} 

\begin{remark} We also need $ \mathscr{R}_{1}(1) = \left \{  \sigma_{0},\, \sigma_{1},\, \sigma_{2} \right \} $ but this is obtained from  $ \mathscr{R}_{1}(\varpi^{(2,1,1,1)})$.  
\end{remark} 

\begin{remark} The appearance of $ \psi $ in one of the representatives listed  in $ \mathscr{R}_{1}(\varpi^{(3,2,2,2)}) $ seems unavoidable. Curiously,  $U \varpi^{(3,2,1,2)}\psi H_{\varsigma_{2}} $ is  the only double coset arising form $ \mathfrak{H} $ 
whose degree vanishes modulo $ q - 1 $. See Lemma \ref{exceptionsvarsigma2}. 
\end{remark} 
\begin{corollary} 
$ H \backslash H \cdot \supp(\mathfrak{h}_{1} ) /  H_{\tau_{1}}'  =  \left \{ 
 \sigma_{0}, \, \sigma _{1}, \sigma_{2}, \, \sigma _{3} \right \} $.    
\end{corollary}
For $ i = 0 , 1, 2, 3 $,  let $ \mathfrak{a}_{\varsigma_{i}} $, $ \mathfrak{b}_{\varsigma_{i}} $, $ \mathfrak{c}_{\varsigma_{i}} $, $ \mathfrak{h}_{\varsigma_{i}}  \in  \mathcal{C}_{\mathcal{\ZZ}}(U \backslash  H / H_{\varsigma_{i}})  $ denote the $ (H,\sigma_{i})$-restrictions of 
$\mathfrak{a}_{1},  \mathfrak{b}_{1}, \mathfrak{c}_{1}, 
 \mathfrak{h}_{1}$ respectively.
Proposition    \ref{R1scrhard} implies that  
\begin{align*} 
\mathfrak{a}_{\varsigma_{0}} & = (U \varpi^{(1,1,1,0)} H_{\varsigma_{0}}) + 
(U\varpi^{(1,1,0,1)}H_{\varsigma_{0}})+ 2 (U\varpi^{(1,1,0,0)}H_{\varsigma_{0}}),  \\
\mathfrak{a}_{\varsigma_{2}} &  =   ( U \varpi^{(1,1,1,0)} H_{\varsigma_{2}})   +  ( U \varpi^{(1,0,0,0)} H_{\varsigma_{2}}) +  (U \varpi^{(1,1,0,1)} H_{\varsigma_{2}}) + ( U \varpi^{(1,0,1,1)} H_{\varsigma_{2}}) + 2( U \varpi^{(1,0,1,0)} H_{\varsigma_{2}})  ,  \\
\mathfrak{a}_{\varsigma_{3}} &  =   ( U \varpi^{-(0,1,0,0)} H_{\varsigma_{3}} )   , \\
\mathfrak{b}_{\varsigma_{0}} & =  (U   \varpi^{(2,2,1,1)}   H _{\varsigma_{0}} ) +  (U \varpi^{(2,1,2,1)}   H_{\varsigma_{0}}) + (U\varpi^{(2,2,0,1)} H _{\varsigma_{0}}) + (U \varpi^{(2,1,1,2)} H_{\varsigma_{0}}) + 4 (U\varpi^{(2,1,1,1)}H_{\varsigma_{0}}) , \\
\mathfrak{b}_{\varsigma_{2}} & = 
(U\varpi^{(2,2,1,1)}H_{\varsigma_{2}})+    (U\varpi^{(2,0,1,1)}H_{\varsigma_{2}})+ (U\varpi^{(2,1,2,1)}H_{\varsigma_{2}})+ 
(U\varpi^{(2,1,2,0)}H_{\varsigma_{2}})+ 
(U\varpi^{(2,1,1,0)}H_{\varsigma_{2}})\, +\\
&\quad \,\,  (U\varpi^{(2,0,2,1)}H_{\varsigma_{2}})+(U\varpi^{(2,0,2,0)}H_{\varsigma_{2}})+
(U\varpi^{(2,0,1,0)}H_{\varsigma_{2}})+
(U\varpi^{(2,1,0,1)}H_{\varsigma_{2}})+
(U\varpi^{(2,1,1,2)}H_{\varsigma_{2}})\, +\\
& \quad \,\,  \,   4(U\varpi^{(2,1,1,1)}H_{\varsigma_{2}}),\\
\mathfrak{b}_{\varsigma_{3}} & =  (U  \varpi^{(1,0,1,0)} H_{\varsigma_{3}}) +  (U \varpi^{(1,-1,1,0)}H_{\varsigma_{3}} )  +  (U \varpi^{(1,0,0,1)}H_{\varsigma_{3}})  ,     \\
\mathfrak{c}_{\varsigma_{0}} & = 
(U\varpi^{(3,2,2,2)}H_{\varsigma_{0}})+
(U\varpi^{(3,3,1,1)}H_{\varsigma_{0}})+(U\varpi^{(3,2,0,1)}H_{\varsigma_{0}}),\\  
\mathfrak{c}_{\varsigma_{2}} & =    
(U\varpi^{(3,2,1,1)}H_{\varsigma_{2}})+
(U\varpi^{(3,2,2,2)}H_{\varsigma_{2}})+   (U\varpi^{(3,1,1,2)}H_{\varsigma_{2}})  
+ ( U \varpi^{(3,2,1,2)} \psi  H_{\varsigma_{2}}) 
+ ( U \varpi^{(3,0,2,1)}
H_{\varsigma_{2}}) \, + \\ & \quad \,\,     ( U \varpi^{(3,1,3,1)}    H_{\varsigma_{2}})  + ( U \varpi^{(3,1,2,0)}  H_{\varsigma_{2}}) , \\
\mathfrak{c}_{\varsigma_{3}} & =     (U  \varpi^{(2,1,1,1)} H_{\varsigma_{3}} ) + (U 
 \varpi^{(2,0,2,0)} H_{\varsigma_{3}}). 
\end{align*}
Using expression (\ref{tfh1}), we  get
\begin{align}  \label{tfhvarsig0}    
\tfh_{\varsigma_{0}} &=  -  (1+\rho^{6}) ( U   H_{\varsigma_{0}}  ) +(1+ 2 \rho^{2} + \rho^{4}) \mathfrak{a}_{\varsigma_{0}}  -  (1+\rho^{2}) \mathfrak{b}_{\varsigma_{0}} + \mathfrak{c}_{\varsigma_{0}}   , 
\\  
\label{tfhvarsig2} \tfh _ {\varsigma_{2}}   & =    - ( 1 + \rho^{6} ) (U H_{\varsigma_{2}})  +  ( 1 +  2 \rho^{2} + \rho^{4} ) \mathfrak{a}_{\varsigma_{2}} - ( 1 + \rho^{2} ) \mathfrak{b}_{\varsigma_{2}}     +  \mathfrak{c}_{\varsigma_{2}}     ,    \\
\tfh_{\varsigma_{3} }    &  =   ( 1 + 2 \rho^{2}  +  \rho^{4} )  \mathfrak{a}_{\varsigma_{3}}  - ( 1 + \rho^{2} ) \mathfrak{b}_{\varsigma_{3}}  + \mathfrak{c}_{\varsigma_{3}} 
\intertext{Now observe that each $ \mathscr{R}(\varpi^{\lambda}) $ in Proposition \ref{R1scrhard} contains a unique representative of the form $ \varpi^{s_{2}(\lambda)}  \sigma_{1} $. Moreover $ \varsigma_{1} = w_{2} \varsigma_{0} $ and $ w_{2} $ normalizes  $ U $ (and $ H $).  So  $ w_{2} U \varpi^{\lambda}  H_{\varsigma_{0}}w_{2} = U \varpi^{ s_{2}(\lambda) } H_{\varsigma_{1}} $ for all $ \lambda \in \Lambda $. Therefore }  \label{tfhvarsig1} 
\mathfrak{h}_{\varsigma_{1}} & = w_{2} \mathfrak{h}_{\varsigma_{0}} w_{2} 
\end{align}
where $ w_{2} $ distributes over each double coset characteristic function.

\subsection{$H$-restrictions of $\tfh_{2}$}

For $ i = 0, 1, 2  $, denote  $ \theta_{i} : = \sigma_{i} $ and $ \theta_{3} :  =  \varpi^{-(1,1,1,1)}  \sigma_{3} $ where $ \sigma_{0}, \sigma_{1}, \sigma_{2}, \sigma_{3} $ are as in 
 (\ref{varsigmamatrices}). For $ i = 0 ,1, 2, 3 $, set  
$ \vartheta_{i} = \sigma_{i}  \tau_{2} \in G $.  Additionally for $ k \in  [\kay]^{\circ} := [\kay] \setminus \left \{ -1 \right \} $, we define $  \tilde{\theta}_{k} = (1, \tilde{\eta}_{k})  \in H'   $ where   \begin{equation} \label{etakmatrix}  \tilde{\eta}_{k} =  \scalebox{0.9}{$\begin{pmatrix}     k & 1  \\
k+1 & 1 \\ & & -1 & k + 1   \\[0.1em]   &  & 1 & - k   \end{pmatrix}$} \in H_{2}'  
\end{equation}   
and set $ \tilde{ \vartheta}_{k}   = \tilde{\theta}_{k}  \tau_{2} \in G   $. Note that $ \tilde{\theta}_{0} = w_{2} w_{3} \theta_{2} w_{3}  $ and $ w_{3}  \tau_{2} =  \tau_{2}w_{3}t_{1} $ where $ t_{1} = \mathrm{diag}(1,1,-1,1,1,-1) $. So  $$ \tilde{\vartheta}_{0} = \tilde{\theta}_{0} \tau_{2} =  w_{2} w_{3} \theta_{2} w_{3} \tau_{2} = 
w_{2} w_{3}  \vartheta_{2}  
 w_{3} t_{1}.  $$

\begin{proposition}   \label{R2scrhard}     We have 
\begin{enumerate}[label = $\bullet$]  
 \setlength\itemsep{0.5em}
\item  $  \mathscr{R}_{2} ( \varpi^{(0,0,0,0)} )  = \{ 1,  \, \theta_{1} , \,  
\theta_{2} , \,  
\tilde{\theta}_{k}    \, |\, k \in [\kay]^{\circ}  
  \}   $,  
\item  $ \mathscr{R}_{2}(\varpi^{(3,2,1,2)}) =$ 
\begin{minipage}[t]{1.0\textwidth}
$\{
\varpi^{(3,2,1,2)},\,
\varpi^{(3,2,2,1)}\theta_{1},\,  
\varpi^{(3,2,1,2)}\theta_{2},\, 
\varpi^{(3,1,2,1)}\theta_{2},\,  
\varpi^{(3,1,2,2)}\theta_{2},\,
\varpi^{(3,1,2,2)}\theta_{3}\}\cup \\[0.3em]  
\{
\varpi^{(3,1, 2, 2)} \tilde{\theta}_{0},\,  
\varpi^{(3,1,1,2)}\tilde{\theta}_{0},\,  
\varpi^{(3,2,1,1)}\tilde{\theta}_{k} 
\, | \,  k \in   [\kay]^{\circ} \}, $
\end{minipage} 
\item  $ \mathscr{R}_{2}(\varpi^{(4,3,1,2)})   =   \{   \varpi^{(4,3,1,2)}, \, 
\varpi^{(4,3,2,1)}\theta_{1}  ,    \varpi^{(4,1,3,2)}\theta_{2} ,     \,  \varpi^{(4,1,2,3)}\tilde{\theta}_{0},   \,  \varpi^{(4,1,3,2)}\theta_{3}          \},      $       
\item $\mathscr{R}_{2}(\varpi^{(4,2,2,3)}) =  \{
\varpi^{(4,2,2,3)},\,   \varpi^{(4,2,3,2)}\theta_{1},\, \varpi^{(4,2,2,3)}\theta_{2},\, \varpi^{(4,2,1,2)}\tilde{\theta}_{0},\,    \varpi^{(4,2,2,3)}\theta_{3} \} $.
\end{enumerate}
\end{proposition}   
\begin{proof} The proof of this result is provided in \S \ref{UHtau2'orbitssec}   
\end{proof}  
\begin{corollary} $ H \backslash  H \cdot \, \supp(\tfh_{2}) / H_{\tau_{2}}'   =   \{ \theta_{0}, \theta_{1}, \theta_{2} , \theta_{3} , \tilde{\theta}_{k} \, | \, k \in [\kay]^{\circ}    \} $.  
\end{corollary}  
\begin{proof} This follows by Lemma  \ref{distinctH2E} and  Proposition  \ref{R2scrhard}.  
\end{proof} 
For $ \vartheta \in 
\{\vartheta_{0}, \vartheta_{1}, \vartheta_{2}, \vartheta_{3}, \tilde{\vartheta}_{k} \, | \, k \in  [\kay]^{\circ} 
\} $, we  let $ \mathfrak{h}_{\vartheta}  \in   \mathcal{C}_{\mathcal{\ZZ}}(U \backslash  H / H_{\vartheta}   )   $ denote the $ (H,\vartheta\tau_{2}^{-1}) $-restriction of $   \mathfrak{h}_{2} $. By the results above,
\begin{align}
\mathfrak{h}_{\vartheta_{0}}  &  =   
(\rho^{2} + 2\rho^{4} +  \rho^{6})  (U H_{\vartheta_{0}} ) -  ( 1 + \rho^{2})     (U  \varpi^{(3,2,1,2)}  H_{\vartheta_{0}} )  +   ( U \varpi^{(4,2,2,3)} H_{\vartheta_{0}})  + ( U \varpi^{(4,3,1,2)}H_{\vartheta_{0}})  ,   \\[0.1em]  
\,  \mathfrak{h}_{\vartheta_{2}}  & =   
( \rho^{2} + 2 \rho^{4} +  \rho^{6} )  (  U H_{\vartheta_{2}  }  )   - ( 1 + \rho^{2} )  
\big (     (  U    \varpi^{(3,2,1,2)}   H _{\vartheta_{2}  }  )  +   ( U   \varpi^{(3,1,2,1)}   H _{\vartheta_{2}  }  )  +    ( U    \varpi^{(3,1,2,2)}   H _{\vartheta_{2}  }  ) 
\big  )   \,  +    \\[0.1em]  
 & \quad  \,  \,      ( U \varpi^{(4,2,2,3)}  H_{\vartheta_{2} } )   + ( U \varpi^{(4,1,3,2)}  H _ { \vartheta_{2} }  )  ,     \\
 \mathfrak{h}_{\vartheta_{3}} &  =  
 ( U \varpi^{(4,2,2,3)}   H _{\vartheta_{3}} )   +  (U \varpi^{(4,1,3,2)}   H _{\vartheta_{3}} )    - ( 1 + \rho^{2} ) (U \varpi^{(3,1,2,2)}  H_{\vartheta_{3}} ),  \\[0.1em]   
\mathfrak{h}_{\tilde{\vartheta}_{k}}    &    =   (\rho^{2} + 2\rho^{4} + \rho^{6} ) ( U H_{\tilde{\vartheta}_{k}}  ) - ( 1 + \rho^{2} )   (   U \varpi^{(3,2,1,1)} H_{\tilde{\vartheta}_{k}}) \\ 
\intertext{ where $ k \in [\kay] \setminus \left \{ 0 , - 1 \right \} $. Observe that  $ H_{\vartheta_{1}} = w_{2} H_{\vartheta_{0}} w_{2} $ and that  in  each set appearing in  Proposition  \ref{R2scrhard}, $ \varpi ^{\lambda} $ for some $ \lambda \in \Lambda $ is listed in that set if and only if $ \varpi^{s_{2}(\lambda)} \theta_{1} $ is.    
So as in the case of $ \mathfrak{h}_{\varsigma_{1}} $, we have    
} 
\mathfrak{h}_{\vartheta_{1}}  &   = w _{2} \mathfrak{h}_{\vartheta_{0}} w _{2}.
\intertext{Similarly we have $ H_{\vartheta_{2}} = w_{2} w_{3} H_{\tilde{\vartheta}_{0}} w_{3} w_{2}  $ and $ \varpi^{\lambda} \vartheta_{2} $ appears in Proposition  \ref{R2scrhard} if and only if $ \varpi^{s_{2}s_{3}(\lambda)} \tilde{\vartheta}_{0}$ does. Therefore  
}    
\mathfrak{h}_{\tilde{\vartheta}_{0}}   &  =  w_{2} w_{3} \mathfrak{h}_{\vartheta_{2}} w_{3} w_{2} .
\end{align} 

\section{Horizontal norm relations} 
\label{HNR}

Let $ X  = \mathrm{Mat}_{2\times 1}(F) $ be  the $ F$-vector space of  size $ 2 $ column vectors over $ F $. We view $ X $ as  a locally compact totally disconnected topological  vector 
 space. Define a right action  $ X \times H   \to H $, $  
(\vec{v}, h)     \mapsto \pr_{1}(h)^{-1} \cdot \vec{v}  $ where  dot denote matrix multiplication. Let $ \mathcal{\OO} $ be an integral domain in which $ \ell $ is invertible and let $ \mathcal{S}_{X} =  \mathcal{S}_{X,\OO} $ denote the $ \OO$-module  of all locally constant compactly supported functions $ X \to \mathcal{O} $. Then $ \mathcal{S}_{X} $  inherits a smooth left  $H  $-action. We define $$ \phi = \ch \left (  \begin{smallmatrix} \Oscr_{F} \\ \Oscr_{F}  \end{smallmatrix}  \right )  \in \mathcal{S}_{X} .   $$  For any compact open subgroup $ V $ of $ H$, we let $ \mathcal{S}_{X}(V) $ denote the submodule $ V $-invariant functions. Let $ \Upsilon_{H} $ denote the collection of all compact open subgroups of $ H $ and  $ \mathcal{P}(H, \Upsilon_{H}) $ denote the category of compact opens (see \cite[\S 2]{CZE}). Then $$  \mathcal{S}_{X} : \mathcal{P}(H, \Upsilon_{H}) \to \mathcal{O}\text{-Mod},  \quad    V \mapsto  \mathcal{S}_{X}(V)   $$
is a cohomological Mackey functor. Note that $ \phi \in \mathcal{S}_{X}(U) $.  For $ g \in G $, let $ H_{g} = H \cap gK g^{-1} $ as before and $ V_{g} \subset H_{g} $ denote the subgroup of all elements $ h \in H_{g} $ such that $ \mathrm{sim}(g) \in 1 +  \varpi  \Oscr_{F} $.   For  $ g \in G $, we denote  by $ \mathfrak{h}_{g} \in \mathcal{C}_{\ZZ}(U \backslash H / H_{g} ) $ the  $ (H,g) $-restriction $ \mathfrak{H} $.  
\begin{theorem} \label{mainzeta} For any $ g \in G $, $ \mathfrak{h}_{g,*}(\phi) $  lies in the image of the trace map $ \pr_{*} :  \mathcal{S}_{X}(V_{g} ) \to  \mathcal{S}_{X}(H_{g} ) $.  
\end{theorem}  

\begin{proof} Since $ \mathfrak{h}_{\eta g \gamma , * } = \mathfrak{h}_{g, *} \circ [\eta]_{H_{g}, H_{\eta g }, * }  $, 
it suffices to  prove the claim for $ g  \in H  \backslash  H  \cdot  \supp(\mathfrak{H}) /  K   $. By the results of the previous section, a complete system of representatives for this double quotient is the set $  \{ \varrho_{0}, \varrho_{1}, \varrho_{2},  \varsigma_{0},  \varsigma_{1},  \varsigma_{2}, 
 \varsigma_{3}, \vartheta_{0}, \vartheta_{1}, \vartheta_{2}, \vartheta_{3},  \tilde{\vartheta}_{k} \, | \, k \in [\kay]^{\circ} \} $.    By the results established in \S  \ref{convolutionsection}, $$ \mathfrak{h}_{g,*}(\phi) \equiv  0 \pmod{q  - 1}  $$ for  all $ g \neq  \vartheta_{3}  $ in this  set and $ \mathfrak{h}_{\vartheta_{3}, *} ( \phi )  =    -  \ch \left (    \begin{smallmatrix} \varpi ^{-1}  \Oscr_{F}^{\times} \\  \varpi^{-2} \Oscr_{F}^{\times}  \end{smallmatrix} \right )   $.  So it suffices to show that $   \chi : =   \ch    \left (  \begin{smallmatrix} 
 \varpi \Oscr_{F}^{\times} \\  \Oscr_{F}^{\times}      \end{smallmatrix}     \right )  \in \mathcal{S}_{X}(H_{\vartheta_{3}}) $  is the trace of a function in $ \mathcal{S}_{X}(V_{\vartheta_{3}}) $. 
 By  \cite[Theorem 3.5.3]{CZE}, it suffices to verify that for all $ \vec{v} \in  \supp(\chi) $, the stabilizer $  \mathrm{Stab}_{H_{\vartheta_{3}}} (\vec{v}) $ of $ \vec{v} $ in $ H_{\vartheta_{3}} $ is contained in $  V_{\vartheta_{3}} $. So let $ \vec{v} = \begin{psmallmatrix}  x \\ y  \end{psmallmatrix} \in  \supp(\chi) $ and $ h = (h_{1}, h_{2}, h_{3} )  \in \mathrm{Stab}_{H_{\vartheta_{3}}}(\vec{v}) $. If we write  $  h_{1} = \begin{psmallmatrix}  a & b \\ c &  d \end{psmallmatrix}  $, then $   \vec{v} \cdot h  = \vec{v} $ is equivalent to $  \vec{v} \cdot  h^{-1} = \vec{v} $ and so     
 \begin{align*}  ( a  - 1 ) x +  by   = 0 , \\
  cx +   ( d - 1 ) y  = 0   .   
  \end{align*}
By Lemma \ref{vart3struclemma}, $ h_{1} \in \GL_{2}(\Oscr_{F}) $  and   $ b  \in \varpi^{2} \Oscr_{F} $. Since $ x \in \varpi \Oscr_{F}^{\times} $, it follows that $ a \in 1 + \varpi \Oscr_{F} $. Similarly 
$ y \in \Oscr_{F}^{\times} $,  $ x \in \varpi \Oscr_{F} ^{\times }$  implies  $ d \in 1 + \varpi \Oscr_{F} $. Thus $ \mathrm{sim}(h) =  ad - bc  \in 1 + \varpi \Oscr_{F } $ and  so  $ h \in  V_{\vartheta_{3}} $.   
\end{proof} 

Now  let $ \tilde{\Gb} : = \Gb \times \GG_{m} $, $ \tilde{G} $ its group of $ F $-points and $ \tilde{K} $ its group of $ \Oscr_{F} $-points. Embed $ \Gb $ into $ \tilde{\Gb }$ via $ 1 \times \mathrm{sim} $ and let $ \tilde{\iota} : \Hb \to \tilde{\Gb} $ denote the embedding $  (1 \times \mathrm{sim}) \circ \iota $. Fix a $ c \in \ZZ $ and define  $$ \tilde{\mathfrak{H}} = \mathfrak{H}_{\mathrm{spin}, c}(\mathrm{Frob}) \in \mathcal{C}_{\ZZ[q^{-1}]}(\tilde{K} \backslash \tilde{G} / \tilde{ K}  ) $$ where $ \mathrm{Frob} = \ch( \varpi \Oscr_{F}^{\times}  ) $.   Let $ \tilde{L}  =  K \times ( 1 + \varpi \Oscr_{F} ) \subset \tilde{K} $. Let $ \Upsilon_{\tilde{G}} $ denote the collection of all compact open subgroups of $ \tilde{G} $ and $ \mathcal{P}(\tilde{G}, \Upsilon_{\tilde{G}}) $ the associated   category.  
\begin{corollary}  \label{mainteoloc}    For any cohomological  Mackey 
functor $ M_{\tilde{G}} : \mathcal{P}(\tilde{G}, \Upsilon_{\tilde{G}}) \to \mathcal{O}\text{-Mod}  $ 
and 
any Mackey  pushforward $ \tilde{\iota}_{*} : \mathcal{S}_{X}  \to M_{\tilde{G}} $, there exists a class $   y    \in M_{\tilde{G}}( \tilde{L}  )  $  such that $$  \tilde{\mathfrak{H}}_{*} \circ  \tilde{\iota}_{U , \tilde{K} ,   *}(\phi  )=  \pr_{ \tilde{L}  , \tilde{K} ,*}(y )  $$  

\end{corollary}

\begin{proof} By the expression in Proposition \ref{Gsp6Heckepolynomial}, it is clear that $ (G, g) $-restriction of $ \tilde{\mathfrak{H}} $ is non-zero only if $ g \in G \tilde{K} $ and the $ (G, 1_{\tilde{G}}  ) $-restriction is $ \mathfrak{H}_{\mathrm{spin},c}(1) $.   The  claim is then a  consequence of  Theorem \ref{mainzeta}, Proposition  \ref{Gsp6Heckepolynomial}  and 
 \cite[Corollary 3.2.13 and 3.2.14]{CZE}. 
\end{proof} 

\subsection{Global relations} \label{HNRglobal}   We now  repurpose our notation for the global setup.  Let $ \Gb $, $ \tilde{ \Gb } = \Gb \times \GG_{m} $, $ \Hb $ be as before.  Fix a set $ S $ of rational  primes. By $ \ZZ_{S} $, be mean the product $ \prod_{\ell \in S } \ZZ_{\ell} $ and by $ \Ab_{f}^{S} $, we mean the group of finite rational adeles away from primes in $ S $.  Let  $ G$,  $ \tilde{G} $, $ H $ denote the group of $ \ZZ_{S} \cdot \Ab_{f}^{S} $ points of $ \Gb $,  $ \tilde{\Gb}   $, $ \Hb $ 
respectively.  Let $ \Upsilon_{\tilde{G}} $ denote the collection of all neat compact open subgroups of $ \tilde{G} $ and $ \Upsilon_{H} $ denote the collection of compact open subgroups of the form $ H \cap \tilde{L} $ where $ \tilde{L} \in  \Upsilon_{\tilde{G}} 
 $. Let $ \mathcal{P}(H, \Upsilon_{H}) $, $ \mathcal{P}(\tilde{G}, \Upsilon_{\tilde{G}}) $ denote the corresponding categories  of compact opens. These satisfy axioms (T1)-(T3) of \cite[\S 2]{CZE}. 
 
Next fix a neat compact open subgroup $ K \subset G $ such that if $  \ell \notin S $ is a rational prime, $ K = K^{\ell} K_{\ell} $ where $ K_{\ell} =  G(\ZZ_{\ell}) $ as before and $ K^{\ell} = K / K_{\ell} \subset \Gb(\Ab_{f}^{\ell} )   $ is the group at primes away from $ \ell $.    Let $ \mathcal{N} $ denote the set of all square free  products of primes away from $ S $  where the empty product  means  $ 1  $.   For each $ n \in  \mathcal{N} $, let $$ K[n]  = K \times  \prod_{\ell \nmid n} \ZZ_{\ell}^{\times}  
\prod_{\ell \mid n } ( 1+ \ell \ZZ_{\ell})  \in \Upsilon_{\tilde{G}} . $$ 
We also denote $ K[1] $ as $ \tilde{K} $.   
Let $ X = \mathrm{Mat}_{2 \times 1} ({\Ab_{f}}) \setminus 
\{ \vec{0} 
\}  
$ and let $ H $ act on $ X $ in a manner analogous to  the local situation. Let $ \mathcal{O} $ be a characteristic zero integral domain such that $ \ell \in \mathcal{O}^{\times} $ for all $ \ell \notin S $. Let  $ \mathcal{S}_{X} =  \mathcal{S}_{X, \mathcal{O}} $ denote the set 
of all 
functions   $ \chi : X  \to \mathcal{O} $ such that $  \chi   = 
f_{S} \otimes \chi^{S }$ where $ f_{S } $ is a fixed locally constant  compactly supported function on  $  \mathrm{Mat}_{2 \times 1} (\ZZ_{S})  $ that is invariant under $  \Hb(\ZZ_{S} ) $ and $  \chi  ^{S} $ is any locally constant compactly supported  function on $ \mathrm{Mat}_{2\times 1}(\Ab_{f}^{S}) $.   Then $$  \mathcal{S}_{X} : \mathcal{P}(H, \Upsilon_{H}) \to  \mathcal{O}\text{-Mod} , \quad   V \mapsto \mathcal{S}_{X}(V) $$ is a CoMack functor with Galois descent.  Let $ U = H \cap \tilde{K} $ and $ \phi \in  \mathcal{S}_{X}(U) $ be the  function  
$
f_{S}  
\otimes   \ch(\widehat{\ZZ}^{S}) $ where $\widehat{\ZZ}^{S} = \prod_{\ell \notin S } \ZZ_{\ell} $ denotes integral adeles away from $ S $. Note that $ \phi^{S} $ is the restricted tensor product of $ \otimes_{\ell \notin S } \phi_{\ell} $ where $ \phi_{\ell}  =    \ch \left ( \begin{smallmatrix}   \ZZ_{\ell} \\  \ZZ_{\ell} \end{smallmatrix} \right )   $. Fix an integer $ c$ and for each $ \ell \in S $, let $$ \tilde{\mathfrak{H}}_{\ell}   =  \mathfrak{H}_{\mathrm{spin},c,\ell}(\mathrm{Frob}_{\ell})  \otimes  \ch( \tilde{K}^{\ell}) \in \mathcal{C}_{\ZZ[\ell^{-1}]}  ( \tilde{K} \backslash \tilde{G} / \tilde{K}) $$
where $ \mathrm{Frob}_{\ell}  =  \ch ( \ell  \ZZ_{\ell} ^  { \times } )    $ is as before.

\begin{theorem}
\label{mainteoglobal}      For any cohomological Mackey functor  $ M_{\tilde{G}} : \mathcal{P}(\tilde{G}, \Upsilon_{\tilde{G}}) \to \mathcal{O}\text{-Mod} $ 
and any Mackey pushforward   $ \tilde{\iota}_{*} :   \mathcal{S}_{X}  \to   M_{\tilde{G}} $,  there exists a collection of classes $ y_{n} \in M_{\tilde{G}}(K[n])   $ indexed by integers $ n \in \mathcal{N} $ such that $ y_{1} = \tilde { \iota } _{U,\tilde{K},*}(\phi) $ and $$ \tilde{\mathfrak{H}}_{*}     
( y_{n}  )   = \mathrm{pr}_{K[n\ell],  K [n],*}(y_{n\ell})   $$
for all $ n , \ell \in \mathcal{N} $ such that $ \ell $ is a prime and $ \ell   \nmid  n  $.   
\end{theorem}    
\begin{proof} Combine  Theorem \ref{mainzeta}, \cite[Theorem 3.4.2]{CZE} and the  results  referred to in Corollary \ref{mainteoloc}.  
\end{proof}  

%% file: DoubleCosets.tex
\part{Proofs} 
\section{Double cosets of   \texorpdfstring{$ \mathrm{GSp}_{6}$}{GSp6}}   \label{U'orbitssec}

Throughout, we maintain the notations introduced in Part 1. 

\subsection{Desiderata} The embedding $ \iota' : \Hb \to \Gb $ identifies the set $ \Phi_{H'} $ of roots of $ \mathbf{H}'$   with   $$    \left \{ \pm \alpha_{0} , \pm \alpha_{2} , \pm \alpha_{3} , \pm ( \alpha_{2} + \alpha_{3}) , \pm ( 2 \alpha_{2} + \alpha_{3} )  \right  \}  \subset \Phi.  
 $$  
The Weyl group $ W'  $ of $ H ' $ is then the subgroup of $ W $ generated by $    s_{0} , s_ {   2    }  ,  s_{3} $ and  $ W'  \cong   S_{2} \times  \left ( (\ZZ / 2 \ZZ ) ^{2}  \rtimes S_{2}   \right  )  $. We let $ \Phi_{H'}^{+} = \Phi^{+} \cap \Phi_{H} $ be the set of positive roots. The base is then $ \Delta_{H'} = \left \{  \alpha_{0}, \alpha_{2}, \alpha_{3} \right \}  $  
and the corresponding Iwahori subgroup $ I' $ of $ H ' $ equals the intersection $ I \cap G $. 
Since the normalizer $ N_{H'} (A) $ of $ A $ in $ H' $ equals the intersection $ N_{G}(A) \cap H' $, the Iwahori Weyl  group 
$ W_{I'} = N_{H'}(A)/A^{\circ} $ is also identified  with a subgroup of $ W_{I}  $. We let $ W_{\mathrm{aff}}' $ denote the affine Weyl group of $ H ' $. 

For notational convenience in referring to the roots  corresponding to the projection $ \Hb_{2}' = \mathrm{GSp}_{4}  $ of $ \Hb'  $, we will denote  
$$  \beta_{0}  =  2 e_{2} - e_{0}  ,  \quad   \quad    \beta_{1} : = e_{2} - e_{3} , \quad  \quad   \beta_{2} = 2 e_{3}  - e_{0}, $$ and  let    $ r_{0}, r_{1}, r_{2} $ denote the reflections associated with $ \beta_{0}, \beta_{1}  ,  \beta_{2} $ respectively. In this notation, the generators of $ W_{\mathrm{aff}}' $ of  $  \Ht $  are given by  $  S_{\mathrm{aff}}'  =   \left  \{   s_{0} , t (  f_{1} ) s_{0},  r_{1}, r_{2} , t( f_{2} ) r_{0}   \right \}  .   $ 
The group $ W_{I'} $ is equals the semidirect product of $ W_{\mathrm{aff}}'  $ with the  cyclic subgroup $ \Omega_{H'} \subset W_{I} $ generated  by $  \omega_{H'}   :=  t(-f_{0}) s_{0} r_{2} r _    {1} r_{2} \in W_{I}  $. The action of $ \omega_{H'}  $ on $ S_{\mathrm{aff}}' $ is given by  $ s_{0} \leftrightarrow t(f_{1}) s_{0} $, $  r_{2} \leftrightarrow t(f_{2}) r_{0}   $  and fixing $ r_{1}  $. It can be  visualized as the  order $ 2 $ automorphism of the extended Coxeter-Dynkin diagram
\begin{equation}    
\dynkin[extended, Coxeter, edge length = 1cm, labels = { t(f_{1}) s_{0}  , s_{0}     }]A{1}   \quad   \quad      \dynkin[extended,Coxeter,
edge length=1cm,
labels={t(f_{2})  r_{0}    , r_{1}  ,  r_{2}    }]
C{2}   \label{dynkingl2gsp4} 
\end{equation} 
A representative element in $N_{H'}(A) $ for $ \omega_{H'} $ is given by $  (\rho_{1}, \rho_{2}) \in \GL_{2}(F) \times_{F^{\times}} \mathrm{GSp}_{4}(F)  $ where  $$ \rho_{1} =  
\begin{pmatrix}   &  1 \\   \varpi  \end{pmatrix}  , \quad  \quad    \rho_{2} = 
\begin{psmallmatrix}  & &  &  1  \\ 
&  &  1  &  \\  
& \varpi & \\ 
\varpi & & &   
\end{psmallmatrix}.
$$
Note that $ \rho $ normalizes $I'$.

\subsection{Intersections with $ H'$}   \label{H'strucsec}       In this subsection, we record some results on the structure of the twisted intersections $   H' \cap \tau _{i} K \tau_{i}^{-1} $.  
\begin{notation}  \label{notationh'}    
If $ h \in H ' $, we will often write  $ h  =   \left (  \begin{smallmatrix} a  & & & b &  \\  & a_{1} & a_{2} & &  b_{1} & b_{2}  \\ & a_{3} & a_{4} & & b_{3} & b_{4} \\ 
c & & &   d  \\   & c_{1} & c_{2} & &  d_{1} & d_{2}  \\  & c_{3} & c_{4}  & &  d_{3} & d_{4}     \end{smallmatrix}   \right  )  $ or $    h = \left (  \left ( \begin{matrix} a  & b \\ c & d  \end{matrix}  \right )  ,   \left  (   \begin{smallmatrix} a_{1} & a_{2} &  b_{1} & b_{2}  \\ a_{3} & a_{4} & b_{3} & b_{4} \\  c_{1} & c_{2} &  d_{1} & d_{2}  \\ c_{3} & c_{4}  &   d_{3} & d_{4}   \end{smallmatrix} 
     \right )  \right )  .   $    
\end{notation} 
 
\begin{lemma}   \label{distincttaui}                        $ H   ' K $, $ H '  \tau_{1} K $ and $  H '  \tau_{2} K $ 
are pairwise disjoint.   
\end{lemma}  
\begin{proof} 
If $  \Ht \tau _{i} K =  \Ht   \tau _{j} K $ for distinct $ i $ and $j $,   then $   \tau_{i}^{-1} h  \tau_{j} \in K $ for some $ h \in H $.
Requiring the entries of $ k  : = \tau_{i} ^{-1} h  \tau _{j} $ to be in $ \Oscr_{F} $, one easily deduces that $ \det ( k  ) \in \varpi \Oscr_{F}^{\times} $, a contradiction.    For instance,    $$  \tau_{1}^{-1} h \tau_{2}   =  \begin{pmatrix}  a &  * & * &  *  &   
\mfrac{a - d_{1} } { \varpi^{2}  }
& * 
  \\[0.1em] 
-  c    & * &  *  &  * 
  &    *       &  *    \\
  & *  & *  & *  &  *  &  *      \\[0.3em]
  c \varpi  &  &  &    *  &  
  \mfrac{c}{\varpi }
  &  \\[0.5em]  
  &  *  & *  &  * &  \mfrac{d_{1}}{\varpi}   & *   \\[0.5em] 
   & *   &  *    &  *   &   *   &  *     
\end{pmatrix} $$
where a $ * $ denotes an expression in the matrix entries of  $ h $ and the empty spaces are zeros.  From the entries displayed above, we see that $ a , c  \in \varpi \Oscr_{F} $ and so the  first column is an integral  multiple of $ \varpi $.      
\end{proof}  

\begin{remark}   \label{Schroder}  
  This also follows by an analogue of Schr\"oder's decomposition  proved in \cite[Theorem 12.1] {Endoscopy}. 
\end{remark}

\begin{notation}   \label{notationT} We let $ W^{\circ}  \subset W'  $ be the Coxeter subgroup generated by $ T : = 
 S_{\mathrm{aff}}'  \backslash   \left \{ s_{0} , r_{1}  \right \} $ and $  U^{\circ}  =  I'
W^{\circ}     I'$  the  corresponding maximal parahoric subgroup  of $H'$.    We let $ \lambda_{\circ}  = (1, 1, 1, 1) $ and  $ \tau_{\circ}  = \varpi^{-\lambda_{\circ}  } \tau_{1} $.  
\end{notation}
As usual,  we denote $ \Ht_{\tau_{\circ} } : = \Ht\cap \tau_{\circ}  K \tau_{\circ} ^{-1} $. Then $ H_{\tau_{1} } ' $ is the conjugate of $ H_{\tau_{\circ}} '  $ by $ 
\varpi^{\lambda_{\circ}} $.  Note that 
$ U^{\circ} $ is exactly the  subgroup of $ \Ht $ whose elements lie in $$    \begin{pmatrix}  \Oscr_{F} &  \varpi^{-1} \Oscr_{F}  \\   \varpi  \Oscr_{F}  &  \Oscr_{F}      \end{pmatrix}   \times  \begin{pmatrix} \Oscr_{F}  & \Oscr_{F} &  \varpi^{-1} \Oscr_{F}  &    \Oscr_{F}   \\ 
\varpi \Oscr_{F}   &  \Oscr_{F} &  \Oscr_{F}  &  \Oscr_{F}   \\    \varpi  \Oscr_{F}  &   \varpi     \Oscr_{F}  &   \Oscr_{F}   &   \varpi     \Oscr_{F}  \\
\varpi \Oscr_{F}  &  \Oscr_{F}  & \Oscr_{F} & \Oscr_{F}  
 \end{pmatrix}   $$ 
 and whose similitude is in $ \Oscr_{F}^{\times} $.    
\begin{lemma}   \label{Httau1}   $ \Ht_{\tau_{\circ} } $ is a subgroup of  $ U^{\circ}    $ and $ \pr'_{2}( \Ht_{\tau_{\circ}} )  = \pr'_{2}( U^{\circ} ) $. 
\end{lemma}    
\begin{proof} Let $ h \in  H _ { \tau_{\circ} }  '  $ and write $ h $ as in Notation \ref{notationh'}. Then   
 \begin{equation*}  
 \label{tau1hexpr}  \tau_{\circ}^{-1} h \tau_{\circ}     =   \begin{pmatrix} 
a &  -   \mfrac{c_{1}}{\varpi}     & -\mfrac{c_{2}}{\varpi} &  b -  \frac{c_{1}   }{\varpi^{2}} &  \mfrac{a-d_{1}}{\varpi }  &   - \mfrac{d_{2}}{\varpi}      \\[0.4em]  
-  \mfrac{c}{\varpi}    &  a_{1}   &   a_{2}     &  \mfrac{a_{1} - d }{  \varpi  }    &   b_{1}  -  \mfrac{c} { \varpi^{2}} &    b_{2}        \\[0.4em]
& a_{3}     &  a_{4}   &    \frac{a_{3}} {  \varpi  }   &    b_{3}    &   b_{4}       \\[0.4em]     
c      &   &   &   d  &    
  \mfrac{c}{\varpi}      &  \\[0.4em] 
&  c_{1}  & c_{2}    &    \mfrac{  c_{1}   } {  \varpi   }        & d_{1} & d_{2}   \\[0.4em]         
& c_{3}        & c_{4}    &     \mfrac{  c_{3 }  }  { \varpi }   &      d_{3}        &  d_{4}      
\end{pmatrix}  \in  K    
\end{equation*}   
From  the  matrix  above, one sees that $ h $ satisfies all the conditions that are satisfied by elements of    $ U^{\circ}   $, e.g., $ c \in \varpi \Oscr_{F} $ and  $ b \in \varpi^{-1} \Oscr_{F} $ and $ \det(h) = \det(\tau_{\circ} h \tau_{\circ}^{-1}  ) \in \det(K) \subset \Oscr_{F}^{\times}$. Therefore   $ \Ht_{\tau_{\circ}}  \subset  U^{\circ}  $. In  particular,  $  \pr_{2}'( \Ht_{\tau_{\circ}} ) \subseteq  \pr_{2}   '  ( U^{\circ} ) $. To see the reverse inclusion, say $  h  = ( h_{1} , h_{2} )  \in  U^{\circ}  $   
and again  write $ h $ as in Notation \ref{notationh'}.  Clearly, 
$ a_{1} d_{1} - b_{1} c_{1} \in \Oscr_{F} $.
Since \begin{align*}   \mathrm{sim}(h_{2}) &  =  a_{1} d_{1} - b_{1} c_{1} + a_{3}d_{3} - b_{3} c_{3} \\
      & \in  a_{1} d_{1} - b_{1} c_{1} + \varpi \Oscr_{F} ,
      \end{align*} 
we may find $ a', d' \in \Oscr_{F} $, $ b ' \in \varpi^{-1} \Oscr_{F} $ and $ c   '     \in \varpi \Oscr_{F}$ such that $ \frac{a' - d_{1}}{\varpi}, \frac{a_{1} - d' }{\varpi} $, $  b' - \frac{c_{1}}{\varpi^{2}} $, $   b_{1} -    \frac{ c'}{ \varpi^{2}} $ are all integral and $ a' d' - b ' c ' =  \mathrm{sim}(h_{2}) $.    Then $ h '   =  \left (   \left  (  \begin{smallmatrix}   a' & b' \\ c' &  d' \end{smallmatrix}   \right ) , h_{2} \right ) \in   \Ht_{\tau_{\circ}}  $ and $ \pr_{2}'(h') =  h_{2} $.       
\end{proof}

\begin{notation}  We let  $  {U}^{\ddagger}  \subset U' $   denote the compact open subgroup of all elements     whose  reduction modulo $ \varpi $  equals  $ \jmath(\mathbf{H}(\kay)) $.        
\end{notation} 

\begin{lemma}   \label{structureofHtau2}     $ \Ht_{\tau_{2}} $ is a subgroup of $  U'  $  and $ \pr_{2}(\tH_{\tau_{2}}) = \pr_{2} ( U^{\ddagger} ) $. 
\end{lemma}  
\begin{proof}  If we write $ h \in H_{\tau_{2}}' $ as in  \ref{notationh'}, then 
\begin{equation*}  \label{sig2hexpr}
\tau_{2}^{-1} h \tau_{2}   =   \begin{pmatrix} 
a &  -c_{1}  & -\mfrac{c_{2}}{\varpi} & \mfrac{b  -c_{1}   }{\varpi^{2}} &  \frac{a-d_{1}}{\varpi^{2}}  &   - \mfrac{d_{2}}{\varpi}      \\[0.5em]  
- c &  a_{1}   &   \mfrac{a_{2}}{\varpi}    &  \mfrac{a_{1} - d }{  \varpi ^{2}  }    &    \mfrac{b_{1} - c } { \varpi^{2}} &   \mfrac{  b_{2}   } { \varpi   }           \\[0.5em] 
& *  &  a_{4}   &    \mfrac{a_{3}} {  \varpi  }   &    \mfrac{  b_{3} } { \varpi }  &   b_{4}       \\[0.5em]    
 *   &   &   &   d  &   c    &  \\[0.3em]  
&  *   & *  &  *   & d_{1} & *        \\[0.5em]         
& *       & c_{4}    &     \mfrac{  c_{3 }  }  { \varpi }   &    \mfrac{  d_{3}  } { \varpi  }       &  d_{4}      
\end{pmatrix}  \in  K 
\end{equation*}   
From  the  matrix  above, one sees that all the entries of $ h $ are integral. Since $ H_{\tau_{2}} $ is compact, $ \mathrm{sim}(h) \in \Oscr_{F}^{\times} $ and so   $ h \in U'  $.  Similarly, it is easy to see from the matrix above that $ \pr_{2}' ( H_{\tau_{2}} ) \subset \pr_{2}'( U^{\ddagger} )   $.  For the reverse inclusion, say $ y \in \pr_{2}( U^{\ddagger})  $ is given. Choose any $ h \in H '  $ such that $  \pr'_{2}(h) = y $ and write $ h $ as in Notation \ref{notationh'}.   Then \begin{align*}  \mathrm{sim} (y)  &   =  a_{1} d_{1} -  b_{1}c_{1} + a_{3} d_{3} - b_{3} c_{3}   \\ &  \in  a_{1} d_{1} - b_{1} c_{1}  +   \varpi^{2}   \Oscr_{F}    
\end{align*}   
We may therefore find $ a', b', c' , d' \in \Oscr_{F} $ which are congruent to $ d_{1} $, $ c_{1} $, $ b_{1} $, $ a_{1} $ modulo $ \varpi^{2} $ such that $ a'd' - b'c' =   \mathrm{sim} (     y) $.  Then $ h ' =  \left (  \left ( \begin{smallmatrix}  a' & b' \\ c ' & d '  \end{smallmatrix}  \right ) ,   y  \right ) \in \Ht_{\tau_{2}} $ and $  \pr_{2} '  ( h')  = y   $. 
\end{proof} 
\begin{notation}  \label{Xscrnot} Let $  \jmath _{\tau} :  \GL_{2} \to \Hb $ be the embedding  given by   the  embedding  \begin{align*}   
\begin{pmatrix} a & b \\ c & d  \end{pmatrix}    
   \mapsto 
   \left ( \left ( \begin{matrix} a & b \\ c & d  \end{matrix} \right )  , \scalebox{1.1}{$ \left (   \begin{smallmatrix} \\  d &     &  c  \\  &   1 \\   b  &  &   a \\ & & &     ad    -  bc      \end{smallmatrix} \right )$} \right ) .    
\end{align*} 
We let  $ \mathscr{X}_{\tau} :=  \jmath_{\tau}(\GL_{2}(\Oscr_{F})  )  $ and $ \ess _{\tau}  \in \mathscr{X}_{\tau}  $ denote    $  \jmath_{\tau}  \left (  \begin{smallmatrix}  & 1 \\  1 &  \end{smallmatrix}  \right  )  $. 
\end{notation}    
\begin{lemma}   \label{Xscrembeds}  For $ i = 0, 1, 2 $, $ \mathscr{X}_{\tau}  $ is a subgroup of $   \Ht_{\tau_{i}}$.   In  particular,  $   \pr'_{1}( \Ht_{\tau_{i}})  =  \GL_{2}(\Oscr_{F} ) $.    
\end{lemma}    
\begin{proof} The first claim is easily verified by checking that $ \tau_{i}^{-1} \mathscr{X}\tau_{i}  \subseteq K   $ for each $i $.  For the second, note that $ \pr_{1}'(H_{\tau_{i}}') $ are  compact open subgroups of $ H_{1} = \GL_{2}(F) $ that contains $  \GL_{2}(\Oscr_{F} )$ and  $ U_{1} = \GL_{2}(\Oscr_{F}) $ is a maximal compact open subgroup of $H_{1} $.        
\end{proof}  
\begin{corollary}   \label{sigepsientrycoro}      If $ h \in H_{\tau_{i}} $, $ a_{1} - d , a - d_{1} , b_{1} - c, b - c_{1} \in \varpi^{i}  \Oscr_{F} $. 
\end{corollary} 
\begin{proof} Follows by   matrix  computations  above.  \end{proof} 

\subsection{Cartan decompositions}    \label{Cartansec}  Throughout this article, we  let    $ \varpi^{\Lambda } $ denote the subset $ \left \{ \varpi^{\lambda} \, | \, \lambda \in \Lambda \right \} $ of $ A $.  
For $ i = 0, 1, 2 $, define  \begin{align}  p_{i}  : \Lambda &  \to U ' \varpi^{\Lambda} \tau_{i}  K,  \quad \quad 
\lambda     \mapsto   U '   \varpi^{\lambda}  \tau_{i} K  . 
\end{align} 
By \cite[Lemma 5.9.2]{CZE}, we have an identification $ U ' \varpi^{\Lambda} H_{\tau_{i}} ' \xrightarrow{\sim}  U' \varpi^{\Lambda} \tau_{i} K $ given by $ U '  \varpi^{\lambda} \Ht_{\tau_{i} } \mapsto  U  ' \varpi^{\lambda} \tau_{i}  K $.  
So we may equivalently view $ p_{i} $ as a map to $ U ' \varpi^{\Lambda}   \Ht_{\tau_{i}} $. For $ i = 0 $, Cartan decomposition for $ H' $ implies the following. 
\begin{lemma}  \label{p0isabijection} $ p_{0} $ induces a bijection $  W' \backslash  \Lambda  \xrightarrow{\sim}  U' \varpi^{\lambda} K $.    
\end{lemma} 

Observe that $ \ess_{\tau}  \in N_{H'}(A^{\circ} ) $ is a lift of the element $ s_{0} r_{0}\in W ' $. 
Moreover $$    \begin{psmallmatrix}
 1&  &  &  
 \\ 
 &  \,  1&  &  \\ 
 &  &  0 &  & &  \varpi \\ 
 &  &  & 1  & & \\ 
 &  &  &  &   1  \\ 
 &  &  -\mfrac{1}{\varpi}&  &  & 0 
\end{psmallmatrix}  \in H_{\tau_{1}}' ,  \quad \quad    \quad \begin{psmallmatrix}
 1&  &  & 
 \\ 
 & 1&  &  \\ 
 &  & 0 &  & & \, 1  \\ 
 &  &  & \, 1 & &  \\ 
 &  &  &  & \,  1  \\ 
 &  & -1  &  &  & \,  0
\end{psmallmatrix} \in H_{\tau_{2}}' .   $$ 
Thus   $  p_{1} $ factors through $ \langle s_{0}r_{0}, t(-f_{3}) r_{2} \rangle \backslash \Lambda  $ and $  p_{2} $ factor through   $  \langle s_{0}r_{0}, 
  r_{2} \rangle  \backslash \Lambda $.  

\begin{lemma} For $ i = 1, 2 $, $ p_{i} ( \lambda) $ is distinct from $  p_{i} ( s_{0} \lambda ) $ if $ \lambda \notin \left \{ s_{0} \lambda ,   r_{0}  \lambda   \right \} $.            
\end{lemma}   
\begin{proof}  Write $ \lambda  = ( a_{0}, a_{1}, a_{2}, a_{3} ) $.  Since $ p_{i} $ factors through $ \langle s_{0} r_{0}  \rangle \backslash \Lambda  $, we may assume by replacing $ \lambda $ with $ s_{0}(\lambda) $ etc.,  that $  2 a_{1} \geq a_{0}$ and $ 2 a_{2} \geq a_{0}  $.   Then we need to show that $ U' 
 \varpi^{\lambda} H_{\tau_{i}} \neq U' \varpi^{s_{0}(\lambda)}  \Ht_{\tau_{i}  } $  whenever  $  2 p_{1}  > p_{0} $ and $ 2p_{2} >  p_{0}    $.  Assume on the contrary that  there  exists an $  h  \in  U'   $ such that $   \gamma  : =  \varpi^{-\lambda}  h  \varpi^{s_{0}(\lambda)}  \in  
 H_{\tau_{i}} $. Write $ h = (h_{1}, h_{2} ) $ as in Notation \ref{notationh'}. Then $  \gamma  = (\gamma_{1} ,  \gamma_{2} ) $ satisfies  
 $$   
 \gamma_{1} =  \begin{pmatrix}  a  \varpi^{p_{0} - 2p_{1}} &   b \\[0.2em] c  &  d  \varpi^{2p_{1} -  p_{0}}  \end{pmatrix}  , \quad \quad   \gamma _{2}  =  \begin{pmatrix}  *  &  *  & * &  * \\ 
 * & * & * & *   \\[0.2em]   c_{1} \varpi^{2p_{2} - p_{0}}  
 & *  & * & *  \\ * & *   & *  &  *      \end{pmatrix}  .  $$
Lemma \ref{Xscrembeds} implies that   $ a \varpi ^ { p_{0}  - 2p_{1}  }  \in  \Oscr_{F} $  and Corollary \ref{sigepsientrycoro} implies that  $ b - c_{1}  \varpi^{2p_{2} -  p_{0} } \in \varpi^{i} \Oscr_{F}  $. Thus $ a ,   b   \in \varpi \Oscr_{F} $. Since $ c, d \in \Oscr_{F} $ as $ h \in U' $, we see that $ \mathrm{sim}(h) =  \det(h_{1}) = ad - bc \in \varpi \Oscr_{F} $,  a  contradiction.   
\end{proof}    
Recall that $ \lambda_{\circ} \in \Lambda $ denotes the cocharacter $ (1,1,1,1) $.  
\begin{lemma}   \label{distinctWTorbits}      If  the $ W  ^ { \circ  }   $-orbits of $ \lambda + \lambda_{\circ} $ and $  \mu  + \lambda_{\circ}  $ are distinct,   $ p_{1}(\lambda ) $ is distinct from $ p_{1}( \mu  )  $.  
\end{lemma}    
\begin{proof} Since $ W^{\circ} $ is a  Coxeter subgroup of the Iwahori Weyl group, there is a bijection \begin{align*}    W^{\circ}  \backslash  W_{I'} / W'    & \xrightarrow {\sim}  U  ^  { \circ }  \varpi^{\Lambda} U'      
 \quad   \quad     W  ^ {\circ } w  W ' \mapsto U ^{\circ} w U ' .
 \intertext{Recall that we have an isomorphism $ W_{I'}   \simeq   \Lambda  \rtimes W  '  $ which sends $ \varpi^{\lambda} \in W_{I}' $ to $  (t(-\lambda), 1)  $. Via this isomorphism, we obtain  bijection $ W^{\circ} \backslash \Lambda \to U ^{\circ} \varpi^{\lambda} U' $ given by $ W ^{\circ} \lambda \mapsto  U ^{\circ} \varpi^{-\lambda} U ' $ and hence a bijection} 
 W^{\circ} \backslash  \Lambda  &   \xrightarrow{\sim}  U ' \varpi^{\lambda} U ^{\circ} , \quad 
 \quad   W^{\circ} \lambda \mapsto  U '\varpi^{\lambda} U ^{\circ}.
 \end{align*} 
Now $  H_{\tau_{\circ}}  \subset  U^{\circ}  $ by Lemma  \ref{Httau1}. So (the inverse of) the bijection above induces  a well-defined   surjection   $ U '  \varpi^{\Lambda}   H_{\tau_{\circ}} '  \to   U'   \varpi^{\Lambda}  U^{\circ}  \xrightarrow{\sim}   W^{\circ} 
     \backslash  \Lambda $. Thus if $ \lambda_{1}, \mu_{1} \in  \Lambda $ are in different $ W^{\circ} $-orbits, $ U ' \varpi^{\lambda_{1}} H_{\tau_{\circ}}'$ is distinct from $ U ' \varpi^{\mu_{1}} H_{\tau_{\circ}} ' $.  Now  apply this to $ \lambda_{1} : = \lambda  +  \lambda_{0} $ and $ \mu_{1} : = \mu + \lambda_{\circ}  $ and use that $ H _{ \tau_{1}} '  = \varpi^{\lambda_{\circ}} H_{\tau_{\circ}}  '\varpi^{-\lambda_{\circ}} $.        
\end{proof}

\begin{lemma} If the $ W' $-orbits of $ \lambda $, $ \mu $ are distinct, $ p_{2}(\lambda) $ is distinct from $ p_{2}(\mu) $. 
\end{lemma} 

\begin{proof} This follows similarly since $  H_{\tau_{2}} '    \subset   U' $. 
\end{proof}

\begin{notation} We denote  $ W_{\tau_{1}} ' = \langle s_{0} r_{0} , t(-f_{3}) r_{2}  \rangle \subset W_{I'} $  and  $  W_{\tau_{2}}' = \langle  s_{0} r_{0} , r_{2} \rangle $. We also denote $ W' $ by $ W_{\tau_{1}}' $ for consistency.  
\end{notation}
\begin{proposition}    \label{cartanhtaui}              For $ i = 0, 1, 2 $,  the maps $ p_{i}$ induce bijections $   W _ { \tau_{i} } '  \backslash  \Lambda    \xrightarrow { \sim } U   ' \varpi^{\lambda } \tau_{i}  K $. 
\end{proposition}

\begin{proof}  Follows  from  the results above. 
\end{proof}

\subsection{Schubert cells}  The decompositions of various double cosets is  accomplished by a recipe proved in \cite[\S 5]{CZE}. Below, we    provide its formulation in the special case of  $ G = \mathrm{GSp}_{6}(F) $.   

Recall that $ I $ denotes the Iwahori subgroup of $ G $ contained in $ U $ whose reduction modulo $ \varpi $ lies in the Borel of $ \Gb(\kay) $ determined by $ \Delta $.  For $ i = 0 ,1 ,2 ,3 $, let  
$ x_{i} : \GG_{a} \to \mathbf{G} $ denote the root group maps 
\begin{align*}  x_{0}   :      u \mapsto    \left (   \begin{smallmatrix}
1 &  &  & &  & \\  &  1 &  &  &  & \\ 
 &  &  1&  &  & \\ 
 \varpi u &  & &  1   &  & \\ 
 &  &  &  &1  & \\ 
 &  &  &  &  & 1
 \end{smallmatrix} \right ), \quad  \!    
    x_{1} : u  \mapsto            
 \begin{psmallmatrix}
 1&   u  &  &&  & \\[0.1em]  
 &   1 &  &  &  & \\ 
 &   &  1   &  &  & \\ 
&  &  & 1  &   & \\[0.1em]  
 &  &  &   - u &  1 &  \\ 
 &  &  &  &   &  1    \end{psmallmatrix}, \quad \!  x_{2 }  : u   \mapsto     \begin{psmallmatrix}
  1&  &  & &  & \\ 
 &   1&  u &  &  & \\[0.1em]  
 &  &   1&  &  & \\ 
 &  &  &  1 &  & \\ 
 &  &  &   &    1 & \\[0.1em]  
 &  &  &  &    -u &  1 
\end{psmallmatrix},  \quad  \!   
x_{3} :  u  \mapsto    
\begin{psmallmatrix}
1&  &  &  &  & \\ 
&   1&  &  &  & \\ 
&  &   1&  &  &   u \\ 
&  &  &  1&  & \\ 
&  &  &  &  1& \\ 
&  &  &  &  &  1
\end{psmallmatrix}   
\end{align*}
and let $ g_{i} : [ \kay ] \to G $ be the maps $ \kappa \mapsto x_{i}(\kappa ) w_{i} $. Then $ Iw_{i} I/ I =  \bigsqcup_{ \kappa \in [\kay] } g_{i}(\kappa) I $ for $ i = 0,  1, 2, 3 $.  
For $ w \in  W_{I}   $, choose a reduced word decomposition $ w=  s_{w,1} s_{w,2}\cdots s_{w, \ell(w) }  \rho_{w} $  where $   s_{w,i}  \in  S_{\mathrm{aff}} $, $ \rho_{w} \in  \Omega $ and  define  \begin{align*} \mathcal{X}_{w} :  [  \kay ] ^ { \ell(w) } &  \to  G \\ 
 (\kappa_{1} ,   \ldots,   \kappa_{\ell(w)}  )& \mapsto    g_{s_{w,1} } ( \kappa_{1} )  \cdots  g_{s_{w}, \ell(w) } ( \kappa_{\ell (w) }   )   \rho_{w}  
\end{align*}    
Here, we have suppressed the dependence on the choice of the reduced word decomposition in light of the following result, 
which is a consequence of the braid relations in Iwahori Hecke algebras.     
\begin{proposition}  \label{SchubertGSp6}   $  I w I = \displaystyle \bigsqcup  \nolimits  _ { \vec{\kappa} \in [\kay]^{\ell(w)}}  \mathcal{X}_{w}(\vec{\kay}) I  $. If $ w $ has minimal possible length in $ w W $, then $ I w K = \bigsqcup_{\vec{\kappa}   \in [\kay]^{\ell(w)} } \mathcal{X}_{w}(\vec{\kappa})K  $.     
\end{proposition} 
\noindent Thus the image of $ \mathcal{X}_{w} $ modulo $ I $ is independent of the choice of decomposition and we have     $ \left  |        \mathrm{im} ( \mathcal{X}_{w})I/I  \right    |     =  q ^ { \ell(w) } $. Moreover,  the  same facts  holds with right $ K $-cosets if $ w $ has the aforementioned   minimal length property. For such $ w $,  $  \ell(w) =    \ell_{\min}( t(-\lambda_{w} )  )    $ where $  \lambda_{w}    \in \Lambda $ is the   unique     cocharacter   such that $ w K  =  \varpi^{\lambda_{w}} K $. We refer to the image of $ \mathcal{X}_{w}$ as a \emph{Schubert cell} since these images are reminiscent of the Schubert cells that appear in the stratification of the classical  Grassmannians. 

Now given a $ \lambda \in \Lambda^{+}  $, a set of representatives for $ U' \backslash K \varpi^{\lambda} K / K $ can be obtained by studying $ U  '    $-orbits on a decomposition for $ K \varpi^{\lambda} K / K $. Let $ W^{\lambda} $ denote the stabilizer of $ \lambda $ in $ W $.  The next result shows that the study of such   orbits  amounts to studying $ U' $-orbits on certain Schubert cells. 

\begin{proposition} There exists a unique $ w =  w_{\lambda} \in W_{I} $ of minimal possible length such that $ K \varpi^{\lambda} K = K w  K $.  If $ [W / W^{\lambda}] $ denotes the set of minimal length representatives in $ W $ for $ W / W^{\lambda} $, then   
$$  K \varpi^{\lambda} K  = \bigsqcup_{\tau} \bigsqcup_{\vec{\kappa} \in [\kay]^{\ell(\tau w  )}  } \mathcal{X}_{\tau w  }( \vec{\kappa}  )    K  . $$
Moreover, $ \ell ( \tau w ) = \ell (\tau) + \ell(w) $ for all $ \tau \in   [ W / W^{\lambda} ]  $.  
\end{proposition} 
In what follows, we will write these Schubert cells for various words in $ W_{I} $. Note  $ W / W^{\lambda} $ is identified with the orbit  $   W \lambda $ of $ \lambda $.  The  set of possible reduced  words decompositions for $ \tau \in [W / W_{\lambda}] $ can be   visualized by a \emph{Weyl orbit diagram}. This is the Hasse diagram on the subset $ [ W / W^{\lambda} ] \subset W $ under the weak left Bruhat order. Via the bijection $ [W/ W^{\lambda} ] \simeq W\lambda $,  the nodes of this diagram can be viewed as elements of $ W \lambda $ and its edges are labelled by  one of the simple reflections in $ \Delta =  \left \{  s_{1} , s_{2} , s_{3}  \right \}  $.  The unique minimal element of this diagram is $ \lambda^{\mathrm{opp}} $ (the unique anti-dominant element in $ W\lambda $)  and  the unique maximal element in this diagram is $ \lambda $. 
\begin{example}   \label{Schubertexample}     Let $ \lambda = (2,2,1,1) $. Then $ \lambda ^ {  \mathrm{opp}} = (2,0,1,1) $ and the Weyl orbit diagram  is  \begin{center} \vspace{0.2em} 

\begin{tikzcd}
{(2,0,1,1)} \arrow[r, "s_1"] & {(2,1,0,1)} \arrow[r, "s_2"] & {(2,1,1,0)} \arrow[r, "s_3"] & {(2,1,1,2)}  \arrow[r, "s_2"] & {(2,1,2,1)}  \arrow[r, "s_1"] & {(2,2,1,1)}
\end{tikzcd}
\end{center} 
By Lemma \ref{lminwords}, we have $ w_{\lambda} = w_{0} \rho^{2} $. So the decomposition of $ K \varpi^{\lambda} K / K $ can be given by six Schubert cells, corresponding to the reduced words $$ w_{0} \rho^{2}, \quad  w_{1} w_{0} \rho^{2} , \quad  w_{2} w_{1} w_{0} \rho^{2},  \quad  w_{3} w_{2} w_{1} w_{0} \rho^{2},  \quad     w_{2}  w_{3} w_{2} w_{1} w_{0} \rho^{2} , \quad    w_{1}   w_{2}  w_{3} w_{2} w_{1} w_{0} \rho^{2} $$
which are obtained by  ``going down" the Weyl orbit diagram. Each cell down this diagram can be obtained from one preceding it by applying two elementary row operations, one for the reflection and one for the root group  map.    We also apply an optional column operation to ``match" the diagonal with  the value of the cocharacter at $ \varpi $ at each node (for aesthetic reasons).   For instance, let $ \varepsilon_{0} = w_{0} \rho^{2} $ and $ \varepsilon_{1} = w_{1} \varepsilon_{1} $. We have  \begin{align*} \mathrm{im}( \mathcal{X}_{\varepsilon_{0}}) K / K   &    =   \scalebox{0.8}{$ \Set*{   \left(\begin{array}{cccccc}
1   &     & &   & &    \\
 & \varpi  &  &  & & \\
 & &   \varpi  & & \\
 x \varpi & & &      \varpi^{2} &  &  \\
 &  &  &   &  \varpi    & \\
 & & &  &  & \varpi 
\end{array}\right)   K  \given  
x   \in  [\kay]    
}$} \\   
\mathrm{im}( \mathcal{X}_{\varepsilon_{1} })K/ K  & =   
  \scalebox{0.8}{$ \Set*{   \left(\begin{array}{cccccc}
 \varpi  &    a     & &   & &    \\
 & 1  &  &  & & \\
 & &   \varpi  & & \\
 & & &      \varpi &  &  \\
 & x \varpi     &  &  - a \varpi  &  \varpi^{2}     & \\
 & & &  &  & \varpi 
\end{array}\right)   K  \given  
a  ,    x  \in  [\kay]    
}.$}
\end{align*} 
Note that for $ \varepsilon =  w_{2} w_{3} w_{2} w_{1} w_{0} \rho^{2} $, our recipe gives 
$$ \mathrm{im}(\mathcal{X}_{\varepsilon}) K / K      
 =   
  \scalebox{0.8}{$ \Set*{   \left(\begin{array}{cccccc}
 \varpi  &   & &   & a  &    \\[0.2em] 
 &  \varpi^{2}  &   c_{1} \varpi   &  a   \varpi   &  z  + cc_{1} + \varpi x  & c \varpi   \\[0.2em] 
 & &   \varpi  & & \\
 & & &    
 \varpi &  &  \\[0.1em] 
 &   &  &  & 1     & \\
 & & &  & - c_{1}    & \varpi 
\end{array}\right)   K  \given  
a  , c, c_{1} ,     x  ,  z  \in  [\kay]    
}$}$$ 
However, we can replace $ z + cc_{1} + \varpi x $ with a variable $  y  $ running over $ [\kay_{2}] $, since for a fixed value of $ c $, $ c_{1} $ and $ a $,   the expression $ z + c c_{1} + \varpi x $ runs over such a set of representatives of $ \Oscr_{F} /\varpi^{2} \Oscr_{F }$  and a column operation between fifth and second columns allows us to choose any such set of representatives. In what follows, such replacements will be made without further  comment. 
\end{example} 

\noindent  \textit{Convention.} To save space, we will often write the descriptors of parameters below the Schubert cells rather than within the set.  We will also write $ \mathcal{X}_{\varepsilon}  $ for the Schubert cell  where we really mean $ \mathrm{im}(\mathcal{X}_{\varepsilon})K/K  $ and omit writing $ K $ next to the matrices.  When drawing Weyl orbit diagrams, we remove all the labels of the nodes as they can be read off by following the labels on the edges.  
\begin{proof}[Proof of Proposition \ref{GSp6decompose}]   That the listed representatives are distinct follows by Lemma   \ref{distincttaui} and Lemma \ref{cartanE}.     The goal therefore  is  to show that  the Schubert  cells reduce to the claimed  representatives in each case.   For each of the words $ w $, we will draw the Weyl orbit diagram beginning in the anti-dominant cocharacter $ \lambda_{w} $ associated with $ w $.  In   these   diagrams, we pick the first vertex and the vertices that only have one incoming arrow labelled $ s_{1} $ (all of which we mark on the diagrams)  and study the $ U'$-orbits on Schubert cells corresponding to these vertices. This suffices  since the orbits of $ U  '    $ on the remaining cells are contained in these by the recursive nature of the cell maps. 
We list all of the  relevant  cells and record all of our conclusions. However since the reduction steps involved are just  elementary row and column operations\footnote{row operations coming from $ \GL_{2}(\Oscr_{F}) \times_{\Oscr_{F}^{\times}} \mathrm{GSp}_{4}(\Oscr_{F})$ and column operations coming from $\mathrm{GSp}_{6}(\Oscr_{F})$},  we only provide detailed justifications for one cell in each case, and leave the remaining for the reader to verify (all of which are  completely  straightforward).    \\

\noindent $ \bullet  $ $ w  =  \rho  $. Here  $ \lambda_{w} = (1,0,0,0) $ and the Weyl orbit diagram is as follows.  
\begin{center}

\begin{tikzcd}
                      &                       &                                             & {} \arrow[rd, "s_{1}"  ]  &                       &                       &    \\
 \arrow[r, "s_{3}"] & {} \arrow[r, "s_{2}"] & {} \arrow[ru, "s_{3}"] \arrow[rd, "s_{1}"' , "\circ" marking   ] &                         & {} \arrow[r, "s_{2}"] & {} \arrow[r, "s_{3}"] & {}  &  &    \\
                      &                       &                                             &  \arrow[ru, "s_{3}"'] &                       &                       &   
\end{tikzcd}
\end{center} Thus there are two cells of interests, corresponding to the words $ \varepsilon_{0} = \rho $ and $ \varepsilon_{1} =  w_{1} w_{2}    w_{3}     \rho $. The cell $ \mathcal{X}_{\rho}  $   obviously reduces to $ \varpi^{(1,1,1,1)}$. As for $ \varepsilon_{1} $, we have      $$          \mathcal{X}_{ \varepsilon_{1}}   =   \scalebox{0.8}{$  \Set*{  
\left(\begin{array}{cccccc}
\varpi  & a & c & z &   & \\
& 1 &   & &  & \\
& & 1 & & & \\
& & & 1 & & \\
& & & -a & \varpi &\\
& & & - c & & \varpi 
\end{array}\right)   \given     a, c, z \in   [  \kay   ] }$}$$
We can eliminate $ z  $ via a row operation. Then we conjugate by reflections $ w_{3 } $ and $ v_{2} =  w_{2} w_{3} w_{2} $ to make the diagonal $ \varpi^{(1,1,1,1)} $ which puts the entries $ a $, $ c $ in the top right $ 3 \times 3 $ block. Conjugation by $ w_{1} $ switches $ a $, $ c $ and one execute Euclidean division (using row/column operations) to make one of $ a $ or $ c $ equal to zero. Conjugating by an  element of $ A^{\circ} $ if necessary, we get $ \varpi ^{(1,1,1,1)} $ or $ \tau_{1}  $ as  possible  representatives from this cell.      \\

\noindent  $  \bullet $ $ w = w_{0}  \rho^{2} $. The Weyl orbit diagram of $  \lambda_{w} =  (2,0,1,1) $ is 
\begin{center}    

\begin{tikzcd}
 \arrow[r, "s_{1}",  "\circ" marking] & \arrow[r, "s_{2}"] & {} \arrow[r, "s_{3}"] & {} \arrow[r, "s_{2}"] & {} \arrow[r, "s_{1}",   "\circ" marking  ] &  { }   
\end{tikzcd}
\end{center} There are three cells of interests corresponding to $ \varepsilon_{0} =  w_{0} \rho^{2} $, $ \varepsilon_{1} = w_{1}  \varepsilon_{0} $ and $ \varepsilon_{2} =   w_{1} w_{2} w_{3} w_{2}  \varepsilon_{1}  $.  The cells $ \mathcal{X}_{\varepsilon_{0}}  $, $ \mathcal{X}_{\varepsilon_{1}}  $ were recorded in Example \ref{Schubertexample} and 
$$   \mathcal{X}_{\varepsilon_{2}} = \scalebox{0.8}{$ \Set*{   \left(\begin{array}{cccccc}
\varpi ^{2}   &  a_{1} \varpi  & c_{1} \varpi &  z + \varpi x   & a \, \varpi  &  c  \, \varpi  \\
 & \varpi  &  & a & & \\
 & & \varpi  &c  & & \\
 & & & 1 &  &  \\
 &  &  &  -a_{1}  &   \varpi    & \\
 & & &  -c_1 &  & \varpi 
\end{array}\right)   \given  \begin{aligned} &  a ,  a_{1} ,  c , c_{1},  \\ &  \,   x,  z   \in  [\kay]   \end{aligned} 
}$}   .  $$
We claim that the $ U  '   $-orbits  on 
\begin{itemize}   [before = \vspace{\smallskipamount}, after =  \vspace{\smallskipamount}]    \setlength\itemsep{0.1em}  
\item  $ \mathcal{X}_{\varepsilon_{0}} $ are represented by  $ \varpi^{(2,2,1,1)} $, 
\item $ \mathcal{X}_{\varepsilon_{1}} $  
are represented by $ \varpi ^{(2,1,2,1) } $, $ \varpi^{(1,1,0,0)} \tau_{1}  $,    
\item   $ \mathcal{X}_{\varepsilon_{2}}$ are represented by $ \varpi^{(2,2,1,1) } ,  \varpi^{(1,1,0,0) }  \tau_{1} $.  
\end{itemize}
We record our steps for reducing $ \mathcal{X}_{\varepsilon_{2}} $.    Eliminate the entry $ z + \varpi x $ using a row operation. Conjugation by $ w_{3} \in U $ (resp., $ w_{2}w_{3}w_{2} \in U   '  $) switches $ a_{1} , a $ (resp., $ c_{1} , c $) and keeps the diagonal $ \varpi^{(2,2,1,1)} $. Using row/column operations, we may make one $ a, a_{1} $  (resp., $ c ,  c_{1} $) zero while still keeping the diagonal $ \varpi^{(2,2,1,1)} $. Without loss of generality,  assume $ a_{1} , c_{1} $ are zero.  Conjugation by $ w_{2} \in U   '   $  switches  $ a $, $ c $ and we may again apply row-column operations to make one of $ a $, $c $ zero, say $ c $. Normalizing by an appropriate  diagonal matrix in $ A^{\circ} $, we get the representatives $ \varpi^{(2,2,1,1)} $ or $ \varpi^{(1,1,0,0)} \tau_{1} $ depending on whether $ a = 0 $ or not. \\

\noindent $\bullet$ $ w = \upsilon_{1} \rho^{2}  $.  We have $ \lambda_{w} =    (2,0,0,1) $ and the Weyl orbit diagram is 
\begin{center}

\begin{tikzcd}
                      &                                            &                                            & {} \arrow[r, "s_{1}",   "\circ" marking  ]  & \arrow[rd, "s_{2}"]                     &                         &                       &    \\
  \arrow[r, "s_{2}"] & {} \arrow[r, "s_{3}"] \arrow[rd, "s_{1}"',  "\circ" marking  ] & {} \arrow[rd, "s_{1}"    ] \arrow[ru, "s_{2}"] &                        &                                            & {} \arrow[r, "s_{3}"]   & {} \arrow[r, "s_{2}"] & {} \\
                      &                                            &    \arrow[r, "s_{3}"']                     & {} \arrow[r, "s_{2}"'] & {} \arrow[ru, "s_{1}"] \arrow[r, "s_{3}"'] & {} \arrow[ru, "s_{1}"'] &                       &   
\end{tikzcd} 
\end{center} 
So we need to study the $ U$-orbits on the analyze Schubert cells corresponding to the words 
$
\varepsilon_{0} =  w $, $  
\varepsilon_{1} = w_{1} w_{2} w $ and $ 
\varepsilon_{2} = w_{1} w_{2} w_{3} w_{2} w . 
$. The cells corresponding to these words are  $$  \mathcal{X}_{\varepsilon_{0}}  =   \scalebox{0.8}{$\Set*{   \left(\begin{array}{cccccc}
1 & & & &  & \\
 & 1 &  &  &  &  \\
 &  & \varpi  & &  & \\
x_{1} \varpi  &  a \varpi  &  &  \varpi^{2} & & \\
a \varpi  & - x \varpi  &  & & \varpi^2  & \\
 &  &  &  &  & \varpi 
\end{array}\right)}$}, \quad \mathcal{X}_{\varepsilon_{1}}  =   \scalebox{0.8}{$\Set*{   \left(\begin{array}{cccccc}
\varpi & a_{1}  &  c  & &  & \\
 & 1 &  &  &  &  \\
 &  & 1  &  &  & \\
 &  &  & \varpi & & \\
 & x_{1} \varpi  &  a \varpi & -a_{1} \varpi &  \varpi  ^ { 2  }   & \\
 &  a \varpi  &  - x \varpi  & - c \varpi     &  &  \varpi ^{2} 
\end{array}\right)}$} $$ 
$$  \mathcal{X}_{\varepsilon_{2}}  =   \scalebox{0.8}{$\Set*{   \left(\begin{array}{cccccc}
\varpi^2  & a_1 +a\,\varpi  & c_1 \,\varpi  & z + \varpi \,x &  & c\,\varpi \\
 & 1 &  &  &  &  \\
 &  & \varpi  & c &  & \\
 &  &  & 1 & & \\
 &  x_1  \varpi   &  & -(a_1 +  a\,\varpi)  & \varpi^2  & \\
 &  &  & -c_1  &  & \varpi 
\end{array}\right)}$} $$  
where $ a, a_{1}, c , c_{1} , x , x_{1},   z \in [\kay] $. 
We claim that the $ U  '$-orbits  on 
\begin{itemize}   
[before = \vspace{\smallskipamount -1.25pt}, after =  \vspace{\smallskipamount-0.25pt}]    
\setlength\itemsep{0.1em}     
\item   $ \mathcal{X}_{\varepsilon_{0}} $ are given by $  \varpi^{(2,2,2,1)} $, $    \varpi^{(1,1,1,0)}  \tau_{1}    $, 
\item  $ \mathcal{X}_{\varepsilon_{1}} $ are given by $ \varpi^{(2,1,2,2)} $, $ \varpi^{(1,1,0,1)} \tau_{1}$, 
\item $ \mathcal{X}_{\varepsilon_{2}} $ are given by $   \varpi^{(2,2,2,1) }  ,  \varpi^{(1,1,1,0)} \tau_{1}   ,  \varpi^{(2,1,1,1)} \tau_{2} $.   \vspace{-0.1em}    
\end{itemize}   We record our analysis for $ \mathcal{X}_{\varepsilon_{2}} $. 
Begin by eliminating  the entries  $  z + \varpi x $ and $ x_{1} \varpi $ using row operations.   Conjugation by $  w_{3} \in  U '  $ switches $ c_{1} $, $ c $ while keeping the diagonal $ \varpi^{(2,2,0,1)} $ and we can apply row-column operations to make either $ c $ or $ c_{1 }$ zero, say $ c_{1} $. Conjugating by $  r_{0} =   w_{2} w_{3} w_{2}  \in  U    ' $, we arrive at   
$$ \scalebox{0.8}{$       \left(\begin{array}{cccccc}
\varpi^2  &  &   &  & a_1 +a\,\varpi  & c\,\varpi \\
 & \varpi^2  &  & a_1 +a\,\varpi  &  &  \\
 &  & \varpi  & c &   &  \\
 &  &  & 1 &  &  \\
 &  &  &  & 1 &  \\
 &  &  &   &  & \varpi 
\end{array}\right)   $}    $$
for some $ a, a_{1}, c \in [\kay]  $. We now divide in two case.   Suppose first that $ c $ is zero. Then $ (a_{1} + a \varpi ) $ is in $ \Oscr_{F}^{\times }$, $ \varpi \Oscr_{F}^{\times} $ or is equal to zero, and we  can normalize by conjugating with an element of $ A^{\circ} $ to get the representatives $ \varpi^{(2,2,2,1)} $, $ \varpi ^{(1,1,1,0)} \tau_{1} $, $ \varpi^{(2,1,1,1)}  \tau_{2} $.  Now suppose  that  $ c \neq 0 $.  Then we  may assume $ a = 0 $ by applying  row-column operations. If now  $ a_{1} \neq 0 $, we may make $ c = 0 $ and normalizing by $ A^{\circ}$  
leads us to the representative  $ \varpi^{(1,1,1,0)} \tau_{1} $. 
If $ a_{1} = 0 $ however, then conjugating by $ w_{2} $ and normalizing by $ A^{\circ}$ gives us the representative   $ \varpi^{(1,1,0,1)} \tau_{1} $. \\

\noindent   $   \bullet $  $ w =  \upsilon_{2}   \rho^{3}  $.  Here $ \lambda_{w} = (3,0,1,1) $ and the Weyl orbit diagram is    
\begin{center}

\begin{tikzcd}
                                            &                                              &                                            & {} \arrow[r, "s_{1}"]                        & {} \arrow[r, "s_{2}"]                       & {} \arrow[rd, "s_{3}"]  &                                             &                         &    \\
                                            &                                              & {} \arrow[ru, "s_{3}"] \arrow[r, "s_{1}", "\circ" marking   ] & {} \arrow[rd, "s_{2}"'] \arrow[ru, "s_{3}"'] &                                             &                         & {} \arrow[rd, "s_{2}"]                      &                         &    \\
                                            & {} \arrow[rd, "s_{1}"] \arrow[ru, "s_{2}"]   &                                            &                                              & {} \arrow[rd, "s_{3}"]                      &                         &                                             & {} \arrow[rd, "s_{1}"]  &    \\
 {}  \arrow[rd, "s_{1}"',  "\circ" marking]  \arrow[ru, "s_{3}"] &                                              & {} \arrow[r, "s_{2}"]                      & {} \arrow[rd, "s_{3}"'] \arrow[ru, "s_{1}"]  &                                             & {} \arrow[r, "s_{2}"]   & {} \arrow[rd, "s_{1}"'] \arrow[ru, "s_{3}"] &                         & {} \\
                                            &   {} \arrow[ru, "s_{3}"'] \arrow[rd, "s_{2}"'] &                                            &                                              & {} \arrow[ru, "s_{1}"'] \arrow[rd, "s_{2}"] &                         &                                             & {} \arrow[ru, "s_{3}"'] &    \\
                                            &                                              & {} \arrow[rd, "s_{3}"']                    &                                              &                                             & {} \arrow[r, "s_{1}"]   & {} \arrow[ru, "s_{2}"']                     &                         &    \\
                                            &                                              &                                            & {} \arrow[r, "s_{2}"']                       & {} \arrow[r, "s_{1}",swap,   "\circ" marking   ] \arrow[ru, "s_{3}"]  &    {} \arrow[ru, "s_{3}"'] &                                             &                         &   
\end{tikzcd}
\end{center}  
There are four cells of interest corresponding to words  $ \varepsilon_{0} =  w_{0} w_{1} w_{2} w_{3} \rho^{3} $, $   \varepsilon_{1} =  w_{1} \varepsilon_{0}  $, $  \varepsilon_{2} =  w_{1} w_{2}  w_{3}  \varepsilon_{0} $ and $   \varepsilon_{3} = w_{1}  w_{2} w_{3} w_{2}  w_{1} \varepsilon_{0} $.    Their   Schubert cells are  
\begin{align*}    
\mathcal{X}_{\varepsilon_{0}}  & = \scalebox{0.8}{$\Set*{\left(\begin{array}{cccccc}
1 &  &  &  &  &  \\
 & \varpi  &  &  &  &  \\
 &  & \varpi  &  &  &  \\
y\varpi  & a\,\varpi^2  & c\,\varpi^2  & \varpi^3  &  & \\
a\,\varpi  &  &  &  & \varpi^2  &  \\
c\,\varpi  &  &  &  &  & \varpi^2 
\end{array}\right) }$,} &  \! \!     &  \mathcal{X}_{\varepsilon_{2}} : = \scalebox{0.8}{$\Set*{  \left(\begin{array}{cccccc}     
\varpi^2  & a_1 +c\,\varpi  & c_1 \,\varpi  & \varpi \,z  &  &  \\
 & 1 &  &  &  &  \\
 &  & \varpi  &  &  &  \\
 &  & & \varpi  &  &  \\
 & - y\varpi & -a\,\varpi^2  & -\varpi \,{\left(a_1 +c\,\varpi \right)} & \varpi^3  &\\
 & -a\,\varpi  &  & -c_1 \,\varpi  &  & \varpi^2  
\end{array}\right)} $}   \\
\mathcal{X}_{\varepsilon_{1}} &  =    \scalebox{0.8}{$\Set*{\left(\begin{array}{cccccc}
\varpi  & a_1  &  &  &  &  \\
 & 1 &  &  &  &  \\
 &  & \varpi  &  &  &  \\
 & a\,\varpi  &  & \varpi^2  &  & \\
a\,\varpi^2  & y  \varpi  & c\,\varpi^2  & -a_1 \,\varpi^2  & \varpi^3  &\\
 & c\,\varpi  &  &  &  & \varpi^2
\end{array}\right) }$,} &  \! \!   &   \mathcal{X}_{\varepsilon_{3}} = \scalebox{0.8}{$\Set*{ \left(\begin{array}{cccccc}
\varpi^3  & a\,\varpi^2 +a_2 \,\varpi  & c\,\varpi^2 +c_2 \,\varpi  &   z   
& a_1 \,\varpi^2  & c_1 \,\varpi^2 \\ 
 & \varpi  &  & a_1  &  &  \\
 &  & \varpi  & c_1  &  &  \\
 &  &  & 1 &  &  \\
 &  &  & - (a_2 + a\,\varpi   )  & \varpi^2  &  \\
 &  &  & -( c_2  +  c\,\varpi  )  &  & \varpi^2 
\end{array}\right)}$}  
\end{align*}
where $ a, a_{1},  a_{2}  ,  c, c_{1}, c_{2}  \in [\kay] $, $   y \in [ \kay_{2}] $ and $ z \in [\kay_{3}] $.  Then we claim that the $ U' $-orbits on   
\begin{itemize} [before = \vspace{\smallskipamount-2pt}, after =  \vspace{\smallskipamount-2pt}]    \setlength\itemsep{0.1em}     
\item $ \mathcal{X}_{\varepsilon_{0}} $ are represented by $ \varpi^{(3,3,2,2)} $, $ \varpi^{(2,2,1,1)} \tau_{1} $, 
\item $ \mathcal{X}_{\varepsilon_{1}}  $ are represented by $  \varpi^{(3,2,3,2)} ,  \varpi^{(2,1,2,1) } \tau _{1} $,   
\item  $ \mathcal{X}_{\varepsilon_{2}} $ are represented by $ \varpi^{(3,2,3,2)}  $,  $  \varpi^{(2,1,2,1)}   \tau _{1} $, 
$  \varpi^{(2,1,1,2)} \tau _{1}  $,       $  \varpi^{(3,2,1,2)} \tau_{2} 
$, 
\item  $ \mathcal{X}_{\varepsilon_{3}} $ are 
represented by  $ \varpi^{ (3,3,2,2) }, \varpi^{(2,2,1,1) } \tau _{1} $,  $  \varpi^{(2,2,0,1)}  \tau  _{1} $,  $ \varpi^{ (3,2,1,2)  }  \tau  _{2} $.  
\vspace{-0.1em}    
\end{itemize} We record our reduction steps for $ \mathcal{X}_{\varepsilon_{3}} $. 
Begin by eliminating the entry $ y   $ by a  row operation.  Observe  that if $ a_{1} $ (resp., $ c_{1}) $ is not zero, then we can assume $ a $ (resp.,  $c$) is zero by row column operations. Moreover, conjugation by $ w_{2} $ switches the places of $ a,  a_{1},  a_{2} $ by $ c, c_{1}, c_{2} $ respectively  and   keeps  the  diagonal  $ \varpi ^ { ( 3,3,1,1)} $.    We have three cases to discuss.   \vspace{0.7em}    

\noindent  \textit{Case 1.} Suppose $ a_{1} = c_{1} = 0 $. Apply row column operations to replace $ a \varpi^{2} + a_{2} \varpi $, $ c \varpi^{2} + c_{2} \varpi $ by their greatest common divisor (with the other entry being zero).  Since we can  swap entries by $ w_{2} $, let's assume that  $ a \varpi^{2} + a_{2} \varpi = 0 $. We may normalize the gcd by an element of $ A^{\circ} $ so that the greatest common divisor is  $ 0 $ or $  \varpi $ or $ \varpi^{2} $.  Now conjugate by  $ s_{2} r_{0} r_{2} = s_{2} (s_{1} s_{0}s_{1}) s_{3}  \in  U  $ to   makes the diagonal $ \varpi^{(3,3,2,2)} $ and put the non-diagonal entries in right place.  Thus this case leads us to    representatives  $  \varpi^{(3,3,2,2)} \tau_{1} $, $ \varpi^{(3,3,2,2)} \tau_{2} $.   \vspace{0.7em}

\noindent    \textit{Case 2.}  Suppose exactly one of $ a_{1} $, $ c_{1} $ is non-zero. Since we can swap these, we may assume  wlog  $ a_{1} \neq 0 $, $ c_{1} = 0 $.  Then we are free to make  $ a = 0 $. Now if  $  a_{2} \neq 0 $, it  can be used to replace the entries  $ a_{1} , c , c_{2} $ by zero.   Conjugating by $ r_{0} r_{2} =  w_{2} w_{3} w_{2} w_{3} $ and normalizing by $ A^{\circ} $ gives us $ \varpi^{(3,3,2,2)} \tau_{2} $. If however $  a_{2} = 0 $, then we can conjugate by $ w_{3} $ to make the diagonal $ \varpi^{(3,3,1,2)} $ while moving the $ c \varpi^{2} + c_{2} \varpi $ entry corresponding to the root group of $ e_{1} + e_{3} - e_{0} $. As $ a_{1} \neq 0 $, we are free to eliminate $ c_{1} $.  There are now two further sub-cases. If $ c_{2} = 0 $, we obtain the representative $ \varpi^{(3,3,1,2)} \tau_{1} $ after normalizing by an element of $ A^{\circ} $. If however $ c_{2} \neq 0 $, we can replace $ a_{1} = 0 $ and  conjugating by $ w_{2} w_{3} \in U $ and normalizing by $ A^{\circ} $ gives us $ \varpi^{(3,3,2,2)} \tau_{2} $.  \vspace{0.7em}

\noindent   \textit{Case 3.}  Suppose both $ a_{1} $, $ c_{1} $ are non-zero.   Then we may assume $ a, c $ are zero. If $ a_{2} $ (resp., $c_{2}$) is not zero, we can  eliminate entries containing  $ a_{1} $ (resp.,  $c_{1} $).  Then an argument similar to Case 2 yields $ \varpi^{(3,3,2,2)} \tau_{2}  $, $ \varpi^{(3,3,1,2)}\tau_{1} $ as representatives.     \\ 

\noindent   $ \bullet     $      $ w =  \upsilon_{3}   \rho^{4} $.  The Weyl orbit diagram  for $  \lambda_{w}   =    (4,1,1,1) $ is the same as for $ (1,0,0,0) $ and so we have to analyze cells of length $  \varepsilon_{0} =  w_{0} w_{1} w_{0} w_{2} w_{1} w_{0} \rho^{4} $ and $ \varepsilon_{1} =  w_{1} w_{2} w_{3}  \varepsilon_{0} $. The two cells are as follows: 
\begin{align*} \mathcal{X}_{\varepsilon_{0}} &  =  \scalebox{0.8}{$  \Set*{\left(\begin{array}{cccccc}
1  &  &  &   &  &  \\
 &  1  &  &  &  &  \\
 &  &  1  &  &  &  \\
x_{2} \,\varpi  & a_1 \,\varpi  & c\,\varpi  & \varpi^2 &  &  \\
a_1 \,\varpi   & -  x_{1} \, \varpi  & - a\, \varpi  &  & \varpi^2  &  \\
c\,\varpi & -  a\, \varpi  & x 
 \,  \varpi  &  &  & \varpi^2 
\end{array}\right) \rho^{2} 
\given    \begin{aligned} &  a ,  a_{1} ,  c, x    \\
&   x_{1}, x_{2}  \in   [\kay ]  
\end{aligned}}  $} \\   \mathcal{X}_{\varepsilon_{1} }   & =    \scalebox{0.8}{$  \Set*{\left(\begin{array}{cccccc}
\varpi^2  &  a_2 + c \,\varpi &  c_{1}  +a \, \varpi  & z +  x\,\varpi  &  &  \\
& 1  & & & & \\
& & 1  & &  &  \\
& & & 1  &  &  \\
& - x_2  \,  \varpi  & - a_{1}  \varpi    & - ( a_{2} + c\,\varpi ) & \varpi^2  &  \\
& -  a_{1}  \varpi   & x_1  \,  \varpi  &  -  ( c_{1} +  a\,\varpi  )  &  & \varpi^2 
\end{array}\right) \rho^{2} \given  \begin{aligned}  &   a ,  a_{1}  , a_{2} ,  c, c_{1} , \\ &   x ,  x_{1}, x_{2},  z       \in   [\kay ]  
\end{aligned}   
}    $}  
\end{align*} We claim that   the $ U'$-orbits on 
\begin{itemize}  [before = \vspace{\smallskipamount-2pt}, after =  \vspace{\smallskipamount-2pt}]    \setlength\itemsep{0.1em}         
\item 
$\mathcal{X}_{\varepsilon_{0}} $ are given by  $  \varpi^{(4,3,3,3)} $, $ \varpi^{(3,2,2,2)}  \sigma _{1} $,   
\item   $ \mathcal{X}_{\varepsilon_{1}} $ are given $  \varpi^{(4,3,3,3)} $, $ \varpi^{(3,2,2, 2)}  \tau_{1} $, $ \varpi^{(4,2,2,3)}  \tau _{2} $,  \vspace{-0.1em}     
\end{itemize} 
We record our analysis for orbits on $ \mathcal{X}_{\varepsilon_{1}} $. 
We can eliminate the entries involving $ a_{1},  x , x_{1} , x_{2} , z $ using row  operations. Conjugating by $ w_{3 }$ and $ w_{2} w_{3} w_{2} $ gives us  $$ \scalebox{0.8}{$  \left(\begin{array}{cccccc}
\varpi^2  & & &  & a_2 +c\,\varpi  &  c_{1}  +  a\,\varpi \\
 & \varpi^2  &  & a_2 +c\,\varpi  &   &  \\
 &  & \varpi^2  & c_{1} +  a\,\varpi &  &  \\
 &  &  & 1 &  &  \\
 &  &  &  & 1 &  \\
 &  &  &  &  & 1
\end{array}\right)   \rho^{2}  $} $$ 
and one can apply Euclidean algorithm to the entries $ c_{1} + a \varpi $, $ a_{2} + c \varpi $ to replace one of them with $ 0 $ and the other by the greatest common divisor which is either $ 0, 1 $ or $ \varpi $.   Conjugating by $ w_{2} $ and normalizing by $ A^{\circ} $ if necessary, we obtain  the  three representatives.  \\ 

\noindent $  \bullet $ $  w =  
\upsilon_{4} \rho^{4} $. We have $ \lambda_{w} = (4,0,2,2) $ and the Weyl orbit diagram is the same as for $ (2,0, 1,1) $.  We need to analyze  the 
Schubert cells  corresponding to $ \varepsilon_{0} = w_{0} w_{1} w_{2} w_{3} w_{2} w_{1} w_{0}  \rho^{4} $, $ \varepsilon_{1} =  w_{1}  w $ and $ \varepsilon_{2} =   w_{1} w_{2} w_{3} w_{2} w_{1}  w $. These cells are 
$$  \mathcal{X}_{\varepsilon_{0}} =   \scalebox{0.8}{$\displaystyle \Set*{  \left(\begin{array}{cccccc}
1 &  &  &  &  & \\
a\,\varpi  & \varpi^2  &  &  &  & \\
c\,\varpi  &  & \varpi^2  &  &  & \\
y\varpi  & a_1 \,\varpi^3  & c_1 \,\varpi^3  & \varpi^4  & -a\,\varpi^3  & -c\,\varpi^3 \\
a_1 \,\varpi  &  &  &  & \varpi^2  & \\
c_1 \,\varpi  &  &  & & & \varpi^2 
\end{array}\right)}$}, \quad  \mathcal{X}_{\varepsilon_{1}} = \scalebox{0.8}{$ \Set*{ \left(\begin{array}{cccccc}
\varpi^2  & a_2 +a\,\varpi  &  &  &  &  \\
 & 1 &  &  &  & \\
 & c\,\varpi  & \varpi^2  &  &  &  \\
 & a_1 \,\varpi  &  & \varpi^2  &  & \\
a_1 \,\varpi^3  & y \, \varpi  & c_1 \,\varpi^3  & - (a_{2} + a \, \varpi) \varpi^{2}   & \varpi^4  & -c\,\varpi^3 \\
 & c_1 \,\varpi  &  &  &  & \varpi^2 
\end{array}\right)}$}     
$$
$$     \mathcal{X}_{\varepsilon_{2}} = \scalebox{0.8}{$ \Set*{ \left(\begin{array}{cccccc}  
\varpi^4  & a_1 \,\varpi^3 +a_3 \,\varpi^2  & c_1 \,\varpi^3 +c_3 \,\varpi^2  &   & a\,\varpi^3 +a_2 \,\varpi^2  & c\,\varpi^3 +c_2 \,\varpi^2 \\
 & \varpi^2  &   & a_2 +a\,\varpi  &   &  \\
  &   & \varpi^2  & c_2 +c\,\varpi  &   &  \\
  &   &  & 1 &  &  \\
 &   &   & -a_3 -a_1 \,\varpi  & \varpi^2  &  \\
 &  &  & -c_3 -c_1 \,\varpi  &   & \varpi^2       
\end{array}\right)}$}      $$ 
where $ a , a_{1}, a_{2}, a_{3},  c , c_{1} , c_{2} , c_{3} \in  [\kay] $ and $ y \in [\kay_{3}] $.  We claim that  the $ U ' $-orbits on  
\begin{itemize}   [before = \vspace{\smallskipamount}, after =  \vspace{\smallskipamount}]    \setlength\itemsep{0.1em}   
\item  $ \mathcal{X}_{\varepsilon_{0}} $   are given by  $  \varpi^{(4,4,2,2)} $, $ \varpi^{(3,3,1,1)} \tau _{1} $,  
\item $ \mathcal{X}_{\varepsilon_{1}} $ are given by 
$  \varpi^{(4, 2, 4, 2 )} $, $ \varpi^{(3,2,0,1)} \tau_{1}  $,  $  \varpi^{(4,3,1,2)}  \tau _{2}  $     
\item  $ \mathcal{X}_{\varepsilon_{2}} $ are given by   
$      \varpi^{(4,4,2,2)} $, $ \varpi^{(3,3,1,1)} \tau  _{1} $,  \vspace{-0.1em}   
\end{itemize}   
Let us record our steps for the reduction of $ \mathcal{X}_{\varepsilon_{1}} $. 
We begin by eliminating the entries involving $ y $, $ c$, $ c_{1} $  using row operations. If $ a_{1} = 0 $, then conjugating $  r_{0} = w_{2} w_{3} w_{2} $ and normalizing by an appropriate element of $ A^{\circ} $, we obtain     $ \varpi^{(4,2,4,2)} $, $ \varpi^{(3,1,3,1)} \tau_{1} $, $ \varpi^{(4,1,3,2)} \tau_{2} $ depending on the valuation of $ a_{2 } + a \varpi $. Now $$ U  ' \varpi^{(3,1,3,1)} \tau_{1} K =  U '  \varpi^{(3,2,0,1)} \tau_{2} K, \quad \quad U ' \varpi^{(4,1,3,2) }   K   =  U ' \varpi^{(4,3,1,2)}  \tau_{2} K $$ 
by Proposition \ref{cartanE}. If however $ a_{1} \neq 0  $, then $ a $ can be made zero via row-column operations. We then  have two  further 
subcases. If $ a_{2} = 0 $, then we can conjugate by $  s_{0} = 
s_{\alpha_{0}} $ and normalize by $ A^{\circ} $ to obtain $ \varpi^{(3,1,3,1)} \tau_{1} $ which is the same as $ \varpi^{(3,2,0,1)}$. On the other hand, if $ a_{2} \neq 0 $,  then $ a_{1} $ can be made zero and   normalizing by $ A^{\circ} $ gives $ \varpi^{(4,1,3,2)} \tau_{2} $ which is the same as $ \varpi^{(4,3,1,2)} \tau_{2} $.   
\end{proof}

\begin{remark} \label{formidableremark}    If one instead tries  to directly study the $ U $-orbits on the double cosets in the proof above, one needs to study far more Schubert cells and distinguish an enormous number of representatives from each other. For instance for $ w = \upsilon_{2} \rho^{3} $, one would need to study $12$  cells instead of $ 4 $.  
\end{remark}  
\section{Double cosets of  \texorpdfstring{$\GL_{2} \times \mathrm{GSp}_{4}$}{GL2 x GSp4}} 
\label{Uorbitssec}
In this section, we record the proofs of various  claims involving the action of $ U $ on double cosets spaces of $ H ' $. Since both $ U$ and $U'$ have a common $ \GL_{2}(\Oscr_{F} ) $ component, the computation of orbits is facilitated   by  studying the orbits of $ U_{2} $ on double cosets of $H _{2}  '   $. This in turn is  achieved by techniques analogous to the one used in \S \ref{U'orbitssec} for decomposing double cosets of parahoric subgroups of an unramified group.

\begin{notation}  \label{notationh2}    
If $ h \in H_{2} \subset H_{2}  ' $, we will often write  $$  h  =   \left (  \begin{smallmatrix} a  & & b  \\  & a_{1} & &  b_{1} \\
  c & & d  \\ &  c_{1} & &  d_{1} \end{smallmatrix}   \right  )   \quad  \text { or }  \quad     h = \left (  \left ( \begin{matrix} a  & b \\ c & d  \end{matrix}  \right )  ,   \left  (   \begin{matrix} a_{1} & b_{1} \\ c_{1} & d_{1}    \end{matrix}   
     \right )  \right )    . $$
We let $ \Lambda_{2} $ denote $ \ZZ f_{0} \oplus \ZZ f_{2} \oplus \ZZ f_{3} $. Given $  \lambda = a_{0} f_{0} + a_{2} f_{2} + a_{3} f_{3}  \in \Lambda_{2} $ as $ (a_{0}, a_{2}, a_{3}) $ and let $ \varpi^{\lambda} $ denote the element  $ \mathrm{diag}(\varpi^{a_{2}}, \varpi^{a_{3}}, \varpi^{a_{0} -a_{2}},  \varpi^{a_{0} - a_{3} })  \in H_{2}' $.    
\end{notation}

\subsection{Projections}  \label{projectionprelim}    
Let $ s : \Hb'_{2} \to \Hb'  $ denote the section of $ \pr_{2}' $ given  by $ \gamma \mapsto   \left (  \left (   \begin{smallmatrix}  \mathrm{sim}(\gamma)  \\ &  1  \end{smallmatrix}   \right )  ,    \gamma   \right  )    $. Fix a compact open subgroup $ V \subset H ' $ such that $ \pr_{1}'(V) = \GL_{2}(\Oscr_{F} )$ and an arbitrary  element $ h = (h_{1}, h_{2} ) \in H ' $. Denote $ V_{2} = \pr_{2}'(V) $.  We  refer to   \begin{align*}   \pr_{h, V}   :  U \backslash U' h  V / V 
&  \to U_{2} \backslash U_{2}' h_{2}  V_{2} / V_{2} , \\  
U  \gamma  V   &    \mapsto U_{2} \pr_{2}(\gamma )   V_{2}
\end{align*}  
as the projection map. We are interested in the fibers of $ \mathrm{pr}_{h,V} $.

\begin{lemma}  \label{pVbijection}  Suppose $ h_{1} \in \GL_{2}(F) $ is diagonal and either $ s(V_{2}) \subset V $ or $ h_{1} $ is central.  If $ \eta \in H_{2}' $ has the same similitude as $ h $ and $ U_{2}  \eta  V_{2}  \in U_{2}  \backslash U_{2} ' h_{2} V_{2}/ V_{2} $, then $   \{  U (h_{1}, \eta) V    \} = \pr_{h,V}^{-1}(U_{2} \eta V_{2} )   $. In particular, $ \mathrm{pr}_{h,V} $ is a bijection.  

\end{lemma} 
\begin{proof} 
Note that  any element of $ U  \backslash U ' h V / V $  can be written as  $ U (1, \gamma) h V $ for some $ \gamma \in  \mathrm{Sp}_{4}(\Oscr_{F}) $ and similarly for elements of $ U_{2} \backslash  U_{2}' h_{2} V_{2}/ V_{2} $.  This immediately implies that $ \pr_{h,V} $ is surjective.  

Suppose now that $ \gamma \in \mathrm{Sp}_{4}(\Oscr_{F} ) $  is such that   $  U (1, \gamma ) h V $ maps to  $ U_{2} \eta V_{2} $ under $ \mathrm{pr}_{h,V}$.  
Then  there  exist $ u_{2} \in U_{2} $, $ v_{2} \in V_{2} $  such that $ \eta  =  u_{2} \gamma  h_{2} v_{2}  $. Taking similitudes, we see that $ \mathrm{sim}(u_{2})   = \mathrm{sim}(v_{2}) ^{-1} $.   Let $ u_{1} = \mathrm{diag}(1, \mathrm{sim}(u_{2}) )  \in \GL_{2}(\Oscr_{F}) $ and set $ u = (u_{1} , u_{2}) \in  U $.
Take $ v = s(v_{2})  \in V $ if $ s(V_{2}) \subset V $ or an arbitrary element in $ (\pr_{2}')^{-1}(V) $ if $ h_{1} $ is central.  Write $ v = (v_{1}, v_{2}) $. Then  
\begin{align*} U (1,\gamma) h V & = U u (1,\gamma) h v V \\
& = U (u_{1}h_{1} v_{1}, u_{2} \gamma_{2} h_{2} v_{2} ) V\\ 
& =  U ( u_{1}v_{1} h_{1}, \eta ) V  \\
&
= U (h_{1}, \eta) h V 
\end{align*}
where we used that $ h_{1} $ commutes with $ v_{1} $ in both cases and  that  $ (u_{1}v_{1}, 1) \in U_{2} $.  
\end{proof}
In case $h_{1}$ is non-central or $ s(V_{2}) \not \subset V $,  one needs to perform an additional check to determine the fibers of $ \pr_{h, V}  $.  Define 
\begin{alignat*}{3}
S^{-} &  = \left \{  \left (   \begin{smallmatrix} 1 & \\  x & 1  \end{smallmatrix}  \right )  \,
|  \,  x \in \Oscr_{F}  \right \} ,  \quad \quad \quad   &     S^{+} & = 
\left \{
\left(\begin{smallmatrix}
& 1 \\ -1 \end{smallmatrix}\right )  
\left ( \begin{smallmatrix} 1 &  \varpi x \\ & 1 \end{smallmatrix} \right ) \, | \, x \in \Oscr_{F}      \right \}.
\end{alignat*} 
For a positive integer $ a  $, define $ S_{a}^{-} $ to be the subset $ S^{-} $ where we require the variable $ x $ to lie  in $  [\kay_{a} ]$ (see \S \ref{generalnot} for notation)  and $ S_{a}^{+} $ the subset of $ S^{+} $ where we require  $ x $ to lie in $  [\kay_{a-1}   ]  $. 
We also  denote $ S^{\pm} = S^{-} \cup S^{+} $ and $ S^{\pm}_{a} = S^{-}_{a} \cup S^{+}_{a} $. 

\begin{corollary} \label{fibersofpV}  Suppose   $ h_{1 }  =   \mathrm{diag}  (\varpi^{u}, \varpi^{v}) $ with  $ u > v $ and  $ \eta \in H_{2}' $ is such that  $ U_{2} \eta V_{2} \in   U_{2} \backslash  U_{2}' h V_{2} / V_{2} $ with   $ \mathrm{sim}(\eta) =    \varpi^{ u+v }  $. Then $$  \pr^{-1}_{h ,  V }( U_{2} \eta V_{2} ) =   \left \{  U (h_{1} \chi ,   \eta )   V  \, | \,  \chi  \in S^{\pm} _{ u -  v}  \text{ and }    \,   U'  (h_{1} \chi , \eta ) V = U' h V    \right \}  
$$  
\end{corollary}

\begin{proof} In the proof of Lemma \ref{pVbijection}, one obtains the equality $ U ( 1, \gamma ) h V = U (h u_{1} v_{1} , \eta V ) $ with $ u_{1} v_{1} \in \SL_{2}(\Oscr_{F} ) $. Now $ u_{1} v_{1} $ can be replaced with a   representative in the quotient $$ \big ( \SL_{2}(\Oscr_{F} )  \cap h_{1}  ^{-1} \SL_{2}(\Oscr_{F}  )  h_{1}  \big  )    \backslash \SL_{2}(\Oscr_{F}) $$ 
and $ 
S^{\pm}_{u-v} $ forms such a set of representatives.        
\end{proof}
\begin{remark} We will need to use the last result for $ V \in \left \{ H_{\tau_{1}}' , H_{\tau_{2}}' \right \} $ when lifting coset representatives $ \eta $  for $ U_{2} \backslash U_{2}' h_{2} V_{2}/ V_{2}  $ to $ U \backslash U' h V/ V $.  In almost all cases,  it will turn out that there is essentially one choice of $ \gamma \in S^{\pm} $ that satisfies $ U ' (h_{1} \gamma , \eta)V = U' h V  $. If  there are more than one element in the fiber, we will invoke  a suitable Bruhat-Tits  decomposition for parahoric double cosets to distinguish them.     \end{remark}

\subsection{The $\mathrm{GSp}_{4} $-players}  \label{players}      
Recall that the roots of $ H_{2}' = \mathrm{GSp}_{4}  $  are  identified with $$ \left \{ \pm \beta_{0}, \pm \beta_{1}, \pm \beta_{2} , \pm (\beta_{1} + \beta_{2})    \right \} .  $$     To compute these decompositions,  we let 
$$  
v_{0} =   \scalebox{0.9}{$ \left (   
\begin{matrix} & & \mfrac{1}{\varpi}\\ 
& 1 & &  \\ 
\varpi & & &  
\\ & & & -1 
\end{matrix} \right )$} , \quad \quad 
v_{1} =   \scalebox{0.9}{$ 
\left ( 
\begin{matrix} &  1\\ 
1&   & & \\ 
& & & 1 \\ 
& & 1 &  
\end{matrix} \right )$} , \quad \quad  
v_{2} =  
\scalebox{0.9}{$\left ( \begin{matrix} 1 & \\ 
&  &  & 1  \\  
& & - 1 \\ 
& 1 & &   
\end{matrix}  \right )$}  $$ 
which respectively represent the reflections $ t(f_{2}) r_{0}, r_{1}, r_{2} $ which generate the affine Weyl group  $ W_{2, \mathrm{aff}}' $ of $ H_{2}' $.  We also denote $
v_{\beta_{0}} = \mathrm{diag}(\varpi, 1, \varpi^{-1}, 1) v_{0} $ which represents the reflection $ r_{0} $   in the root $ \beta_{0}$.   For $ i = 0, 1, 2 $, let $ y_{i} : \GG_{a} \to \Hb_{2}'  $ be the maps $$ y_{0} : u \mapsto    \scalebox{0.9}{$   
\left ( \begin{matrix}  1 &  \\  &  1  &  &  \\   u  \varpi   &  & 1  &  \\ & & & 1  \end{matrix} \right ) $},  \quad 
y_{1} : u \mapsto  \scalebox{0.9}{$  \left (   \begin{matrix} 1  & u \\[0.1em]  &  1 \\ & & 1 \\ & & - u & 1  \end{matrix}    \right )$} , \quad 
y_{2}  : u \mapsto    \scalebox{0.9}{$\left (  
\begin{matrix} 1  &  \\ &  1  &  & u \\ & & 1 \\ & & & 1  \end{matrix}  \right   ) $}   . $$
If $ I_{2}' \subset H'_{2} $ denote the Iwahori subgroup given by $ \pr_{2}(I') $, then $ y_{i}([\kay]) $ forms a set distinct $ q $  representatives for the quotients $ I_{2}'/ I_{2}' \cap v_{i} I_{2}' v_{i} $.  For each $ i = 0, 1, 2 $, let  $$ h_{r_{i}} : [\kay] \to H_{2} ' , \quad \kappa \mapsto y_{i}(\kappa) v_{i} . $$
Let $ W_{I_{2}'} $ denote the Iwahori Weyl group of $H_{2}' $ and  $ l : W_{I_{2}'} \to \ZZ $ denote the length function induced by $ \pr_{2}(S_{\mathrm{aff}}') = \left \{ r_{1}, r_{2} , t(f_{2}) r_{0} \right \} $.   For $ v \in W_{I_{2}'} $ and $  v = r_{v,1} r_{v,2} \cdots  r_{v, l} \omega_{v}   $ (where $ r_{v,i} \in \pr_{2}(S_{\mathrm{aff}}') $, $ \omega_{v} \in \pr_{2}( \Omega_{H'}) $ is a power of $ \pr_{2}(\omega_{H'})$)  is a reduced word decomposition, we set 
\begin{align*} \mathcal{Y}_{v} : [\kay]^{l(v)} & \longmapsto  H_{2} ' \\
(\kappa_{1}, \ldots,  \kappa_{l(v)})   &  \mapsto h_{r_{v,1}}(\kappa_{1})    \cdots h_{r_{v,l(v)}}(\kappa_{l(v)}) \rho_{2,v} 
\end{align*}  
where $ \rho_{2,v} \in H $ is the element representing $ \omega_{v}  $.   For a compact open subgroup $ V \subset H_{2}' $, we let $ \mathcal{Y}_{v}/V $ to denote the coset space  $ \mathrm{im}(\mathcal{Y}_{v})V/V $, which we will also refer to as a Schubert cell.  

\subsection{Orbits on $ U' h U'/U'$}
\label{UU'orbsubsec}
Let $ W_{2} $ denote the Weyl group of $ H_{2} =  \GL_{2} \times_{\GG_{m}} \GL_{2} $. We can identify $ W_{2} $ as the subgroup of $ W_{2}' $ generated by $ r_{0} $ and $ r_{2} $. For $ \eta \in H_{2}' $, denote $ H_{2} \cap \eta U_{2}' \eta^{-1} $ by $ H_{2,\eta} $.  Then the map 
\begin{equation} \label{mydclemma} U_{2} \varpi^{\Lambda_{2}} \eta U_{2}' \to U_{2} \varpi^{\Lambda_{2}} H_{2,\eta} \quad U_{2} \varpi^{\lambda} \eta U_{2} ' \mapsto U_{2} \varpi^{\lambda} H_{2,\eta} \end{equation} 
is a bijection.  Let $ \eta_{1} $, $ \eta_{2} $ denote the projection of $ \varrho_{1} $, $\varrho_{2} $ given in  (\ref{varrhomatrices}) to $ H_{2} '$.  Explicitly,  
\begin{equation}  \label{eta12matrices} \eta_{1}
=   \scalebox{0.9}{$ \left (  \begin{matrix} \varpi  &  & & 1 \\ & \varpi & 1 \\[0.1em] & &  1  \\ &  & & 1 \end{matrix} \right )$} , \quad \eta_{2} =  \scalebox{0.9}{$\left ( \begin{matrix}  \varpi^{2} &  & & 1  \\ & \varpi^{2} & 1  \\[0.1em]   &  & 1 \\ &  & & 1 \end{matrix}  \right )$}   
\end{equation} 
 
\begin{lemma}  \label{distinctetai}  The cosets $ H_{2} U_{2}' $, $ H_{2} \eta_{1} U_{2}' $ and $ H_{2}\eta_{2}U_{2}' $ are pairwise disjoint. 
\end{lemma} 
\begin{proof} This is similar to Lemma   \ref{distincttaui}. See also Remark \ref{Schroder}.  
\end{proof} 
\begin{lemma}  \label{distinctetaireflections}  The map $ W_{2} \backslash \Lambda_{2} \to U_{2} \varpi^{\Lambda_{2}} U _{2}' $ given by $ W_{2} \lambda  \mapsto  U_{2} \varpi^{\lambda} U_{2}' $ is a bijection.   If $ \lambda, \mu \in \Lambda_{2} $ are not in the same $ W_{2} $-orbit, then  $ U_{2} \varpi^{\lambda} \eta_{1}  U_{2}'$ is distinct from  $ U_{2} \varpi^{\mu} \eta_{1}  U_{2}' $. 
\end{lemma} 
\begin{proof}   The first claim follows by the bijection (\ref{mydclemma})  and Cartan decomposition for $ H_{2} $. It is easily verified that $ H_{2,\eta_{1}}  \subseteq U_{2} $, so  the second claim also follows by  Cartan decomposition for $ H_{2}  $.    
\end{proof}

\begin{lemma}  \label{etaireflections}  For $ i = 1 ,2 $ and any $ \lambda \in \Lambda_{2} $, $ U_{2} \varpi^{\lambda} \eta_{i} U_{2}' = U_{2}  \varpi^{r_{0}r_{2}(\lambda)   } \eta_{i} U_{2}'  $.  
\end{lemma} 

\begin{proof} This follows by noting that $ \eta_{i}^{-1} v_{\beta_{0}} v_{2} \eta_{i} 
 \in U_{2}' $ for $ i = 1 , 2 $ and $ v_{\beta_{0}} v_{2} \in U_{2} $.   
\end{proof}   
 
\begin{proof}[Proof of Proposition \ref{R0scrhard}] For $ h \in H' $,   Let $ \mathscr{R}(h) $ denote the double coset space $ U_{2} \backslash U_{2}' h U_{2}'/ U_{2}' $. 
By Lemma \ref{pVbijection}, it suffices to establish that 
\begin{enumerate}[label = (\alph*)]   
\item $ \mathscr{R}(\varpi^{(1,1,1)}) = 
\left \{\varpi^{(1,1,1)}, \, \eta_{1} \right \} $, 
\item $ \mathscr{R}(\varpi^{(2,2,1)}) = \left \{\varpi^{(2,2,1)}, \, \varpi^{(2,1,2)},\,  \varpi^{(1,1,0)} \eta_{1}  \right \} $,
\item $ \mathscr{R}(\varpi^{(2,2,2)}) = 
\left \{\varpi^{(2,2,2)},\, \varpi^{(1,1,1)}\eta_{1},\,   
\eta_{2} \right \} $ 
\item $ \mathscr{R}(\varpi^{(3,3,2)}) = 
\left \{\varpi^{(3,3,2)},\,  \varpi^{(3,2,3)},\, \varpi^{(2,2,1)}\eta_{1},\,
\varpi^{(2,2,0)}\eta_{1},\,
\varpi^{(1,1,0)}\eta_{2} \right \}$, 
\item $ \mathscr{R}(\varpi^{(4,4,2)})  = \left \{ \varpi^{(4,4,2)}, \, \varpi^{(4,2,4)},\, \varpi^{(3,3,1)} \eta_{1},\,  
\varpi^{(2,2,0)}\eta_{2}\right \}  
$. 
\end{enumerate} 
It is easy to check using Lemma \ref{distinctetai} and  Lemma \ref{distinctetaireflections}  that the listed elements in each case represent distinct double cosets. It remains to show that they form a complete set of representatives. Here we again use the recipe given by \cite[\S 5]{CZE}. As before, we will write the parameters of the below them and omit writing $ U_{2} ' $ next to the matrices.  
\\ 

\noindent (a) \& (b) These were calculated in \cite[Proposition 9.3.3]{CZE}.   \\

\noindent (c) We have $ U_{2}' \varpi^{(2,2,2)} U_{2}' = U_{2}' v_{0}v_{1}v_{0} \rho^{2} _{2}  U_{2}' $ and $ v_{0} v_{1} v_{0} \rho^{2}_{2}  $ is of minimal possible length. The Weyl orbit diagram of $(2,2,2) $ is 
\begin{center} 

\begin{tikzcd}
 \arrow[r, "r_{2}"] & \arrow[r, "r_{1}"] & {} \arrow[r, "r_{2}"] &   {} 
\end{tikzcd}
\end{center}
So we need to analyze the cells corresponding to the first and the third node, which are of length $ 3 $ and $ 5 $ respectively. Let  $ \varepsilon_{0} = v_{0}v_{1}v_{0} \rho^{2}_{2} $ and $   \varepsilon  _{1} = v_{1} v_{2}v_{0}v_{1}v_{0} \rho^{2}_{2} $ be the words corresponding to these nodes. We have \begin{align*}   
 \mathcal{Y}_{\varepsilon_{0}}  / U_{2}'  & = \scalebox{0.9}{$     
\Set*{\left(\begin{array}{cccc}
1 & & & \\
& 1 & & \\[0.1em] 
x_{1} \varpi & a \varpi  & \varpi^{2} &  \\[0.1em] 
a \varpi   & x \varpi  &   & \varpi^{2}  
\end{array}\right)}$}, \quad \quad  
 \mathcal{Y}_{\varepsilon_{1}}   / U_{2}'    = \scalebox{0.9}{$  
\Set*{\left(\begin{array}{cccc}
\varpi^{2} & a_{1}+a\varpi & y+\varpi x  \\[0.1em]
& 1  & & \\
& & 1 &\\[0.1em]
& x_{1} \varpi  & -(a_{1} + a \varpi) & \varpi^{2}
\end{array}\right)}$} 
\end{align*} 
where $ a, a_{1} , x , y $ run over $ [\kay] $. 
For the first cell, eliminate $ \varpi x_{1} $, $ \varpi x $ via row operations and conjugate by $ v_{\alpha_{0}}v_{0} $.  For the second, eliminate $ y+ \varpi x $, $ \varpi x_{1} $ similarly and conjugate by $v_{2} $. The resulting matrices are {\setlength{\abovedisplayskip}{7pt}\setlength{\belowdisplayskip}{7pt}\[ \scalebox{0.9}{$ \left( 
\begin{matrix}  
\varpi^{2} & & & a\varpi \\
& \varpi^{2} & a\varpi\\[0.1em]
& & 1 \\
& & & 1
\end{matrix}  \right )$}, \quad \quad    \scalebox{0.9}{$ \left( 
\begin{matrix}  
\varpi^{2} & & & a_{1} + a\varpi \\
& \varpi^{2} & a_{1} + a\varpi\\[0.1em]  
& & 1 \\
& & & 1
\end{matrix}  \right )$} \]}
respectively. By conjugating with appropraite diagonal matrices,  the left matrix can be simplified to $ \varpi^{(2,2,2)} $ or $ \varpi^{(1,1,1)}\eta_{1} $ depending on whether $ a $ is zero or not. Similarly the second one simplifies to one of $ \varpi^{(2,2,2)}$, $\varpi^{(1,1,1)}\eta_{1} $, $ \eta_{2} $.  \\

\noindent (d)  We have  $ U_{2} ' \varpi^{(3,3,2)} U_{2}'  = U_{2}' v_{0}v_{1}v_{2}\rho^{3}_{2}   U_{2'} $ with $ v_{0} v_{1} v_{2} \rho^{3}_{2} $ of minimal  possible length. The Weyl orbit diagram of $ (3,3,2)$ is

\begin{center} 
\begin{tikzcd}
& {} \arrow[r, "r_2"]  & {}  \arrow[r,"r_{1}"] & {}   \arrow[rd,"r_{2}"] &  {}    \\
\arrow[ru,"r_{1}"]\arrow[rd,"r_{2}"]& &  & & {}   \\
& {} \arrow[r, "r_1"] & {} \arrow[r, "r_2"] {} & {} \arrow[ru, "r_1"] & {}    
\end{tikzcd}
\end{center} 
There are four cells to analyze which have lengths $3 $, $4$, $5$ and $ 6 $.  These correspond to $ \varepsilon _{1} = v_{0} v_{1} v_{2} \rho^{3}_{2} $, $ \varepsilon  _ { 2} = v_{1} \varepsilon  _{1} $, $  \varepsilon _{3}  =  v_{1} v_{2}  \varepsilon  _{1} $ and  $ \varepsilon _{4} = v_{1} v_{2} v_{1}  \varepsilon _ {1}  . $      
The matrices in the corresponding cells are as follows: 
\begin{alignat*}{4}  \mathcal{Y}_{\varepsilon _{0}}/ U_{2}'  &  =  \scalebox{0.9}{$\Set*{\left(\begin{array}{cccc}
1 & & & \\
& \varpi & & \\[0.2em] 
z \varpi  & a \varpi^{2} & \varpi^{3} &  \\[0.2em] 
a \varpi  & &   & \varpi^{2}  
\end{array}\right)}$},& 
 \mathcal{Y}_{\varepsilon _{2}}  / U_{2}'   &= \scalebox{0.9}{$\Set*{\left(\begin{array}{cccc}
\varpi^{2} & a_{1} + a\varpi & y_{1} \varpi & \\[0.2em] 
& 1 & & \\
&  & \varpi & \\[0.2em] 
 &  z \varpi   & (a_{1} + a \varpi)\varpi  & \varpi^{3}  
\end{array}\right)}$},  \\
\mathcal{Y}_{\varepsilon _{1}} / U_{2}'  & = \scalebox{0.9}
{$\Set*{\left(\begin{array}{cccc}
\varpi & a_{1} & & \\
&1 & & \\[0.2em] 
& a \varpi & \varpi^{2}  &  \\[0.2em] 
a \varpi^{2 }  & z \varpi  & -a_{1} \varpi^{2}   & \varpi^{3}  
\end{array}\right)},$}  
&
\quad  \quad \mathcal{Y}_{\varepsilon _{3}} / U_{2} '     &=  \scalebox{0.9}
{$\Set*{\left(\begin{array}{cccc}
\varpi^{3} & (a + a_{2} \varpi) \varpi  & y_{1} +  z \varpi & a_{1} \varpi^{2}  \\
& \varpi &  a_{1}  & \\[0.2em] 
 &  & 1 &  \\[0.2em] 
 & &  -(a_{2}+  a \varpi )  &   \varpi^{2}   
\end{array}\right)}$}   
\end{alignat*} 
where $  a, a_{1} , a_{2} , y_{1} 
 \in [\kay] $ and $ z \in  [\kay_{2} ]  $.    From these matrices and using elementary row/column operations\footnote{A slightly non-obvious operation is $ \varpi^{(2,0,2)} \eta_{2} \to  \varpi^{(2,2,0)} \eta_{2} $ obtained from  Lemma \ref{etaireflections}.}  arising from $ U_{2} $, $U_{2}'$,  one  can deduce that the orbits of $ U $ on  
\begin{itemize} 
\item $  \mathcal{Y}_{\varepsilon_{0}}/ U_{2} ' $ are given by  $  \varpi^{(3,3,2)} $, $ \varpi^{(2,2,1)}\eta_{1} $, 
\item $ \mathcal{Y}_{\varepsilon_{1}}/ U_{2}' $ are given by $ \varpi^{(3,2,3)}  $, $ \varpi^{(2,2,0)} \eta_{1} $,  $ \varpi^{(2,1,2)}\eta_{1} $,   
\item $ \mathcal{Y}_{\varepsilon_{2}}/ U_{2}'  $ are given by $ \varpi^{(3,2,3)} $, $\varpi^{(2,1,2)}\eta_{1} $, $ \varpi^{(1,1,0)} \eta_{2} $,  
\item $ \mathcal{Y}_{\varepsilon_{3}}/ U_{2} '     $ are given by   $    \varpi^{(3,3,2) } $, $ \varpi^{(2,2,0)} \eta_{1}  $, $ \varpi^{(2,2,1)} \eta_{1}   $, $ \varpi^{(1,1,0)} \eta_{2} $.  \\ 
\end{itemize}

\noindent (e)  We have   $ U_{2}'  \varpi^{(4,2,2)}  U_{2}'  =  U_{2}' v_{0}v_{1}v_{2}v_{1}v_{0} \rho^{4}_{2} U_{2}' $ and $  v_{0} v_{1}v_{2} v_{1}v_{0}\rho^{4}_{2} $ is of minimal possible length. The Weyl orbit diagram for $ (4,2,2) $ is   
\begin{center} 
\begin{tikzcd}
 \arrow[r, "r_{1}"] & \arrow[r, "r_{2}"] & {} \arrow[r, "r_{1}"] &   {} 
\end{tikzcd}
\end{center}
So we have three cells to check, corresponding to $ \pi_{1} = v_{0} v_{1} v_{2} v_{1} v_{0} \rho^{4}_{2} $, $ \sigma_{2} = v_{1}  \sigma_{1} $ and $ \sigma_{3} = v_{1} v_{2} v_{1} \sigma_{1} $. The matrices in the corresponding cells are as follows:  

\begin{align*}  \mathcal{Y}_{\varepsilon_{0}} /U_{2} ' =  \scalebox{0.9}{$\Set*{\left(\begin{array}{cccc}
1 & & & \\
a \varpi & \varpi ^{2}  & & \\[0.2em] 
z \varpi &  a_{1} \varpi^{3} & \varpi^{4} &  - a\varpi^{3}   \\[0.2em] 
a_{1} \varpi  & &   & \varpi^{2}  
\end{array}\right)}$},  
\quad  \mathcal{Y}_{\varepsilon_{1}}/ U_{2}'   = \scalebox{0.9}
{$\Set*{\left(\begin{array}{cccc}
\varpi^{2}  & a_{2} + a\varpi  & & \\
&1 & & \\[0.2em] 
& a_{1}  \varpi & \varpi^{2}  &  \\[0.2em] 
a_{1} \varpi^{3}  &  z \varpi  & - ( a_{2} + a \varpi)\varpi^{2}   & \varpi^{4}  
\end{array}\right)}$}   
\end{align*}
\begin{align*}   \mathcal{Y}_{\varepsilon_{2}}/U_{2}'  =     
&\scalebox{0.9}{$\Set*{\left(\begin{array}{cccc}
\varpi^{4} & (a_{1} + a_{3} \varpi)\varpi^{2}   & y_{1} + z \varpi &  (a_{2} + a\varpi) \varpi^{2}\\[0.2em] 
& \varpi^{2} & a_{2} + a \varpi   & \\
&  & 1  & \\[0.2em] 
&  & (a_{3} + a_{1} \varpi)  & \varpi^{2} 
\end{array}\right)}$}\quad \quad \quad
\end{align*} 
where $ a, a_{1}, a_{2}, a_{3}, y_{1} \in [\kay] $ and $ z \in  [\kay_{3} ]  $. From these, one deduces that  the orbits of $ U $ on  
\begin{itemize} 
\item $ \mathcal{Y}_{\varepsilon_{0}}/U_{2}'  $ are given by $  \varpi^{(4,4,2)} $, $ \varpi^{(3,3,1)} \eta _{1} $, 
\item $ \mathcal{Y}_{\varepsilon_{1}}/U_{2}' $ are given by   $ \varpi^{(4,2,4)} $, $ \varpi^{(3,1,3)  } \eta_{1} $, $ \varpi^{(2,2,0)} \eta_{2} $
\item  $ \mathcal{Y}_{\varepsilon_{2}}/U_{2}  '  $ are given by  $ \varpi^{(4,4,2)}$, $ \varpi^{(3,3,1)}\eta_{1} $, $ \varpi^{(2,2,0)}\eta_{2} $. 
\end{itemize}
Note that we make use of $ U_{2} \varpi^{(2,2,0)} \eta_{2} U_{2}' = U_{2} \varpi^{(2,0,2)} \eta_{2} U_{2}' $ which holds by Lemma \ref{etaireflections}.      
\end{proof} 
\subsection{Orbits on $ U'h H_{\tau_{1}}'/H_{\tau_{1}}'$}     \label{UHtau1orbitsec}        The proof of  Proposition \ref{R1scrhard} is based on Lemma  \ref{fibersofpV}. To compute  the decompositions of the projections of $ U'\varpi^{\lambda} H_{\tau_{1}}' $ to $ H_{2}' $, it will be convenient to work the with the conjugate $ H_{\tau_{\circ}}' $ of $ H_{\tau_{1}}' $ introduced in Notation \ref{notationT}. This is done since the projection $   U_{2}^{\circ}  :   =  \pr_{2}'(H_{\tau_{\circ}}')$ is a (standard) maximal  parahoric subgroup of $ \mathrm{GSp}_{4}(F) $. It is possible to perform these computations with $ \pr_{2}'(H_{\tau_{1}}') $ instead, but this requires us to introduce a different Iwahori subgroup of $ \mathrm{GSp}_{4}$.

Recall that $ W_{2}' $ denotes the Weyl group of $ H_{2}' $ and  $ W_{I_{2}'} $ the Iwahori Weyl group. Let $ W^{\circ}_{2} $ denote Coxeter  subgroup of $ W_{I_{2}'} $ generated by $ T_{2} := \left \{ t(f_{2})r_{0} ,  r_{2 }  \right \}  $. Each coset $ W_{2}' w W_{2}^{\circ} \in W_{2}' \backslash W_{I_{2}'} / W_{2}^{\circ}  $ contains a unique element of minimal possible length which we refer to as \emph{$ (W_{2}', W_{2}^{\circ})$-reduced} element. We let $ [W_{2}' \backslash W  _ {I_{2}'} / W_{2}^{\circ}]  $ denote the subset of $ W_{I_{2}'} $ of all $ (W_{2}', W_{2}^{\circ}) $-reduced elements. If $ w \in W_{I_{2}'} $ is such a  reduced element, the intersection $$ W_{2, w}' := W_{2}' \cap w W_{2}^{\circ} w^{-1} $$ is a Coxeter subgroup of $ W_{2}' $ generated by $ T_{2,w}  :=  w T_{2}w^{-1} \cap  W_{2}'   $. Then each coset in $ W_{2}'/ W_{2,w}' $  contains a unique element of minimal possible length. The set of all  representatives elements for $ W_{2}'/W_{2,w}'$ of minimal length  denoted by $ [ W_{2}' / W_{2,w}'  ]$.   
Then the   decomposition  recipe of \cite[Theorem 5.4.2]{CZE} says the following. 
\begin{proposition} For any $ w \in [ W_{2}' \backslash W_{I_{2}}' / W_{2}^{\circ} ] $, $$ U_{2}' w  U^{\circ}_{2}  = \bigsqcup_{\tau }  \bigsqcup_{\vec{\kay} \in [\kay]^{l(\tau w)}} \mathcal{Y}_{\tau w}(\vec{\kay})  U_{2}^{\circ}       $$ 
where $  \tau $ runs over $ [W_{2}'/W_{2,w}']$. 
\end{proposition} 
\begin{remark} Note that $ l(\tau w) = l(\tau) + l(w) $ for $ \tau \in [W_{2}' / W_{2,w}'] $ and $ w \in   [W_{2}'\backslash W_{I_{2}'}/ W_{2}^{\circ}] $.    
\end{remark} 
\begin{lemma}   \label{U2circwords}           For each $ \lambda \in \Lambda_{2}^{+}$, the element $ w = w_{\lambda} \in W_{I_{2}'} $ specified is the unique element in $ W_{I_{2}}' $ of minimal possible length such that $ U_{2}' \varpi^{\lambda}  U_{2}^{\circ}  = U_{2}' w  U_{2}  ^  {    \circ}  $ 
\begin{itemize} 
\item $ \lambda = (1,1,1) $, $ w =  \rho_{2}$
\item $ \lambda = (2,2,2) $, $ w  =   v_{0} v_{1} \rho^{2}_{2} $ 
\item $ \lambda = (3,3,2)  $, $ w  =  v_{0}  v_{1} \rho^{3}_{2}  $  
\item $ \lambda = (3,2,3) $, $ w = v_{0} v_{1} v_{2} v_{1} \rho^{3}_{2} $
\item $ \lambda = (4,4,2) $, $ w = v_{0} v_{1} v_{2} v_{1} \rho^{4}_{2}  $ 
\end{itemize}  
\end{lemma} 
\begin{proof} It is easy to verify the equality of cosets for each $ \lambda $ and $ w $. To check that the length is indeed minimal, one can proceed as follows.     Under the isomorphism, $ U_{2}' \backslash H_{2}'/U_{2}^{\circ} \simeq W_{2}'\backslash  W_{I_{2}'}/ W_{2}^{\circ}   $, the coset $ U_{2}' \varpi^{\lambda}  U_{2} ^  { \circ } $ corresponds to $ W_{2}' t(-\lambda) U_{2} ^{\circ} 
   $. The minimal possible length of elements in $ W_{2}' t(-\lambda ) U_{2} ^  { \circ } $ is the same as that for $ U_{2} ^{ \circ } t(\lambda) W_{2}' $ (taking inverse establishes a bijection). One can then use analogue of (\ref{lmin})   for $ \mathrm{GSp}_{4} $ to find the 
minimal possible length in each of $ \gamma t(\lambda) W_{2}' $ for every  $ \gamma \in W^{\circ}_{2}   = \left \{ 1, r_{2} , t(f_{2}) r_{0}, t(f_{2})r_{0}r_{2} \right \} $. For instance, $$ W^{\circ}_{2}  t(3,2,3) = \left \{ t(3,2,3), t(3,2,0) \right \}  $$  and the minimal lengths of elements in  $ t(3,2,3)W_{2}'  $ is $ 4 $ while that of  $t(3,2,0)W_{2}' $  is $ 5 $.       
\end{proof} 
\begin{lemma}  \label{distinctH2U2circ}  $ 1, v_{1}, \eta_{1} $ and $ \eta_{2} $ 
represent distinct classes in $ H_{2}  \backslash H_{2}' /  U_{2}^{\circ} $.    
\end{lemma}    
\begin{proof} We need to show that for distinct  $ \gamma  , \gamma'  \in \left \{ 1, v_{1}, \eta_{1}, \eta_{2} \right \} $, $ \gamma ^{-1} h \gamma '  \notin H $ for any $ h \in H $.  Writing $ h $ as in Notation \ref{notationh2},   we  have 
$$
hv_{1}   =   \scalebox{0.9}{$\begin{pmatrix} & a &  & b \\ a_{1} & & b_{1} &  \\  & c & & d \\ c_{1} & & d_{1}  & \end{pmatrix}$} ,  \quad 
 h \eta_{i}    =  \scalebox{0.9}{$\begin{pmatrix}  a \varpi^{i}  & & b & a    \\ &   a_{1}  \varpi^{i} & a_{1} &  b_{1}  \\[0.1em]  c \varpi^{i}  & &d & c  \\   &   c _{1} \varpi^{i}  &  c_{1} & d_{1}   
\end{pmatrix}$}, \quad  
v_{1}  h \eta_{i}  = 
\scalebox{0.9}{$\begin{pmatrix}  & a_{1} \varpi ^{i}  & a_{1} &  b_{1} \\ a \varpi ^{i} & & b & a   \\  &  c_{1} \varpi ^{i } & c_{1}  &  d_{1}  \\ c \varpi & & d & c  \end{pmatrix}$}$$ 
where $ i = 1, 2 $ and  
$$ 
 \eta^{-1}_{1} h \eta_{2} =   
\scalebox{0.9}{$\begin{pmatrix}
a \varpi& - c_{1} \varpi & \frac{b-c_{1}}{\varpi}  &  \frac{a - d _{1}}{\varpi}   \\[0.1em]  
-c \varpi &   a_{1}  \varpi & \frac{a_{1} -d_{1}}{\varpi} &  \frac{b_{1}-c}{\varpi}  \\[0.1em]  c \varpi^{2}  & &d & c  \\   &   c _{1} \varpi^{2}  &  c_{1} & d_{1}   
\end{pmatrix}$}. $$  
If any of $  hv_{1} \in  U_{2}^{\circ} $,  then $ a_{1}, c_{1} \in \varpi \Oscr_{F} $. Since all entries of $ hv _{1} $ are integral, this would mean $ \det(h v_{1} ) \in \varpi \Oscr_{F} $, a contradiction. If $ h\eta_{i} \in U_{2}^{\circ} $, then all entries of $ h $  excluding $ b $ are integral and $ b \in \varpi^{-1}  \Oscr_{F} $. Since the  first two columns of $ h \eta_{i} $ are integral multiples of $ \varpi $, this would still make $ \det ( h \eta_{i} ) \in \varpi \Oscr_{F} $, a contradiction.  Similarly for $ v_{1} h \eta_{i} $.  Finally, $ \eta_{1}^{-1} h \eta_{2} \in  U_{2}^{\circ} $ implies  that $ c, d, c_{1}, d_{1} \in \Oscr_{F} $ and the top right  $ 2 \times 2 $ block implies $ a , a_{1} \in \Oscr_{F} $. So again, the first two columns are integral multiples of $ \varpi $ making $ \det( \eta_{1} ^{-1} h \eta_{2} ) \in \varpi  \Oscr_{F} $, a contradiction.        
\end{proof} 
\begin{notation} For this subsection only, we  let  $ \mathscr{R}_{V}(h) $, denote the double coset space $ U_{2} \backslash   U_{2}' h V / V  $   where   $ h \in H_{2}' $ and $ V \subset H_{2}' $ a compact open subgroup. 
\end{notation} 
\begin{proposition}  \label{U2circhard}   We have 
\begin{enumerate}[label = \normalfont (\alph*)] 
\item $ \mathscr{R}_{U_{2}^{\circ}}(\varpi^{(2,1,1)}) = \left\{\varpi^{(2,1,1)}, \,  \varpi^{(2,1,1)}v_{1} , \,   \varpi^{(1,1,0)} \eta_{1}    \right \}  $ 
\item $  \mathscr{R}_{U_{2}^{\circ}}(\varpi^{(2,2,2)}  )   = \left \{\varpi^{(2,2,2)},  \,  
\varpi^{(2,2,2)}v_{1}, \,
\varpi^{(1,0,0)}\eta_{1}, \,
\varpi^{(1,0,1)}\eta_{1}, \, 
\varpi^{(1,1,1)}\eta_{1}, \eta_{2} \right \} $
\item $ \mathscr{R}_{U_{2}^{\circ}}(\varpi^{(3,2,3)}) = \left \{ \varpi^{(3,2,3)}, \,
\varpi^{(3,3,2)}v_{1}, \,
\varpi^{(2,0,1)}\eta_{1}, \,
\varpi^{(2,1,2)}\eta_{1}, \,
\varpi^{(1,0,1)}\eta_{2}     \right \}   $  
\end{enumerate}  
and $ \mathscr{R}_{U_{2}}^{\circ}(\varpi^{(2,2,1)}) =  \mathscr{R}_{U_{2}^{\circ}}(\varpi^{(2,1,1)})$. 
\end{proposition}

\begin{proof} That the representatives are distinct follows by Lemma \ref{distinctH2U2circ} and by checking that $ H_{2} \cap \eta_{1} U_{2}^{\circ} \eta_{1}^{-1} $ is contained in an Iwahori subgroup of $ H_{2} $ (see e.g., the argument in Lemma \ref{distinctWTorbits}).    As usual, we show  that all the orbits are represented by  studying the $U_{2}$-orbits on  Schubert cells. Note that $$ W_{2}' = \left \{ 1, r_{1}, r_{2}, r_{2}r_{1}, r_{1} r_{2}, r_{1}r_{2}r_{1}, r_{2}r_{1}r_{2},  r_{2}r_{1}r_{2}r_{1}  \right \}  $$
and $ (r_{2}r_{1})^{2} = (r_{1} r_{2})^{2}   $.  \\ 

\noindent (a) $ w = \rho_{2}^{2} $.  We have $ W_{2,w} ' = W_{2}' \cap W_{2} ^{\circ}  = \langle r_{2} \rangle $, so $  W_{2}'/W_{2,w} =  
\left \{ W_{2,w}' , r_{1} W_{2,w}', r_{2}r_{1} W_{2,w}' , r_{1}r_{2} r_{1} W_{2,w}' \right \} . $
 So $ [W_{2}'/W_{2,w}]  = \left \{ 1 , r_{1}, r_{2} r_{1} , r_{1}  r_{2}  r_{1}  \right \} $. Thus to study  $ \mathscr{R}_{U_{2}^{\circ}}(w) $,  it suffices to study the $U_{2}$-orbits on cells corresponding to $ \varepsilon_{0} =  \rho_{2}^{2}$,  $   \varepsilon_{1} =  r_{1}\rho_{2}^{2} $ and $ \varepsilon_{2} = r_{1} r_{2} r_{1}\rho_{2}^{2}  $. Now $ \mathcal{Y}_{\varepsilon_{0}}/U_{2}^{\circ} = \varpi^{(2,1,1)}U_{2}^{\circ} $ and 
\begin{align*}  \Yvar{1}/U_{2}^{\circ} =   \Set*{\scalebox{0.9}{$\begin{pmatrix}a \varpi & \varpi \\
\varpi \\  & &  & \varpi \\ 
 & & \varpi & -a \varpi\end{pmatrix} $}} \quad  \quad \quad   \Yvar{2}/U_{2}^{\circ}   =  \scalebox{0.9}{$\Set*{\begin{pmatrix} y &  a_{1} \varpi  &   \varpi  &  - a \varpi  \\
 a \varpi &  \varpi \\
 \varpi \\ 
 -a_{1} \varpi  & & & -\varpi \end{pmatrix}}$}
\end{align*}
where $ a, a_{1}, y \in [\kay] $. For $ \Yvar{1}/U_{2}^{\circ} $, the case $ a = 0 $ clearly leads to $ \varpi^{(2,1,1)}v_{1} $. If $ a \neq 0 $, then we can multiply by $ \mathrm{diag}(a^{-1},1,1,a^{-1}  )  $ on the left and $ \mathrm{diag}(1,a,a,1) $ on the right to assume $ a = 1 $. We then hit with $ v_{\beta_{0}} \in U_{2} $ on left and $ v_{0} \in U_{2}^{\circ}  $ on right to arrive at $\mathrm{diag}(-1,1,1,-1)$ (which we can ignore) times 
\[
\begin{psmallmatrix} 
&  & & \varpi \\
& &  1   \\[0.1em] 
&  \varpi & 1   \\
\varpi^{2}  &  & &  \varpi  
\end{psmallmatrix}. 
\]
Now a simple column operation and a left multiplication by a diagonal matrix in the compact torus transforms this into $ \varpi^{(1,1,0)}\eta_{1} $. As for $  \Yvar{2}/U_{2}^{\circ} $, begin by eliminating $ y $ with a row operation. Then note that conjugation by $ v_{2} $ swaps $ a $ with $ a_{1} $ and reverses all signs. So  after applying operations involving second and  fourth row and columns, we may assume that wlog that $ a_{1} = 0 $. Right multiplication by $ v_{2} $ yields the matrix \[ 
\begin{psmallmatrix} 
\varpi^{2}  & & & a \varpi \\ & \varpi & a \\[0.2em]   & & 1 \\ & &  & \varpi   \end{psmallmatrix} .  
\] 
which results in either $ \varpi^{(2,2,1)}$ (which represents the same class as  $ \varpi^{(2,1,1)} $) or $ \varpi^{(1,1,0)}\eta_{0} $. So all in all, we have three representatives: $ \varpi^{(2,1,1)} $, $ \varpi^{(2,1,1)}v_{1} $, $ \varpi^{(1,1,0)}\eta_{1} $. \\

\noindent  (b) $ w = v_{0} v_{1} \rho^{2}_{2} $. Here $ w W^{\circ}_{2}w^{-1} = \langle t(f_{3})r_{2}, t(f_{2})r_{0} \rangle $, 
so $ W_{2,w}' $ is trivial. So we  need  to analyze cells corresponding to $ \varepsilon_{0} = v_{0}v_{1}\rho_{2}^{2} $, $  \varepsilon_{1} = v_{1} \varepsilon_{0} $, $ \varepsilon_{2} = v_{1} v_{2}  \varepsilon_{0} $ and $  \varepsilon_{3} =  v_{1} v_{2}v_{1} \varepsilon_{0}  $.  
The corresponding cells are   
\begin{alignat*}{4}  \mathcal{Y}_{\varepsilon _{0}}/ U_{2} ^ { \circ  }    &  =  
\scalebox{0.9}{$\Set*{
\left(\begin{array}{cccc}
 & & & 1\\ 
 \varpi & & & \\[0.2em] 
a  \varpi^{2} &  \varpi^{2} &  &  \varpi x  \\[0.2em] 
  & &  - \varpi  &  a \varpi   
\end{array}\right)}$},& 
 \mathcal{Y}_{\varepsilon _{2}}  / U_{2}^{\circ }    &=
 \scalebox{0.9}{$\Set*{
\left(\begin{array}{cccc}
y \varpi & &  - \varpi & a_{1} + a \varpi   \\[0.2em] 
&  & &  1 \\
-\varpi  &  & & \\[0.2em] 
 (a_{1} +   a   \varpi  )    \varpi          &   \varpi^{2}  &  &   x  \varpi 
\end{array}\right)}$},
\\
\mathcal{Y}_{\varepsilon _{1}} / U_{2} ^{\circ }  & = \scalebox{0.9}
{$\Set*{\left(\begin{array}{cccc}
\varpi &    &   &  a_{1} \\
& & &  1 \\[0.2em] 
& -\varpi &  &  a \varpi   \\[0.2em] 
a \varpi^{2 }  &  \varpi^{2}  &  a_{1} \varpi    & x \varpi     
\end{array}\right)},$}  
&
\quad \quad  
\mathcal{Y}_{\varepsilon _{3}} / U_{2}^{\circ}     &=  \scalebox{0.85}
{$\Set*{\left(\begin{array}{cccc}
(a_{2} + a\varpi) \varpi   &   \varpi^{2}  & a_{1} \varpi  & y + \varpi x \\
\varpi &  &  &  a_{1} \\[0.2em] 
 &  &   & 1   \\[0.2em] 
 & &  \varpi  &  -(a_{2}+  a \varpi )   
\end{array}\right)}$}   
\end{alignat*} 
where $ a, a_{1}, a_{2}, x, y \in [\kay] $.  Using similar arguments on these, one deduces that the orbits of  $U_{2} $ on 
\begin{itemize} 
\item $ \Yvar{0}/U_{2}^{\circ} $ are represented by $ \varpi^{(2,2,2)}v_{1} $, $ \varpi^{(1,0,0)}\eta_{0} $,
\item $ \Yvar{1}/U_{2}^{\circ} $ are represented by $ \varpi^{(2,2,2)} $, $ \varpi^{(1,0,1)}\eta_{0} $, $ \varpi^{(1,1,1)}\eta_{1} $, 
\item $   \Yvar{2}/U_{2}^{\circ} $  are  represented by $ \varpi^{(2,2,2)} $, $ \varpi^{(1,1,1)}\eta_{1} $, $ \eta_{2} $ 
\item $ \Yvar{3}/U_{2}^{\circ} $ are represented by $ \varpi^{(1,0,0) }  \eta_{1}  $, $ \varpi^{(1,0,1)}\eta_{1}  $, $ \eta_{2}  $. \\     
\end{itemize}  

\noindent (c)  $  w = v_{0} v_{1} v_{2} v_{1 }\rho_{2}^{3}$. Here $ w W^{\circ}_{2}w^{-1} =  \langle r_{2} , t(3f_{2})r_{0} \rangle  $ which means that  $ W _{2, w}' = \langle  r_{2}  \rangle  $. So as in part (a), we have $  [W_{2}'/W_{2,w}] = \left \{ 1, r_{1} ,  r_{2} r_{1}, r_{1} r_{2} r_{1}  \right \} $.  Again,  we   have three cells to analyze, which correspond to $ \varepsilon_{0} = v_{0}v_{1}v_{2}v_{1} \rho_{2}^{3} $, $ \varepsilon_{1} =  v_{1}  \varepsilon_{0} $ and $ \varepsilon_{2} = v_{1} v_{2} v_{1}  \varepsilon_{0}  $.  The  corresponding  cells  are      
\begin{align*}  \mathcal{Y}_{\varepsilon_{0}} /U_{2}  ^ {  \circ   }  =   \scalebox{0.9}{$\Set*{\left(\begin{array}{cccc}
 & & & 1  \\
 & &   \varpi   & a \varpi \\[0.2em] 
- a\varpi^{3} & \varpi^{3} &  a_{1}\varpi^{2}  & (x + y \varpi) \varpi    \\[0.2em] 
\varpi^{2}    & &  &   a_{1} \varpi
\end{array}\right)}$},  
\quad  \mathcal{Y}_{\varepsilon_{1}} /U_{2}^{\circ}    = \scalebox{0.9}
{$\Set*{\left(\begin{array}{cccc}
&  & \varpi & a_{2} + a\varpi  \\
&  & & 1  \\[0.2em] 
\varpi^{2}  &   & &  a_{1}  \varpi  
\\[0.2em] 
-(a_{2} + a\varpi) \varpi^{2}  &  \varpi^{3}  &  a_{1}  \varpi^{2}      & \varpi ( x + \varpi y ) 
\end{array}\right)}$}   
\end{align*}   
\begin{align*}   \mathcal{Y}_{\varepsilon_{2}}/U_{2}^{\circ}  = 
&\scalebox{0.9}{$\Set*{\left(\begin{array}{cccc}
-(a_{2} + a\varpi) \varpi^{2} &  \varpi^{3}  &  (a_{3} + a_{1} \varpi) \varpi    & z \\[0.2em] 
&   &  \varpi  & a_{2} + a \varpi  \\
&  &  &  1  \\[0.2em] 
- \varpi^{2} &  &  &  - (a_{3} + a_{1} \varpi) 
\end{array}\right)}$}\quad \quad \quad  
\end{align*} 
where $ a,a_{1},a_{2}, a_{3} , x,y \in [\kay] $. From these, we deduce that
\begin{itemize}  
\item  $ \Yvar{0}/U_{2}^{\circ} $  are represented by   $ \varpi^{(3,3,2)} v_{1}  
$, $ \varpi^{(2,0,1)} \eta_{1} $, 
\item $ \Yvar{1}/U_{2}^{\circ} $  are  represented  by $ \varpi^{(3,2,3)}  $,  $ \varpi^{(2,1,2) } \eta_{1}   $, $ \varpi^{(1,0,1)}\eta_{2} $, 
\item  $ \Yvar{2}/U_{2}^{\circ} $  are  represented by $  \varpi^{(3,3,2)} v_{1} $, $  \varpi^{(2,0,1)} \eta_{1} $, $ \varpi^{(1,0,1)} \eta_{2}    $. \qedhere 
\end{itemize} 
\end{proof} 
We can use Proposition \ref{U2circhard} to obtain representatives for the remaining words computed in  Lemma  \ref{U2circwords} without computing Schubert cells.  
\begin{corollary} \label{U2circeasy}  We have 
\begin{enumerate}[label = \normalfont (\alph*)]
\item $ \mathscr{R}_{U^{\circ}_{2}}(\varpi^{(1,1,1)}) = \left \{   \varpi^{(1,1,1)} v_{1},  \,  \varpi^{(1,1,1)} , \,   \eta_{1}   \right \} $ 
\item $ \mathscr{R}_{U^{\circ}_{2}}(\varpi^{(3,3,2)})  = \left \{  \varpi^{(3,2,3)}v_{1} , \,   \varpi^{(3,3,2)}, \, 
\varpi^{(2,2,1)}\eta_{1},\, 
\varpi^{(2,2,0)}\eta_{1},\,
\varpi^{(2,1,0)}\eta_{1},\,  \varpi^{(1,1,0)}\eta_{2}   
\right \}  $  
\item  $ \mathscr{R}_{U^{\circ}_{2}} ( \varpi^{(4,4,2)}  )   = \left \{ \varpi^{(4,2,4)}v_{1}  , \,  \varpi^{(4,4,2)}, 
\varpi^{(3,3,1)}\eta_{1}, \, 
\varpi^{(3,2,0)}\eta_{1}, \,    \varpi^{(2,2,0)}\eta_{2} \right \}  $  
\end{enumerate} 
\end{corollary}
\begin{proof} Since the class of $ \rho_{2} $ normalizes $ W_{2}^{\circ} $ (see diagram  (\ref{dynkingl2gsp4})) and $ \rho_{2} $ normalizes the Iwahori subgroup $I_{2}' $, it normalizes $ U_{2}^{\circ} $. Thus  for any integer $ k $, the representatives for $ \mathscr{R}_{U_{2}^{\circ}}(h\rho^{k}_{2}) $ can be obtained from $ \mathscr{R}_{U_{2}^{\circ}}(h) $  by multiplying representatives on the right by $ \rho^{k}_{2} $.  Now we have the following relations:
$$ \rho_{2} U_{2}^{\circ}   =  \varpi^{(1,0,1)} v _{1} U_{2}^{\circ}, \quad \quad   v_{1} \rho_{2} U_{2}^{\circ}  =  \varpi^{(1,1,0)} U_{2}^{\circ},  \quad \quad 
  \eta_{i} \rho_{2} U_{2}^{\circ}  = v_{\alpha_{0}}v_{2} \varpi^{(1,1,0)}  \eta_{i}   U_{2}^{\circ}     $$
for $ i = 1 , 2 $.    
By Lemma \ref{U2circwords}, parts  (a), (b) and  (c) are obtained by the corresponding parts of  Proposition \ref{U2circhard}. For instance, $ \varpi^{(1,1,0)} \eta_{1} \in \mathscr{R}_{U_{2}^{\circ}}(\varpi^{(2,1,1)}) $ corresponds to $ \varpi^{-(1,0,1)}\eta_{1} \rho_{2}  $ and   \begin{align*}  U_{2}  \varpi^{-(1,0,1)}\eta_{1} \rho_{2} U_{2} &  =  U_{2}  \varpi^{-(1,0,1)}  v_{\alpha_{0}} v_{2} \varpi^{(1,1,0)} \eta_{1}  U_{2}^{\circ} \\
& = U_{2}v_{\alpha_{0}} v_{2} \eta_{1} U_{2}^{\circ}  = U_{2} \eta_{1} U_{2}^{\circ} 
\end{align*}
which gives the representative $ \eta_{1} $ in $ \mathscr{R}_{U_{2}^{\circ}}(\varpi^{(1,1,1)}) $.          
\end{proof}
Let us denote  $ U_{2}^{\dagger} : = \pr_{2}'(H_{\tau_{1}}') $.   Since $ H_{\tau_{1} } $ is the conjugate of $ H_{\tau_{\circ}} $ by $ \varpi^{(1,1,1,1)}$,  $ U_{2}^{\dagger} $ is the conjugate of $ U^{\circ}_{2} $ by $ \varpi^{(1,1,1) }  $. Set 
\begin{equation} \label{eta0matrix}     \eta_{0}:  = \eta_{1} \varpi^{-(1,1,1)} = \scalebox{0.9}{$  
     \left ( \begin{matrix} 1  &  & & 1 \\ & 1  & 1 \\[0.1em]   & &  1  \\ &  & & 1 \end{matrix}  \right )$}  .
     \end{equation} 
Then $ \eta_{i} \varpi^{-(1,1,1)} = \eta_{i-1} $ for $ i = 1 , 2 $. Moreover $ v_{1} $  commutes with $  \varpi^{(1,1,1)} $. So  by Proposition \ref{U2circhard} and Corollary \ref{U2circeasy}, we obtain the following.     
\begin{corollary}    \label{U2dagger}   We  have    
\begin{itemize} 
\item   $ \mathscr{R}_{U_{2}^{\dagger}}(\varpi^{(1,0,0)}) = \left\{\varpi^{(1,0,0)}, \,
\varpi^{(1,0,0)} v _{1}, \,
\varpi^{(1,1,0)} \eta_{0}  \right \} $, 
\item  $ \mathscr{R}_{U_{2}^{\dagger}}(\varpi^{(1,1,1)}) = \left \{ \varpi^{(1,1,1)}, \, 
\varpi^{(1,1,1)} v _{1}, \, 
\varpi^{(1,0,0)}\eta_{0} , \,
\varpi^{(1,0,1)}\eta_{0}, \, 
\varpi^{(1,1,1)}\eta_{0} , \, 
\eta_{1} \right \}  $
\item  $  \mathscr{R}_{U_{2}^{\dagger}}(\varpi^{(2,1,2)}) = \left \{ \varpi^{(2,1,2)},  \,
\varpi^{(2,2,1)} v _{1},  \, \varpi^{(2,0,1)}  \eta_{0}, \,  \varpi^{(2,1,2)} \eta_{0},  \varpi^{(1,0,1)}\eta_{1}   \right \} $
\item $ \mathscr{R}_{U_{2}^{\dagger}}(\varpi^{(0,0,0)})  =  \left \{ 1,  v _{1},   \eta_{0}  \right \} $
\item $  \mathscr{R}_{U_{2}^{\dagger}}(\varpi^{(2,2,1)}) =  \left \{  \varpi^{(2,1,2)} v _{1} , \,   \varpi^{(2,2,1)}, \, 
\varpi^{(2,2,1)}\eta_{0},
\varpi^{(2,2,0)}\eta_{0},\,
\varpi^{(2,1,0)}\eta_{0},   \,  \varpi^{(1,1,0)}\eta_{1}   
\right \}   $          
\item $ \mathscr{R}_{U_{2}^{\dagger}}(\varpi^{(3,3,1)}) = \left \{    \varpi^{(3,1,3)}v _{1}  , \,  \varpi^{(3,3,1)}, 
\varpi^{(3,3,1)}\eta_{0}, \, 
\varpi^{(3,2,0)}\eta_{0}, \,    \varpi^{(2,2,0)}\eta_{1}  \right \}  $ 
\end{itemize}  
\end{corollary}

\begin{remark}   \label{transreps211dagger}  Note that Proposition  \ref{cartanhtaui} implies that for 
$  \mathscr{R}_{U^{\dagger}_{2}}(\varpi^{\lambda}) =  \mathscr{R}_{U_{2}^{\dagger}} (\varpi^{r_{0}(\lambda)})  = \mathscr{R}_{U_{2}^{\dagger}}(\varpi^{r_{2}(\lambda) - (0,0,1)} ) $ 
Thus Corollary \ref{U2dagger} records the decompositions of the all the projections in  Proposition  \ref{R1scrhard}. 
\end{remark}

Next,   we study the fibers of the projection $ \pr_{\lambda} :  \mathscr{R}_{1}(\varpi^{\lambda })  \to  \mathscr{R}_{U_{2}^{\dagger}}(\varpi^{\pr_{2}(\lambda)})  $ and use Corollary \ref{transreps211dagger} to calculate coset representatives given in Proposition \ref{R1scrhard}.  Let us denote by $ \Lambda^{\alpha_{0} > 0}  $ the set of all $ \lambda \in \Lambda $ such that $ \alpha_{0}(\lambda)  > 0 $. We first  specialize Corollary 
\ref{fibersofpV} to the case of $ H_{\tau_{1}}' $.    
\begin{corollary}  \label{fibersofHtau1}  For any $ \lambda  = (a, b, c ,  d    ) \in \Lambda^{\alpha_{0} > 0 } $ and $ \eta \in H_{2}' $ such that $ U_{2} \eta U_{2}^{\dagger}  \in  U_{2}  \backslash 
 U_{2} ' \varpi^{\pr_{2}(\lambda)} U_{2}^{\dagger}  / U_{2}^{\dagger}  $ with $ \mathrm{sim}(\eta) = a $, the fiber of $ \pr_{\lambda} $ above $ U_{2} \eta U_{2}^{\dagger}$ is $$ \left \{ U  
 ( \varpi^{(a,b)} \chi ,   \eta  ) H_{\tau_{1}}'  \, | \,  \chi \in S^{\pm}_{1} \text{ and } U ' (\varpi^{(a,b)} \chi, \eta ) H_{\tau_{1}}' =  U' \varpi^{\lambda} H_{\tau_{1}}'  \right \}  $$  
\end{corollary} 

\begin{proof} This follows since $ \left ( \begin{psmallmatrix} 1 & \\ x & 1 \end{psmallmatrix} , 1 \right ) $,  $ \left (  \begin{psmallmatrix} 1 & x \\ & 1 \end{psmallmatrix} , 1 \right ) $ lie in $      H_{\tau_{1}}'$ if $ x \in \varpi \Oscr_{F} $. 
\end{proof}     
For $ x \in \Oscr_{F} $,  define \begin{equation} \label{varkappamatrices} \varkappa^{+}(x) =    \left (  \scalebox{0.9}{$\left ( \begin{matrix} 1 & x \\  & 1  \end{matrix} \right )$}  , \scalebox{1.1}{$ \left (   \begin{smallmatrix} \\  1 &     &   \\  &   1 \\    x   &  &   1  \\ & & &      1    \end{smallmatrix} \right )$} \right )  ,    \quad \quad   \,  \varkappa^{-}(x) =    \left ( \scalebox{0.9}{$\left ( \begin{matrix} 1 &   \\ x & 1  \end{matrix} \right )$} , \scalebox{1.1}{$ \left (   \begin{smallmatrix} \\  1 &     &   x   \\  &   1 \\    &  &   1  \\ & & &      1    \end{smallmatrix} \right )$} \right )   .
\end{equation} 
Note that $ \varkappa^{\pm}(x) \in H_{\tau_{1}} ' $ as these elements are in the subgroup $ \mathscr{X}_{\tau}  \subset     \Ht_{\tau_{1}}     $ introduced in Notation \ref{Xscrnot}.  We let $ \kappa^{\pm}_{1}(x) = \pr_{1}'( \varkappa_{\pm} (x))  $ and $ \kappa^{\pm}_{2}(x) = \pr_{2}' (\varkappa^{\mp}(x)) $ denote their projections.

\begin{lemma} \label{liftsforsigma0}  If   $ \lambda \in \Lambda_{\alpha_{0}}^{>}  $, then   
$ U \varpi^{\lambda}  \chi   H_{\tau_{1}}' \in \left \{ U \varpi^{\lambda} H_{\tau_{1}}',  \,   U \varpi^{s_{0}(\lambda)} H_{\tau_{1}'} \right \} $ for any $ \chi \in S_{1}^{\pm} \times  \left     \{ 1 \right \}  $. 

\end{lemma}

\begin{proof}    Write $  \chi  = (\chi_{1}, 1 ) $. If $ \chi_{1} \in S^{+} = \left \{ \begin{psmallmatrix} & 1 \\- 1  \end{psmallmatrix} \right \} $, the claim is clear. If $ \chi \in S_{1}^{-} $, let $ x \in \Oscr_{F} $ be such that $ \chi_{1} =   \kappa_{1}^{-}(x)  $. The case $ x = 0 $ is also obvious, so we assume that $ x \in \Oscr_{F}^{\times}    $.  Observe  that  \begin{align*}  U \varpi^{\lambda} \chi 
H_{\tau_{1}}' &  = U \varpi^{\lambda} \chi  \varkappa^{-}(-x) H_{\tau_{1}}'   \\ & =  
U \varpi^{\lambda} \left ( 1 ,  \kappa_{2}^{+}(-x) \right )  H_{\tau_{1}}' . 
\intertext{If $ 2c - a \geq 0 $, the conjugate of $ \left ( 1 ,  \kappa^{+}_{2}(-x) \right ) $ by $ \varpi^{\lambda} $ lies in $ U $ and so our double coset equals $ U \varpi^{\lambda} H_{\tau_{1}}' $. 
If  $ 2c - a < 0  $ however, then the conjugate of  $  \left ( \kappa^{+}_{1}(-x^{-1}) ,  \kappa^{-}_{2}(x^{-1})  \right ) $ by $ \varpi^{\lambda} $ lies in $ U $.  
So} U \varpi^{\lambda} \gamma H_{\tau_{1}}' 
 &  = U \varpi^{\lambda} \left  (\kappa_{1}^{+}(-x^{-1}),  \kappa_{2}^{-}(x^{-1})\kappa_{2}^{+}(-x) \right  ) H_{\tau_{1}}' 
\end{align*}
Now note that $$ \left  (\kappa_{1}^{+}(-x^{-1}),  \kappa_{2}^{-}(x^{-1})\kappa_{2}^{+}(-x) \right  ) \cdot  \varkappa ^{+}(x^{-1}) 
= \left (  1 ,  
\begin{psmallmatrix} 
& &  -x  \\ & 1  \\ 1/x
\\ & & & 1 
\end{psmallmatrix}  
\right ) . 
$$ 
From this, it follows that $ U \varpi^{\lambda} \chi    H_{\tau_{1}} ' = U \varpi^{r_{0}(\lambda)}H_{\tau_{1}}' $ which equals $ U \varpi^{s_{0}(\lambda)} H_{\tau_{1}}  '    $.       
\end{proof} 
\begin{corollary}   \label{liftsforsigma1}        For $ \lambda \in \Lambda^{\alpha_{0}  >} $, 
$ U \varpi^{\lambda} \sigma_{1}  \chi  H_{\tau_{1}}' \in \left \{ U \varpi^{\lambda}   \sigma_{1} H_{\tau_{1}}',  \, U \varpi^{s_{0}(\lambda)} \sigma_{1}   H_{\tau_{1}}'   \right \} $ 
for any $  \chi   \in S^{\pm}   _  {  1 }   \times \left \{ 1 \right \}  $. 
\end{corollary} 
\begin{proof} 
Since  $ v   _{1} $ normalizes $ U_{2} $ and $ \sigma_{1} = (1,  v_{1}  ) $ commutes with $ \chi $ 
we see that $ U \varpi^{\lambda} \sigma_{1}  \chi H_{\tau_{1}}' =  \sigma _{1}  U \varpi^{r_{1}(\lambda)}   \chi 
 H_{\tau_{1}}  '     $. The claim  now 
 follows from part by  noting that $ r_{1}$ commutes with $ s_{0} $.
 \end{proof}
Next we record results on double cosets involving $ \sigma_{2} $. 

\begin{lemma} \label{Hvarsig2inIwahori}  $ H \cap \sigma_{2} H_{\tau_{1}}' \sigma_{2}^{-1}  =  H \cap \varsigma_{2} K \varsigma_{2} ^{-1}  
  $ is contained in the Iwahori subgroup  $   J_{\varsigma_{2}} $ of  triples $ (h_{1}, h_{2}, h_{3}) \in U $ where $ h_{2} $ reduces to an upper triangular matrix modulo $ \varpi $ and $ h_{1}, h_{3} $ reduce to lower triangular matrices.   
\end{lemma} 
\begin{proof} This follows by a stronger result established in  Lemma  \ref{varsig2struclemma}.  
\end{proof} 

\begin{lemma}  \label{liftsforsigma2U'}  Suppose $ \lambda = (a, b, c, d ) \in \Lambda_{\alpha_{0}}^{>} $ and let $ \chi = (\chi_{1}, 1) \in S^{\pm}_{1}  \times \left \{ 1 \right \}  $.  
\begin{itemize}   
\item If $ \chi _{1} =   \begin{psmallmatrix} 1  \\  x & 1 \end{psmallmatrix}  \in S^{-}_{1} $,  then  $$  \quad  \,     U '  \varpi^{\lambda}\sigma_{2} \chi  H_{\tau_{1}}'  =   \begin{cases} 
U  '\varpi^{\lambda} H_{\tau_{1}}' & \text{ if  }   \,\,   x \in \Oscr_{F}^{\times}, \,  c + d \geq a , \,  2c  \geq a \,  \text{  or if  }  \, \,  x = 0, \,   c+ d \geq a  ,      \\
U' \varpi^{r_{0}(\lambda)} H_{\tau_{1}}' & \text{ if } \,\,  x \in \Oscr_{F}^{\times} , \,  c + d \geq a > 2c  , \\
U   '    \varpi^{r_{1} r_{0}(\lambda) } H_{\tau_{1}}'  &  \text{ if }  \,\,  x \in \Oscr_{F}^{\times }  ,  \,        2d > a > c+ d  , \\
U  '  \varpi^{r_{0}r_{1} r_{0}(\lambda)} H_{\tau_{1}}'  & \text{ if } \,\,  x \in \Oscr_{F}^{\times}, \,   a > c + d, \,    a \geq 2d \,     \text{ or if } \,  x  = 0,    \,    a > c + d .     \end{cases}  $$

\item  If $ \chi_{1} =    \begin{psmallmatrix}  &  1  \\  -  1   &   \end{psmallmatrix}   \in S^{+}_{1} $, then   $$  U '  \varpi^{\lambda}\sigma_{2} \chi  H_{\tau_{1}}'  =   \begin{cases}  U  '\varpi^{r_{0}(\lambda)} H_{\tau_{1}}' & \text{ if  }  
\, c + d \geq a    \\

U  '  \varpi^{r_{1} r_{0}(\lambda)} H_{\tau_{1}}'  & \text{ if } 

a > c + d  \end{cases}  $$      
\end{itemize} 
\end{lemma} 
\begin{proof} For $ x \in \Oscr_{F} $,  define  $$ \nu_{2}^{+}  = \begin{psmallmatrix} 1  & &  & -1 \\ & \,  1 & - 1 \\[0.1em]   &  & \,   \,   1 \\   &  & & 1  \end{psmallmatrix} ,  \quad  \quad   \quad   
     \nu_{2}^{-}  = \begin{psmallmatrix} 1 & \\ &  \,   \,  1 \\[0.2em]  & -1 & \,  1 \\%[0.1em] 
-1  & &&  \, \,   1  \end{psmallmatrix} $$  and set $ \nu^{\pm}  = (1, \nu_{2}^{\pm}) $. Note that  $ \nu ^ {  -  }  \in H_{\tau_{1}}' $. Now if $ c + d \geq a   $, we have $ \varpi^{\lambda} \nu^{+} \varpi^{-\lambda} \in U' $. Thus \begin{align*}  U' \varpi^{\lambda} \sigma_{2} \chi H_{\tau_{1}}'  = U ' \varpi^{\lambda} \nu^{+}   \sigma_{2} \chi H_{\tau_{1}}' = U ' \varpi^{\lambda} \chi 
 H_{\tau_{1}}'  
\end{align*}  
since $ \nu^{+} \sigma_{1} = ( 1, \nu_{2}^{+}\eta_{0} ) = (1, 1) $. If on the other hand $ a > c + d $, then $ \varpi^{\lambda}  \nu^{-} \varpi^{-\lambda} \in U' $, and so 
\begin{align*} U ' \varpi ^ { \lambda} \sigma_{2} \chi H_{\tau_{1}}'    &  = U \varpi^{\lambda}  \nu^{-}  \sigma_{2} \chi \nu^{-}   H_{\tau_{1}}'    \\
& = U' \varpi^{\lambda}  \left ( \chi_{1} ,  \nu_{2}^{-}   \begin{psmallmatrix}  & & & 1 \\ 
 &&  1 \\[0.1em] &  -1  & 1  \\%[0.3em] 
 -1 & &  & 1  \end{psmallmatrix}  \right )  H_{\tau_{1}}  
 \\ 
 & =   U' \varpi^{\lambda}  \left ( \chi_{1} ,    \begin{psmallmatrix}  & & & 1 \\ 
 &&  1 \\[0.1em] &  -1  &   \\%[0.3em] 
 -1 & &  &   \end{psmallmatrix}  \right )  H_{\tau_{1}}  
 \\  
 & = U ' \varpi^{r_{0} r_{1} r_{0} (\lambda)} \chi  H_{\tau_{1}}'  .
\end{align*} 
The rest of the proof now proceeds along exactly the same lines as Lemma  \ref{liftsforsigma0} which determines the classes of $ U \varpi^{\mu} \chi H_{\tau_{1}}' $ for any $ \chi \in S^{\pm}  _ { 1 } \times \left \{ 1 \right \} $, $\mu \in \Lambda $ and  hence those of $ U' \varpi^{\mu} \chi H_{\tau_{1}}' $.  
\end{proof} 

\begin{lemma} \label{liftsforsigma2U}  Suppose  $ \lambda \in \Lambda^{\alpha_{0} >0} $ satisfies  either $ \beta_{0}(\lambda) \geq 0 $ or $ \beta_{2}(\lambda) \leq 0 $. Then  $ U\varpi^{\lambda} \sigma_{2} \chi H_{\tau_{1}} =  U \varpi^{\lambda} \sigma_{2}  H_{\tau_{1}}' $  for any $ \chi \in S^{-}_{1}   \times \left \{ 1 \right \}  $.  
\end{lemma} 
\begin{proof} Write $ \chi =  (\chi_{1}, 1) $ and $ \chi_{1} =  \begin{psmallmatrix}  1&  \\  x & 1 \end{psmallmatrix} $. If  $ 2c \geq a $.  replace $ \varpi^{\lambda} \sigma_{2} \chi $ with $ \varpi^{\lambda} \sigma_{2}  \varkappa^{-}(-x)  = \varpi^{\lambda} \sigma_{2} \kappa_{2}^{+}(-x) $. If $ 2d \geq a $, replace $ \varpi^{\lambda} \sigma_{2} \chi $ with $ \varpi^{\lambda} \sigma_{2} \chi \nu^{-}  \varkappa^{-}(-x)  =   \varpi^{\lambda} \sigma_{2} \nu^{-} \kappa_{2}^{+}(-x)   $ where $ \nu^{-} \in H_{\tau_{1}}' $ is as in  proof     Lemma \ref{liftsforsigma2U'}.  Explicitly if $ \lambda = (a,b,c,d) $, then  $$ \varpi^{\lambda} \sigma_{2} \kappa_{2}^{+}(-x) =  \left (  \varpi^{(a,b)}  , \left ( \begin{smallmatrix}  \varpi^{c}  & &  - x  \varpi^{c} & \varpi^{c} \\ & \varpi^{d} & \varpi^{d} \\[0.2em]  
 & & \varpi^{a-c}  \\ &  & &  \varpi^{a-d} \end{smallmatrix} \right )  \right ) , \, \, \, \varpi^{\lambda}  \sigma_{2}      \nu^{-} \kappa_{2}^{+}(-x)   = \left  ( \varpi^{(a,b)} ,   \left ( \begin{smallmatrix}   & & & \varpi^{c} \\ 
 & & \varpi^{d} \\[0.2em]  &  \,  \varpi^{a-c} &  \varpi^{a-c}  \\[0.2em]    \varpi^{a-d }    &  & - x \varpi^{a-d}  & \varpi^{a-d} \end{smallmatrix} \right)  \right )   
$$  
Now an obvious row operation transforms these into $ \varpi^{\lambda} \sigma_{2} $.   
\end{proof}

 \begin{lemma}   \label{liftsforsigma3}      Suppose $ \lambda = (a,b,c,d) \in \Lambda^{\alpha_{0} > 0} $ satisfies  $ 2c  + 1 \geq a $ and $ c + d \geq a $. Then

$$ U \varpi^{\lambda} \sigma_{3}  \chi  H_{\tau_{1}}'  =  \begin{cases}  U \varpi^{\lambda} \sigma_{3} H_{\tau_{1}}' &  \text{ if } \chi \in S_{1}^{-} \times \{ 1 \}  \\
 U \varpi^{s_{0}(\lambda) - f_{1}}  \sigma_{3}      H_{\tau_{1}}'   & \text{ if }   \chi \in S_{1}^{+} \times \{ 1 \}     
\end{cases}  $$ 
Moreover $  U ' \varpi^{\lambda}  \sigma_{3}  H_{\tau_{1}}' =  U' \varpi^{\lambda+  \lambda_{\circ}} H_{\tau_{1}}'  $  and $ U' \varpi^{s_{0}(\lambda)-f_{1} } \sigma_{3} H_{\tau_{1}}' = U' \varpi^{s_{0}(\lambda) - f_{1} + \lambda_{\circ} } H_{\tau_{1}}' = U ' \varpi^{s_{0}(\lambda+ \lambda_{\circ})} H_{\tau_{1}}' $.

\end{lemma}   
\begin{proof}  This is  entirely similar Lemma \ref{liftsforsigma2U} and \ref{liftsforsigma2U'}.  
\end{proof}    
\begin{proof}[Proof of Proposition \ref{R1scrhard}]     The proof in each case 
goes by 
applying either Lemma \ref{pVbijection} or  Corollary \ref{fibersofHtau1} to  the coset representatives computed in  Corollary \ref{U2dagger}. In the latter case, we   will need to  determine the fibers  of the projection   $$ \pr_{\mu } :   \mathscr{R}_{1}( \varpi^{\mu} ) \to \mathscr{R}_{U^{\dagger}_{2}}(\varpi^{\pr_{2}(\mu)}) $$ for a given $ \mu = (u_{0}, u_{1}, u_{2},  u_{3}) \in \Lambda $ above each $ \varpi^{(a,c,d)} \gamma  \in \mathscr{R}_{U_{2}^{\dagger}}(\varpi^{\pr_{2}(\mu)}) $ where $ \gamma \in \left \{ 1, v_{1} , \eta_{0},  \eta_{1}  \right \} $. Let $ \lambda = (a, b, c, d ) $ where $ b  = u_{1}  $  if $ \gamma \in \left \{ 1, v_{1}  , \eta_{0} \right \} $  and $  u_{1}     - 1 $ if $ \gamma =  \eta_{1} $. Let  $ i \in \left \{0,1,2,3\right\}  $ be the unique integer 
such that $ \pr_{2}(\sigma_{i}) = \gamma $. Then the fiber consists of cosets of the form $ U \varpi^{\lambda} \sigma_{i} \chi  H_{\tau_{1}}' $ where $ \chi  \in S^{\pm}_{1} $.   

Let us first address the case where $ \gamma \in \left \{ 1,  v_{1} \right \} $.  
Note that in each of the projections computed in Corollary \ref{U2dagger},  there is a unique element of the form $ \varpi^{(a,c,d)}\gamma  $. So the projection of $ \mathscr{R}_{1}(\mu) $ in each case has a unique element of this form (see Remark \ref{transreps211dagger}). Lemma \ref{liftsforsigma0} and Corollary \ref{liftsforsigma1}  tell us that the fiber  of  $ \pr_{ 2, \mu }  $  above each such element is   contained in    $ \left \{ \varpi^{\lambda} \gamma , \varpi^{s_{0}(\lambda)}    \gamma \right \} $. If $ \lambda \neq s_{0}(\lambda) $, then Corollary \ref{cartanhtaui} implies 
that only one of $ \varpi^{\lambda}\gamma $ or $ \varpi^{s_{0}(\lambda)}\gamma $ can belong to the fiber. 
Thus the fiber is necessarily a  singleton. 
Now $ U \varpi^{\mu} H_{\tau_{1}}' $, $ U  \varpi^{r_{1}(\mu)}   \sigma_{1}  H_{\tau_{1}} ' $ are clearly subsets of $ U' \varpi^{\mu}  H_{\tau_{1}}' $ and their projections  $ \varpi^{(u_{0}, u_{2}, u_{3})  }  $, $\varpi^{(u_{0},u_{3},u_{2})} v_{1} $ respectively  have the desired form\footnote{that is,  $ U\varpi^{(u_{0},u_{2},u_{3})} H_{\tau_{1}}' = U \varpi^{(a,c,d)} H_{\tau_{1}}'$ and $U \varpi^{(u_{0},u_{3},u_{2})}v_{1}H_{\tau_{1}}'  = U \varpi^{(a,c,d)}v_{1}H_{\tau_{1}}'$}. So  we  are free to choose $ \varpi^{\mu}  $ as the representative element in the fiber if $ \gamma = 1 $ and $ \varpi^{s_{2}(\mu)} \sigma_{1} $  if $ \gamma = v_{1} $.

The case where $  \gamma = \eta_{0} $ requires a closer  case-by-case analysis.   Here we need study the possible values $ \chi \in S_{1}^{\pm} $  such that $ U' \varpi^{\lambda} \sigma_{i}  \chi H_{\tau_{1}} ' =  U ' \varpi^{\mu} H_{\tau_{1}}'  $.  
We let $ C(\lambda) = \left \{ \lambda,  r_{0}(\lambda) , r_{1} r_{0} (\lambda), r_{0} r_{1}  r_{0}( \lambda)  \right \}  $. In each case, we compute the intersection $ C(\lambda) \cap W_{\tau_{1}}'\mu $,  using which we read off the possible values of $ \chi $ from Lemma \ref{liftsforsigma2U'}, i.e., we only consider $ \chi $ for which $ U '\varpi^{\lambda} \sigma_{2} \chi  H_{\tau_{1}}' = U' \varpi^{\mu'} H_{\tau_{1}}'$ for $ \mu' \in C(\lambda) \cap W_{\tau_{1}}' $ which is a necessary condition by Proposition  \ref{cartanhtaui}.  We then use Lemma \ref{liftsforsigma2U} to simplify these cosets if possible. In most cases, this results in a single element in the fiber. For $ \gamma = \eta_{1} $, the analysis is similar but much easier and we will  only need Lemma \ref{liftsforsigma3} to decide the elements of the fiber.  
\\

\noindent $ \bullet $ $  \mu =  (1,1,1,0) $. \\[0.3em]    The projection is $ \mathscr{R}_{U_{2}^{\dagger}}(\varpi^{(1,1,0)} )  =  \mathscr{R}_{U_{2}^{\dagger}} ( \varpi^{(1,0,0)} ) = \left \{ \varpi^{(1,0,0) } , \varpi^{(1,0,0)  } v_{1}  ,    \varpi^{(1,1,0) }  \eta_{0}     \right \} $. To determine the lift of  $   \varpi^{(1,1,0)} \eta_{0} $,  let $ \lambda  : = (1,1,1,0)$. Then $ C(\lambda) = \left \{ (1,1,1,0), (1,1,0,0) \right \} $ and $ \varpi^{(1,1,0,0)} \notin U' \varpi^{\mu} H_{\tau_{1}}'  $  by Lemma    \ref{cartanhtaui}.    Lemma  \ref{liftsforsigma2U'} tells us for $ \chi \in S^{\pm}_{1} $, $ U' \varpi^{\lambda} \sigma_{2} \chi H_{\tau_{1}} = U' \varpi^{\mu} H_{\tau_{1}}' $ only when 
$ \chi \in S^{-}   $.    But then $ U \varpi^{\lambda} \sigma_{2} \chi  H_{\tau_{1}} ' = U \varpi^{\lambda} \sigma_{2} H_{\tau_{1}}' $  by  Lemma   \ref{liftsforsigma2U}.   Thus $ \varpi^{(1,1,1,0)}  \sigma_{2} $ is the unique element of the fiber  above  $ \varpi^{(1,1,0)}  \eta_{1}  $.       \\

\noindent   $ \bullet $  $  \mu  =  (1,1,0,1 ) $. \\[0.3em]  
We have  $ \mathscr{R}_{U^{\dagger}_{2}}(\varpi^{(1,0,1)}) =  \mathscr{R}_{U^{\dagger}_{2}}(\varpi^{(1,1,1)})      = \left \{ \varpi^{(1,1,1)} , \,  \varpi^{(1,1,1) }  v_{1} , \,   \varpi^{(1,0,0)}\eta_{0} , \,  \varpi^{(1,0,1)} \eta_{0} , \,  \varpi^{(1,1,1)  } \eta_{0} , \, \eta_{1}  \right \} .  $ 
Let $ \lambda_{1} = (1,1,0,0) $, $ \lambda_{2} = (1,1,0,1) $, $ \lambda_{3} = (1,1,1,1) $. Then  
$  
C(\lambda_{i})  \cap W_{\tau_{1}}\mu
 = 
\left \{ (1,1,0,1)  \right \}    $.    
For $ \lambda_{1} $ and $ \lambda_{3} $, the only choice is $ \chi = \begin{psmallmatrix} & 1 \\ -1  \end{psmallmatrix}  $ and so the unique elements in the fibers above $ \varpi^{(1,0,0)}\eta_{0} $ and  $ \varpi^{(1,1,1)}\eta_{0} $  are respectively $  \varpi^{s_{0}(\lambda_{1})}  \sigma_{2}   = \varpi^{(1,0,0,0)} \sigma_{2}  $ and  $ \varpi^{s_{0}(\lambda_{3})} \sigma_{2}  =  \varpi^{(1,0,1,1)} \sigma_{2}$.  For $ \lambda_{2}$,  
the only  choice is $ \chi = \begin{psmallmatrix} 1 \\ & 1 \end{psmallmatrix} \in S^{-} $, $ \varpi^{(\lambda_{2})}\sigma_{2} $ is the unique element in the fiber above  $ \varpi^{(1,0,1)}\eta_{0} $.

For $ \eta_{1} $, the unique element above is $ \varpi^{-f_{1}}  \sigma_{3}  $ by  Lemma  \ref{liftsforsigma3},  since $ \lambda_{\circ} = (1,1,1,1) $ does not belong to $  W_{\tau_{1}}' \mu $ but $ (1,0,1,1) =  s_{0}( \lambda_{\circ})   $ does.  \\

\noindent  $ \bullet $   $ \mu =  (1,1,0,0) $.  \\[0.3em]   This is similar to the first case  except now we work with $ \varpi^{(1,1,0,0)} \in C(\lambda) $ where $ \lambda = (1,1,1,0) $. In this case, the only possible choice for $ \chi  =    \begin{psmallmatrix} & 1 \\ -1 &  \end{psmallmatrix} $.  The fiber is therefore $ U \varpi^{\lambda} \sigma_{2} \chi H_{\tau_{1}}' = U \varpi^{s_{0}(\lambda) } \sigma_{2} H_{\tau_{1}}  $ and we take the representative $  \varpi^{(1,0,1,0)}\sigma_{2}  $.  \\

\noindent  $ \bullet $  $  \mu  =  ( 2,2,1,1) . $ \\[0.3em]
The projection is $ \mathscr{R}_{U_{2}^{\dagger}}(\varpi^{(2,1,1)}) = \left \{ \varpi^{(2,1,1)}, \varpi^{(2,1,1)} v_{1},  \varpi^{(2,1,1)} \eta_{0} \right \} $. Lemma \ref{liftsforsigma2U'} implies that $ U '  \varpi^{(2,2,1,1)} \sigma_{2} \chi  H_{\tau_{1}}' $ coincides with $  U ' \varpi^{(2,2,1,1)} H{\tau_{1}}'$ for any $ \chi \in S^{\pm}_{1} $. Now if $ \chi \in S^{-} $, then $ U \varpi^{(2,2,1,1)}\sigma_{2} \chi H_{\tau_{1}} ' = U \varpi^{(2,2,1,1)}\sigma_{2} H_{\tau_{1}}' $ by Lemma \ref{liftsforsigma2U}.    If however $ \chi = \begin{psmallmatrix} & 1 \\ -1  \end{psmallmatrix} $, then $ U ' \varpi^{(2,2,1,1)} \sigma_{2} H_{\tau_{1}}' = U' \varpi^{(2,0,1,1)} H_{\tau_{1}}' $. Thus the fiber  above $ \varpi^{(2,1,1)} \eta_{0} $ consists of $$ U \varpi^{(2,2,1,1)}\sigma_{2} H_{\tau_{1}}'  \, \text {  and  }  U \varpi^{(2,0,1,1)}\sigma_{2} H_{\tau_{1}}'. 
$$  These are distinct elements of the fiber, since $ U \varpi^{(2,2,1,1)} H_{\varsigma_{2}}  \subset U \varpi^{(2,2,1,1)}  J_{\varsigma_{2}} $, $ U\varpi^{(2,0,1,1)} H_{\varsigma_{2}} \subset U \varpi^{(2,0,1,1)} J_{\varsigma_{2}}  $ by  Lemma \ref{Hvarsig2inIwahori} and $ U \backslash H / J_{\varsigma_{2}} \simeq  \Lambda $.  \\

\noindent $\bullet $ $ \mu = ( 2,1,2,1) $, $(2,1,1,2)$,  $ (2,1,1,1) $. \\[0.3em]  These are handled by Lemma \ref{pVbijection}. \\

\noindent $  \bullet $ $ \mu = (2,2,0,1) $\\[0.3em]  
The projection is $ \mathscr{R}_{U_{2}^{\dagger}}(\varpi^{(2,2,1)}) = \left \{     \varpi^{(2,1,2)}v_{1} , \,   \varpi^{(2,2,1)}, \, 
\varpi^{(2,2,1)}\eta_{0},
\varpi^{(2,2,0)}\eta_{0},\,
\varpi^{(2,1,0)}\eta_{0},   \,  \varpi^{(1,1,0)}\eta_{1}   \right \} $. Let $ \lambda_{1} = (2,2,2,1) $, $ \lambda_{2} = (2,2,2,0) $, $ \lambda_{3} = (2,2,1,0) $. Then for  $ \chi \in S^{\pm}_{1} $ and any $ i = 1, 2,  3 $, the double coset   $ U ' \varpi^{\lambda_{i}} \sigma_{2} \chi H_{\tau_{1}}' $ coincides with 
$ U ' \varpi^{(2,2,0,1)} H_{\tau_{1}}'  $ only when $ \chi  =  \begin{psmallmatrix} &  1 \\ - 1 \end{psmallmatrix} $. This gives the three desired representatives. 

As for $ \varpi^{(1,1,0)}  \eta_{1} $, the unique element in the fiber  is $ \varpi^{s_{0}(1,1,1,0) -  f_{1}} \sigma_{3}  =   \varpi^{(1,-1,1,0)} \sigma_{3} $  by Lemma \ref{liftsforsigma3}  since $ (1,1,1,0) + \lambda_{\circ} =  (2,2,2,1)   \notin W_{\tau_{1}}' \mu $  but  $   (2,0,2,1) = s_{0}(1,1,1,0) + s_{0} ( \lambda_{\circ}   ) = ( 2,0,2,1) \in W_{\tau_{1}}' $.  \\

\noindent $ \bullet $   
$ \mu = (3,2,2,2)  $   \\[0.3em] 
We have $ \mathscr{R}(\varpi^{(3,2,2)}) = \varpi^{(2,1,1)}\mathscr{R}_{U_{2}^{\dagger}}(\varpi^{(1,1,1)}) $, so
$$ 
\mathscr{R}_{U_{2}^{\dagger} } (\varpi^{(3,2,2)}) =  \left \{ \varpi^{(3,2,2)},  \varpi^{(3,2,2)}v_{1},   \varpi^{(3,1,1)}\eta_{0},  \varpi^{(3,1,2)}\eta_{0},  \varpi^{(3,2,2)}\eta_{0},  \varpi^{(2,1,1)}\eta_{1}   \right  \}. $$ 
Let $ \lambda_{1} = (3,2,1,1) $, $ \lambda_{2} = (3,2,1,2) $, $ \lambda_{3} =(3,2,2,2) $. Then $ C(\lambda_{i}) \cap W_{\tau_{1}}' \mu =  \left \{  (3,2,2,2) \right \} $ for all $ i $. For  $ \lambda_{1} $ and $ \lambda_{3} $, Lemma \ref{liftsforsigma2U'} forces $ \chi $ to be in  $ S^{-}_{1} $, and  Lemma \ref{liftsforsigma2U}   allow us to conclude that  $ U \varpi^{\lambda_{1}}\sigma_{2} H_{\tau_{1}}'  $, $ U\varpi^{\lambda_{2}}\sigma_{2}  H_{\tau_{1}'}  $ are the only elements of the fibers above $ \varpi^{(3,1,1)}\eta_{0} $, $ \varpi^{(3,2,2)}\eta_{0} $ respectively.  For $ \lambda_{2} $, the possible choices are $ \chi = \begin{psmallmatrix} & 1 \\ -1  \end{psmallmatrix} $ or $ \chi = \begin{psmallmatrix}  1 &  \\ x  & 1  \end{psmallmatrix} $ for $ x \in \Oscr_{F}  ^{\times }$. In the latter case, we have $ U \varpi^{\lambda_{2}} \sigma_{2} \chi  H_{\tau_{1}}' = U \varpi^{\lambda_{2}} \sigma_{2} \psi H_{\tau_{1}}' $
since the conjugate of $ \varpi^{\lambda}\sigma_{2}\chi $ by $  \mathrm{diag}(x, 1, x , 1 , x, 1) \in U \cap  H_{\tau_{1}}' $ equals $ \varpi^{\lambda_{2}} \sigma_{2} \psi $. So the fiber above $ \varpi^{(3,1,2)}\eta_{0} $ contains $$  U \varpi^{s_{0}(\lambda_{2})}\sigma_{2}  H_{\tau_{1}}' \,  \text{ and } \,   U  \varpi^{\lambda_{2}} \sigma_{2} \psi H_{\tau_{1}}' . $$
Since  $ U \varpi^{\lambda_{2}} \psi J_{\varsigma_{2}}  = U \varpi^{\lambda_{2}} J_{\varsigma_{2}} $, so the same argument used in the case $ \mu = (2,2,1,1) $ shows that the two displayed elements are distinct.

For $  \varpi^{(2,1,1)} \eta_{1} $, the only element in the fiber is $ \varpi^{(2,1,1,1)}\sigma_{3} $ by Lemma \ref{liftsforsigma3}, since $  (3,2,2,2) = (2,1,1,1) + \lambda_{\circ}  $ belongs to $ W_{\tau_{1}}' $ but $ (3,1,2,2) =  s_{0} (2,1,1,1) + s_{0} ( \lambda_{\circ}  )  $ does not.     \\      

\noindent  $\bullet$ $ \mu =  (3,3,1,1) $ \\[0.3em] 
The projection is $ \varpi^{(2,1,1)}  \cdot 
 \mathscr{R}_{U_{2}^{\dagger}}(\varpi^{(1,0,0)}) = \left \{ \varpi^{(3,1,1)}, \varpi^{(3,1,1)} v_{1}, \varpi^{(3,2,1)}\eta_{0} \right \} $. For  $ \lambda = (3,3,2,1)$, the only possibility is $ \chi = \begin{psmallmatrix} & 1 \\ -1  \end{psmallmatrix} $ which gives the representative $ \varpi^{(3,0,2,1)}\sigma_{2} $ in the fiber above $ \varpi^{(3,2,1)}\eta_{0} $.   \\          

\noindent $ \bullet $  $ \mu   =  ( 3 , 2 , 0 ,   1 )  $. \\[0.3em] 
The projection is $ \mathscr{R}(\varpi^{(3,3,1)}) = \left \{ \varpi^{(3,1,3)}v_{1}, \varpi^{(3,3,1)}, \varpi^{(3,3,1)}\eta_{0},  \varpi^{(3,2,0)}\eta_{0},  \varpi^{(2,2,0)}\eta_{1} \right \}  $. Set  $ \lambda_{1} : = (3,2,3,1) $ and $ \lambda_{2}  : = (3, 2,2,0)  $.  Then $ C(\lambda_{i})  \cap W_{\tau_{1}}'\mu = \left \{ \mu \right \} $ for $ i = 1, 2 $. In both cases, the only possibility is $ \chi  =  \begin{psmallmatrix} & 1 \\ -1 & \end{psmallmatrix} $ which gives the rerpresentatives $ \varpi^{(3,1,3,1)}\sigma_{2} $, $ \varpi^{(3,1,2,0)}\sigma_{2} $ above $ \varpi^{(3,3,1)}\eta_{0} $, $ \varpi^{(3,2,0)}\eta_{0} $ respectively. 

For $ \varpi^{(2,2,0)}\eta_{1} $, the unique element in the fiber is $    \varpi^{(2,0,2,0)}  \sigma_{3}     $ since $ (2,1,2,0) + \lambda_{\circ} \notin W_{\tau_{1}}' $ but $  (3,1,3,1) =  s_{0} (2,1,2,0) + s_{0}(\lambda_{\circ}) \in W_{\tau_{1}}'\mu $.       
\end{proof}  
\subsection{Orbits on $U' h H_{\tau_{2}}'/ H_{\tau_{2}}'   $}  \label{UHtau2'orbitssec}  
Let $   E   =  \pr_{2}(U^{\ddagger}) $ denote the projection of the group $ U^{\ddagger} $. Thus     $  E \subset U_{2}  '  $ is the \emph{endohoric}\footnote{a portmanteu of Iwahori and endoscopic} subgroup of all  elements whose reduction modulo $ \varpi $ lies in $   \mathbf{H}_{2} (\kay)  =  \GL_{2}(\kay) \times _{\kay^{\times}  } \GL_{2}(\kay) $.  
For $ a ,  b \in  \Oscr_{F}    $, let   $$ \gamma(u, v)    =    
    \begin{pmatrix}    1  & u  &  &   v \\ 
                          &  1 &   v       \\
                           & &  1  \\
                            & & -u &  1  
\end{pmatrix}  .   $$

\begin{lemma}
\label{IwYlemma}  $  I_{2} ' 
E  / E   = \bigsqcup_{a,b \in  [\kay] }  
\gamma (a, b ) E  
$.  
\end{lemma} 
\begin{proof}  Let $ \mathbf{N}'_{2} $ (resp., $\mathbf{N}_{2}$)  denote the unipotent radical of the Borel subgroup of $ \Hb_{2}' $ (resp., $\Hb_{2}$) determined by $ \left \{ \beta_{0} , \beta_{2} \right \} $.   Let $ Z \subset E $ the subgroup of all elements  that  reduce modulo $ \varpi $ to the Borel subgroup of $ H_{2} $. 
 Then  $ Z  = I_{2}' \cap E $ and so $$ I_{2}' E /  E   \simeq  I_{2}'/  Z    \simeq \mathbf{N}_{2}' (\kay)/ \mathbf{N}_{2}(\kay) . $$ 
 Now $ | \mathbf{N}_{2}'(\kay) | = q^{4} $  and $ | \mathbf{N}_{2}(\kay) | = q^{2} $ and so $ | I_{2} ' / Z |  = q^{2} $ and it is easily seen that the reduction of $ \gamma(u,v) $ for $ u, v \in [\kay] $ form a complete set of representatives for $ \mathbf{N}_{2}'(\kay) / \mathbf{N}_{2}(\kay) $.  
\end{proof}    
Let $ v_{1} $ be as in \S \ref{players}  and $  \eta_{0}, \eta_{1} $ be as in (\ref{eta0matrix}),  (\ref{eta12matrices}). Recall (\ref{etakmatrix}) that for $ k \in [\kay] $, we  denote 
\begin{equation} 
\tilde{\eta}_{k} =  \scalebox{0.9}{$\begin{pmatrix}     k & 1  \\
k+1 & 1 \\ & & -1 & k + 1   \\[0.1em]   &  & 1 & - k   \end{pmatrix}$}  
\end{equation} 
and  $ [\kay]^{\circ}  = [\kay] \setminus \left \{ -1 \right \} $.  
\begin{lemma}    \label{distinctH2E} $ 1 $, $v_{1} $, $\eta_{0}, \eta_{1},  $ and $ \tilde{\eta}_{k} $ for $ k \in [\kay]^{\circ}  $ represent pairwise  distinct classes in $ H_{2} \backslash H_{2}'/ E $.          
\end{lemma}  
\begin{proof} 
This is 
handled as in Lemma \ref{distinctH2U2circ}.  The matrix formulas shown  therein for $ \eta_{i} $, $ i =1 , 2 $  also apply for $ i = 0 $ and it is easy to deduce the pairwise distinction for $ 1, v_{1}, \eta_{0}, \eta_{1} $ from these formulas. Let us distinguish the class of $ \tilde{\eta}_{k} $ for $ k \in [\kay]^{\circ} $ from $ \gamma \in \left \{ 1, v_{1} ,\eta_{0}, \eta_{1} \right \}  $. Write $ h \in H $ as in Notation \ref{notationh2}. Then  
$$ h \tilde{\eta}_{k}  
= \begin{pmatrix}  
ak & a  & -b  & b(k+1) \\ * & * & * & * \\   * & * &  *  & *  \\  * & * & * &*  
\end{pmatrix},  \quad   
\eta_{i}^{-1} h \tilde{\eta}_{k}  
= \begin{pmatrix} * & * & * & * 
\\ * & * & * & * \\ 
ck & c & -d  & d (k+1)   \\ 
*&  *& * & *  \end{pmatrix}$$ 
for $ i = 0, 1 $.  
If $ h \tilde{\eta}_{k} \in E $, we see from the entries shown above that the first row is a multiple of $ \varpi $ which makes $ \det(h\tilde{\eta}_{k}) \in \varpi \Oscr_{F}$, a contradiction. Since $ v_{1} $ just swaps the rows of $ h\tilde{\eta}_{k} $, the same argument applies to $ v_{1} h \tilde{\eta}_{k}  $. Similarly for $ \eta_{i}^{-1} h \tilde{\eta}_{k} $.  Finally for $ k ,k' \in  [\kay]^{\circ} $ and $ k \neq k' $, we see from the matrix    $$ 
 (\tilde{\eta}_{k})^{-1} h \tilde{\eta}_{k'}     =       \begin{pmatrix}  *  &  a_{1} -  a  &  * & - b - b_{1} k' \\ a k'  - a_{1} k &  * & - b - b_{1} k   &  * \\  *  & *  &   *  & *     \\ * & * &  * &  *  \end{pmatrix} $$ 
 that $ (\tilde{\eta}_{k})^{-1} h \tilde{\eta}_{k'} $ lies in $ E $ only if $ a, b \in \varpi \Oscr_{F} $. But since $ \tilde{\eta}_{k} , \tilde{\eta}_{k'} \in U_{2}' $ and $ E \subset U_{2}' $, we also have $ h \in U_{2}' $. But then $ a, b \in \varpi \Oscr_{F}$ implies that $ \mathrm{sim}(h) = ad - bc \in \varpi  \Oscr_{F} $, a contradiction. 
\end{proof}  

\begin{remark} Note that $ U v_{2} \tilde{\eta}_{-1}v_{2} E = U \eta_{0} E $.    
\end{remark} 
\begin{lemma}   \label{cartanE}       For $  \gamma  \in \left \{ \eta_{0},  \tilde{\eta}_{0}    \right \} $, the map $ \Lambda_{2} \to U_{2} \varpi^{\Lambda}   \gamma E $, $ \lambda \mapsto  U_{2} \varpi^{\lambda}  \gamma   E $ is a bijection.   
\end{lemma} 
\begin{proof} It is easy to see that  $ H_{2} \cap \gamma E \gamma^{-1} $ are  contained in certain  Iwahori subgroups of $H_{2} $.  So the Bruhat-Tits decomposition along with the identification $ U_{2} \varpi^{\Lambda} \gamma E \xrightarrow{\sim}  U \varpi^{\Lambda} (H_{2} \cap \gamma E \gamma^{-1})  $ implies the result.     
\end{proof} 
\begin{lemma}  \label{Iwahoripassesthrough}    If $ \lambda \in \Lambda_{2} $  is such that  $ \beta_{1}(\lambda), \beta_{2}(\lambda)  \in \left \{ 0 ,1 \right \}  $, then $  U_{2}' \varpi^{\lambda} E  =  U_{2}' \varpi^{\lambda}  I_{2}'  E $.     
\end{lemma}  
\begin{proof} The conditions ensure that $ \varpi^{\lambda} I_{2}' \varpi^{-\lambda}  \subset U_{2}' $.      
\end{proof} 
\begin{notation} For this subsection only, we  let  $ \mathscr{R}_{E}(h) $, denote the double coset space $ U_{2} \backslash   U_{2}' h E/ E  $   for   $ h \in H_{2}' $.   
\end{notation} 
\begin{proposition} \label{Ehard}    We have  
\begin{enumerate}[label = \normalfont(\alph*)]    \setlength\itemsep{0.5em}
\item $ \mathscr{R}_{E} (\varpi^{(0,0,0)}) =  
\left \{   1  ,  v_{1} , \eta_{0} ,    \tilde {\eta}_{k} \, | \,  k \in [\kay]^{\circ} \right \} $.   
\item  $   \mathscr{R}_{E}(\varpi^{(1,1,1)}) =$ 
\begin{minipage}[t]{1.0\textwidth} 
$\{\varpi^{(1,1,1)},\,  \varpi^{(1,1,1)}v_{1},\, \varpi^{(1,0,1)}\eta_{0},\, \varpi^{(1,1,0)}\eta_{0},\,  \varpi^{(1,1,1)}\eta_{0},\,  
\eta_{1}\} \cup \\[0.4em] 
\Set*{\varpi^{(1, 1, 1)}\tilde{\eta}_{0},\,  \varpi^{(1,0,1)}\tilde{\eta}_{0},\,  \varpi^{(1,0,0)}\tilde{\eta}_{k}       \given k \in [\kay]^{\circ}}$ 
\end{minipage}
\item     $\mathscr{R}_{E}(\varpi^{(2,2,1)}) 
=  \Set * {  \varpi^{(2, 2, 1)}, \,   \varpi^{(2,1,2) }  v_{1} ,  \,    \varpi^{(2,2,1) }\eta_{0} ,  \,   \varpi^{(2,1,2)} \tilde {\eta}_{0}  ,  \,     \varpi^{(1,1,0) } \eta_{1}   }  $.    
\end{enumerate}      
\end{proposition} 
\begin{proof} By Lemma \ref{distinctH2E} and \ref{cartanE}, the elements listed in part (a), (b), (c) represent distinct classes.  We show that these also form a full set of representatives. 
Say for $ \lambda \in \Lambda_{2}$   is such that $ 0 \leq \beta_{1} ( \lambda) , \beta_{2}(\lambda)  \leq  1 $ and say $ U_{2}' \varpi^{\lambda} I_{2}'  =   \bigsqcup_{\gamma \in \Gamma} \gamma \tilde{I}_{2} $  for some finite set $ \Gamma $.   
Then by Lemma \ref{Iwahoripassesthrough} and \ref{IwYlemma},  

\begin{equation}  \label{Edecomrecipe}    U _{2}   '  h  E    =    U_{2}' h  I_{2}' E        =       \bigcup_{\gamma \in \Gamma }  U_{2} \gamma   I_{2}'  E    =      \bigcup_{\substack{\gamma  \in \Gamma \\  u ,v  \in [\kay] } }  U_{2} \gamma  \gamma_{u, v }   E 
\end{equation} 
Since $ (0,0,0) $, $ (1,1,1) $, $ (2,2,1) $ satisfy the condition of Lemma \ref{Iwahoripassesthrough}, the decomposition (\ref{Edecomrecipe}) applies. Now  we can compute the set $ \Gamma $ for each $ \lambda $ by replacing $ \varpi^{\lambda} $ with $ w \in W_{I_{2}'}  $ of  minimal  possible  length   such that $ U_{2}' \varpi^{\lambda} E = U_{2}'w E $ and invoking the analogue of Proposition \ref{SchubertGSp6}    for $ \mathrm{GSp}_{4} $. 
Since we are only interested in computing the double cosets $ U_{2} \gamma \gamma_{u, v}  E $ appearing in $ U_{2}'   w  E   $,  we only need to study the  cells corresponding to $$   \varepsilon_{0} : =  w , \quad    \varepsilon_{1} :  = r_{1}w, \quad   \varepsilon_{2} : = r_{1} r_{2}w , \quad  \varepsilon_{3} :  =  r_{1}r_{2}r_{1}w  .     $$ 
   Thus we need to study the classes in  $ U_{2} \backslash H_{2}' / E $ of  $ \left \{  \mathcal{Y}_{\varepsilon_{i}}(\vec{\kappa}) \gamma_{u,v} E \, |  \,  \vec{\kappa} \in [\kay]^{l(\varepsilon_{i})} , \,   u, v \in [\kay]     \right \} $ for each $  i = 0    ,    1, 2 , 3 $. We will refer to these sets Schubert cells  as  well and as usual,  abuse notation to  denote them by  $  \Yvar{i}E/E $. \\    

\noindent (a)  Here $ w = 1 $ and the four cells are  
\begin{alignat*}{4}  \mathcal{Y}_{\varepsilon _{0}}E/ E &  =  
\scalebox{0.9}{$\Set*{
\left(\begin{array}{cccc}
 1 &  u   &  & v \\ 
  & 1  & v  & \\[0.2em] 
 &  &  1 &  \\[0.2em] 
  & &  - u  &  
\end{array}\right)}$},& 
 \mathcal{Y}_{\varepsilon _{2} } E /  E   &=
 \scalebox{0.9}{$\Set*{
\left(\begin{array}{cccc}
a &    y + au &  vy - u & av + 1 \\[0.2em] 
1 &  u  & &   v \\
  & 1  & v  & \\[0.2em] 
  &  - a  & - (av+1)  &   
\end{array}\right)}$},\\
\mathcal{Y}_{\varepsilon _{1}} E / E  & = \scalebox{0.9}
{$\Set*{\left(\begin{array}{cccc}
a&  au + 1    &  v &  av \\
1& u & &   v \\[0.2em] 
&   &  - u  &   1  \\[0.2em] 
  &  &  au + 1   &  -a     
\end{array}\right)},$}  
&
\quad 
\mathcal{Y}_{\varepsilon _{3}} E  / E   &=  \scalebox{0.9}
{$\Set*{\left(\begin{array}{cccc}
z   &  u z +  a_{1}  & a u + a_{1} v + 1  &  vz - a \\
a & au + 1 &  v  &  av  \\[0.2em] 
 1&  u  &   & v    \\[0.2em] 
 -a_{1} & -a_{1} u  &  u  & -a_{1}v - 1  
\end{array}\right)}$}   
\end{alignat*} 
where $ a, a_{1}, u, v , y \in   [\kay] $ and $ z  :  = y + aa_{1} $.  Note that the  $ \varepsilon_{1}$-cell is obtained from $ \varepsilon_{0}$-cell by multiplying on the left by $ y_{1}(a)v_{1} $. If $ a = 0 $, the orbits of $ U_{2} $ are $ v_{1} $ times  those of $ \varepsilon_{0} $-cell since $ v_{1} $ normalizes $ U_{2} $. Similarly we can assume that $ a \neq 0 $ in $ \varepsilon_{2}$-cell and $ a_{1} \neq 0 $ in $  \varepsilon_{3} $-cell.   

Consider the $ \varepsilon_{0}$-cell. Conjugation by $ v_{2} $ swaps the entries $ u $, $ v $ and row column operations arising from $ U_{2} $, $ E $   allow us to make at least one of  $u$, $v$ zero. So say $ u = 0 $. Then we obtain either identity or $ \eta_{0} $ as representative from this cell.  Next consider the $ \varepsilon_{1} $-cell. As observed above,  the case  $ a = 0 $ leads to orbits of $ v_{1} $ and $ v_{1} \eta_{0} $ and we have $ U_{2} v_{1} \eta_{0} E = U_{2}   \tilde{\eta}_{0} E $. If $ a \neq 0 $,  we apply the following sequence of row-column operations:  $$ \scalebox{0.9}{$\left(\begin{array}{cccc}
a&  au + 1    &  v &  av \\
1& u & &   v \\[0.2em] 
&    &  - u  &   1  \\[0.2em] 
  &  &  au + 1   &  -a     
\end{array}\right)$} \longrightarrow   \scalebox{0.9}{$\left(\begin{array}{cccc}
a&  au + 1    &  &  av \\
1& u & -v/a  &   v \\[0.2em] 
&     &  - u  &   1  \\[0.2em] 
  &  &  au + 1   &  -a     
\end{array}\right)$}     \longrightarrow     \scalebox{0.9}{$\left(\begin{array}{cccc}
a&  au + 1    &  &  av \\
1& u & uv  &    \\[0.2em] 
&     &  - u  &   1  \\[0.2em] 
  &  &  au + 1   &  -a     
\end{array}\right)$} $$ $$   \longrightarrow   
  \scalebox{0.9}{$\left(\begin{array}{cccc}
a&  au + 1    & a uv   &  av \\
1& u &  &    \\[0.2em] 
&   &  - u  &   1  \\[0.2em] 
  &  &  au + 1   &  -a     
\end{array}\right)$} \longrightarrow    \scalebox{0.9}{$\left(\begin{array}{cccc}
a&  au + 1    &    &  \\
1& u &   &    \\[0.2em] 
&   &  - u  &   1  \\[0.2em] 
  &  &  au + 1   &  -a     
\end{array}\right)$}   \longrightarrow    \scalebox{0.9}{$\left(\begin{array}{cccc}
1&  au + 1    &  &  \\
1& au &   &    \\[0.2em] 
&   &  - au  &   1  \\[0.2em] 
  &  &  au + 1   &  -1     
\end{array}\right)$.}
$$ 
Let us denote $ k = au $. The structure of $ Y $ allows us to restrict $ k \in [\kay] $. Conjugating this matrix by $ v_{\beta_{0}}v_{\beta_{2}} $ and scaling by $ -1 $ gives us the matrix $ \tilde{ \eta}_{k} $ if $ k \in [\kay]^{\circ} $, i.e., $ au \neq - 1$. If $ au = -1 $ however,  then conjugating by $ v_{2} $ further  gives us $ \eta_{0} $. So the $   \varepsilon_{1} $-cells decomposes into $U_{2}$-orbits of $ v_{1},  \eta_{0} $ and $ \tilde{\eta}_{k} $ for $ k \in [\kay]^{\circ} $. 

For the case of $ \varepsilon_{2} $-cell and $ a \neq 0 $, use   
$$ 
 \scalebox{0.9}{$
\left(\begin{array}{cccc}
a &    y + au &  vy - u & av + 1 \\[0.2em] 
1 &  u  & &   v \\
  & 1  & v  & \\[0.2em] 
  &  - a  & - (av+1)  &   
\end{array}\right)$} \longrightarrow 
 \scalebox{0.9}{$
\left(\begin{array}{cccc}
a &   &  -auv-u & av + 1 \\[0.2em] 
1 &  u  & &   v \\
  & 1  & v  & \\[0.2em] 
  &  - a  & - (av+1)  &   
\end{array}\right)$} \longrightarrow    
 \scalebox{0.9}{$
\left(\begin{array}{cccc}
a &    &  & av + 1 \\[0.2em] 
1 &  u  & uv + u/a &   v \\
  & 1  & v  & \\[0.2em] 
  &  - a  & - (av+1)  &   
\end{array}\right)$} $$  $$  \longrightarrow    \scalebox{0.9}{$
\left(\begin{array}{cccc}
a &    &  & av + 1 \\[0.2em] 
1 &   & &   v \\
  & 1  & v  & \\[0.2em] 
  &  - a  & - (av+1)  &   
\end{array}\right)$}  \longrightarrow    \scalebox{0.9}{$
\left(\begin{array}{cccc}
 &   1 &   v  & \\[0.2em] 
& a  &  (av + 1)  &    \\
-a  &   &   & - (av +  1)   \\[0.2em] 
1  &  &  &  v   
\end{array}\right)$}    \longrightarrow    \scalebox{0.9}{$
\left(\begin{array}{cccc}
v &   1 &      & \\[0.2em] 
(av+1)& a  &  &    \\
 &   & - a  & av + 1   \\[0.2em] 
  &  &  1 &  - v   
\end{array}\right)$}      $$
and multiply on the left by $ \mathrm{diag}(a, 1,  1, a ) $ and $ \mathrm{diag}(1, a^{-1}  ,  a^{-1}  ,1) $ on the right to arrive at the same situation as the $ \varepsilon_{1}$-cell.  Finally the case for $ \varepsilon_{3} $-cell with $ a_{1} \neq 0 $, use 
$$  \scalebox{0.9}
{$\left(\begin{array}{cccc}
z   &  u z +  a_{1}  & a u + a_{1} v + 1  &  vz - a \\
a & au + 1 &  v  &  av  \\[0.2em] 
 1&  u  &   & v    \\[0.2em] 
 -a_{1} & -a_{1} u  &  u  & -a_{1}v - 1  
\end{array}\right)$}  \longrightarrow    \scalebox{0.9}
{$\left(\begin{array}{cccc}
  &   a_{1}  & a u + a_{1} v + 1  &   - a \\
a & au + 1 &  v  &  av  \\[0.2em] 
 1&  u  &   & v    \\[0.2em] 
 -a_{1} & -a_{1} u  &  u  & -a_{1}v - 1  
\end{array}\right)$}     $$ 
$$\scalebox{0.9}
{$\left(\begin{array}{cccc}
  &   a_{1}  & a u + a_{1} v + 1  &   - a \\
 &  1 &  (au + a_{1}v)/a_{1}   &  -a/a_{1}    \\[0.2em] 
 & u & &   v   \\[0.2em] 
 -a_{1} & -a_{1} u  &  u  & -a_{1}v - 1  
\end{array}\right)$}    \longrightarrow  \scalebox{0.9}
{$\left(\begin{array}{cccc}
  &   a_{1}  & a u + a_{1} v + 1  &  
  \\
 &  1 & (au+ a_{1}v)/a_{1}  &   \\[0.2em] 
 1&   u  &  
 &    (au + a_{1} v ) / a_{1}    \\[0.2em] 
 -a_{1} & -a_{1} u  &  u  & - (au + a_{1}v + 1 )  
\end{array}\right)$}    . $$
Next  substitute $   k_{1}  =  au + a_{1} v $ and use 
$$ \scalebox{0.8}
{$\left(\begin{array}{cccc}
  &   a_{1}  & k_{1}  + 1  &  
  \\
&  1 &   k_{1}/a_{1}    &   \\[0.2em] 
 1&   u  &  
 &    k  _  { 1  }   / a_{1}    \\[0.2em] 
 -a_{1} & -a_{1} u  &  u  & - k_{1}  - 1   
\end{array}\right)$}     \longrightarrow 
   \scalebox{0.8}
{$\left(\begin{array}{cccc}
  &   a_{1}  & k_{1} + 1  &  
  \\
 &  1 &  k_{1}  /a_{1}  &   \\[0.2em] 
 1&     &   - u ( k_{1} + 1 )  /  a_{1}  
 &    k_{1} / a_{1}    \\[0.2em] 
 -a_{1} & -a_{1} u  &  u  & - k_{1}  - 1   
\end{array}\right)$}     \longrightarrow   
   \scalebox{0.8}
{$\left(\begin{array}{cccc}
  &   a_{1}  & k_{1} + 1  &  
  \\
 &  1 &   k_{1} / a_{1 }  &   \\[0.2em] 
 1&     &    
  &    k_{1} / a_{1}    \\[0.2em] 
 -a_{1} & -a_{1} u  & -uk_{1}       & - k_{1} - 1   
\end{array}\right)$}  $$ $$ \longrightarrow         \scalebox{0.8}
{$\left(\begin{array}{cccc}
  &   a_{1}  & k_{1} + 1  &  
  \\
 &  1 &  k_{1} / a_{1}  &   \\[0.2em] 
 1&     &   
 &    k_{1} / a_{1}    \\[0.2em] 
 -a_{1} & &    & - k_{1} - 1   
\end{array}\right)$} \longrightarrow    \scalebox{0.85}
{$\left(\begin{array}{cccc}
  &   a_{1}  & k_{1}  + 1  &  
  \\
 &  1 &  k_{1}   / a_{1}  &   \\[0.2em] 
 1&     &    
 &    k_{1 }  / a_{1}    \\[0.2em] 
 -a_{1} & &   & - k _{1}  - 1   
\end{array}\right)$}  
\longrightarrow  \scalebox{0.8}
{$\left(\begin{array}{cccc}
  k_{1} + 1 &   a_{1}  &   &  
  \\
  k_{1}/a_{1} &  1 &  &   \\[0.2em] 
 &     &    
 1 &    - k_{1 }  / a_{1}    \\[0.2em] 
 & & - a_{1}  &  k _{1}  + 1   
\end{array}\right)$}  . $$
Now multiply by $ \mathrm{diag}(-1,-a_{1},a_{1},1) $ on the left, $ \mathrm{diag}(1,-a_{1}^{-1},-a_{1}^{-1},1) $ on the right and use the substitution $ k = -k_{1} - 1 $. If $ a_{1} = 0 $ in the $ \varepsilon_{3}$-cell, then one gets $ v_{1}, 1 , v_{1}\eta_{0} $, $ v_{1}\tilde{\eta}_{k}$ and the latter two can be replaced with $ \tilde{\eta}_{0} $, $ \tilde{\eta}_{k'} $ where $ k' = -(k+1) $.   \\

\noindent (b) We have $ w = \rho _{2}  $ and the four cells are  
\begin{alignat*}{4}  \mathcal{Y}_{\varepsilon _{0}}E/ E &  =  
\scalebox{0.85}{$\Set*{
\left(\begin{array}{cccc}
  &     &   - u   & 1 \\ 
  &   & 1  & \\[0.2em] 
 &  \varpi  &  v \varpi   &  \\[0.2em] 
\varpi     &   u \varpi &   &    v \varpi   
\end{array}\right)}$},& 
 \mathcal{Y}_{\varepsilon _{2} } E /  E   &=
 \scalebox{0.85}{$\Set*{
\left(\begin{array}{cccc}
\varpi &   u \varpi & y - au  & a + v \varpi  \\[0.2em] 
 &     & - u &   1  \\
  &    &  1   & \\[0.2em] 
  &  - \varpi    &  - a - v \varpi  &   
\end{array}\right)}$},
\\
\mathcal{Y}_{\varepsilon _{1}} / U_{2} ^{\circ }  & = \scalebox{0.85}
{$\Set*{\left(\begin{array}{cccc}
 &   &    1 - au  &  a  \\
 & & - u &   1 \\[0.2em] 
\varpi &    u  \varpi  &   &   v \varpi \\[0.2em] 
 -a \varpi   & (1-au) \varpi &  v  \varpi     &  -a   v \varpi    
\end{array}\right)},$}  
&
\quad 
\mathcal{Y}_{\varepsilon _{3}} E  / E   &=  \scalebox{0.85}
{$\Set*{\left(\begin{array}{cccc}
- a  \varpi   &  ( 1 - au) \varpi    & a_{1}  + v  \varpi  -  u z  &  z - av \varpi  \\
 & &   1 - au  &  a   \\[0.2em] 
 &    &    -u  &  1    \\[0.2em] 
- \varpi  &  -  u  \varpi  &   a_{1} u   & -a_{1} - v \varpi  
\end{array}\right)}$}   
\end{alignat*} 
where $ a, a_{1} , y , u , v  \in [\kay] $ and $ z  = y + a a_{1}  $ in the $ \varepsilon_{3}$-cell.    Using  analogous arguments on these cells, one deduces that the $U_{2}$-orbits on 
\begin{itemize}
    \item $ \Yvar{0}E/ E $  are  represented by $  \varpi^{(1,1,1)}v_{1}   $, $  \varpi^{(1,0,1)} \tilde{\eta}_{0} $,  $ \varpi^{(1,1,1)} \tilde{\eta}_{0} $, 
    \item $ \Yvar{1}E/E $ are  represented by   $ \varpi^{(1,1,1)} $, $  \varpi^{(1,1,1)}\eta_{0} $, $ \varpi^{(1,1,0)}\eta_{0}  $ when $ a $ equals zero\footnote{these are obtained by applying $ v_{1} $ to the representatives of the $ \varepsilon_{1} $-cell} and $\varpi^{(1,0,0)}\tilde{\eta}_{k} $ for $ k \in [\kay] $ when $ a $ is non-zero,  
    \item   $ \Yvar{2}E/E $ are  represented by   $ \varpi^{(1,1,1)} $, $  \varpi^{(1,1,1)}\eta_{0} $, $ \varpi^{(1,1,0)}\eta_{0}  $ when $ a $ equals zero and $ \eta_{1} $ when $ a \neq 0 $ 
    \item      $ \Yvar{3}E/E $ are  represented by   $ \varpi^{(1,1,1)}  v_{1}   $, $  \varpi^{(1,0,1)}\tilde{\eta}_{0} $, $ \varpi^{(1,1,1)}\tilde{\eta}_{0}  $, $ \tilde{\eta}_{k} $ for $ k \in [\kay] $ when $ a_{1} $ equals zero and  $ \eta_{1} $ when $ a_{1} \neq 0 $. \\   
\end{itemize}   

\noindent (c) In this case, $ w = v_{0}  \rho^{2}_{2}   $ and the four  cells are    
\begin{alignat*}{4}  \mathcal{Y}_{\varepsilon _{0}}E/ E &  =  
\scalebox{0.85}{$\Set*{
\left(\begin{array}{cccc}
  &     &   1   &  \\ 
  &   \varpi  &  v  \varpi     & \\[0.2em] 
\varpi ^{ 2}  &   u \varpi^{2}    &   x \varpi   &  v \varpi^{2}  \\[0.2em] 
   & &   u  \varpi   &    - \varpi   
\end{array}\right)}$},& 
 \mathcal{Y}_{\varepsilon _{2} } E /  E   &=
 \scalebox{0.85}{$\Set*{
\left(\begin{array}{cccc}
 &   y  \varpi  & a + (u + vy)\varpi &  - \varpi   \\[0.2em] 
 &     &  1 &     \\
  &   \varpi  &   v \varpi    & \\[0.2em] 
 - \varpi^{2}  &  - ( a + u \varpi )  & - (x + av )  \varpi    &  - v \varpi ^{2}   
\end{array}\right)}$},
\\
\mathcal{Y}_{\varepsilon _{1}} / U_{2} ^{\circ }  & = \scalebox{0.85}
{$\Set*{\left(\begin{array}{cccc}
 &   \varpi    &   a + v \varpi  &   \\
 & & 1  &   \\[0.2em] 
 &     &   u \varpi  &   - \varpi \\[0.2em] 
\varpi^{2}   & u  \varpi ^{2 } &  ( x -  a u )   \varpi     &  ( a +  v   \varpi ) \varpi    
\end{array}\right)},$}  
&
\quad 
\mathcal{Y}_{\varepsilon _{3}} E  / E   &=  \scalebox{0.85}
{$\Set*{\left(\begin{array}{cccc}
 \varpi^{2}   &    (  a _{1} + u \varpi   )  \varpi    & z  &  ( a + v \varpi ) \varpi  \\
 & \varpi  & a  +  v  \varpi  &   \\[0.2em] 
 &    &   1   &    \\[0.2em] 
  &  &  -  a_{1} - u \varpi  &  \varpi  
\end{array}\right)}$}   
\end{alignat*}  
where $ a, a_{1}, x , y, u, v \in [\kay] $ and $ z $ denotes $  y + aa_{1} + ( x - au + a_{1} v ) \varpi  $  in the $ \varepsilon_{3} $-cell.   From  these,  one   deduces that the orbits of $U_{2} $ on 
\begin{itemize} 
\item $ \Yvar{0}E/E $ are represented by $ \varpi^{(2,2,1)   } $, $   \varpi^{(2,2,1)} \eta_{0} $,  
\item    $ \Yvar{1}E/E $ are represented by $ \varpi^{(2,1,2)   } v_{1}  $, $   \varpi^{(2,1,2)} \tilde{\eta}_{0} $ when $ a  = 0 $ and   $ \varpi^{(1,1,0)} $ when $ a \neq 0 $,
\item      $ \Yvar{2}E/E $ are represented by $ \varpi^{(2,1,2)   } v_{1}  $, $   \varpi^{(2,1,2)} \tilde{\eta}_{0} $ when $ a  = 0 $ and   $ \varpi^{(1,1,0)} $ when $ a \neq 0 $, 
\item    $ \Yvar{3}E/E $ are represented by $ \varpi^{(2,2,1)   } $, $   \varpi^{(2,2,1)} \eta_{0} $ when both $a,  a_{1} $ are $ 0 $ and  $ \eta_{1} $ when at least  one of $ a, a_{1} $ is non-zero.  \qedhere     
\end{itemize}  
\end{proof}

\begin{remark} The result above implies that (the reductions of) $ 1, v_{1} $ and $ \tilde{\eta}_{k}$ for $ k \in [\kay] $ form a complete system 
 of representatives for  $ \mathbf{H}_{2}(\kay) \backslash \mathbf{H}_{2}'(\kay) / \mathbf{H}_{2}(\kay) $.     
\end{remark} 

\begin{corollary}   \label{Eeasy} $ \mathscr{R}_{E}(\varpi^{(2,1,2)}) = \left \{   \varpi^{(2,2,1)} v_{1}  ,  \, \varpi^{(2,1,2)},  \,     \varpi^{(2,0,1)} \tilde{\eta}_{0}  , \,  \varpi^{(2,1,2)} \eta_{0} , \,   \varpi^{(1,0,1)} \eta_{1}         \right \} $ 
\end{corollary} 
\begin{proof} 
First note that  $ U ' \varpi^{(2,1,2)} = U' v_{0} v_{1}\rho_{2}^{2}   $.  Since $  v_{1} $ normalizes $ E  $ and $ \rho_{2}^{2} \in H' $ is central, $  U_{2} '  v_{0} v_{1} \rho_{2}^{2} E = U_{2} ' v_{0} \rho_{2}^{2} E v_{1}    $.  So the result follows by  Proposition \ref{R2scrhard} (c).  
\end{proof} 
Now we address the lifts of these cosets to $ H' $. Let $ S_{1}^{\pm} $ be as in    \S  \ref{projectionprelim} 
\begin{lemma} \label{Htau2easylifts}  Suppose $ \lambda $ is in $ \Lambda^{+}  $.  Then for any $ \chi \in S_{1}^{\pm} $, 
$ U \varpi^{\lambda} \chi H_{\tau_{2}}' \in \left \{ U \varpi^{\lambda} H_{\tau_{2}}' , U \varpi^{s_{0}(\lambda)}H_{\tau_{2}}'  \right \} $ and $ U \varpi^{r_{1}(\lambda)} \theta_{1}\chi H_{\tau_{2}}' $ $\in \left \{ U\varpi^{r_{1}(\lambda)}\theta_{1} H_{\tau_{2}}' , U \varpi^{s_{0}r_{1}(\lambda)} \theta_{1} H_{\tau_{2}}' \right \} $. 
\end{lemma} 
\begin{proof}  This first part is  proved in the same manner as Lemma \ref{liftsforsigma0}. Since $ \theta_{1} = \sigma_{1} = w_{2}  $ normalizes $ U $,  commutes with $ \chi $, $ w_{\alpha_{0}} $   and $ w_{2} \varpi^{\lambda} =   \varpi^{r_{1}(\lambda)} w_{2}   $,  the second claim also follows easily.     
\end{proof} 
\begin{lemma}  \label{liftsforU'Htau2'}  Let $ \lambda \in \Lambda^{\alpha_{0} > 0 } $ and $ \chi =  ( \chi_{1}, 1 ) $ where $ \chi_{1} \in S^{\pm}_{1} $.    
\begin{enumerate} [label =  \normalfont (\alph*)]  

\item  Suppose $ (\beta_{1} + \beta_{2})(\lambda) \geq 0 $. Then  $$ 
U '  \varpi^{\lambda}\theta_{2} \chi  H_{\tau_{2}}' =  \begin{cases}  U' \varpi^{\lambda} H_{\tau_{2}}' & \text{ if }  \,  \chi_{1} = 1 \,  \text{ or if } \,  \, \chi_{1} \in S^{-}_{1} \setminus \left \{ 1 \right \}\!, \, \beta_{0}(\lambda) \geq 0  \\ 
U  ' \varpi^{s_{0}(\lambda)   }     
H_{\tau_{2}}'    &  \text{ if }  \chi_{1} \in S^{-}_{1}  \setminus \left \{ 1 \right \}\!, \,   \beta_{0}(\lambda) < 0 \,  \text{ or if }  \, \,  \chi_{1} \in S_{1}^{+} 
\end{cases} $$ 
\item Suppose  $ \beta_{1} ( \lambda) \leq 
0 $. Then $$ U ' \varpi^{\lambda} \tilde{\theta}_{0}  \chi   H_{\tau_{2}}' = \begin{cases} U '  \varpi^{r_{1}(\lambda)}  H_{\tau_{2}}  '   & \text { if }  \, \chi_{1} = 1 \, \text{ or if } \, \, \chi_{1} \in S_{1}^{-}\setminus \left \{1 \right \}\!, \,   \beta_{2}(\lambda) \geq 0, 
\\  U ' \varpi^{s_{0}r_{1}(\lambda) } H_{\tau_{2}}' & \text{ if }  \,   \chi_{1} \in S_{1}^{-}\setminus \left \{ 1 \right \}\!,  \,   \beta_{2}(\lambda) < 0   \text{ or  if }  \chi_{1} \in S^{+}_{1}   \end{cases}  $$ 
\item  Suppose $  \beta_{1}( \lambda) = 0 $.  Then for any $ k \in [\kay] $, $$  U'  \varpi^{\lambda} \tilde{\theta}_{k} \chi H_{\tau_{2}}'  =     \begin{cases}     U' \varpi^{\lambda}  H_{\tau_{2}}'  & \text{ if } \,  \chi_{1} = 1 \text{ or if } \, \, \chi_{1} \in S_{1}^{-} \setminus \left \{ 1 \right \}\!, \, \beta_{0}(\lambda) \geq 0    \\  U'  \varpi^{s_{0}(\lambda)} H_{\tau_{2}}'  & \text{ if }  \chi_{1} \in S^{-}_{1}  \setminus  \! \left \{ 1 \right \}, \,  \beta_{0}(\lambda) < 0 \text{ or if } \, \, \chi_{1} \in S^{+}  \end{cases} .     $$  
 \item    Suppose $ (\beta_{1} + \beta_{2})(\lambda) \geq 1 $. Then  $$ 
U '  \varpi^{\lambda}\theta_{3} \chi  H_{\tau_{2}}' =  \begin{cases}  U' \varpi^{\lambda 
} H_{\tau_{2}}' & \text{ if }  \,  \chi_{1} = 1 \,  \text{ or if } \,  \, \chi_{1} \in S^{-}_{1} \setminus \left \{ 1 \right \}\!, \, \beta_{0}(\lambda) \geq  0 \\ 
  U  ' \varpi^{s_{0}(\lambda 
  )   }     
 H_{\tau_{2}}'    &  \text{ if }  \chi_{1} \in S^{-}_{1}  \setminus \left \{ 1 \right \}\!, \,   \beta_{0}(\lambda) <  0 \,  \text{ or if }  \, \,  \chi_{1} \in S_{1}^{+} 
 \end{cases} $$ 
\end{enumerate}  

\end{lemma}    
\begin{proof} In each of the parts (a), (c) and (d), the assumption made implies  the equality $ U' \varpi^{\lambda} \gamma = U' \varpi^{\lambda} $ where $ \gamma $ denotes $ \sigma_{2}, \sigma^{k} , \sigma_{3} $. In part (b), the assumption implies that $  U' \varpi^{\lambda} \sigma^{0} = U'\varpi^{r_{1}(\lambda)} $. Using this and the fact that the matrix $ \varkappa^{-}(-x) $ in (\ref{varkappamatrices}) lies in $ H_{\tau_{2}}' $ for $ x \in \Oscr_{F} $, one easily deduces each of the claims. 
\end{proof}

\begin{proof}[Proof of  Proposition \ref{R2scrhard}]  For $ \mathscr{R}_{2}(1) $ (resp., $ \mathscr{R}_{2}(\varpi^{(4,2,2,3)}) $), the result is obtained by applying Lemma \ref{pVbijection} to  Proposition \ref{Ehard} (resp., Corollary \ref{Eeasy}). The other two cases are handled by studying the fibers of the projection $$ \pr_{\mu} : \mathscr{R}_{2}(\varpi^{\mu}) \to  \mathscr{R}_{E}(\varpi^{\pr_{2}(\mu)}) $$ 
using Corollary \ref{fibersofpV}. That is, if $ \mu \in \left \{ (3,2,1,2), (4,3,1,2) \right \} $ and  $ \varpi^{(a,c,d)} \gamma $ lies in $ \mathscr{R}_{E}(\varpi^{\pr_{2}(\mu)}) $ for some $  \gamma \in \{ 1, v_{1}, \eta_{0},   \varpi^{-(1,1,1)} \eta_{1},  \tilde{  \eta}_{k} \, | \,    k \in  [\kay]^{\circ}  \} $, the fiber $ \pr_{\mu }$   above $ \varpi^{(a,c,d)} \gamma $ consists of all elements of the form $ \varpi^{\lambda}\hat{\gamma} \chi  $  where  $ \hat{\gamma} \in \{1, \theta_{1}, \theta_{2},  \theta_{3},  \tilde{\theta}_{k} \, | \, k \in [\kay]^{\circ}   \} $ satisfies $ \pr_{2}(\hat{\gamma}) = \gamma $, the cocharacter  $ \lambda = (a,b,c,d) \in \Lambda^{\alpha_{0} > 0} $ is such that $  b  = \pr_{1}'(\varpi^{\mu}) $ and $ \chi \in S_{\alpha_{0}(\mu)}^{\pm}   $ is arbitrary.  Note that $ \alpha_{0}(\mu)   = 1 $ for both $ \mu $.    \\

\noindent $ \bullet $ $ \mu = (3,2,1,2) $ \\[0.3em]   The projection is  $ \mathscr{R}_{E}(\varpi^{(3,1,2)} ) = \mathscr{R}_{E}(\varpi^{(3,2,2)})  = \varpi^{(2,1,1)}  \mathscr{R}_{E}(\varpi^{(1,1,1)}) $ which by Proposition \ref{Ehard}(b),  equals    $$ \left \{ \varpi^{(3,2,2)},\,  \varpi^{(3,2,2)}v_{1},\, \varpi^{(3,1,2)}\eta_{0},\, \varpi^{(3,2,1)}\eta_{0},\,  \varpi^{(3,2,2)}\eta_{0},\,
\varpi^{(2,1,1) } \eta_{1}, \,  \varpi^{(3, 2, 2)}\tilde{\eta}_{0},\,  \varpi^{(3,1,2)}\tilde{\eta}_{0},\,  \varpi^{(3,1,1)}\tilde{\eta}_{k}  \, | \,     k \in [\kay]^{\circ} \right \}  $$ 
By Lemma  \ref{Htau2easylifts} and Proposition \ref{cartanhtaui}, the fibers above $ \varpi^{(3,2,2)} $ and $ \varpi^{(3,2,2)}v_{1} $ are singletons. Since $ \varpi^{(3,2,1,2)}$, $ 
 \varpi^{(3,2,2,1)}\sigma_{1} $ clearly belong to $ \mathscr{R}(\varpi^{\mu})$, we choose these as the representative elements above the corresponding fibers.    For the remaining elements of $ \mathscr{R}_{E}(\varpi^{(3,2,2)})$, one deduces from Lemma \ref{liftsforU'Htau2'} that $ \chi $ must be either identity or in $ S_{1}^{+}  $ in each case (but not both), and the corresponding unique representative in the fiber is  easily obtained.    \\
 
\noindent $ \bullet $ $ \mu = (4,3,1,2) $ \\[0.3em]  
The projection $   \varpi^{(2,1,1)} \cdot   \mathscr{R}_{E}(\varpi^{(2,2,1)}) = \left \{ \varpi^{(4,3,2)},\, 
\varpi^{(4,2,3)}v_{1},\,  
\varpi^{(4,3,2)}\eta_{0},\, \varpi^{(4,2,3)}\tilde{\eta}_{0},\,
\varpi^{(3,2,1)}\eta_{1}
\right \} $. Again, we decide the lifts for $ \varpi^{(4,3,2)} $, $ \varpi^{(4,2,3)}v_{1} $ using Lemma \ref{Htau2easylifts} and use Lemma \ref{liftsforU'Htau2'} to show that $ \chi \in S^{-}_{1} $ is the only possible for choice  for each of the remaining representatives in $ \mathscr{R}_{E}(\varpi^{(4,3,2)} )  $.  
\end{proof} 

%% file: Convolutions.tex
 \section{Convolutions}

\label{convolutionsection}

Recall that $ X  $ denotes the topological  vector   space  $ \mathrm{Mat}_{2 \times 1}(F) $ and $ \mathcal{S} =  \mathcal{S}_{\mathcal{O}, X} $ denotes the set of all locally constant compactly supported $ \mathcal{O}$-valued  functions on $ X$.  The space $ X $  admits a continuous  right action of $ H_{1} =  \GL_{2}(F) $ via left matrix multiplication by inverse and we extend this action to $ H $ via  $ \pr_{1} : H \to H_{1}  $. These induce left actions of $ H_{1} $ and $ H $ on $ \mathcal{S} $.  If $ \mathfrak{p}  $ is an ideal of $ \mathcal{O} $ and $ \xi_{1}, \xi_{2} \in \mathcal{S}  $, we write $ \xi_{1} \equiv \xi_{2} \pmod { \mathfrak{p} }  $ if $ \xi_{1}(x) - \xi_{2}(x) \in \mathfrak{p} $ for all $ x \in X $.    If $ V $ is a compact open subgroup of $ H_{1} $ or $ H $, we let $ \mathcal{S}(V) $ denote the space of $ V $-invariants of $ \mathcal{S} $. 
If $ m, n $ are integers, we let 
\begin{align*} X_{m,n } &  =  \left \{  \begin{psmallmatrix} x \\ y \end{psmallmatrix} \, | \, x \in \varpi^{m} \Oscr_{F} ,  \,  y \in  \varpi^{n}  \Oscr_{F} \right \}
\end{align*} 
which are  compact open subset of $ X $. We denote  $$ \phi_{(m,n)} : = \ch(X_{m,n} ) , \quad \quad  \bar{\phi}_{(m,n)} = \phi_{(-m,-n)}. $$   We let $ z_{0} $ denote the inverse of the central element $ \rho_{1}^{2} =  \mathrm{diag}(\varpi, \varpi) \in   H_{1} $. 
For $ n $ a positive integer, we let $ U_{\varpi^{n}} $ denote the subgroup of all elements in $ U $  whose reduction modulo $ \varpi^{n} $ is identity in $ \mathbf{H}(\kay/\varpi^{n} ) $. For $ \lambda \in \Lambda $, we define the \emph{depth} of $ \lambda $ to be $ \mathrm{dep}(\lambda) : = \max \{   \pm \alpha_{0}(\lambda) , \pm \beta_{0}(\lambda) ,  \pm \beta_{2}(\lambda) \}  $. Then for $ \lambda $ of depth at most $ n $,   $  \varpi^{-\lambda}  U_{\varpi^{n}}  \varpi^{\lambda} 
 \subset U $.     

\begin{notation}  \label{notationh}    
We will often write $ h  = (h_{1}, h_{2}, h_{3}) \in  \GL_{2}(F) \times_{F^{\times}  }  \GL_{2}(F) \times_{F^{\times}}  \GL_{2}(F)  \subset \mathrm{GSp}_{6}(F) $ as  $$ h  =   \left (  \begin{smallmatrix} a  & & &  b  \\  & a_{1} & & & b_{1} \\
 & & a_{2} & & & b_{2} \\
  c & &  & d  \\ &  c_{1} & &  &  d_{1} \\
  & &  c_{2}  & && d_{2} \end{smallmatrix}   \right  )   \quad  \text { or }  \quad     h = \left (  \left ( \begin{matrix} a  & b \\ c & d  \end{matrix}  \right )  ,   \left  (   \begin{matrix} a_{1} & b_{1} \\ c_{1} & d_{1}    \end{matrix}   
     \right )  ,     \left  (   \begin{matrix} a_{2} & b_{2} \\ c_{2} & d_{2}    \end{matrix}   
     \right )   \right )  .   $$  
If we wish to refer to another element in $ H $, we will write $ h '$ and all its   entries will be 
 adorned with a prime.  Given $ a, b \in \ZZ $, we write $ \varpi^{(a,b)} $ to denote $ \mathrm{diag}(\varpi^{b} , \varpi^{a-b }  ) \in  \GL_{2}(F)  $.  
\end{notation}

\subsection{Action of $\GL_{2}$}  It will be useful to record a few general results on convolution of Hecke operators of $ \GL_{2}(F) $ with $ \phi $. Let $  \mathcal{T} _{u,v} $ denote the double coset Hecke operator $ [U_{1} \, \mathrm{diag}(\varpi^{u}, \varpi^{v} ) \,   
U_{1}] $.   
It acts on $ \mathcal{S}(U_{1})  $ and in particular, on $ \phi \in  
\mathcal{S}(U_{1}) $. It is clear that $  \mathcal{T}_{u,v}(\phi) = \mathcal{T}_{v,u}(\phi)$ and $ \mathcal{T}_{u,u}(\phi) = \phi_{(u,u)} $.   
\begin{lemma}    \label{Tablemma} $ \mathcal{T}_{u,v}(\phi) =  \phi_{(v,v)} + q^{u-v}\phi_{(u,u)} + \sum_{ i=1}^{u-v-1} (q^{i} - q^{i-1}) \phi_{(i+v, i+v)} $ when $  u > v $.  Here the sum in the expression  is zero if $ u - v  = 1 $. 
\end{lemma} 
\begin{proof} Let $ \xi = \mathcal{T}_{u,v}(\phi) =  \sum_{ \gamma } \gamma \cdot \phi $ where $ \gamma $ runs over representatives of of $ U_{1} \mathrm{diag}(\varpi^{u} , \varpi^{v} )  U_{1} / U _{1} $.  Translating everything by $ (z_{0})^{v} $, it suffices to establish our formula  when  $ v = 0 $. Then $ u \geq 1 $ and \[ U_{1}  
\begin{psmallmatrix} \varpi^{u} \\  & 1 \end{psmallmatrix}U_{1} / U_{1} =  \bigsqcup_{\kappa \in [\kay_{u}] } \begin{psmallmatrix} \varpi^{u} & \kappa \\ & 1 \end{psmallmatrix}  U_{1} \sqcup \bigsqcup _{ \kappa \in [\kay_{u-1}] }  \begin{psmallmatrix} 1 \\ \varpi \kappa   & \varpi^{u}   \end{psmallmatrix} U_{1}  . \] 
From the  decomposition above, we see that  $ \xi(\vec{v}) = q^{i} $ whenever  $  v  \in ( X_{i,i} \setminus  X_{i,i+1} )  \cup ( X_{i,i+1} \setminus X_{i+1,i+1}  )  = X_{i,i} \setminus X_{i+1,i+1} $ for all $ i \in \left \{ 0, 1 , \ldots, u  - 1 \right \} $  
and that  $ \xi(\vec{x} ) = q^{u} + q^{u-1} $ when $  \vec{x}  \in X_{u,u} $.  
\end{proof} 

Let $ \mathcal{T}_{u,v,*} : = \mathcal{T}_{-u,-v} =  [U_{1} \mathrm{diag}(\varpi^{u}, \varpi^{v})U_{1}]_{*} $  denote the dual (or transpose)  of $ \mathcal{T}_{u,v} $.   
\begin{corollary} \label{Tab*}      If $ u \neq v  $, then $ \mathcal{T}_{u,v,*} (\phi) \equiv (z_{0}^{u} + z_{v}^{v}) \cdot \phi  \pmod{q - 1} $ and $ \mathcal{T}_{u,u,*}(\phi) = z_{0}^{u}  \cdot   \phi  $.     
\end{corollary} 
\begin{proof}  This is clear by Lemma \ref{Tablemma}. 
\end{proof} 
Let $I_{1}^{+} $ denote the Iwahori subgroup of $ U_{1} = \GL_{2}(\Oscr_{F}) $ of upper triangular matrices and $ I_{1}^{-} $ the Iwahori subgroup of lower triangular matrices. For $ u,  v  $ integers, let $ \mathcal{I}_{u,v}^{\pm} $ denote the double coset Hecke operator $  [I_{1}^{\pm} \mathrm{diag}(\varpi^{u},\varpi^{v}) U_{1} ] $.

\begin{lemma}   \label{Iuvlemma}     Let $ u, v $ be integers. Then 
$$  \mathcal{I}_{u,v}^{+} (\phi)  = \begin{cases}  \displaystyle{q^{u-v} \phi_{(u,u)}  +  \sum_{i=0}^{u-v-1} 
  q^{i}\rho_{1}^{2(i+v)} \cdot ( \phi - \phi_{(0,1)} ) } 
  & \text{ if } \, \,  u \geq  v  \\[0.3em]   

q^{v-u-1} \phi_{(v-1, v)}  +  \displaystyle{ \sum_{i=0}^{v-u-2} 
  q^{i}  \rho_{1}^{2(i+u)} \cdot ( \phi_{(0,1)} - \phi_{(1,1)}) } 
  & \text { if } \, \,    u <  v
  \end{cases}  
  $$
and 
$$  \mathcal{I}_{u,v}^{-} (\phi)  = \begin{cases} \displaystyle{ q^{v-u} \phi_{(v,v)}  +  \sum_{i=0}^{v-u-1} 
  q^{i}  \rho_{1}^{2(i+u)} \cdot ( \phi - \phi_{(1,0)})} 
  & \text{ if }  u \leq  v  \\[0.3em]   

\displaystyle{ q^{u-v-1} \phi_{(u,u-1)}  +  \sum_{i=0}^{u-v-2} q^{i}  \rho_{1}^{2(i+v) }  \cdot  (\phi_{(1,0)} - \phi_{(1,1)})} 
  & \text { if }  u >   v  
  \end{cases}  
  $$
  where $ \rho _{1}^{2}  =   z_{0}^{-1}   =   \begin{psmallmatrix} \varpi &  \\ & \varpi \end{psmallmatrix} $.

\end{lemma}

\begin{proof}  The first equality is established in the same manner as Lemma \ref{Tablemma} using the  decompositions \[ I^{+}_{1}   \begin{psmallmatrix} \varpi^{u} \\  & 1 \end{psmallmatrix}   U_{1} / U_{1}  =    \bigsqcup_{\kappa \in [\kay]_{u}} \begin{psmallmatrix} \varpi^{u}   &  \kappa \\[0.2em]   &   1 \end{psmallmatrix},       \quad \quad  \quad   I_{1}^{+} \begin{psmallmatrix} 1  \\    &  \varpi^{v}  \end{psmallmatrix}   U_{1} / U_{1}  =    \bigsqcup_{\kappa \in [\kay]_{v-1}}  \begin{psmallmatrix} 1 \\[0.2em]  \kappa  \varpi   &  \varpi^{v}  \end{psmallmatrix}  \] 
which hold for integers $ u \geq 0 $, $ v  \geq  1  $. The second is obtained from the first by notation that $ I_{1}^{-} $, $ I_{1}^{+} $ are conjugates of each  other by the reflection matrix $ \begin{psmallmatrix}  &  1  \\ 1  & \  \end{psmallmatrix} $.  
\end{proof} 

\subsection{Convolutions with restrictions of $\mathfrak{h}_{0}$}  This subsection is devoted to computing $ \mathfrak{h}_{\varrho_{i},*}(\phi) $ for $ i = 0 ,1 ,2  $.   Recall that $$ 
\varrho_{0} =  
\scalebox{1.1}{$
\left(\begin{smallmatrix}
1   & & & &      \\   
&     1   & &     \\[0.1em] 
& &  1   &  \\   
& & &   1 \\
& & & &   1      \\ 
& & & & & 1  
\end{smallmatrix}\right)$}
, \quad    \quad  \varrho_{1} =  \scalebox{1.1}{$\left(\begin{smallmatrix}
\varpi  & & & &   \\ 
& \varpi &   & & &  1 \\
&  & \varpi & & 1  \\%[0.2em]  
& & &  1 \\
& & & &  1  \\ 
& & & &  &  1 
\end{smallmatrix}\right)$} ,    \quad  \quad 
\varrho_{1} =  \scalebox{1.1}{$\left(\begin{smallmatrix}
\varpi  & & & &   \\ 
& \varpi^{2} &   & & &  1 \\
&  & \varpi^{2} & & 1  \\%[0.2em]  
& & &  \varpi \\
& & & &  1  \\ 
& & & &  &  1 
\end{smallmatrix}\right).$}
$$

\begin{proposition}    Modulo $ q - 1 $,  
\begin{enumerate}  [label = \normalfont(\alph*), itemsep= 0.3em, after = \vspace{0.1em} , before = \vspace{0.1em} ] 

\item $ \mathfrak{a}_{\varrho_{0},*}(\phi) \equiv    (  6  +  16  z_{0}  +  6z_{0}^{2} ) \phi   $ 
\item $ \mathfrak{b}_{\varrho_{0},*}(\phi) \equiv  4 ( 1 + z_{0} ^{3}  +  6 z_{0} + 6 z_{0}^{2} )  \phi  $
\item  $ \mathfrak{c}_{\varrho_{0},*}(\phi)  \equiv   (   ( z_{0} + 1 )^{4} - 2 z_{0}^{2}  ) \phi   $ 
\end{enumerate}  
and $ \mathfrak{h}_{\varrho_{0},*}( \phi ) \equiv   0  $ 
\end{proposition}
\begin{proof} For $ \lambda = (a,b,c,d) \in \Lambda $, the map \begin{align*} U \varpi^{\lambda} U / U  &  \longrightarrow  ( U_{1} \varpi^{(a,b)} U_{1}/ U_{1} )  \times (U_{1} \varpi^{(a,c)}U_{1}/U_{1}) \times (U_{1} \varpi^{(a,d)} U_{1} / U_{1})  \\ 
(h_{1}, h_{2}, h_{3} ) U & \longmapsto (h_{1}U_{1} , h_{2} U_{1} , h_{3} U_{1}) 
\end{align*} 
is a bijection.  Corollary \ref{Tab*} implies that  $$ [U \varpi^{\lambda} U / U  ] ( \phi ) =    | U_{1} \varpi^{(a,c)} U_{1}/U_{1} | \cdot | U_{1} \varpi^{(a,d)}U_{1}/ U_{1} | \cdot  (z_{0}^{b} + z_{0}^{a-b})  \phi  . $$
Now $ |  U_{1} \varpi^{(u,v)} U_{1}/U_{1} | \equiv  1 \text{ or } 2  \pmod{q-1} $ depending on whether $ 2v - u = 0 $ or not. So parts (a)-(c) are all easily obtained.   
Now recall from (\ref{tfhvarrho0}) that  $$ \mathfrak{h}_{\varrho_{0},*}(\phi) = 
(1 + \rho^{8}) (U) - (1+\rho^{6}) ( U 
\varpi^{(1,1,1,1)}  U ) +(1+ 2 \rho^{2} + \rho^{4}) \mathfrak{a}_{\varrho_{0}} -(1+\rho^{2}) \mathfrak{b}_{\varrho_{0}} +  \mathfrak{c}_{\varrho_{0}} $$  
Using our formulas, we find that 
\begin{align*} \mathfrak{h}_{\varrho_{0},*}(\phi) &  \equiv     \Big ( (  1+  z_{0}^{4} )   - 4 ( 1 + z_{0} ^{3}   ) ( 1 + z_{0} ) +  ( 1 + z_{0})^{2} ( 6 + 16 z_{0} + 6z_{0}^{2}) - 4( 1 + z_{0} ) (  1 + z_{0}^{3} + 6z_{0} + 6z_{0}^{2} )  \, \\
  & \quad \, \, \, + (z_{0}+1)^{4} - 2z_{0}^{2} \Big ) \phi  
\end{align*}   
and one  verifies that  the  polynomial  expression in $ z_{0} $  above    is  identically  zero.  
\end{proof}

\begin{notation}  \label{strucvarrho1not}  

 Let $ \mathbf{P} : = \GL_{2} \times_{\GG_{m} } \GL_{2} $  and   define embeddings   
\begin{alignat*}{4} \imath _{\varrho_{1}}   : \mathbf{P}   & \hookrightarrow \mathbf{H} , &  \quad \quad     \jmath_{\varrho_{1}} :  \mathbf{P}  & \hookrightarrow  \mathbf{H}  \\ 
(\gamma _{1}, \gamma _{2}) & \mapsto ( \partial \gamma _{1} \partial^{-1} ,   \gamma _{2} ,   \ess \gamma_{2} \ess 
)   & \quad \quad \quad   \quad   (\gamma_{1}, \gamma_{2} )    
& \mapsto     ( \partial  \gamma _{1} \partial^{-1},    \ess \gamma _{2} \ess ,  \gamma_{2}   ) 
\end{alignat*} 
where  $ \ess =    \begin{psmallmatrix} & 1  \\ 1 \end{psmallmatrix}  $ and $  \partial  :  = \ess \rho_{1} \ess  =  \begin{psmallmatrix} & \varpi \\ 1 \end{psmallmatrix} $. We let $ \mathscr{X}_{\varrho_{1}} $ denote the common image $ \imath_{\varrho_{1}}(P^{\circ} ) $, $ \jmath_{\varrho_{1}}(P^{\circ}  )  $.  We denote by $ M_{\varrho_{1}} $ (resp., $ M_{\varrho_{1}}')$ denote the subgroup of $U_{\varpi} $ in which the first and second (resp., first and third) components are identity.  We also let $$ \mathbf{\pr}_{2,3} : \Hb \to \mathbf{P} \quad \quad  (h_{1}, h_{2}, h_{3}) \mapsto (h_{2}, h_{3} )  .  $$ 
Finally, we let 
 $ \mathscr{Y}_{\varrho_{1}} $, $ L_{\varrho_{1}}, L'_{\varrho_{1}}, P_{\varpi}^{\circ}  $ denote respectively  the projections of  $ \mathscr{X}_{\varrho_{1}} $, $ M_{\varrho_{1}} $, $ M'_{\varrho_{1}}  $,  $ U_{\varpi} $  under $ \pr_{2,3} $.     
\end{notation}

\begin{lemma}  \label{varrho1struclemma}   $ H_{\varrho_{1}} =   \mathscr{X}_{\varrho_{1}}  M_{\varrho_{1}}    =  \mathscr{X}_{\varrho_{1}}  M _ { \varrho_{1}}  '  $.   
\end{lemma}

\begin{proof}  Writing $ h \in H $ as in \ref{notationh}, we see that  $$ \varrho_{1} ^{-1} h  \varrho_{1}  = 
\begin{pmatrix} 
a &  &  & \mfrac{b}{\varpi } &  &  \\[0.5em]  
 & a_1  & -c_2  &  & \mfrac{b_1 -  c_2 }{\varpi } & \mfrac{a_1 -  d_2 }{\varpi }\\[0.5em] 
 & -c_1  & a_2  &  & \mfrac{a_2  -  d_1 }{\varpi } & \mfrac{b_2  - c_1 }{\varpi }\\[0.5em] 
c\,\varpi  &  &  & d &  &  \\[0.5em] 
 & c_1 \,\varpi  &  &  & d_1  & c_1 \\[0.5em] 
  &  & c_2 \,\varpi  &  & c_2  & d_2 
\end{pmatrix}  $$  
From this, one immediately sees that $  h = (h_{1}, h_{2}, h_{3} ) \in H_{\varrho_{1}  } $ if and only if $ \partial^{-1}  h_{1} \partial, h_{2}, h_{3}  \in U_{1} $ and the modulo $ \varpi $ reductions of $ h_{2} $, $ \ess h_{3} \ess $ coincide. So $ H_{\varrho_{1}} \supset \mathscr{X}_{\varrho_{1}} $, $ M_{\varrho_{1}} $,  $ M_{\varrho_{1}}' $. To see that $ H_{\varrho_{1}} $ equals the stated products, we note that for any $ h = (h_{1}, h_{2}, h_{3} ) \in H_{\varrho_{1}} $, $ \iota_{\varrho_{1}} ( \partial^{-1} h_{1}^{-1}  \partial ,  h_{2} ^{-1} ) \cdot h \in  M_{\varrho_{1}}  $ and $  \jmath_{\varrho_{1}}(\partial^{-1} h_{1} ^{-1}  \partial , h_{3} ^{-1}  ) \cdot  h \in  M_{\varrho_{1}}'  $.  
\end{proof} 
\begin{proposition}  \label{tfhvarrho1action}  Modulo $ q - 1 $, 
\begin{enumerate} [label = \normalfont(\alph*), itemsep= 0.3em, after = \vspace{0.1em} , before = \vspace{0.1em} ] 
\item $ \mathfrak{a}_{\varrho_{1},*}(\phi) \equiv 2 ( 1 + 3z_{0} + z_{0}^{2})  
\phi  $,   
\item $  \mathfrak{b}_{\varrho_{1},*}(\phi)  \equiv  ( 1 + 10 z_{0}   + 10 z_{0}^{2} + 
 z_{0}^{3}  ) 
\phi     $   
\item  $ \mathfrak{c}_{\varrho_{1},*}  (\phi) \equiv 2z_{0}  ( 1   +  z_{0}  )  
\phi   $ 
\end{enumerate} 
and $ \mathfrak{h}_{\varrho_{1},*}(\phi)  \equiv 0     $.  
\end{proposition} 
\begin{proof}  
For $ \lambda = (a,b,c,d)  \in \Lambda $, let $ \xi_{\lambda} =   [ U \varpi^{\lambda} H_{\varrho_{1}}]_{*}( \phi) $.  Then  Lemma \ref{varrho1struclemma} implies that $$ \xi_{\lambda} =  \big |  
P^{\circ} \backslash P^{\circ}  \varpi^{(a,c,d)}  \pr_{2,3}(H_{\varrho_{1}}) \big  | \cdot   [ U_{1} \varpi^{(a,b)}  \partial U_{1} \partial^{-1}]_{*} (\phi )  
$$ 
Now Corollary  \ref{Tab*} implies that \begin{align*}   [ U_{1} \varpi^{(a,b)}  \partial U_{1} \partial^{-1}]_{*} (\phi )  &    =   \partial \cdot  \mathcal{T}_{b+1, a -  b, *}(\phi) 
\equiv  \begin{cases}   (z_{0}^{b+1} + z_{0}^{a-b} )  \, \phi _{(0,1)}  & \text{ if } a \neq 2b + 1 \\ 
 z_{0}^{b+1}   \cdot  \phi_{(0,1)   }  & \text{ if } a = 2b + 1 
\end{cases} 
\end{align*}
If  moreover  $ |\beta_{0}(\lambda)|,| \beta_{2}(\lambda) | \in \left \{ 0,1 \right \} $,  then $  P^{\circ}  \varpi^{(a,c,d)}   \pr_{2,3}( H_{\varrho_{1}} )   $ simplifies to $    P^{\circ}  \varpi^{(a,c,d)} \mathscr{Y}_{\varrho_{1}} $. So in this case, \begin{equation}   \label{varrho1easy}  |  P^{\circ} \backslash  P^{\circ}    \varpi^{(a,c,d)} \pr_{2,3}(H_{\varrho_{1}}) | = [ \mathscr{Y} _{\varrho_{1}} :  \mathscr{Y} 
 _{\varrho_{1}} \cap P^{\circ}_{(a,c,d)} ]  \end{equation} 
where $ P_{(a,c,d)} ^ { \circ}  :  =   \varpi^{-(a,c,d)}  P^{\circ} \varpi^{(a,c,d)}   $. Since $\mathscr{Y} _{\varrho_{1}} \simeq \GL_{2}(\Oscr_{F} ) = U_{1} $   (via the projection $ \mathbf{P} \to \Hb_{1} $, $ (\gamma_{1}, \gamma_{2} ) \mapsto \gamma_{1} $), the index on the RHS of (\ref{varrho1easy})  can be found by comparing the intersection $  \mathscr{Y}_{\varrho_{1}} \cap P^{\circ}_{(a,c,d)} $ with the Iwahori subgroups $ I_{1}^{\pm} $ in $ U_{1} $. One easily sees that that the RHS of (\ref{varrho1easy})   is congruent to $ 1 $ or $ 2 $ modulo $ q+1 $, and that the former only happens if and only if $ \beta_{0}(\lambda) = \beta_{2}(\lambda) = 0 $. This takes care of  the index calculations for all the  Hecke operators in parts (a)-(c) except for $ (U \varpi^{(2,1,2,0)} H_{\varrho_{1}  } ) $.  Here,  we  invoke \cite[Lemma 5.9.3]{CZE}. More precisely, we use that  $ \pr_{2,3}(H_{\varrho_{1}}) = 
P_{\varpi}^{\circ} \mathscr{Y}_{\varrho_{1}}  $ and the result in \emph{loc.cit.} implies that  $$  | P^{\circ}  \backslash P^{\circ} \varpi^{(2,2,0)  } P_{\varpi}^{\circ} \mathscr{Y} _{\varrho_{1}} ] =   e^{-1} \cdot  |  P ^  {\circ }  \backslash    P^{\circ} \varpi^{(2,2,0)} P^{\circ} _{\varpi} | \cdot |   (  P_{\varpi} ^{\circ}  \cap    \mathscr{Y} _{\varrho_{1}}  ) 
\backslash   \mathscr{Y} 
 _{\varrho_{1}} |    
$$
where $ e =  [  \mathscr{Y}_{\varrho_{1}} P_{\varpi}^{\circ} 
\cap P_{(2,2,0)}^{\circ} :   P_{\varpi}^{\circ} \cap  P_{(2,2,0)}^{\circ} ] $.  Now $ P_{\varpi} ^{\circ} \cap Y_{\varrho_{1}} $ is identified with $ I_{1}^{\pm} $, and $ [U_{1} : I_{1} \cap I_{1}^{-}] = q ( q+  1 ) $ and similarly $ |  P^{\circ} \backslash P ^{\circ}  \varpi^{(2,2,0)} P^{\circ}_{\varpi}   | = q^{2}  $.   Moreover   $ \mathscr{Y}_{\varrho_{1}}  P_{\varpi}^{\circ} 
\cap P_{(2,2,0)}^{\circ}  =   (\mathscr{Y}_{\varrho_{1}} \cap P_{(2,2,0)}^{\circ}  ) \cdot  ( P_{\varpi}^{\circ} \cap P_{(2,2,0)}^{\circ}  ) $, which implies that $$ e =   [ \mathscr{Y} _{\varrho_{1}} \cap P_{(2,2,0)}^{\circ}   :  \mathscr{Y} _{\varrho_{1}} \cap  P_{\varpi}^{\circ} \cap  P_{(2,2,0)}^{\circ}  ] . $$ 
from which it is not too hard to see that $ e  = q $.  
It follows that $ \xi_{(2,1,2,0)} \equiv   2  z_{0}  \cdot \phi $. Now  recall from (\ref{tfhvarrho1}) that 
    $$ \mathfrak{h}_{\varrho_{1}} = 
  - ( 1 + \rho^{6} ) (U H_{\varrho_{1}}) +  ( 1 +  2 \rho^{2} + \rho^{4} ) \mathfrak{a}_{\varrho_{1}} - ( 1 + \rho^{2} ) \mathfrak{b}_{\varrho_{1}}    +  \mathfrak{c}_{\varrho_{1}}  $$ 
  So we see that
\begin{align*}  \mathfrak{h}_{\varrho_{1}, *} ( \phi  )  &   \equiv \Big ( \! - ( 1 + z_{0}^{3})  ( 1 + z_{0} ) + (1 + z_{0})^{2}  
( 2 + 6z_{0} +  2 z_{0} ^{2} )   - ( 1 + z_{0 } ) ( 1 + 10 z_{0} + 10 z_{0} ^{2} + z_{0} ^{3} )  \\
&  \quad \, \, + 2  z_{0} ( 1 + z_{0}  ) ^ { 2 }  \Big )  
\end{align*} 
which is zero since  the  polynomial expression in $ z_{0} $  vanishes.   
\end{proof}
\begin{notation}   \label{strucvarrho2not} 
Let $ \mathbf{P} $, $ \ess $ be as in Notation \ref{strucvarrho1not},   $  \imath _{\varrho_{2}}   : \mathbf{P}    \hookrightarrow \mathbf{H}   $ be the given by $  
(\gamma _{1}, \gamma _{2}) 
 \mapsto ( 
\gamma _{1} 
,   \gamma _{2} ,   \ess \gamma_{2} \ess  ) $ and $ \mathscr{X}_{\varrho_{2}} =  \imath_{\varrho_{2}}(P^{\circ} )  $.   
Let $ \mathrm{pr}_{2,3} : \Hb \to \mathbf{P} $ be the projection as before and let $ \mathscr{Y}_{\varrho_{2}} $,  $  P_{\varpi^{2}}^{\circ}  $ denote respectively  the projections of  $ \mathscr{X}_{\varrho_{0}} $, $ U_{\varpi^{2}} $  under $ \pr_{2,3} $.     
\end{notation}

\begin{lemma} $ H _ {\varrho_{2}} = \mathscr{X}_{\varrho_{2}}   U_{\varpi^{2} } $. 
\end{lemma} 
\begin{proof} If $ h \in H $ is written as in Notation \ref{notationh}, then 
$$  \varrho_{2}^{-1} h \varrho_{2} =   
\begin{pmatrix} 
a &  &  & b &  &  \\[0.5em]  
 & a_1  & -c_2  &  & \mfrac{b_1 -  c_2 }{\varpi^{2} } & \mfrac{a_1 -  d_2 }{\varpi ^{2} }\\[0.5em] 
 & -c_1  & a_2  &  & \mfrac{a_2  -  d_1 }{\varpi ^{2}} & \mfrac{b_2  - c_1 }{\varpi ^{2}}\\%[0.5em] 
 c  &  &  & d &  &  \\%[0.5em] 
 & c_1 \,\varpi  &  &  & d_1  & c_1 \\[0.5em] 
  &  & c_2 \,\varpi  &  & c_2  & d_2 
\end{pmatrix} . 
$$ 
Now an argument similar to Lemma  \ref{strucvarrho2not} yields the desired factorization. 
\end{proof}  
\begin{proposition}  \label{tfhvarrho2action}  Modulo $ q - 1 $, 
$ \mathfrak{h}_{\varrho_{2},*}(\phi)  \equiv 0     $.  
\end{proposition} 
\begin{proof} Recall from (\ref{tfhvarrho2})   that $$ \mathfrak{h}_{\varrho_{2}} =  ( 1 + \rho^{2} + \rho^{4} ) (U H _{\varrho_{2}} )  -  ( 1 +  \rho ^{2}  )  (  U   \varpi^{(1,1,0,1)}  H_{\varrho_{2}} )   + (  U \varpi^{( 2,2,1,1)} H_{\varrho_{2}} ) +  (   U  \varpi^{(2,2,1,0)} H_{\varrho_{2}} ) . $$ 
If $ \lambda  = ( a , b, c, d )   \in \Lambda $ has depth at most $ 2 $, then $ U\varpi^{\lambda} H_{\varrho_{2}} =  U \varpi^{\lambda}  \mathscr{X} _{\varrho_{2}} $ by  Lemma  \ref{strucvarrho2not} and so $$ [U \varpi^{\lambda}  H_{\varrho_{2}}]_{*}( \phi )  =  | P^{\circ} \backslash  P^{\circ}\varpi^{(a,c,d)} 
 \mathscr{Y}_{\varrho_{2}} | \cdot \mathcal{T}_{b,a-b,*}(\phi) . $$
Now $ \mathcal{T}_{b, a-b,*}(\phi) $ is computed (modulo $ q-1$) by Corollary \ref{Tab*}. As $ \mathscr{Y}_{\varrho_{2}} \simeq \GL_{2}(\Oscr_{F} )$,  $ | P^{\circ} \backslash  P^{\circ}\varpi^{(a,c,d)} 
 \mathscr{Y}_{\varrho_{2}} | \equiv 1 \text{ or } 2 \pmod{q  - 1 }  $ depending on whether $ c = d = 2a $ or not. So one finds that 
 \begin{align*} [U H_{\varrho_{2}}]_{*}(\phi) & = \phi, \\ 
 [U \varpi^{(1,1,0,1)} H_{\varrho_{2}}  ]_{*}( \phi)  & \equiv2  ( 1 + z_{0}  ) \phi ,  \\
 [U \varpi^{(2,2,1,1)} H_{\varrho_{2}}]_{*}(\phi)  & \equiv ( 1 + z_{0}^{2}) \phi  , \\
 [ U \varpi^{(2,1,2,0)} H_{\varrho_{2}}]_{*}(\phi) &  \equiv  2   z_{0} \cdot  \phi . 
 \end{align*}
 From these, the claim easily  follows.  
\end{proof} 
\subsection{Convolutions with restrictions of  $\mathfrak{h}_{1}$}    In this subsection, we compute the convolution  $ \mathfrak{h}_{\varsigma_{i},*}(\phi) $ for $ i = 0 , 1 ,2 , 3 $. Recall that $ \varsigma_{i} = \sigma_{i} \tau_{1} $. Explicitly, 
$$ \varsigma_{0} =  
\scalebox{1.1}{$
\left(\begin{smallmatrix}
\varpi \, \,  & & & & 1   \\   
&     \varpi \,  \, & &  1  \\[0.1em] 
& &  \varpi   &  \\   
& & &   1 \\
& & & &   1      \\ 
& & & & & 1  
\end{smallmatrix}\right)$}
, \quad  \!  \varsigma_{1} =  \scalebox{1.1}{$\left(\begin{smallmatrix}
\varpi \, \,  & & & & 1  \\ 
& &  \varpi  \, \,  &  \\
&  \varpi \, \, & & 1  \\[0.2em]  
& & &  1 \\
& & & & & 1  \\ 
& & & &  1 
\end{smallmatrix}\right)$}, \quad  \!    \varsigma_{2} =   \scalebox{1.1}{$ \left ( \begin{smallmatrix}  
\varpi  & &  &   &   1    \\ 
& \varpi   &   &  1  & & 1      \\ & & \varpi  & &  1 \\ & & &   1 \\ & & & &  1 \\ & & & &  & 1 
\end{smallmatrix}\right )$} ,  \quad \!      \varsigma_{3} =   \scalebox{1.1}{$  \left ( \begin{smallmatrix}  
\varpi ^{2}  & &  &   &  \varpi    \\ 
& \varpi^{2}    &   &   \varpi & &   1    \\ & & \varpi^{2}  & & 1 \\ & & &  1 \\ & & & & 1 \\ & & & &  & 1 
\end{smallmatrix} \right ) $} $$ 

\begin{notation}  \label{strucvarsig1not} Let $ \mathbf{P} $, $ \ess $ and $ \partial $ be as in \S \ref{strucvarrho1not} and define  embeddings 
\begin{alignat*}{4} \imath _{\varsigma_{0}}   : \mathbf{P}   & \hookrightarrow \mathbf{H} , &  \quad \quad     \jmath_{\varsigma_{0}} :  \mathbf{P}  & \hookrightarrow  \mathbf{H}  \\ 
(\gamma _{1}, \gamma _{2}) & \mapsto (\gamma _{1}, \ess  \gamma _{1}  \ess,  \partial    \gamma _{2}  \partial ^{-1}   )   & \quad \quad \quad   \quad    (\gamma _{1},  \gamma _{2})    & \mapsto     (\ess  \gamma _{1} \ess,  \gamma _{1}, 
\partial \gamma _{2}   \partial^{-1}   ) 
\end{alignat*} 
We denote $ \mathscr{X}_{\varsigma_{0}}  $, the common images $   \imath_{\varsigma_{0}}(P^{\circ} ) = \jmath_{\varsigma_{0}}( P^{\circ} ) $. 
We let $ M_{\varsigma_{0}} $ (resp., $ M_{\varsigma_{0}} '$)  denote the subgroup of $ U_{\varpi} $ in which the first and third (resp., second and third) components are identity.
We also let $ \mathbf{\pr}_{1,2} : \Hb \to \mathbf{P} $ denote the projection $ ( h_{1}, h_{2}, h_{3}) \mapsto  (h_{1}, h_{2} ) $.  Finally, we let $ \mathscr{Y}_{\varsigma_{0}} $, $ L_{\varsigma_{0}}, L'_{\varsigma_{0}}, P_{\varpi}^{\circ}   \subset P $ the projections of $ \mathscr{X}_{\varsigma_{0}} $, $ M_{\varsigma_{0}} $, $ M'_{\varsigma_{0}}  $,  $ U_{\varpi} $   respectively. 
\end{notation} 

\begin{lemma}  \label{varsig0struclemma}   $ H_{\varsigma_{0}} =   \mathscr{X}_{\varsigma_{0}}  M_{\varsigma_{0}}    =  \mathscr{X}_{\varsigma_{0}}  M _ { \varsigma_{0} } '  $.   
\end{lemma} 

\begin{proof}  Writing $ h \in H $ as in Notation \ref{notationh}, we see that $$  \varsigma_{0} ^{-1} h \varsigma_{0} =   
\begin{pmatrix}
a & -c_1  & & \mfrac{b - c_1 }{\varpi } & \mfrac{a - d_1 }{\varpi } &  \\[0.5em] 
- c  & a_1  &  & \mfrac{a_1 - d  }{\varpi } & \mfrac{b_1 -  c }{\varpi } &  \\[0.5em] 
 &  & a_2  &  &  & \mfrac{b_2 }{\varpi }\\[0.5em] 
c \varpi  &  &  & d &  c &  \\[0.5em]    
 & c_1 \,\varpi  &  & c_1  & d_1  &  \\[0.5em] 
 &  & c_2 \,\varpi  &  &  & d_2 
\end{pmatrix}
$$  
Then one easily verifies that $ \mathscr{X}_{\varsigma_{0}} $, $ N_{\varsigma_{0}} $,  $ M_{\varsigma_{0}} $, $ M'_{\varsigma_{0}} $  are contained in $ H_{\varsigma_{0}} $. On the other hand if $ h = (h_{1}, h_{2}, h_{3}) \in H_{\varsigma_{0}} $, the above matrix is in $ K $ which implies that  $  h_{1}$, $ h_{2}$ and $  \partial ^{-1} h_{3}  \partial 
 \in  \GL_{2}(\Oscr_{F} ) $. It follows that  $     \eta : =  \imath_{\varsigma_{0}}    \left  ( h_{1}, \partial h_{3}  \partial  ) \right   )  ,
\gamma :=  \jmath_{\varsigma_{0}} (h_{2}, \partial^{-1}h_{3} \partial)   \in  \mathscr{X}_{\varsigma_{0}}    $ and $  \eta  ^{-1}  h  \in  M_{\varsigma_{0}} $, $ \gamma^{-1} h \in  M'_{\varsigma_{0}} $.    
\end{proof}

\begin{proposition}  \label{tfhvarsig0action}  Modulo $ q - 1 $, 
\begin{enumerate} [label = \normalfont(\alph*), itemsep= 0.3em, after = \vspace{0.1em} , before = \vspace{0.1em} ] 
\item $ \mathfrak{a}_{\varsigma_{0},*}(\phi) \equiv  5  (1+z_{0})  \phi  $,   
\item $  \mathfrak{b}_{\varsigma_{0},*}(\phi)  \equiv   ( 4  +  14 z_{0} +  4 z_{0} ^{2} )  \phi     $   
\item  $ \mathfrak{c}_{\varsigma_{0},*}  (\phi) \equiv  ( 1 +  z_{0})^{3} \cdot   
\phi   $ 
\end{enumerate} 
and $ \mathfrak{h}_{\varsigma_{0},*}(\phi)   =   
\mathfrak{h}_{\varsigma_{1},*}(\phi)    \equiv 0 $. 
\end{proposition} 
\begin{proof}   Let $ \lambda = (a,b,c,d ) \in \Lambda $ and $ \xi_{\lambda} $ denote $ [H_{\varsigma_{0}} \varpi^{\lambda} U ](\phi) $. Let   $ Q^{\circ} : = \GL_{2}(\Oscr_{F})  $ and $ Q^{\diamond}  \subset \GL_{2}(F) $ the  conjugate of $ Q^{\circ} $  by $ \partial =    \begin{psmallmatrix} & \varpi \\ 1 \end{psmallmatrix}  $. 
Lemma  \ref{varsig0struclemma} implies that 
\begin{align*} H_{\varsigma_{0}} \varpi^{\lambda} U / U  &  \to  
\pr_{1,2}  \left ( H_{\varsigma_{0}} \varpi^{\lambda} U / U  \right )  \times 
Q^{\diamond}  
\varpi^{(a,d)} Q^{\circ} / Q^{\circ} \quad  \\ 
(\gamma_{1}, \gamma_{2}, \gamma_{3} ) U  &  \mapsto  \left ( (\gamma_{1}, \gamma_{2}) \pr_{1,2}(U) ,  \gamma_{3} Q^{\circ}  \right )  
\end{align*} 
is a bijection. Now $ \left  |Q^{\diamond} \varpi^{(a,d)}  Q^{\circ} / Q^{\circ}   \right |  = \left | Q^{\circ}  \varpi^{(a-1,d) } Q^{\circ} / Q^{\circ}  \right | $ which equals $ q^{|a-1-2d|}(q+1) $ if $ a -1 \neq 2d $ and $ 1 $ otherwise. It remains to describe $ \mathrm{pr}_{1,2}\left ( H_{\varsigma_{0}} \varpi^{\lambda} U / U   \right )   \subset P / P^{\circ}   $.  By  Lemma  \ref{varsig0struclemma},   $ \pr_{1,2}(  H_{\varsigma_{0}} )  = \mathscr{Y}_{\varsigma_{0}} L_{\varsigma_{0}} =  \mathscr{Y}_{\varsigma_{0}} L'_{\varsigma_{0}} $. If $ | \beta_{0}(\lambda) | \leq 1  $ (resp., $ |\alpha_{0}(\lambda)| \leq 1  $),  then  the conjugate of $ L_{\varsigma_{0}} $ (resp., $ L_{\varsigma_{0}}'  $) by $ \varpi^{\lambda} $ is contained in $ U $. So if $ \mathrm{min} \left \{ |  \alpha_{0}(\lambda) |  , |  \beta_{0}(\lambda) |  \right \} \in \left \{ 0 ,1 \right \} $, we have \begin{equation}    \label{varsig0easy}   \pr_{1, 2} \left (  H_{\varsigma_{0}} \varpi^{\lambda}  U / U    \right )    = \mathscr{Y} _{\varsigma_{0}}  \varpi^{(a,b,c)} P^{\circ} /  P^{\circ}   \end{equation}  
where we write $ \varpi^{(a,b,c)} $ for $ \pr_{1,2}(\varpi^{\lambda} ) $.  
To describe a system of representatives for  $ \mathscr{Y}_{\varsigma_{0}}  \varpi^{(a,b,c)} P^{\circ} / P^{\circ} $, it suffices to describe one for $  \mathscr{Y} _{\varsigma_{0}} /  ( \mathscr{Y} _{\varsigma_{0}} \cap P^{\circ}_{(a,b,c)} )  $ where $$ P^{\circ}_{(a,b,c)}  : = \varpi^{(a,b,c)} P^{\circ }  \varpi^{-(a,b,c) } $$ denotes the conjugate of $ P^{\circ} $ by $ \varpi^{(a,b,c)} $.      Since $ \mathscr{Y} _{\varsigma_{0}} $ is isomorphic to $  \GL_{2}(\Oscr_{F}) $ (via the projection $ \mathbf{P} \to \mathbf{H}_{1} $, $ (\gamma_{1}, \gamma_{2}  ) \mapsto \gamma_{1} $),    this can be done by viewing  intersection $ \mathscr{Y}_{\varsigma_{0}} \cap P^{\circ}_{(a,b,c)} $ as a subgroup of $  U_{1} =  \GL_{2}(\Oscr_{F} )$ and comparing it with the  Iwahori subgroups $ I_{1}^{\pm} $. For this purpose, it will be convenient to introduce the quantities
$$ u_{\lambda}  = \mathrm{max}  \left \{  0, \alpha_{0}(\lambda) ,    -  \beta_{0}(\lambda) \right \} ,  \quad   \quad   v_{\lambda}  =  \mathrm{max} \left  \{  0 ,  - \alpha_{0}(\lambda  ) , \beta_{0}(\lambda)  \right \} . $$ 
These describe the valuations of the upper right and lower left entries of a matrix in  $ \mathscr{Y}_{\varsigma_{0}} \cap P^{\circ}_{(a,b,c)}   $. 

The case where $ \mathrm{min} \left \{ |a_{0}(\lambda) |  , |  \beta_{0}(\lambda)  | \right \} \geq 2 $ requires a little more work (though it will only occur once in this proof).  Here  
we invoke \cite[Lemma 5.9.3]{CZE} for the product $ \pr_{1,2}(H_{\varsigma_{0}}  )  =  \mathscr{Y}_{\varsigma_{0} }  P_{\varpi}^{\circ} $. Thus  
\begin{equation}   \label{varsig0hard}    \ch \left (  \pr_{1,2}   ( H_{\varsigma_{0}} \varpi^{\lambda} U  ) \right ) =  e^{-1} \sum  _{\gamma } \ch (  \gamma  P_{\varpi}^{\circ}  \varpi^{(a,b,c)} P^{\circ} )  
\end{equation} and  where $ \gamma $ runs over  (the finite set) $ \mathscr{Y}_{\varsigma_{0}} / \mathscr{Y}_{\varsigma_{0}} \cap P_{\varpi}^{\circ} $ and  
$ e =  e_{(a,b,c)}   : =  \big  [ \pr_{1,2} (   H _ { \varsigma_{0}} 
 )   \cap  P_{(a,b,c)} ^{\circ} \, :  \,  P^{\circ}_{\varpi} \cap P^{\circ}_{(a,b,c)}   \big   ]   $.  So the function $ \xi_{\lambda} $ can be computed by first computing $ \ch(P^{\circ}_{\varpi} \varpi^{(a,b,c)} P^{\circ} ) \cdot \phi $, then summing the translates of the result by representatives of $ \mathscr{Y}_{\varsigma_{0} } /  ( Y_{\varsigma_{0}} \cap P_{\varpi}^{\circ} ) $ and dividing the coefficients by $ e $. \\

\noindent (a) Recall that $ \mathfrak{a}_{\varsigma_{0}} =   (U \varpi^{(1,1,1,0)} H_{\varsigma_{0}}) + 
(U\varpi^{(1,1,0,1)}H_{\varsigma_{0}})+ 2 (U\varpi^{(1,1,0,0)}H_{\varsigma_{0}}) $.  Let $$  \lambda_{1} :  = (1,0,0,1) , \quad  \lambda_{2} :  = 
(1,0,1,0) , \quad \lambda_{3} : =  (1,0,1,1) .  $$    Then  $ \mathfrak{a}_{\varsigma_{0}, *} (\phi )  =  z_{0} \cdot ( \xi_{\lambda_{1} } + \xi_{\lambda_{2}} +  2 \xi_{\lambda_{3}} )  $. For each $ \lambda_{i} $, the formula  (\ref{varsig0easy})  applies. For $ \lambda  =  (a,b,c,d) \in  \left \{ \lambda_{2}, \lambda_{3}  \right \}  $, $ u_{\lambda} = 0 $ and $ v_{\lambda } = 1 $, so $ \mathscr{Y}_{\varsigma_{0}} \cap \varpi^{(a,b,c)} P^{\circ} \varpi^{-(a,b,c)} $ is identified with $ I_{1}^{+} $ and one easily sees that \begin{align*} 
 \xi_{\lambda_{2} }  & =  (q+1)  \, \mathcal{T}_{0,1}(\phi) \equiv 2 (\phi + \phi_{(1,1)})   \\     \xi_{\lambda_{3}} &  = \mathcal{T}_{0,1} (\phi ) \equiv \phi + \phi_{(1,1)}   
\end{align*} 
modulo $ q - 1 $.  
For $ \lambda = \lambda_{1} $, $ u_{\lambda} =  v_{\lambda} = 1 $ and we see that $ \mathscr{Y}_{\varsigma_{0}} \cap \varpi^{(1,0,0)} P^{\circ} \varpi^{-(1,0,0)} $ is identified with $ I_{1}^{+} \cap  I_{1}^{-} $. Thus a system of representatives for $  \mathscr{Y} _{\varsigma_{0}} /  ( \mathscr{Y} _{\varsigma_{0}} \cap \varpi^{(1,0,0)} P^{\circ} \varpi^{-(1,0,0) }  ) $ is obtained by multiplying a system of representatives for $  U_{1} / I_{1}^{+} $ with that for $ I_{1}^{+} / I_{1}^{+} \cap I_{1}^{-} $. So  $$ \xi_{\lambda_{1}} = 
\sum_{ \gamma \in U_{1}    / I_{1}^{+} }  \gamma  \sum_{\eta \in I_{1}^{+}/(I_{1}^{+}\cap I_{1}^{-}) } \eta  \varpi^{\lambda_{1}}  \cdot \phi   .  $$
Now $ \eta \varpi^{\lambda_{1}   }  \cdot \phi = \varpi^{\lambda_{1} } \cdot \phi $ for any $ \eta \in I_{1}^{+} $. So  the inner sum  equals $  q \phi $. The outer sum then evaluates to $ q ( \phi + q \phi_{(1,1)}) $. 
Thus  $ \xi_{\lambda_{1}} \equiv   (   \phi + \phi_{(1,1)}  ) $.
Putting everything together 
gives part (a).    \\%[0.5em] 

\noindent (b) Recall that  $$ \mathfrak{b}_{\varsigma_{0}}  =  (U   \varpi^{(2,2,1,1)}   H _{\varsigma_{0}} ) +  (U \varpi^{(2,1,2,1)}   H_{\varsigma_{0}}) + (U\varpi^{(2,2,0,1)} H _{\varsigma_{0}}) + (U \varpi^{(2,1,1,2)} H_{\varsigma_{0}}) + 4 (U\varpi^{(2,1,1,1)}H_{\varsigma_{0}}) . $$  For $ \mu   \in \left \{  (2,1,1,2), (2,1,1,1) \right \} $,  it is easy to see that   $$ [  U \varpi^{\mu} H_{\varsigma_{0}}]_{*}(\phi) \equiv 2  z_{0} \cdot \phi  $$ 
For $ \mu_{1} =  (2,2,1,1) $ and $ \mu_{2} = (2,1,2,1) $,  arguments similar to part (a)  reveal that $$ [U \varpi^{\mu_{1}} 
H_{\varsigma_{0}} ]_{*}  ( \phi ) \equiv  2 ( 1 + z_{0}^{2})  \phi   ,  \quad \quad 
 [U \varpi^{\mu_{2}} H_{\varsigma_{0}} ]_{*} ( \phi ) \equiv 4 z_{0} \cdot \phi . $$ 
This leaves $ \mu = (2,2,0,1) $. Denote $ \lambda = (4,2,2,2) - \mu = ( 2,0,2,1) $ and let $ e $ denote $ e_{(2,0,2)} $.  It is easy to see that $ \pr_{1,2}(H_{\varsigma_{0}}) \cap P^{\circ}_{(2,0,2)} $ is equal to the product of $ \mathscr{Y}_{\varsigma_{0}} \cap P_{(2,0,2)}^{\circ} $ with $  P_{\varpi} ^{\circ} \cap P_{(2,0,2)}^{\circ} $ and therefore $ e  =  [  \mathscr{Y}_{\varsigma_{0}} \cap P_{(2,0,2)}^{\circ} \, : \, 
 \mathscr{Y}_{\varsigma_{0}} \cap P_{\varpi} \cap  P_{(2,0,2)}^{\circ}  ] $.  
From this, one finds that $ e =  q $.  
Next we  compute that $$ \ch( P _ { \varpi  }  ^{\circ} \varpi^{(2,0,2)}  P^{\circ} ) \cdot  \phi   =    q \big (  \phi_{(0,1)}  -  \phi_{(1,1)} + q \phi_{(1,2)}  \big   )  . $$
Since $  \mathscr{Y}_{\varsigma_{0}} \cap P_{\varpi}^{\circ} \subset \mathscr{Y}_{\varsigma_{0}}  $ is identified with $ I_{1}^{+} \cap I_{1}^{-}  \subset \GL_{2}(\Oscr_{F} )  $, 
the  expression (\ref{varsig0hard})  reads   \begin{align*}  \xi_{\lambda}  &  =   \sum _ { h \in U_{1 }   / I _{1} ^ {  + } \cap  I_{1} ^{- }  } 
   h ( \phi_{(0,1)} - \phi_{(1,1)} + q \phi_{(1,2)} )  \\
& = 
\sum_ { \gamma \in   U_{1} / I_{1}^{+} }    
\gamma   \sum _ { \eta \in I_{1}^{+} / ( I_{1}^{+} \cap I_{1}^{-}) } \eta  ( \phi_{(0,1)} - \phi_{(1,1)} + q \phi_{(1,2)} ) 
\end{align*}
Then the inner sum is just multiplication by $ q $. The outer sum then evaluates to $$
q ( \phi +  q \phi_{(1,1)}  
)  -  q  ( q + 1 ) \phi_{(1,1) }  + q  ^ { 2 }   ( \phi_{(1,1)} +  q \phi_{(2,2)} )   =  q  \phi  + q  ( q- 1 )  \phi_{(1,1)} +  q^{3}   \phi_{(2,2) } $$ So  
we see that $\xi_{\lambda}  
= q \phi + q ( q - 1) \phi _{(1,1) }  + q^{3} \phi_{(2,2)} $  and  therefore  $$ [ U \varpi^{(2,2,0,1)} H_{\varsigma_{0}}]_{*}(\phi) =  (q+1)  z_{0}^{2} \cdot 
 \xi_{\lambda} \equiv 2   ( 1   + 
 z_{0}^{2} ) \phi. $$   Putting everything together, we find that  \begin{align*} 
\mathfrak{b}_{\varsigma_{0},*}( \phi)  &  \equiv 2(1+z_{0}^{2})\phi +  4 z_{0} \cdot \phi + 2 ( 1 + z_{0}^{2}) \phi +  2 z_{0}  \cdot  \phi +  4 ( 2 z_{0} \cdot \phi ) \\
& = ( 4 +  4z_{0}^{2} + 14 z_{0} )   \phi . 
\end{align*}

\noindent (c) 
We have $ \mathfrak{c}_{\varsigma_{0}} = 
(U\varpi^{(3,2,2,2)}H_{\varsigma_{0}})+
(U\varpi^{(3,3,1,1)}H_{\varsigma_{0}})+  (U\varpi^{(3,2,0,1)}H_{\varsigma_{0}}) $. For each of the three Hecke operators, the formula (\ref{varsig0easy}) applies and we find that \begin{align*}  [ U \varpi^{(3,2,2,2)}H_{\varsigma_{0} } ]_{*}  ( \phi ) & \equiv 2( z_{0} + z_{0}^{2} ) \phi , \\ 
[U \varpi^{(3,3,1,1)} H_{\varsigma_{0}}]_{*} (\phi) & \equiv  ( 1 + z_{0}^{3}) \phi ,  \\ 
[ U \varpi^{(3,2,0,1)}  H_{\varsigma_{2}} ] _{* } ( \phi )   &   \equiv (z_{0}   + z_{0}^{2} ) \phi 
\end{align*} 
from which (c) follows. \\

\noindent Now recall that  $ \mathfrak{h}_{\varsigma_{0}} = -(1+ \rho^{6}) (U H_{\varsigma_{0}})  +  ( 1+  2 \rho^{2} + \rho^{4} ) \mathfrak{a}_{\varsigma_{0}}  - ( 1 + \rho^{2} )   \mathfrak{b}_{\varsigma_{0}} +   \mathfrak{c}_{\varsigma_{0}} $. It is easy to see that $ [U H_{\varsigma_{0}}]_{*}(\phi) = ( q + 1  ) \phi $. So  by parts (a)-(c), we see  that 
\begin{align*}  \mathfrak{h}_{\varsigma_{0}, * }   ( \phi  )  &  \equiv    - 2  ( 1 +   z_{0}  ^{3} )   \phi      + ( 1 + z_{0})^{2}  \Big (  5 ( 1 + z_{0})  \phi \Big  )    -    ( 1 + z_{0}) \Big ( 4 + 14 z_{0} +  4 z_{0}  ^{2}   \Big ) \phi  +  (1 + z_{0})^{3}  \phi    \\
& =  ( 1 +z_{0} ) \Big ( \! - 2 + 2z_{0} - 2z_{0}^{2}  + 5 ( 1+ z_{0})^{2} - 4 - 14 z_{0} - 4z_{0}^{2} + ( 1 + z_{0}  )^{2}  \Big  )  \phi  
\\ &   = 0 
\end{align*} 
Finally since  $ \mathfrak{h}_{\varsigma_{1}} = w_{2}  \mathfrak{h}_{\varsigma_{0}} w_{2} $ (\ref{tfhvarsig1}) and $ w_{2} $ only swaps the second and third components of $ H $ and $ w_{2} $ normalizes $ U $, we see that $ \mathfrak{h}_{\varsigma_{1},*}(\phi) = \mathfrak{h}_{\varsigma_{0},*}(\phi) $. 
\end{proof}

\begin{notation}  \label{strucvarsig2not} We let $ A_{\varsigma_{2}} $ denote the intersection $ A \cap \varsigma_{2} K \varsigma_{2}^{-1} $ and   $ J_{\varsigma_{2}} \subset U  $ denote the Iwahori subgroups of triples $ (h_{1}, h_{2}, h_{3} ) \in U $ such that $ h_{1} , h_{3} $ reduce modulo $ \varpi $ to lower triangular matrices and $ h_{2} $ reduces to an upper triangular matrix.  We denote by $ M_{\varsigma_{2}}$ the three parameter additive  subgroup of all triples $ h = (h_{1},h_{2},h_{3}) \in U $ such that  $$   h_{1} = \begin{psmallmatrix} 1 \\  x & 1  \end{psmallmatrix}, \quad h_{2} = \begin{psmallmatrix} 1 & y  \\  &   1 \end{psmallmatrix} , \quad  h_{3}  =   \begin{psmallmatrix}  1 \\[0.2em]   y - x + \varpi z & 1  \end{psmallmatrix} $$ 
where $ x , y , z \in \Oscr_{F} $  are  arbitrary and by $ N_{\varsigma_{2}} $ the three  parameter  subgroup of all triples $  (h_{1}, h_{2},  h_{3}) $ of the form 
$$
h_{1} = 
\begin{psmallmatrix}  
1  & x \varpi  \\[0.1em]
& 1    
\end{psmallmatrix} , 
\quad    h_{2} = 
\begin{psmallmatrix} 
1  \\  
y \varpi   & 1  
\end{psmallmatrix},  \quad    h_{3} =   \begin{psmallmatrix}   1  &   z \varpi  \\[0.1em]    &  1    \end{psmallmatrix}  $$  
where $ x , y , z \in \Oscr_{F} $ are arbitrary. 
Finally, we let $ L_{\varsigma_{2}} $ the one-parameter subgroup of $U$ all triples of the form $ (1,1, \begin{psmallmatrix} 1 \\  z & 1 \end{psmallmatrix}  ) $ where $ z \in \Oscr_{F} $.   
\end{notation}

\begin{lemma}    \label{varsig2struclemma}   $ H_{\varsigma_{2}} $ is the product of $ A_{\varsigma_{2}} $, $ M_{\varsigma_{2}} $, $ N_{\varsigma_{2}} $ and $    J_{\varsigma_{2}} $ is the product of $ A^{\circ} $, $ H_{\varsigma_{2}} $, $ L_{\varsigma_{2}} $ where these products can be taken in any order.  
\end{lemma} 

\begin{proof}  It is easily verified that $ \varsigma_{2}^{-1} M_{\varsigma_{2}} \varsigma_{2} $, $ \varsigma_{2}^{-1} N \varsigma_{2} $ are contained in $ K $, so that $ M_{\varsigma_{2}} , N_{\varsigma_{2}} $ are subgroups of $ H_{\varsigma_{2}} $. 
Let $ h\in H_{\varsigma_{2}} $ and write $ h $ as in Notation \ref{notationh}. Then $$ \varsigma_{2} ^{-1} h  \varsigma_{2}   =    
\begin{pmatrix}
a & -c_1  &  & \mfrac{b-c_1 }{\varpi } & \mfrac{a -  d_1 }{\varpi } & -\mfrac{c_1 }{\varpi }\\[0.5em] 
-c & a_1  & -c_{2}  & \mfrac{a_1 -  d}{\varpi 
 } & \mfrac{b_1 - c - c_2 }{\varpi  } & \mfrac{a_1 -  d_2 }{\varpi }\\[0.5em] 
& -c_1  & a_2  & -\mfrac{c_1 }{\varpi } & \mfrac{a_2 - d _1 }{\varpi } &  \mfrac{ b_2 -c_1  } { \varpi  }  \\[0.5em] 
c\,\varpi  & & & d & c & \\[0.5em]  
& c_1 \,\varpi  & & c_1  & d_1  & c_1  \\[0.5em]  
 & & c_2  \varpi  & & c_{2}  & d_2
\end{pmatrix} . $$
It follows that $ h \in U $ and   $   c_{1}, b_{2} , b  \in  \varpi  \Oscr_{F}  $. In particular, $ H_{\varsigma_{2}} \subset J_{\varsigma_{2}} $ and  $ a, a_{1}, a_{2}, d , d_{1}, d_{2}  \in  \Oscr_{F}   ^ {  \times  } $.     Let $ m \in M_{\varsigma_{2}} $ be defined with parameters $ x =  - c/a  $, $ y = -b_{1}/d_{1} $   and   $ z =  -(c_{2}/a_{2} +y-x)/\varpi $ (see Notation \ref{strucvarsig2not}).  Write $ h' = m h $ as in Notation \ref{notationh} and  let $ n \in N_{\varsigma_{2}} $ be defined  with  paramaters $ x = -b'/d'\varpi $, $ -c_{1}'/a_{1}'\varpi $, $ z  =  - c_{2}'/ a_{2}' \varpi $ (see Notation \ref{strucvarsig2not}). Then  $  nmh  $ lies in $ A $, and hence in $ A _{\varsigma_{2}} $.   Thus $ H_{\varsigma_{2}} =  M_{\varsigma_{2}} N_{\varsigma_{2}} A_{\varsigma_{2}} $. Similarly we can show $ H _{\varsigma_{2}} =  N _{\varsigma_{2}} M_{\varsigma_{2}}   A_{\varsigma_{2}} $. Since $ A_{\varsigma_{2}} $ normalizes both $ M_{\varsigma_{2}} $, $ N_{\varsigma_{2}} $, the product  holds in  all possible  orders. This establishes the first claim. The second is established in completely analogous way.         
\end{proof}  
\begin{corollary}  \label{basicallyIwahorivarsig2}  If $ \lambda \in \Lambda $ satisfies $ \beta_{2}  (\lambda)   \leq  0  $, then 
$U \varpi^{\lambda} H_{\varsigma_{2}} = U \varpi^{\lambda} J_{\varsigma_{2}}$.  
\end{corollary} 
\begin{proof} This follows by Lemma \ref{varsig2struclemma} since $ \varpi^{\lambda} L_{\varsigma_{2}} \varpi^{-\lambda}  \subset U $ if $ \beta_{2}(\lambda)  \leq 0   $.  
\end{proof}  

Corollary \ref{basicallyIwahorivarsig2} reduces the computation of $ [U\varpi^{\lambda} H_{\varsigma_{2}}]_{*}(\phi) $ to $ [U\varpi^{\lambda} J_{\varsigma_{2}}]_{*}(\phi) $ for almost all Hecke operators appearing  in $ \mathfrak{h}_{\varsigma_{2},*}  $, which we  can be calculated   efficiently  using  Lemma  \ref{Iuvlemma}.   The few exceptions are handled below. 
\begin{lemma}  \label{exceptionsvarsigma2}   Modulo $ q - 1 $, we have 
\begin{enumerate} [label = \normalfont(\alph*), itemsep= 0.3em, after = \vspace{0.1em} , before = \vspace{0.1em} ]
\item $ [U \varpi^{(1,1,0,1)} H_{\varsigma_{2}} ]_{*}(\phi) \equiv   ( 1 + z_{0} ) \phi -  z_{0}  \cdot  \phi_{(1,0)}   ,  $ 
\item $ [U \varpi^{(1,0,1,1)} H_{\varsigma_{2}}]_{*}(\phi) \equiv z_{0} \cdot  \phi_{(1,0)}  , $ 
\item $  [U\varpi^{(2,1,1,2)} H_{\varsigma_{2}}]_{*} (\phi) \equiv  z_{0} \cdot \phi  $ 
\item $ [U \varpi^{(3,2,1,2)}\psi H_{\varsigma_{2}}]_{*}( \phi) \equiv 0 
$ 
\end{enumerate}
\end{lemma}  
\begin{proof}   
For $ \lambda \in \Lambda $, we will denote $  \xi_{\lambda} := [H_{\varsigma_{2} }  \varpi^{\lambda} U ](\phi) $.  \\

\noindent (a) This equals $ z_{0} \cdot \xi_{\lambda} $ where $ \lambda = (1,0,1,0) $. Since $ \lambda $ has depth one, we have   $ H_{\varsigma_{2}}  \varpi^{\lambda} U = M_{\varsigma_{2}} \varpi^{\lambda} U $. Now  
\[ M_{\varsigma_{2}} \varpi^{\lambda} U / U = \left \{  \left (  \begin{psmallmatrix} 1 \\[0.1em]  x & \varpi  \end{psmallmatrix} 
, 
 \begin{psmallmatrix} \varpi & y \\[0.1em] &  1  \end{psmallmatrix} ,  \begin{psmallmatrix} 1   \\[0.1em]  y - x & \varpi  \end{psmallmatrix}    \right ) U \, | \, x, y \in \Oscr_{F} \right \}.
 \] 
and it is easy to see that a system of representatives for  $ M_{\varsigma_{2}} \varpi^{\lambda} U/U $ is obtained by allowing the parameters $ x, y $ in the set above to run over $ [\kay] $. Using this system, one calculates that $ \xi_{\lambda}  =   \phi - \phi_{(1,0)} - \phi_{(1,1)} $. \\

\noindent (b)  This equals $ z_{0} \cdot \xi_{\lambda} $ where $ \lambda = (1,1,0,0) $. As in part (a), we have $  H_{\varsigma_{2}} \varpi^{\lambda} U  = M_{\varsigma_{2}} \varpi^{\lambda} U $ and it is easy to see that \[ M _{\varsigma_{2}} \varpi^{\lambda} U / U = \left \{ \left (1 , 1 ,  \begin{psmallmatrix} 1 \\[0.1em]   t  &   1 \end{psmallmatrix}  \right )  \varpi^{\lambda} U \, | \, t \in \Oscr_{F} \right \} . \]
A set of representatives is obtained by allowing the parameter $ t $ to run over elements of $ [\kay] $. Thus $ \xi_{\lambda} = q \phi_{(1,0)} $.  \\  

\noindent (c) This  expression equals $ z_{0}^{2} \cdot  \xi_{\lambda} $ where $ \lambda  =   (2,1,1,0) $. As the first two components of $ \varpi^{\lambda} $ are central and $ \beta_{2}(\lambda) \geq 0 $, we have $ \varpi^{-\lambda} N_{\varsigma_{2}}\varpi^{\lambda} \subset U $ and so $ H_{\varsigma_{2}} \varpi^{\lambda} U = M _{\varsigma_{2}} \varpi^{\lambda} U / U $.  Using the centrality of the first two componetns again, we see that $$   M_{\varsigma_{2}} \varpi^{\lambda} U / U  = \left \{  \left ( 1 , 1 ,  \begin{psmallmatrix}  1 & \\[0.1em] u & 1 \end{psmallmatrix}   \right )  \varpi^{\lambda} U  \, | \, u \in \Oscr_{F} \right \} . $$
From this, we see that a system of representatives is given by letting the parameter $ u $ run over $ [\kay_{2}] $. Thus $ \xi_{\lambda} = q^{2} \phi_{(1,1)} $. \\

\noindent  (d)  It suffices to  show that $  [ H_{\varsigma_{2}} \psi^{-1} \varpi^{(1,0,1,0)} U ](\phi) \equiv 0  $. Let us denote $ \psi^{-1} \varpi^{(1,0,1,0) } $ by $ \eta $. It is straightforward to verify that $ \eta ^{-1} N_{\varsigma_{2}} \eta \subset U $, so that  $ H_{\varsigma_{2}} \eta U / U = A_{\varsigma_{2}}  M_{\varsigma_{2}} \eta U / U $.  Elementary manipulations show that \vspace{0.4em} $$  A_{\varsigma_{2}} M_{\varsigma_{2} } \eta U / U = \left \{  \left  (  \begin{psmallmatrix}  1 \\[0.1em]  s   & \varpi   \end{psmallmatrix} , \begin{psmallmatrix} \varpi   & t  \\[0.1em]   &  1 \end{psmallmatrix}  ,  \begin{psmallmatrix} 1 & \\[0.1em]  u + s - t  & \varpi  \end{psmallmatrix} \right ) U  \, \,  | \,  \,  s, t \in \Oscr_{F} ,  u \in \Oscr_{F}^{\times} \right \} $$ 
where we used that $ (a, a_{1}, a_{1}, d, d_{1} , d_{2} ) \in A_{\varsigma_{2}} $ if and only if $ a , d \in \Oscr_{F}^{\times }$ with $ a \equiv d_{1} \equiv  a_{2}  \pmod {\varpi} $ and $ d \equiv a_{1} \equiv d_{2} \pmod{\varpi} $. Let $$ C(s,t,u) : =   \left  (  \begin{psmallmatrix}  1 \\[0.1em]  s   & \varpi   \end{psmallmatrix} , \begin{psmallmatrix} \varpi   & t  \\[0.1em]   &  1 \end{psmallmatrix}  ,  \begin{psmallmatrix} 1 & \\[0.1em]  u + t - s  & \varpi  \end{psmallmatrix} \right )  $$ 
where $s, t  \in \Oscr_{F} $, $ u \in \Oscr_{F}^{\times }$.   Then $ C(s,t,u) U = C(s',t',u') U $ if and only if $ s \equiv s' $, $ t \equiv t' $, $ u \equiv u' $ modulo $ \varpi $. Thus  a system of representatives for $ H_{\varsigma_{2}} \eta U / U  $    is given by $ C(s,t,u) $ where $ s, t $ run over $ [\kay] $ and $ u $ runs over $ [\kay]^{\times} $. Thus for each fixed $ s , t $, there are $ q - 1 $ choices of $ u $ from which it easily follows that the function $ [H_{\varsigma_{2}} \eta U ](\phi) $ vanishes modulo $ q - 1 $.  
\end{proof}

\begin{proposition}  \label{tfhvarsig2action}  Modulo $ q - 1 $,  we have 
\begin{enumerate} [label = \normalfont(\alph*), itemsep= 0.3em, after = \vspace{0.1em} , before = \vspace{0.1em} ] 
\item $ \mathfrak{a}_{\varsigma_{2},*}(\phi) \equiv  2(1+z_{0})  \phi + 2z_{0} \cdot  \phi_{(1,0)}   $,   
\item $  \mathfrak{b}_{\varsigma_{2},*}(\phi)  \equiv   ( 1 + 6z_{0} + z_{0}^{2})  \phi + 3z_{0}(1+z_{0})\phi_{(1,0)}   $ 
\item  $ \mathfrak{c}_{\varsigma_{2},*}  (\phi) \equiv   z_{0}(1+z_{0})  \phi + z_{0}(1+z_{0})^{2} \phi _{(1,0)}   $ 
\end{enumerate} 
and $ \mathfrak{h}_{\varsigma_{2},*}(\phi)  \equiv 0 $. 
\end{proposition}

\begin{proof}  For $ \lambda = (a,b,c,d) \in \Lambda $, let $ \xi _{\lambda } $ denote $ [U \varpi^{\lambda} H_{\varsigma_{2}}]_{*}(\phi) $.  If $ \beta_{2}(\lambda) = 2d - a \leq 0 $, then $ \xi_{\lambda} = [U \varpi^{\lambda} J_{\varsigma_{2}}]_{*}(\phi) $. It is easily seen from Lemma \ref{Iuvlemma} and the decompositions given therein  that $$ [U \varpi^{\lambda} J_{\varsigma_{2}}]_{*}(\phi) \equiv  \mathcal{I}_{-b,a-b}^{-}(\phi)    \pmod{q-1}    $$
This  formula in conjunction with Lemma \ref{exceptionsvarsigma2} can be used to calculate all Hecke operators. 
For instance, 
we have $ \mathfrak{a}_{\varsigma_{2}} =     ( U \varpi^{(1,1,1,0)} H_{\varsigma_{2}})   +  ( U \varpi^{(1,0,0,0)} H_{\varsigma_{2}}) +  (U \varpi^{(1,1,0,1)} H_{\varsigma_{2}}) + ( U \varpi^{(1,0,1,1)} H_{\varsigma_{2}}) + 2( U \varpi^{(1,0,1,0)} H_{\varsigma_{2}})  $ and we compute    
\begin{multicols}{2} 
\begin{itemize} 
\item $ [ U \varpi^{(1,1,1,0)} H_{\varsigma_{2}}]_{*}(\phi)  \equiv  ( 1 + z_{0} ) \phi  -  z_{0} \cdot \phi_{(1,0)}  $, 
\item $ [U \varpi^{(1,0,0,0)}H_{\varsigma_{2}}]_{*}(\phi) \equiv z_{0} \cdot \phi_{(1,0)} $, 
\item $ [U \varpi^{(1,1,0,1)} H_{\varsigma_{2}}]_{*}(\phi )   \equiv ( 1 + z_{0} ) \phi -  z_{0} \cdot \phi_{(1,0)} $, 
\item $ [U \varpi^{(1,0,1,1)} H_{\varsigma_{2}}]_{*}(\phi)  \equiv  z_{0} \cdot \phi_{(1,0)} $, 
\item $  2 [U \varpi^{(1,0,1,0)} H_{\varsigma_{2}} ] (\phi) \equiv  2  z_{0}  \cdot  \phi_{(1,0)} $.  
\end{itemize} 
\end{multicols}  
\noindent Now adding all these retrieves the  expression in part (a). Similarly for parts (b) and (c).  \\ 

\noindent Now   $ \mathfrak{h}_{\varsigma_{2}} = - ( 1 + \rho^{6}) (U H_{\varsigma_{2}}) + ( 1 +  2 \rho^{2} + \rho^{4}) \mathfrak{a}_{\varsigma_{2}}  - ( 1 +  \rho^{2})   \mathfrak{b}_{\varsigma_{2}} +    \mathfrak{c}_{\varsigma_{2}} $ from (\ref{tfhvarsig2}). 
Therefore    
\begin{align*}  \mathfrak{h}_{\varsigma_{2},*}( \phi   )  &  \equiv (- 1 - z_{0}^{3}) \phi +  (1+z_{0})^{2}  \Big  ( 2(1 + z_{0} )\phi  + 2z_{0} \cdot \phi_{(1,0)}    \Big  ) \,  - \\[0.1em] 
& \quad  \,  \,   (1+z_{0})  \Big  ( (1+6z_{0} + z_{0}^{2})\phi + 3z_{0}(1+z_{0})\phi_{(1,0)}   \Big  )  + z_{0}(1+z_{0}) \phi + z_{0}(1+z_{0})^{2} \phi_{(1,0)}  \\[0.2em] 
& =  \left (  -1-z_{0}^{3} + 2(1+z_{0})^{3} - (1+z_{0})(1+6z_{0} + z_{0}^{2})  + z_{0}(1+z_{0}) \right )  \phi \,  +  \\[0.2em]  
& \quad \,   \left (  2 z_{0} ( 1+  z_{0} )  ^{2}  - 3z_{0}  ( 1 + z_{0} ) ^{2}  + z_{0} ( 1+ z_{0})^{2}   \right ) \phi_{(1,0)}  = 0   
\qedhere  
\end{align*} 
\end{proof}

\begin{notation}  \label{strucvarsig3not} As usual, we let $ A_{\varsigma_{3}} $ denote the intersection $ A \cap \varsigma_{3} K \varsigma_{3}^{-1} $. 
We denote by $ M_{\varsigma_{3}}$ the three parameter additive  subgroup of all triples $ h = (h_{1},h_{2},h_{3}) \in U $ such that 
\[h_{1} = \begin{psmallmatrix} 1 \\[0.1em]  x/\varpi  & 1  \end{psmallmatrix}, \quad h_{2} = \begin{psmallmatrix} 1 & \, \, \,  y  \\[0.2em]   &  \, \, \,   1 \end{psmallmatrix} , \quad  h_{3}  =   \begin{psmallmatrix}  1 \\[0.2em]   y - x \varpi  +  z\varpi^{2}  & \, \, \,  1  \end{psmallmatrix} \] 
where $ x , y , z \in \Oscr_{F} $  are  arbitrary and by $ N_{\varsigma_{2}} $ the three  parameter  subgroup of all triples $  (h_{1}, h_{2},  h_{3}) $ of the form 
\[
h_{1} = 
\begin{psmallmatrix}  
1 \,\, \,  & x \varpi^{2}  \\[0.2em]
& 1    
\end{psmallmatrix} , 
\quad    h_{2} = 
\begin{psmallmatrix} 
1  &  \\[0.2em]   
y \varpi   & \, \,\,  1  
\end{psmallmatrix},  \quad    h_{3} =   \begin{psmallmatrix}   1  & y \varpi +   z \varpi^{2}  \\[0.2em]    &   1    \end{psmallmatrix}  \]  
where $ x , y , z \in \Oscr_{F} $ are arbitrary. 
\end{notation} 

\begin{lemma}    \label{varsig3struclemma}   $ H_{\varsigma_{3}} = M_{\varsigma_{3}} N_{\varsigma_{3} } A_{\varsigma_{3}} = N_{\varsigma_{3}}  M_{\varsigma_{3}}  A_{\varsigma_{3}} $.  
\end{lemma}

\begin{proof}  Writing $ h \in H $ as in Notation \ref{notationh}, we find that $$ \varsigma_{3} ^{-1} h \varsigma_{3} =  
\begin{pmatrix} 
a & -c_1 \,\varpi & & \mfrac{b}{\varpi^2 }-c_1  & \mfrac{a - d_1 }{\varpi } & -\frac{c_1 }{\varpi }\\[0.5em] 
-c\,\varpi  & a_1  & -c_2  & \mfrac{a_1 - d}{\varpi } & \mfrac{b_1 - c_2 - c \varpi^{2} }{\varpi^2 }  & \mfrac{a_1  - d_2 }{\varpi^2 }\\[0.5em] 
& -c_1  & a_2  & -\mfrac{c_1 }{\varpi } & \mfrac{a_2- d_1 }{\varpi^2 } & \mfrac{b_2 
 - c_1 }{\varpi^2 }\\[0.5em] 
c\,\varpi^2  &  &  & d & c\,\varpi  &  \\[0.5em] 
 & c_1 \,\varpi^2  &  & c_1 \,\varpi  & d_1  & c_1 \\[0.5em] 
 &  & c_2 \,\varpi^2  &  & c_2  & d_2 
\end{pmatrix} .  
$$
From this matrix, we easily see that $ H_{\varsigma_{3}}  $ contains $ M_{\varsigma_{3}} $ and $ N_{\varsigma_{3}} $. We also see that if $ h \in H_{\varsigma_{3}} $, then  all entries of $ h $ except for $ c $ are integral. Moreover,  $ c_{1}, b_{2} \in \varpi \Oscr_{F} $,  $ b \in \varpi^{2} \Oscr_{F} $ and $ c \in \varpi^{-1} \Oscr_{F} $.   Thus $ a, a_{1} ,a_{2}, d, d_{1}, d_{2} \in \Oscr_{F}^{\times} $.  An argument analogous to Lemma  \ref{varsig2struclemma} applies to yield the desired decompositions.   
\end{proof}

\begin{proposition}  \label{tfhvarsig3action}   Modulo $ q - 1 $, we have 
\begin{enumerate} [label = \normalfont(\alph*), itemsep= 0.3em, after = \vspace{0.1em} , before = \vspace{0.1em} ]
\item $\mathfrak{a}_{\varsigma_{3},*}(\phi)  =  z_{0} \cdot  \phi_{(2,0)}   
$ 
\item $ \mathfrak{b}_{\varsigma_{3},*}(\phi) \equiv  ( z_{0}^{2} + z_{0} )\cdot  \phi_{(2,0)}  + z_{0}  \cdot \phi_{(1,0)}  $ 
\item $ \mathfrak{c}_{\varsigma_{3},*} (\phi) \equiv   (z_{0}^{2} + z_{0}) \cdot  \phi_{(1,0)} $

\end{enumerate}
and $ \mathfrak{h}_{\varsigma_{3},*}  (\phi) \equiv     0  $ 
\end{proposition}

\begin{proof}  For $ \lambda \in \Lambda $, let $ \xi_{\lambda} $ denote $ [H_{\varsigma_{2}} \varpi^{\lambda} U ] ( \phi ) $.   \\

\noindent (a) This equals $ \xi_{f_{1}} $.   From Lemma \ref{varsig3struclemma}, we find that  $   H_{\varsigma_{3}} \varpi^{f_{1}}  U / U = \left \{  \varpi^{f_{1}} U \right \}  $, so that  $ \xi_{f_{1}} = \varpi^{f_{1}} \cdot \phi  =  \phi_{(1,-1) }   $.   \\   

\noindent (b)  Recall that $   \mathfrak{b}_{\varsigma_{3}} =     (U  \varpi^{(1,0,1,0)} H_{\varsigma_{3}}) +  (U \varpi^{(1,-1,1,0)}H_{\varsigma_{3}} )  +  (U \varpi^{(1,0,0,1)}H_{\varsigma_{3}}) $. Let $ \lambda_{1} =  (1,1,0,1) $, $   \lambda_{2} = (3,3,1,2)  $ and $ \lambda_{3} =  (1,1,1,0) $. This $ \mathfrak{b}_{\varsigma_{3},*}(\phi) = z_{0} \cdot \xi_{\lambda_{1}} + z_{0}^{2}  \cdot \xi_{\lambda_{2}} + z_{0}  \cdot  \xi_{\lambda_{3}} $. From Lemma \ref{varsig3struclemma}, we find that \begin{align*} H_{\varsigma_{3}} \varpi^{\lambda_{1}} U / U  & = \left \{ \varpi^{\lambda_{1}} U \right \}   \\
H_{\varsigma_{3}} \varpi^{\lambda_{2}} U / U  & =   \left \{   \left  (  \begin{psmallmatrix}  \varpi^{2}   &  x \varpi  \\[0.1em]   & 1   \end{psmallmatrix} , \begin{psmallmatrix}  1    &   \\[0.1em]   &    \varpi   \end{psmallmatrix}  ,  \begin{psmallmatrix} \varpi & \\[0.1em] & 1 \end{psmallmatrix} \right ) U    \, \,  | \, \, x \in \Oscr_{F}  \right \}  \\ 
H_{\varsigma_{3}}  \varpi^{\lambda_{3}} U / U  & =  \left \{     \left  (  \begin{psmallmatrix}  \varpi   &    \\[0.1em]   & 1   \end{psmallmatrix} , 
\begin{psmallmatrix} 
 \varpi     &  y   \\[0.1em] 
&    1   \end{psmallmatrix}  ,
\begin{psmallmatrix} 
 1  & \\[0.1em]
y &  \varpi 
\end{psmallmatrix} \right ) U  \, \,  | \, \,  y \in \Oscr_{F}  \right \} .   
\end{align*}  
So $ H_{\varsigma_{3}}\varpi^{\lambda_{i}} U / U$ is a singleton for $ i = 1 $ and a complete system of representatives for $ i = 2 $ (resp., $ i = 3 $) is given by  letting the parameter $ x $ (resp., $y $) run over $ [\kay] $. One then  easily finds that $ \xi_{\lambda_{1}} =  \phi_{(1,0) } $, $ \xi_{\lambda_{2}} =  \phi_{(2,0)} - \phi_{(2,1)} + \phi_{(3,1)} $ and $ \xi_{\lambda_{3}} =  q \phi_{(1,0)} $. \\

\noindent (c) Recall that $ \mathfrak{c}_{\varsigma_{3}} =  
(U  \varpi^{(2,1,1,1)} H_{\varsigma_{3}} ) + (U 
 \varpi^{(2,0,2,0)} H_{\varsigma_{3}}) $. Let $ \lambda_{1} = (0,0,0,0) $ and $ \lambda_{2} =  (2,2,0,2) $. Then $ \mathfrak{c}_{\varsigma_{3},*}(\phi) =    z_{0}\cdot \xi_{\lambda_{1}} + z_{0} ^{2}  \cdot \xi_{\lambda_{2}} $. Using Lemma \ref{varsig3struclemma},    we find that \begin{align*} H_{\varsigma_{3}} \varpi^{\lambda_{1}} U / U  & = \left \{ \left ( \begin{psmallmatrix} 1 \\[0.01em]  x / \varpi  & \, \, 1   \end{psmallmatrix} , 1 , 1 \right )  U \right \}   \\
H_{\varsigma_{3}} \varpi^{\lambda_{2}} U / U  & =   \left \{   \left  (  \begin{psmallmatrix}  \varpi^{2}   &  \\[0.1em]   & 1   \end{psmallmatrix} , \begin{psmallmatrix}  1    &   \\[0.1em]  y \varpi     &    \varpi  ^{2}   \end{psmallmatrix}  ,  \begin{psmallmatrix} \varpi ^{2} &  y \varpi  \\[0.1em] & 1 \end{psmallmatrix} \right ) U    \, \,  | \, \, x \in \Oscr_{F}  \right \}  
\end{align*} 
So a system for representative cosets for $ H_{\varsigma_{3}} \varpi^{\lambda_{1}} U/ U $ (resp., $ H_{\varsigma_{3}} \varpi^{\lambda_{2}}U/U $ is obtained by letting $ x $ (resp., $y $) run over $ [\kay]  $.  Using this, we compute that $ \xi_{\lambda_{1}}  = \phi_{(0,-1)} - \phi_{(1,-1)} + \phi_{(1,0)} $  and $ \xi_{\lambda_{2}} =   \phi_{(2,0)} $. \\

\noindent Finally, we have $ \mathfrak{h}_{\varsigma_{3}} =    ( 1 + 2 \rho^{2}  +  \rho^{4} )  \mathfrak{a}_{\varsigma_{3}}  - ( 1 + \rho^{2} ) \mathfrak{b}_{\varsigma_{3}}  + \mathfrak{c}_{\varsigma_{3}}   $, so 
\begin{align*}  \mathfrak{h}_{\varsigma_{3},  *}(\phi)   &   \equiv  ( 1 + z_{0})^{2} \left ( z_{0} \cdot \phi_{(2,0)}) \right )  - ( 1 +  z_{0} ) \left (  z_{0}^{2} + z_{0} ) \phi_{(2,0)}  + z_{0}  \cdot  \phi_{(1,0)}  \right )  +  (z_{0}^{2}+ z_{0})  \phi_{(1,0)}  \\
& =  \left ( z_{0} ( 1+ z_{0})^{2}  - (1+z_{0}) (z_{0}^{2} + z_{0} ) \right ) \phi_{(2,0)}  + \left  ( z_{0}^{2} + z_{0} - (1+z_{0}) z_{0}  \right )  \phi_{(1,0) }   \\
& = 0    \qedhere  
\end{align*} 
\end{proof}

\subsection{Convolutions with restrictions of $ \mathfrak{h}_{2}$}
In this subsection, we compute the convolution  $ \mathfrak{h}_{\vartheta,*}(\phi) $ for $ \vartheta \in \left \{ \vartheta_{0}, \vartheta_{1}, 
\vartheta_{2}, \vartheta_{3} \right \} $ $ \cup  \{  \tilde{\vartheta}_{k} \, | \, k \in  [\kay]^{\circ}  \} $. These matrices are as follows: 
$$ \vartheta_{0} =  
\scalebox{0.85}{$
\left(\begin{matrix}
\varpi \, \,  & & & & \mfrac{1}{\varpi}   \\   
&     \varpi \,  \, & & \mfrac{1}{\varpi}  \\[0.1em] 
& &  1 \, \,  &  \\   
& & &   \mfrac{1}{\varpi} \\
& & & &  \mfrac{1}{\varpi}      \\ 
& & & & & 1  
\end{matrix}\right)$}
, \quad  \vartheta_{1} =  \scalebox{0.85}{$\left(\begin{matrix}
\varpi \, \,  & & & & \mfrac{1}{\varpi}   \\ 
& &  1 \, \,  &  \\
&  \varpi \, \, & & \mfrac{1}{\varpi} \\[0.5em]  
& & &   \mfrac{1}{\varpi} \\
& & & & & 1  \\ 
& & & &  \mfrac{1}{\varpi}      \\ 
\end{matrix}\right)$}, \quad    \vartheta_{2} =   \scalebox{0.85}{$ \begin{pmatrix}  
\varpi  & &  &   &   \mfrac{1}{\varpi}   \\ 
& \varpi   &   &   \mfrac{1}{\varpi}   &  & 1  \\ & & 1  & & \mfrac{1}{\varpi} \\ & & &  \mfrac{1}{\varpi} \\ & & & &  \mfrac{1}{\varpi} \\ & & & &  & 1 
\end{pmatrix}$} 
$$
$$ \vartheta_{3} =    \scalebox{0.85}{$   \begin{pmatrix}  
\varpi  & &  &   &   \mfrac{1}{\varpi}   \\ 
& \varpi   &   &   \mfrac{1}{\varpi}   &   &  \mfrac{1}{\varpi}   \\[0.4em]  & & 1  & & \mfrac{1}{\varpi^{2}} \\[0.4em]  & & &  \mfrac{1}{\varpi} \\ & & & &  \mfrac{1}{\varpi} \\ & & & &  & 1 
\end{pmatrix}$} ,  \quad      \tilde{\vartheta}_{k} =    \scalebox{0.85}{$   \begin{pmatrix}  
\varpi  & &  &   &   \mfrac{1}{\varpi}   \\[0.5em]  
& k\varpi  & 1 &   \mfrac{k}{\varpi}   &   &   \\[0.5em]   
& (k+1)\varpi& 1&\mfrac{k+1}{\varpi} \\[0.5em]  & & &   \mfrac{1}{\varpi} &  \\ & & & & -   \mfrac{1}{\varpi}  & k + 1 \\[0.5em]   & & & & \, \, \, \,  \mfrac{1}{\varpi}     & -k
\end{pmatrix}$} $$
where $ k \in [\kay]^{\circ}  = [\kay] \setminus \left \{ -1 \right \}  $. Recall that $ H _{\vartheta} $ denotes the intersection $ H \cap \vartheta K \vartheta ^{-1} $.    
\begin{lemma}  \label{HvtisinU}   $ H_{\vartheta}$ is a subgroup of $ U $  for $ \vartheta \in \{ \vartheta_{0}, \vartheta_{1}, \vartheta_{2}, \tilde{\vartheta}_{k} \, | \, k \in [\kay]^{\circ} \} $. 
\end{lemma}
\begin{proof}  Since $ \theta =  \vartheta\tau_{2}^{-1} \in U $ and $ H_{\tau_{2}}' \subset U $ by Lemma \ref{structureofHtau2}, we see that $ H_{\vartheta} = H \cap  \theta H_{\tau_{2}}' \theta^{-1} \subset U $. 
\end{proof}

\begin{notation} Let $  \mathscr{X}_{\vartheta_{0}} \subset U $ denote the subgroup of all triples $ (h_{1}, h_{2}, h_{3} )$  where $ h_{2}=   \begin{psmallmatrix}  & 1 \\ 1 &  \end{psmallmatrix} h_{1} \begin{psmallmatrix}  & 1 \\ 1 &  \end{psmallmatrix} $.    

\end{notation}

\begin{lemma}  \label{vart0structure}    $ H_{\vartheta_{0}} $ equals the product $ \mathscr{X}_{\vartheta_{0}}  U_{\varpi^{2}} $. 
\end{lemma}   
\begin{proof} Let $ h \in U $ and write $ h $ as in Notation \ref{notationh}. Then $ h\in H_{\vartheta_{4}} $ if and only if $$ \vartheta_{0}^{-1} h \vartheta_{0} =  \begin{pmatrix} 
a & -c_1  &  & \mfrac{b - c_1 }{\varpi^2 } & \mfrac{a - d_1 }{\varpi^2 } &  \\[0.4em] 
-c & a_1  &   & \mfrac{a_1 - d}{\varpi^2 } & \mfrac{b_1  - c}{\varpi^2 } &  \\
  &   & a_2  &   &   & b_2 \\
c\,\varpi^2  &   &   & d & c &  \\
 & c_1 \,\varpi^2  &  & c_1  & d_1  &  \\
 & & c_2  &  &   & d_2  \end{pmatrix}  \in K. $$ 
It follows that $ \mathscr{X}_{\vartheta_{0}} $, $ U_{\varpi^{2} } $ are both contained in $ H_{\vartheta_{0}} $ and hence so is their product. If  $ h = (h_{1}, h_{2}, h_{3} )  \in H_{\vartheta_{0}} $ is arbitrary, let $ \gamma = ( h_{1}^{-1}, h_{2}' , h_{3}) $ where $ h_{2}' =     \begin{psmallmatrix}  & 1 \\ 1 &  \end{psmallmatrix} h_{1}^{-1} \begin{psmallmatrix}  & 1 \\ 1 &  \end{psmallmatrix} $. Then $ \gamma \in \mathscr{X}_{\vartheta_{0}} $ and $ \gamma h = ( 1, h_{2}'h_{2} , 1 )  \in H_{\vartheta_{0}} $ and it is easily seen from the matrix formula above (applied to $ \gamma h $ in place of $ h $) that $ \gamma h \in U_{\varpi^{2}} $.  Thus $ h = \gamma^{-1} \cdot \gamma h \in  \mathscr{X}_{\vartheta_{0}}  U_{ \varpi^{2}} $ which  establishes the reverse   inclusion.  
\end{proof}

\begin{proposition} \label{tfhvart0action} 
Modulo $ q - 1 $, we have

\begin{enumerate}[label = \normalfont(\alph*), itemsep= 0.3em, after = \vspace{0.1em} , before = \vspace{0.1em} ] 
\item $ [U H_{\vartheta_{0}}]_{*}(\phi) =  \phi $,

\item   $[U \varpi^{(3,2,1,2)}  H_{\vartheta_{0}}]_{*} (\phi) 
\equiv 2 (z_{0}^{2} + z_{0})   \phi   $ , 
\item  $[U \varpi^{(4,2,2,3)}  H_{\vartheta_{0}}]_{*} (\phi) \equiv 2 z_{0}^{2} \cdot  \phi  $,  
\item  $[U \varpi^{(4,3,1,2)}     H_{\vartheta_{0}}]_{*} (\phi) \equiv  (z_{0} ^{3} + z_{0} )  \cdot   \phi.  $  
\end{enumerate}

\noindent and $ \mathfrak{h}_{\vartheta_{0}, *}(\phi) = \mathfrak{h}_{\vartheta_{1},*}(\phi)  \equiv 0   $.  
\end{proposition}    
\begin{proof} Part (a) is clear since $ H_{\vartheta_{0}} \subset U $. Let $ \lambda \in \Lambda $ be such that $ \mathrm{dep}(\lambda) \leq 2 $. Then  Lemma \ref{vart0structure} implies that  $  H_{\vartheta_{0}} \varpi^{\lambda} U = \mathscr{X}_{\vartheta_{0}} \varpi^{\lambda} U $. Let us denote  $  \mathbf{P}  : = \GL_{2}(F) \times _{F^{\times}} \GL_{2}(F) $ and let  $ P $, $ P^{\circ} $ denote  the  groups of $ F $, $ \Oscr_{F} $-points of $ \mathbf{P} $ respectively.   Consider the embedding $$ \imath : \mathbf{P} \hookrightarrow \mathbf{H} , \quad \quad \quad  (h_{1}, h_{2}) \mapsto (h_{1}, \ess h_{1}  \ess, h_{2} ) $$ where $ \ess = \begin{psmallmatrix} & 1 \\ 1 & \end{psmallmatrix} $. Then $ \imath $ identifies  $ P^{\circ} $ with $ \mathscr{X}_{\vartheta_{0}} $. If $ \lambda = (a,b,c,d) $ satisfies $ b = a - c $, then we also have $ \varpi^{\lambda} \in \imath(P) $ and we write $ \varpi^{(a,b,d)} \in P $ for the pre-image.  Then  $$  P^{\circ} \varpi^{(a,b,d)} P^{\circ} / P^{\circ} \to  \mathscr{X}_{\vartheta_{0}} \varpi^{\lambda}   U / U , \quad  \quad \quad  \gamma P \mapsto \imath(\gamma) 
 U  $$ 
is a bijection.  It follows  $ \lambda = (a,b,c,d) $ satisfying $ b = a - c $ and with $ \mathrm{dep}(\lambda) \leq 2 $, we have   $$ [  H_{\vartheta_{0}} \varpi^{\lambda} U ]  (\phi)  =    \left |  U_{1} \backslash  U_{1} \varpi^{(a,d)} U_{1}  \right |  \cdot    T_{b,a-b,*}(\phi) .  $$ 
Parts (b), (c), (d)  are then easily obtained using  Corollary \ref{Tab*} and the formula above.   
Now recall that 
\begin{align*} \mathfrak{h}_{\vartheta_{0}}  &  =   \rho^{2}(1 + 2\rho^{2} +  \rho^{4})  (U H_{\vartheta_{0}} ) -  ( 1 + \rho^{2})     (U  \varpi^{(3,2,1,2)}  H_{\vartheta_{0}} )  +   ( U \varpi^{(4,2,2,3)} H_{\vartheta_{0}})  + ( U \varpi^{(4,3,1,2)}H_{\vartheta_{0}}) .
\end{align*} 
So putting everything together, we have \begin{align*} \mathfrak{h}_{\vartheta_{0},*}(\phi) & \equiv \left (   z_{0}(1+2z_{0}+z_{0}^{2}) - (1 + z_{0})(2z_{0}^{2}+ 2z_{0})   + 2z_{0}^{2} + ( z_{0}^{3} + z_{0})  \right ) \phi = 0   
\end{align*} 
Since $ \mathfrak{h}_{\vartheta_{1}} =  w_{2} \mathfrak{h}_{\vartheta_{0}} w_{2} $ and conjugation by $ w_{2} $ only swaps the second and third components of $ H $, we obtain the equality $ \mathfrak{h}_{\vartheta_{0}, * }(\phi) = \mathfrak{h}_{\vartheta_{1},*}(\phi) $.  
\end{proof}  

%We now  proceed to compute the convolution of $ \mathfrak{h}_{\vartheta_{2},*}(\phi) $.  
\begin{notation}  \label{strucvtheta2not}  Let $ A_{\vartheta_{2}} = A \cap \vartheta_{2} K \vartheta_{2}^{-1} $ and $ U_{\varpi^{2}} $ the subgroup of all elements in $ U $ that reduce to identity   modulo $ \varpi^{2}  $. We let $ M_{\vartheta_{2}}$ be the subgroup of all triples $ h = (h_{1},h_{2},h_{3}) \in U $ such that  $$   h_{1} = \begin{psmallmatrix} 1 \\  x & 1  \end{psmallmatrix}, \quad h_{2} = \begin{psmallmatrix} 1 & y  \\  &   1 \end{psmallmatrix} , \quad  h_{3}  =   \begin{psmallmatrix}  1 \\ y - x & 1  \end{psmallmatrix} $$ 
where $ x , y \in \Oscr_{F} $ satisfy $ x - y \in \varpi  \Oscr_{F} $. We define $ N_{\vartheta_{2}} $ to be the two parameter  subgroup of triples $  (h_{1}, h_{2},  h_{3}) $ given by 
  $$  h_{1} = \begin{psmallmatrix}  1  & x \varpi  \\[0.1em]  & 1    \end{psmallmatrix} , \quad    h_{2} = 
\begin{psmallmatrix}   1  \\  x  \varpi   &  1    \end{psmallmatrix},  \quad    h_{3} =   \begin{psmallmatrix}   1  &   y  \\   &  1    \end{psmallmatrix}  $$  
where $ x , y \in \Oscr_{F} $ are arbitrary. 
\end{notation}

\begin{lemma}  \label{vart2structure}      $ H_{\vartheta_{2}}  =M_{\vartheta_{2}} N_{\vartheta_{2}}  A _{\vartheta_{2}}  U_{\varpi^{2}} =  N_{\vartheta_{2}} M_{\vartheta_{2}} A _{\vartheta_{2}} U_{\varpi^{2}} $.  
\end{lemma}

\begin{proof} That $ M_{\vartheta_{2}} $, $ N_{\vartheta_{2}} , U_{\varpi^{2}} $ are subgroups of $ H_{\vartheta_{2}} $ is easily verified by checking that their conjugates by $ \vartheta_{2}^{-1} $ are in $ K $ , so $ H_{\vartheta_{2}} $ contains the product. If  $ h  \in H_{\vartheta_{2}} \subset U  $ is arbitrary, then  $$ \vartheta^{-1}_{2} h \vartheta_{2} =   \begin{pmatrix} 
a &  - c_{1}  & & \mfrac{b-c_1}{\varpi^2} &   \mfrac{a-d_1 }{\varpi^2} 
& -\mfrac{c_1 }{\varpi }\\[0.5em]
-c & a_1  & -\mfrac{c_2 }{\varpi } &   \mfrac{a_{1}-d }{\varpi^2}
& \mfrac{b_1 - c - c_2 }{\varpi^2 } & \mfrac{a_1 - d_2 }{\varpi }
\\[0.5em]
 & -   c_{1} \varpi   & a_2  &  - \mfrac{c_{1}} { \varpi } &  \mfrac{a_2   - d_1 }{\varpi } 
 &  b_2 -c_1
 \\[0.5em] 
c\,\varpi^2  
&  &  & d & c &  \\
 &  c_1 \varpi^2  
 &  & c_1  & d_1  &  c_{1} \varpi  \\[0.5em]  
 &   &  c_{2}   &   &  \mfrac{c_{2}}{\varpi} & d_2 
 \end{pmatrix}   \in   K.  $$
From the matrix, we see that $ b $, $ c_{1}$, $ c_{2}$, $ b_{1} - c$, $ a - d_{1}    \in \varpi \Oscr_{F}  $. In particular, $ a, a_{1}, a_{2},  d , d_{1},  d_{2}   \in \Oscr_{F}^{\times} $. Let $ m \in M_{\vartheta_{2}} $  be defined with $ x = -c/a $, $ y = -b_{1}/d_{1} $ (see Notation \ref{strucvtheta2not}). Then $ h' = mh $ satisfies $  b_{1}' =  c' = 0 $. Then $ c_{2}' \in \varpi^{2} \Oscr_{F} $. If we define $ n \in N_{\vartheta_{2}} $ 
with $ x = -b'/d'\varpi $, $ y = -b'_{2}/d' $, we find that $ h'' $ satisfies  $ b_{1}'' =  c'' = 0 $ (inherited from $ h' $) and  $ b'' = b_{2}'' = 0 $. The latter condition   forces  $ c_1'' \in \varpi^{2} \Oscr_{F} $. Now $ h'' $  clearly lies in the product $ A_{\vartheta_{2} } U_{\varpi^{2} } $ which proves the first equality. The second follows similarly by first using  $  N_{\vartheta_{2}} $ to make the entries $ b $, $ b_{2} $ in $ h $ zero.      
\end{proof}

\begin{proposition}  Modulo $ q - 1 $,  we have

\begin{multicols}{2} \begin{enumerate}[label = \normalfont(\alph*), itemsep = 0.3em] 
\item $ [ U H   _ { \vartheta_{2} } ] _{*}( \phi ) = \phi $, \vspace{0.2em} 
\item   $[U \varpi^{(3,2,1,2)}  H_{\vartheta_{2}}]_{*} (\phi) \equiv      (z_{0}^{2} + z_{0} )   \phi    -  \bar{\phi}_{(1,2)}   $, 
\item  $[U \varpi^{(3,1,2,1)}  H_{\vartheta_{2}}]_{*} (\phi) \equiv \bar{\phi}_{(1, 2 ) } $,  
\item  $[U \varpi^{(3,1,2,2)}     H_{\vartheta_{2}}]_{*} (\phi)  =  \bar{\phi} _ { ( 1, 2 ) }  $,  
\item  $[U \varpi^{(4,2,2,3)}     H_{\vartheta_{2}}]_{*} (\phi) \equiv    z_{2}^{2} \cdot \phi  $,
\item  $[U \varpi^{(4,1,3,2)}     H_{\vartheta_{2}}]_{*} (\phi) \equiv    ( z_{0} + 1 )\cdot \phi_{(1, 2) } -  z_{0}^{2}   \cdot   \phi    $ 
\end{enumerate}  
\end{multicols}

\noindent and $ \mathfrak{h}_{\vartheta_{2}, *} (\phi)  = \mathfrak{h}_{\tilde{\vartheta}_{0}, *} ( \phi )   \equiv  0 $. 
\end{proposition}    
 
\begin{proof} Part (a) is immediate since $ H_{\vartheta_{2}}  \subset U $.  For $ \lambda \in \Lambda $, let $ \xi_{\lambda} = [H_{\vartheta_{2}} \varpi^{\lambda} U](\phi) $. If $ \lambda $ depth at most $ 2 $, then  $ H _{\vartheta_{2}} \varpi^{\lambda} U / U = M_{\vartheta_{2}} N_{\vartheta_{2}} \varpi^{\lambda} U/  U $  by Lemma  \ref{vart2structure}. If moreover $ \lambda $  has depth one and  $ \beta_{2}(\lambda) \leq 0 $, then we also have  $ M_{\vartheta_{2}} N_{\vartheta_{2}} \varpi^{\lambda} U = M _{\vartheta_{2}} \varpi^{\lambda} U $. Similarly if $ \alpha_{0}(\lambda), \beta_{2}(\lambda) \geq 0 $ and $ \beta_{0}(\lambda) \leq 0 $,  then   $ H_{\vartheta_{2}} \varpi^{\lambda} U / U = N_{\vartheta_{2}} \varpi^{\lambda} U / U $. \\ 

\noindent (b) We need to compute  $ z_{0}^{2} \cdot \xi_{\lambda} $ where $ \lambda = (1,0,1,0) $. Then $ \mathrm{dep}(\lambda) = 1 $ and $ \beta_{2}(\lambda) = -1 $, so   $ H_{\vartheta_{2}}  \varpi^{\lambda} U/U = M_{\vartheta_{2}} \varpi^{\lambda} U /  U  $.  It is then easily seen that the quotient $ M_{\vartheta_{2}} / M_{\vartheta_{2}} \cap \varpi^{\lambda} U \varpi^{-\lambda} $ has cardinality $ q $ with representatives given by elements with parameters $ x = y $ running over $ [\kay] $ (see Notation \ref{strucvtheta2not}). From this,  one finds that  $ \xi_{\lambda} = \phi - \phi_{(1,0)} +  q  \phi_{(1,1)} $.    \\

\noindent  (c) We need to compute $ z_{0}^{2} \cdot \xi_{\lambda} $ where  $ \lambda = (1,1,0,1) $.  Here $ \alpha_{0}(\lambda) = \beta_{2}(\lambda) = 1 $ and $ \beta_{0}(\lambda) = - 1 $, so $ H_{\vartheta_{2}} \varpi^{\lambda} U / U = N_{\vartheta_{2}} \varpi^{ \lambda} U /  U   $. This  coset space has  cardinality $ q $ and a set of representatives is $ \gamma \varpi^{\lambda} $ where $ \gamma \in  N_{\vartheta_{2} } $ runs over elements defined with $ x = 0 $ and $ y \in [\kay] $ (see Notation \ref{strucvtheta2not}). So $ \xi_{\lambda} = q \varpi^{\lambda}  \cdot  \phi =     q  \phi_{(1,0)} $.   \\

\noindent (d) If $ \lambda = -(3,1,2,2) $, then $ H_{\vartheta_{2}}  \varpi^{\lambda} U / U = M_{\vartheta_{2}} \varpi^{\lambda} U / U $ as in part (b) and its easy to see that this equals $ \varpi^{\lambda} U/U $. So $ \xi_{\lambda} = \varpi^{\lambda} \cdot \phi = \bar{\phi}_{(1,2)}$.   \\

\noindent (e)  We need to compute $ z_{0}^{3} \cdot \xi_{\lambda} $ where $ \lambda =  (2,1,1,0) $.  As the first and second  components of $ \varpi^{\lambda} $ are central and $ \beta_{2}(\lambda) = - 2 < 0 $, we see that $ H_{\vartheta_{2}}  \varpi^{\lambda} U / U =  M_{\vartheta_{2}} \varpi^{\lambda} U / U $. From the structure of $ M $, we see that a set of 
 representatives is given by $ \gamma \varpi^{\lambda} $ where $ \gamma =  \left (1,1, \begin{psmallmatrix} 1 \\ \varpi z & 1 \end{psmallmatrix} \right )  $ and $ z $ running over $ [\kay] $. So $ \xi_{\lambda} = q \varpi^{\lambda}  \cdot \phi = z_{0}^{-1} \cdot \phi $.  \\

\noindent (f) This equals $ z_{0}^{3} \cdot  \xi_{\lambda} $ where $ \lambda = ( 2,2,0,1) $.   Then  $ H _{\vartheta_{2}} \varpi^{\lambda} U / U = N_{\vartheta_{2}} \varpi^{\lambda} U / U $.  A set of representatives for this quotient is $ \gamma \varpi^{\lambda} $ where $ \gamma $ runs over elements of $ N_{\vartheta_{2}} $  defined with $ y = 0 $ and $ x \in [\kay]  $. From this, one calculates that $ \xi_{\lambda} $ vanishes on $ ( X \setminus X_{1,0} ) \cup ( X_{1,1}  \setminus  X_{2,1}) $, takes value one on $ X_{1,0} \setminus X_{1,1} $ and $ q $ on $ X_{2,1} $. So $ \xi_{\lambda} = \phi_{(1,0)} - \phi_{(1,1)}   +  q \phi_{(2,1)} $ and $ z_{0}^{3}  \cdot  \xi_{\lambda} = \bar{\phi}_{(2,3)} - z_{0}^{2}\phi + q \bar{\phi}_{(1,2)}   $.  \\ 

\noindent  Now recall that \begin{align*}  \mathfrak{h}_{\vartheta_{2}}  & = \rho^{2}( 1  + 2 \rho^{2} +  \rho^{4} )  (  U H_{\vartheta_{2}  }  )   - ( 1 + \rho^{2} ) \Big(     (  U    \varpi^{(3,2,1,2)}   H _{\vartheta_{2}  }  )  +   ( U   \varpi^{(3,1,2,1)}   H _{\vartheta_{2}  }  )  +    ( U    \varpi^{(3,1,2,2)}   H _{\vartheta_{2}  }  ) \Big )  \\
 &  +  ( U \varpi^{(4,2,2,3)}  H_{\vartheta_{2} } )   + ( U \varpi^{(4,1,3,2)}  H _ { \vartheta_{2} }  )  
 \end{align*}
By  parts (a)-(f), we see that 
\begin{align*} \mathfrak{h}_{\vartheta_{2},*}(\phi) &  \equiv   z_{0}(1+z_{0})^{2} \cdot \phi - (1+z_{0}) \Big (  (z_{0}^{2} + z_{0}) \cdot \phi - \bar{\phi}_{(1,2)}  + \bar{\phi}_{(1,2)}  +  \bar{\phi}_{(1,2)}  \Big  ) +     \\
&  \quad  \,  \,  z_{0}^{2} \cdot \phi + ( z_{0}+1) \cdot \phi_{(1,2)} - z_{0}^{2} \cdot \phi  \\
&  =  \Big (  z_{0}(1 + z_{0})^{2} - ( 1 + z_{0}) (z_{0}^{2} + z_{0} )  \Big ) \cdot \phi   - ( 1 + z_{0}) \bar{\phi}_{(1,2)} + (1 + z_{0}) \cdot \bar{\phi}_{(1,2)}   \\
&  = 0 
\end{align*}
modulo $ q -  1    $.  Since $ \mathfrak{h}_{\tilde{\vartheta}_{0}} $ is the conjugate of $ \mathfrak{h}_{\vartheta_{2}} $ by $ w_{2} w_{3} $ and this only affects the second and third components of $ H $, we see that $ \mathfrak{h}_{\tilde{\vartheta}_{0}} (\phi) = \mathfrak{h}_{\vartheta_{2}} (\phi) $. This completes the proof.  
\end{proof}

\begin{lemma}      \label{vart3struclemma}   Let $ I_{\vartheta_{3}} \subset U $ denote the subgroup of triples $ (h_{1}, h_{2}, h_{3} ) $ such that modulo $ \varpi^{2}$, $ h_{1} $ reduces  to   a  lower triangular matrix and $ h_{2} $, $ h_{3} $ reduce to upper triangular matrices. Then $ H_{\vartheta_{3}}  \subset  I_{\vartheta_{3}} $.  
\end{lemma} 
\begin{proof} Write $ h \in H_{\vartheta_{3}} $ as in Notation \ref{notationh}. Then $$  \vartheta_{3} ^{-1} h\vartheta_{3}  =      
\begin{pmatrix} 
a&*& &-\mfrac{b - c_{1}}{\varpi^{2}} 
&*&-\mfrac{c_{1}}{\varpi^{2}} \\[0.5em] 
*&a_{1}&-\mfrac{c_{2}}{\varpi^{2}}&*&\mfrac{b_{1}- c}{\varpi^{2}} - \mfrac{c_{2}}{\varpi^{4}}  &  *  
\\[0.5em]   &  * & a_{2}  & *  & * &    b_{2} -  \mfrac{c_{1}}{\varpi^{2}}  \\[0.5em]  
* & & & d  &  c &   \\[0.5em]
&  * &  & *  & d_{1} &  *   \\[0.5em]  
& & *  
& &  * 
&d_{2} 
\end{pmatrix} \in K  $$
Since all entries of this matrix must be integral, it is easily seen that  $ h \in U $ and  that  $  b , c_{1}, c_{2} \in \varpi^{2} \Oscr_{F} $. 
\end{proof}

\begin{proposition}  \label{vart3action}        We have 
\begin{enumerate}[label = \normalfont(\alph*),  itemsep = 0.3em, after = \vspace{0.3em}, before = \vspace{0.3em}  ] 
\item $[U \varpi^{(3,1,2,2)} H_{\vartheta_{3}}]_{*} (\phi)  = \bar{\phi}_{(1,2)}$ 
\item   $[U \varpi^{(4,2,2,3)}  H_{\vartheta_{3}}]_{*} (\phi)  = \bar{\phi}_{(2, 2) } $  
\item   $[U \varpi^{(4,1,3,2)}  H_{\vartheta_{3}}]_{*} (\phi)  =  \bar{\phi}_{(1, 3)} $  
\end{enumerate}    
and $ \mathfrak{h}_{\vartheta_{3}, *} ( \phi )  = -  \ch \begin{psmallmatrix} \varpi ^{-1}  \Oscr_{F}^{\times} \\  \varpi^{-2} \Oscr_{F}^{\times}  \end{psmallmatrix} $. 
\end{proposition}  
\begin{proof} For $ \lambda \in \Lambda $, let $ \xi_{\lambda} $ denote $ [ H  _   { \vartheta_{3}} \varpi^{\lambda} U ]  ( \phi ) $.  If  each of $ \alpha_{0}(\lambda) , -\beta_{0}(\lambda) , -\beta_{2}(\lambda)$ lies in $  \left  \{ 0 ,1 ,2 \right \} $, then $ \varpi^{-\lambda} I_{\vartheta_{3}}  \varpi^{\lambda}  \subset U $.  So for such $ \lambda  $,  $  H _ {\vartheta_{3}    }  \varpi^{\lambda} U  = \varpi^{\lambda} U $ and so $ \xi_{\lambda} = \varpi^{\lambda} \cdot \phi $.  Parts (a), (b), (c) then follow immediately.  Now recall that 
$$ \mathfrak{h}_{\vartheta_{3}}  =  - ( 1 + \rho^{2} ) (U \varpi^{(3,1,2,2)} H_{\vartheta_{3}} )  +    ( U \varpi^{(4,2,2,3)}   H _{\vartheta_{3}} )   +  (U \varpi^{(4,1,3,2)}   H _{\vartheta_{3}} )  . $$ 
Using parts (a)-(c), we find that   
Therefore   \begin{align*} 
\mathfrak{h}_{\vartheta_{3}, *} ( \phi )  
 & = - ( 1 + z_{0}  ) \bar{\phi} _{(1,2)}  
 + \bar{\phi}_{(2,2)}  + \bar{\phi}_{(1,3)}   \\
 & = \bar{\phi}_{(2,2)} - \bar{\phi}_{(1,2)}  + \bar{\phi}_{(1,3)} - \bar{\phi}_{(2,3)}    \\ 
 & =   -  \ch \begin{psmallmatrix} \varpi ^{-1}  \Oscr_{F}^{\times} \\  \varpi^{-2} \Oscr_{F} \end{psmallmatrix}  +  \ch \begin{psmallmatrix} \varpi ^{-1}  \Oscr_{F}^{\times} \\  \varpi^{-3} \Oscr_{F}  \end{psmallmatrix} \\ &   =  -  \ch \begin{psmallmatrix} \varpi ^{-1}  \Oscr_{F}^{\times} \\  \varpi^{-2} \Oscr_{F}^{\times}   \end{psmallmatrix}  \qedhere 
 \end{align*}  
\end{proof}    
\begin{notation} For $ k \in [\kay]  \setminus \left \{ 0 , -  1 \right \}  $, let $  \tilde{\mathscr{X}}_{k} \subset U $ denote the subgroup of all triples $ (h_{1}, h_{2}, h_{3} )$ where $$ h_{1} = \begin{psmallmatrix} a &  b \\ c & d \end{psmallmatrix} , \quad h_{2} = \begin{psmallmatrix} d &    - c k  \\  -  b / k  & a \end{psmallmatrix} , \quad  h_{3} =  \begin{psmallmatrix} d &  c ( k + 1 ) \\ b / ( k +  1 )  & a  \end{psmallmatrix} .  $$
That is, $ h_{1} \in \GL_{2}(\Oscr_{F}) $ is arbitrary and $ h_{2}, h_{3} $ are certain conjugates of $ h_{1} $ by anti-diagonal matrices.    Recall that $ U_{\varpi } $ denotes the subgroup of $ U $ which reduces to the trivial group  modulo $ \varpi $.  
\end{notation}    
\begin{lemma} For $ k \in [\kay] \setminus \left \{ 0 , -1 \right  \} $, $ H_{\tilde{\vartheta}_{k}} $ is   equal to  the product of $ \tilde{\mathscr{X}}_{k} $ with $ U _{\varpi }  \cap H_{\tilde{\vartheta}_{k}} $.   
\end{lemma} 

\begin{proof} 
It is  straightforward to verify that $ \tilde{\mathscr{X}}_{k} \subset H_{\tilde{\vartheta}_{k}} $ by checking that the matrix $ \tilde{\vartheta}_{k}^{-1} \tilde{\mathscr{X}}_{k} \tilde{\vartheta}_{k} $ has all its entries integral. This implies that the reduction of $  H_{\tilde{\vartheta}_{k}} $ modulo $ \varpi $ contains the reduction of  $ \tilde{\mathscr{X}}_{k} $ modulo $ \varpi $. Thus  $ H_{\tilde{\vartheta}_{k}} $ contains the product $ \tilde{\mathcal{X}}_{k}  \cdot   ( U _{\varpi} \cap H_{\tilde{\vartheta}_{k}}) $.   For the reverse inclusion,  write $ h \in H_{\tilde{\vartheta}_{k}} $ as in Notation \ref{notationh}. 
Then $$ \tilde{\vartheta}_{k}^{-1}  h \tilde{\vartheta}_{k}  =    
\begin{pmatrix}  a & * & *  &  \mfrac{b - c_{1} k^{2} - c_{2}(k+1)^{2}}{\varpi^{2}} &  \mfrac{ a + d_{1} k - d_{2}   ( k+1)  } { \varpi^{2} }  &  * \\[0.5em] 
-c & * & \mfrac{a_{2}-a_{1}}{\varpi}
&     \mfrac{ a_{2} ( k + 1 ) - a_{1} k  -d } { \varpi^{2}}     &   
\mfrac{ b_{1} + b_{2} - c }  { \varpi }   &   -  \mfrac{ b_{2} k  + b_{1} ( k+1)   } { \varpi  }     \\[0.5em]  
&   *  &  * &  *  & *  &   *      \\ 
* &  &  & d & c & &    \\  
& * & * & * & * & *    \\   
 & *   & *  &  \mfrac{c_{1}k + c_{2} (k+1)}{
\varpi}  &  \mfrac{d_{2} - d_{1}}{\varpi} & *    \end{pmatrix}  \in K   $$  
As the displayed entries must be integral (and the entries of $ h $ are also integral by Lemma \ref{HvtisinU}), one easily deduces all the congruence conditions on entries of $ h $ for its reduction to lie in the reduction of $ \tilde{\mathscr{X}}_{k} $.  For instance, we have $ b_{2} k \equiv -b_{1} (k+1) $ and $ b_{1} + b_{2} \equiv c $ modulo $ \varpi $, which implies  that $ b_{1} \equiv - ck   $. 
\end{proof}    

\begin{proposition} Modulo $ q - 1$,  
\begin{enumerate} 
[label = \normalfont(\alph*),  itemsep = 0.3em, after = \vspace{0.3em}, before = \vspace{0.3em}  ]  
\item  $  [U H_{ \tilde{ \vartheta} _{k}  } ]_{*}(\phi ) =  \phi $,  
\item  $ [ U \varpi^{(3,2,1,1)} H_{\tilde{\vartheta}_{k}}]_{*}(\phi)  \equiv (z_{0}^{2} + z_{0} ) \cdot \phi     $ 
\end{enumerate} and 
$ \mathfrak{h}_{\tilde{\vartheta}_{k},*}(\phi)  \equiv 0  $ for all $ k \in [\kay] \setminus  \left \{ 0 , -1  \right \} $.  
\end{proposition}   

\begin{proof} Part (a) is trivial since $ H_{\tilde{\vartheta_{k}}} \subset U $. For part (b), let $ \lambda = -(3,2,1,1) $. Then $ \mathrm{dep}(\lambda) = 1 $ and so $ H_{\vartheta_{2}} \varpi^{\lambda} U = \tilde{\mathscr{X}}_{k} \varpi^{\lambda} U $. An   argument analogous to Proposition  \ref{tfhvart0action} shows that there is a bijection  $$ U_{1} \varpi^{-(3,2) }  U_{1}  /  U_{1}  \to   \tilde{\mathscr{X}}_{k} \varpi^{\lambda} U / U  $$  
(where $  U_{1} = \GL_{2}(\Oscr_{F}) $)  using which one obtains the equality $  [H_{\tilde{\vartheta}_{k} }\varpi^{\lambda} U](\phi)  =  \mathcal{T}_{2,1,*}(\phi) $. Corollary  \ref{Tab*} then implies the claim. Now recall that 
$$ 
\mathfrak{h}_{\tilde{\vartheta}_{k}}       =  \rho^{2} (1 + 2\rho^{2} + \rho^{4} ) ( U H_{\tilde{\vartheta}_{k}}  ) - ( 1 + \rho^{2} )   (   U \varpi^{(3,2,1,1)} H_{\tilde{\vartheta}_{k}})  . $$
So   
$ \mathfrak{h}_{\tilde{\vartheta}_{k},*}(\phi) \equiv \big (   z_{0}(1+z_{0})^{2} \phi - (1+z_{0}) ( z_{0}^{2}+ z_{0}) \big )  \phi   = 0  $. 
\end{proof}    

%% file: main.bbl
\providecommand{\bysame}{\leavevmode\hbox to3em{\hrulefill}\thinspace}
\providecommand{\MR}{\relax\ifhmode\unskip\space\fi MR }
\providecommand{\MRhref}[2]{%
  \href{http://www.ams.org/mathscinet-getitem?mr=#1}{#2}
}
\providecommand{\href}[2]{#2}
\begin{thebibliography}{BGCLRJ23}

\bibitem[AS01]{Asgari}
Mahdi Asgari and Ralf Schmidt, \emph{\href{https://doi.org/10.1007/PL00005869}{Siegel modular forms and representations}}, Manuscripta Math. \textbf{104} (2001), no.~2, 173--200. \MR{1821182}

\bibitem[BG14]{Buzz}
Kevin Buzzard and Toby Gee, \emph{\href{https://doi.org/10.1017/CBO9781107446335.006}{The conjectural connections between automorphic representations and {G}alois representations}}, Automorphic forms and {G}alois representations. {V}ol. 1, London Math. Soc. Lecture Note Ser., vol. 414, Cambridge Univ. Press, Cambridge, 2014, pp.~135--187. \MR{3444225}

\bibitem[BGCLRJ23]{gclj}
José~Ignacio Burgos~Gil, Antonio Cauchi, Francesco Lemma, and Joaquín Rodrigues~Jacinto, \emph{\href{https://arxiv.org/abs/2204.05163}{Tempered currents and {D}eligne cohomology of {S}himura varieties, with an application to $\mathrm{GSp}_6$}}, 2023, to appear in Cambridge Journal of Mathematics.

\bibitem[CGS]{EulerGU22}
Antonio Cauchi, Andrew Graham, and Syed Waqar~Ali Shah, \emph{Euler systems for exterior square motives}, in preparation.

\bibitem[CRJ20]{AJ}
Antonio Cauchi and Joaquín Rodrigues~Jacinto, \emph{\href{https://doi.org/10.25537/dm.2020v25.911-954}{Norm-compatible systems of {G}alois cohomology classes for {${\bf GSp}_6$}}}, Doc. Math. \textbf{25} (2020), 911--954. \MR{4151876}

\bibitem[CRJS]{EulerG2}
Antonio Cauchi, Joaquín Rodrigues~Jacinto, and Syed Waqar~Ali Shah, \emph{Euler systems for motives of {G}alois group of type {${G}_2$}}, in preparation.

\bibitem[Dis23]{Dis}
Daniel Disegni, \emph{{E}uler systems for conjugate-symplectic motives}, 2023, available at \url{https://disegni-daniel.perso.math.cnrs.fr/euler-theta.pdf}.

\bibitem[GS23]{GS}
Andrew Graham and Syed Waqar~Ali Shah, \emph{\href{ https://doi.org/10.1112/plms.12566}{Anticyclotomic {E}uler systems for unitary groups}}, Proc. Lond. Math. Soc. (3) \textbf{127} (2023), no.~6, 1577--1680. \MR{4673434}

\bibitem[HJS20]{hsu2020euler}
Chi-Yun Hsu, Zhaorong Jin, and Ryotaro Sakamoto, \emph{\href{https://arxiv.org/abs/2011.12894}{Euler Systems for $\mathrm{GSp}_4 \times \mathrm{GL}_2$}}, 2020.

\bibitem[KS23]{Kret}
Arno Kret and Sug~Woo Shin, \emph{\href{https://doi.org/10.4171/jems/1179}{Galois representations for general symplectic groups}}, J. Eur. Math. Soc. (JEMS) \textbf{25} (2023), no.~1, 75--152. \MR{4556781}

\bibitem[Lan01]{Lansky}
Joshua~M. Lansky, \emph{\href{https://doi.org/10.2140/pjm.2001.197.97}{Decomposition of double cosets in {$p$}-adic groups}}, Pacific J. Math. \textbf{197} (2001), no.~1, 97--117. \MR{1810210}

\bibitem[Loe21]{loe}
David Loeffler, \emph{\href{https://doi.org/10.1007/s10884-020-09844-5}{Spherical varieties and norm relations in {I}wasawa theory}}, J. Th\'{e}or. Nombres Bordeaux \textbf{33} (2021), no.~3, part 2, 1021--1043. \MR{4402388}

\bibitem[LSZ22a]{gu21}
David Loeffler, Christopher Skinner, and Sarah~Livia Zerbes, \emph{\href{https://doi.org/10.1007/s00208-021-02224-4}{An {E}uler system for {$\rm GU(2,1)$}}}, Math. Ann. \textbf{382} (2022), no.~3-4, 1091--1141. \MR{4403219}

\bibitem[LSZ22b]{LSZ}
\bysame, \emph{\href{https://doi.org/10.4171/jems/1124}{Euler systems for {$\rm GSp(4)$}}}, J. Eur. Math. Soc. (JEMS) \textbf{24} (2022), no.~2, 669--733. \MR{4382481}

\bibitem[PS18]{PS}
Aaron Pollack and Shrenik Shah, \emph{\href{https://doi.org/10.1353/ajm.2018.0018}{The spin {$L$}-function on {$\rm GSp_6$} via a non-unique model}}, Amer. J. Math. \textbf{140} (2018), no.~3, 753--788. \MR{3805018}

\bibitem[Sha22]{AES}
Syed Waqar~Ali Shah, \emph{\href{https://dash.harvard.edu/handle/1/37372227?show=full}{On an approach to automorphic {E}uler systems}}, ProQuest LLC, Ann Arbor, MI, 2022, Thesis (Ph.D.)--Harvard University. \MR{4464230}

\bibitem[Sha23a]{distpolylog}
\bysame, \emph{\href{https://arxiv.org/abs/2309.10938}{On distribution relations of polylogarithmic Eisenstein classes}}, 2023.

\bibitem[Sha23b]{CZE}
\bysame, \emph{On constructing zeta elements for {S}himura varieties}, 2023.

\bibitem[Sha24]{EulerGSp6}
\bysame, \emph{Euler systems for motives of {S}iegel modular sixfolds}, 2024, in preparation.

\bibitem[Wei09]{Endoscopy}
Rainer Weissauer, \emph{\href{https://doi.org/10.1007/978-3-540-89306-6}{Endoscopy for {${\rm GSp}(4)$} and the cohomology of {S}iegel modular threefolds}}, Lecture Notes in Mathematics, vol. 1968, Springer-Verlag, Berlin, 2009. \MR{2498783}

\end{thebibliography}
